\documentclass[
a4paper, 
fontsize=12pt, 
twoside, 
openright, 
numbers=noenddot, 
]{scrreprt}

\usepackage[utf8]{inputenc}        
\usepackage[T1]{fontenc}             % f�r "modernere" Schriftkodierung
\usepackage[english]{babel}          % Sprachanpassungen f�r Deutsch (neue Rechtsschreibung)
\usepackage{lmodern}
\usepackage{csquotes} 
\usepackage[intlimits]{amsmath}      % f�r erweiterte mathematische Konstrukte
\usepackage{amsfonts}                % f�r erweiterte mathematische Schriften
\usepackage{amssymb}                 % f�r erweiterten Symbolvorrat
\usepackage{amsthm}
\usepackage{array}                   % f�r erweiterte Tabellensatzm�glichkeiten
\usepackage{enumerate}               % f�r komfortablere Aufz�hlungen
\usepackage{latexsym}
\usepackage{graphicx}
\usepackage{multirow}
\usepackage{mathtools}
\usepackage[plainpages=false,pdfpagelabels,bookmarks=true]{hyperref}
\usepackage[a4paper,left=3cm,right=3cm,top=3cm,bottom=3.5cm,includeheadfoot, headsep=1cm, footskip=2cm,bindingoffset=0mm]{geometry}	
\usepackage{faktor}
\usepackage{paralist}
\usepackage{mathrsfs}
\usepackage{lastpage}
\usepackage{cleveref}
\usepackage{pdfpages}
\usepackage{blindtext}
\usepackage{tcolorbox}
\usepackage{enumitem}
\usepackage{tikz}
\usepackage{youngtab}
\usepackage{tkz-euclide}
\usepackage{tabularx}
\usepackage{longtable}
\usepackage{bbm}

\usetikzlibrary{cd, matrix,arrows,decorations.pathmorphing, babel, positioning, calc}

\setlist[enumerate, 1]{nosep, label=(\arabic*)}
\setlist[enumerate, 2]{nosep, label=(\roman*)}
\setlist[itemize, 1]{nosep, label={\bfseries $\bullet$}}
\setlist[itemize, 2]{nosep, label=$\circ$} 

\hypersetup{
	pdftitle={Random matrices}, 
	pdfauthor={Ricardo Schnur},
	colorlinks,
	linkcolor=blue,
	citecolor=blue,
	urlcolor=blue,
	linktoc=all,
}

\allowdisplaybreaks[3]				% Zeilenumbruch in align etc [x], x=1,...,4

\addtokomafont{disposition}{\rmfamily}

      \theoremstyle{plain}
      
      \newtheorem{theorem}{Theorem}
      \newtheorem*{theorem*}{Theorem}
      \newtheorem{lemma}[theorem]{Lemma}
      \newtheorem*{lemma*}{Lemma}
      \newtheorem{cor}[theorem]{Corollary}
      \newtheorem*{cor*}{Corollary}
      \newtheorem{prop}[theorem]{Proposition}
      \newtheorem*{prop*}{Proposition}

      \theoremstyle{definition}
      
      \newtheorem{definition}[theorem]{Definition}
      \newtheorem*{definition*}{Definition}
      \newtheorem{notation}[theorem]{Notation}

 \newtheorem{exercise}{Exercise}    
\newtheorem*{exercise*}{Exercise}

      \theoremstyle{remark}
      
      \newtheorem{remark}[theorem]{Remark}
      \newtheorem*{remark*}{Remark}     
      \newtheorem{example}[theorem]{Example}
      \newtheorem*{example*}{Example}    
      \newtheorem*{question*}{Question}    
      \newtheorem*{fact*}{Fact}

\newcommand{\N}{\mathbb{N}}
\newcommand{\Z}{\mathbb{Z}}

\newcommand{\R}{\mathbb{R}}
\newcommand{\C}{\mathbb{C}}

\let\P\relax
\newcommand{\P}{\mathbb{P}}

\newcommand{\G}{\mathcal{G}}

\newcommand{\A}{\mathcal{A}}

\newcommand{\PP}{\mathcal{P}}

\newfont{\suet}{suet14}
\DeclareTextFontCommand{\scrpt}{\suet}

\newcommand{\pair}{\mathcal{P}_2}
\newcommand{\ncpair}{\mathcal{NC}_2}

\newcommand{\ee}{\varepsilon}
\newcommand{\tov}{\overset{v}{\to}}
\newcommand{\tow}{\overset{w}{\to}}
\newcommand{\Nto}{\overset{N\to\infty}{\longrightarrow}}
\newcommand{\GUE}{\hbox{\textsc{gue}}}
\newcommand{\GUEN}{\hbox{\textsc{gue(n)}}}
\newcommand{\GOE}{\hbox{\textsc{goe}}}
\newcommand{\GSE}{\hbox{\textsc{gse}}}
\newcommand{\GOEN}{\hbox{\textsc{goe(n)}}}
\newcommand{\ff}{\varphi}

\newcommand{\abs}[1]{\left|{#1}\right|}
\newcommand{\norm}[1]{\left\|{#1}\right\|}

\newcommand{\floor}[1]{\left\lfloor {#1} \right\rfloor}

\newcommand{\ev}[1]{\mathbb{E}\left[{#1}\right]}
\newcommand{\evpartition}[2]{\mathbb{E}_{#1}\left[{#2}\right]}
\newcommand{\prob}[1]{\mathbb{P}\left[{#1}\right]}
\newcommand{\var}[1]{\mathbb{V}\text{ar}\left[{#1}\right]}
\newcommand{\vari}[1]{\mathbb{V}\text{ar}^{(i)}\left[{#1}\right]}

\let\Re\relax
\let\Im\relax
\DeclareMathOperator{\Re}{Re}
\DeclareMathOperator{\Im}{Im}

\DeclareMathOperator{\sgn}{sgn}

\DeclareMathOperator{\tr}{tr}
\DeclareMathOperator{\diag}{diag}
\DeclareMathOperator{\id}{id}

\DeclareMathOperator{\Tr}{Tr}

\DeclareMathOperator{\vol}{vol}

\DeclareMathOperator{\Ai}{Ai}
\DeclareMathOperator{\Tab}{Tab}

\newcommand{\td}{\,\mathrm{d}}
\newcommand{\bs}{\backslash}

\begin{document}

%\pagestyle{fancy}
%\cfoot{\thepage \ / \pageref{LastPage}}
%\lhead{Summer 18}
%\chead{Random matrices}

\includepdf{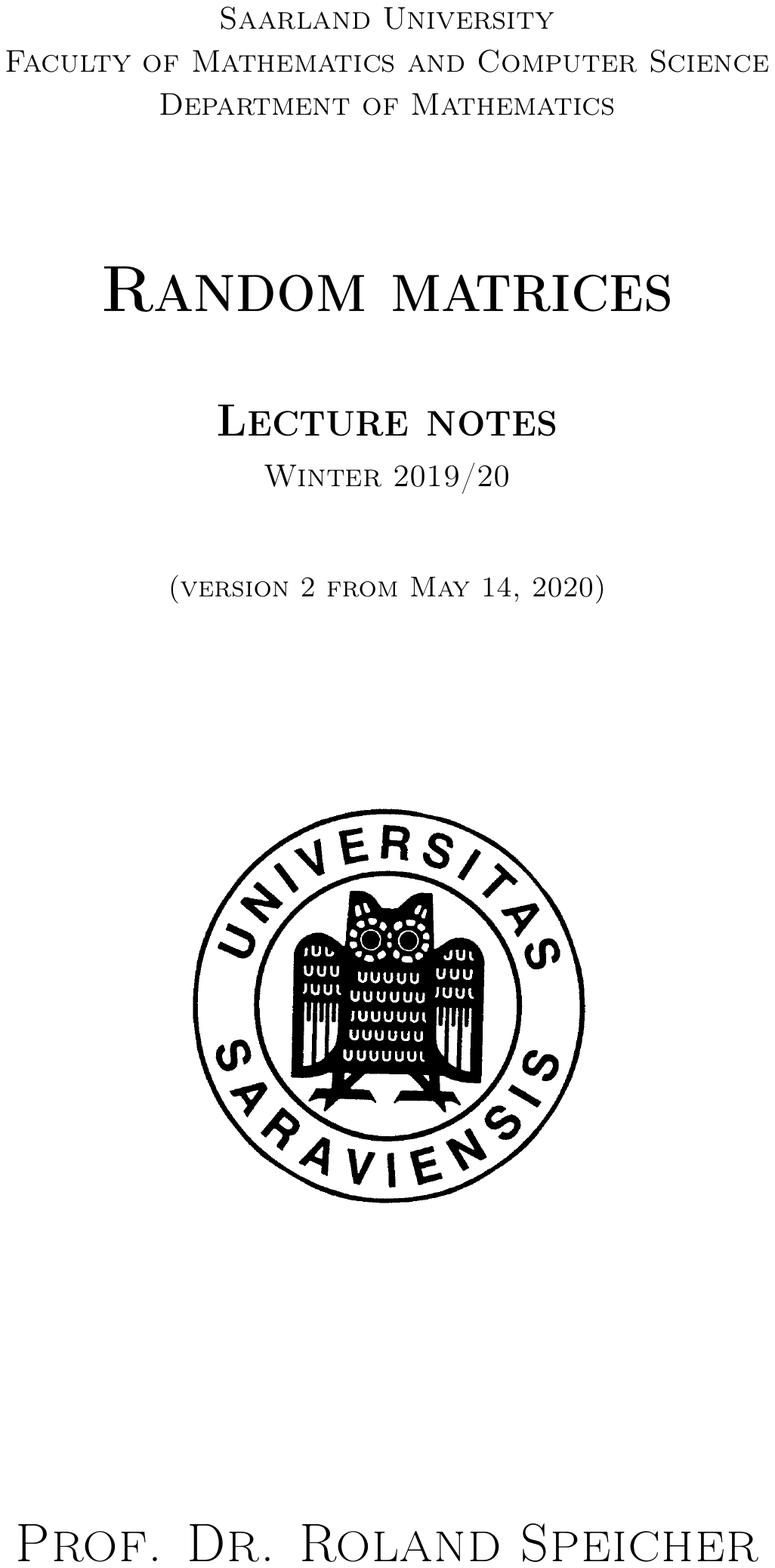}

This in an introduction to random matrix theory, giving an impression of some of the most important aspects of this modern subject. In particular, it
covers the basic combinatorial and analytic theory around Wigner's semicircle law, featuring also concentration phenomena, and the Tracy--Widom distribution of the largest eigenvalue. The circular law and a discussion of Voiculescu's multivariate extension of the semicircle law, as an appetizer for free probability theory, also make an appearance.

This manuscript here is an updated version of a joint manuscript with Marwa Banna from an earlier version of this course; it relies substantially on  the sources given in the literature;
in particular, the lecture notes of Todd Kemp were inspiring and very helpful at various places.

The material here was presented in the winter term 2019/20 at Saarland University in 24 lectures of 90 minutes each. The lectures were recorded and can be found online at \url{https://www.math.uni-sb.de/ag/speicher/web_video/index.html}.

$$\includegraphics[width=2.8in]{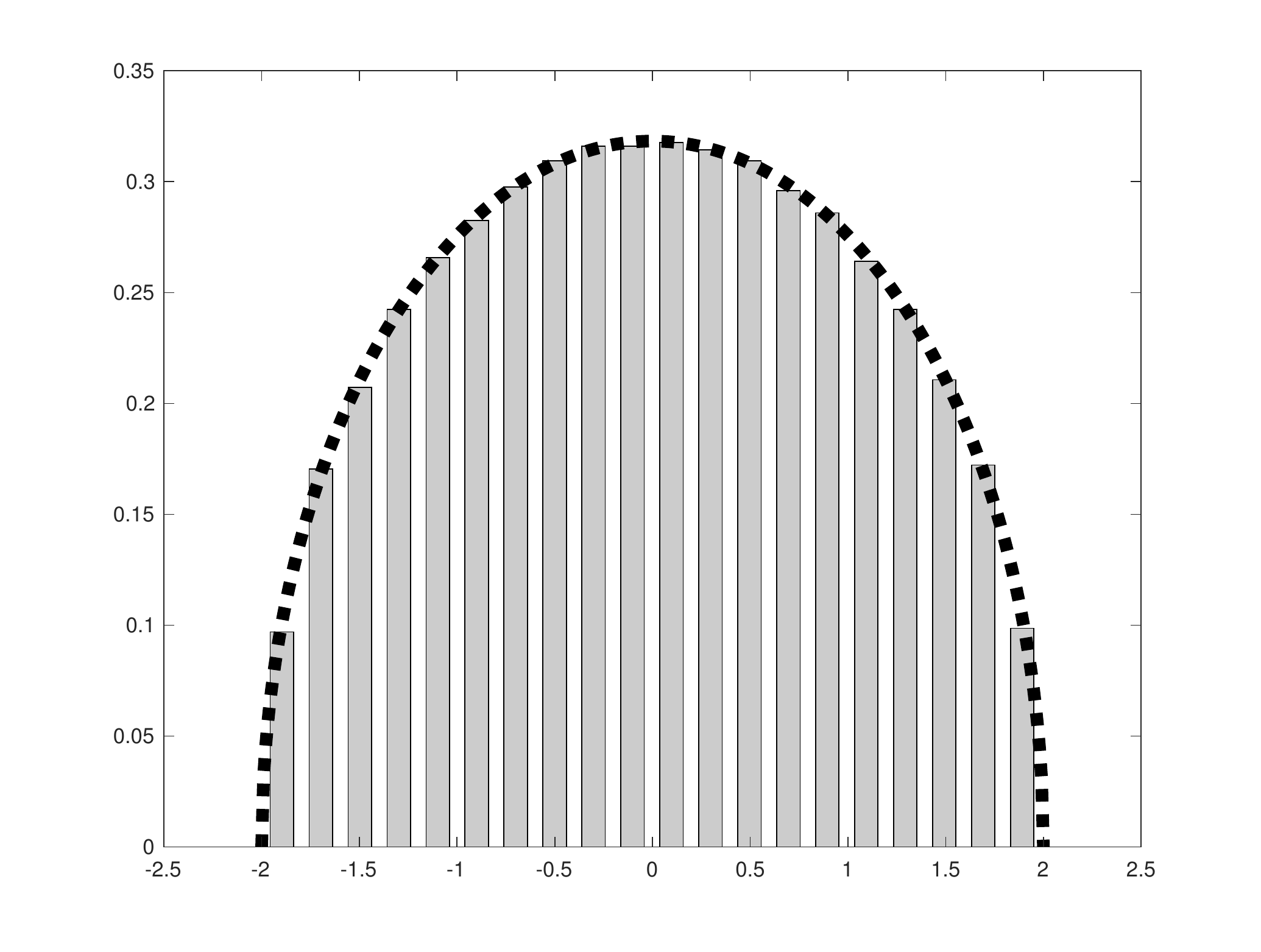}\quad
\includegraphics[width=2.8in]{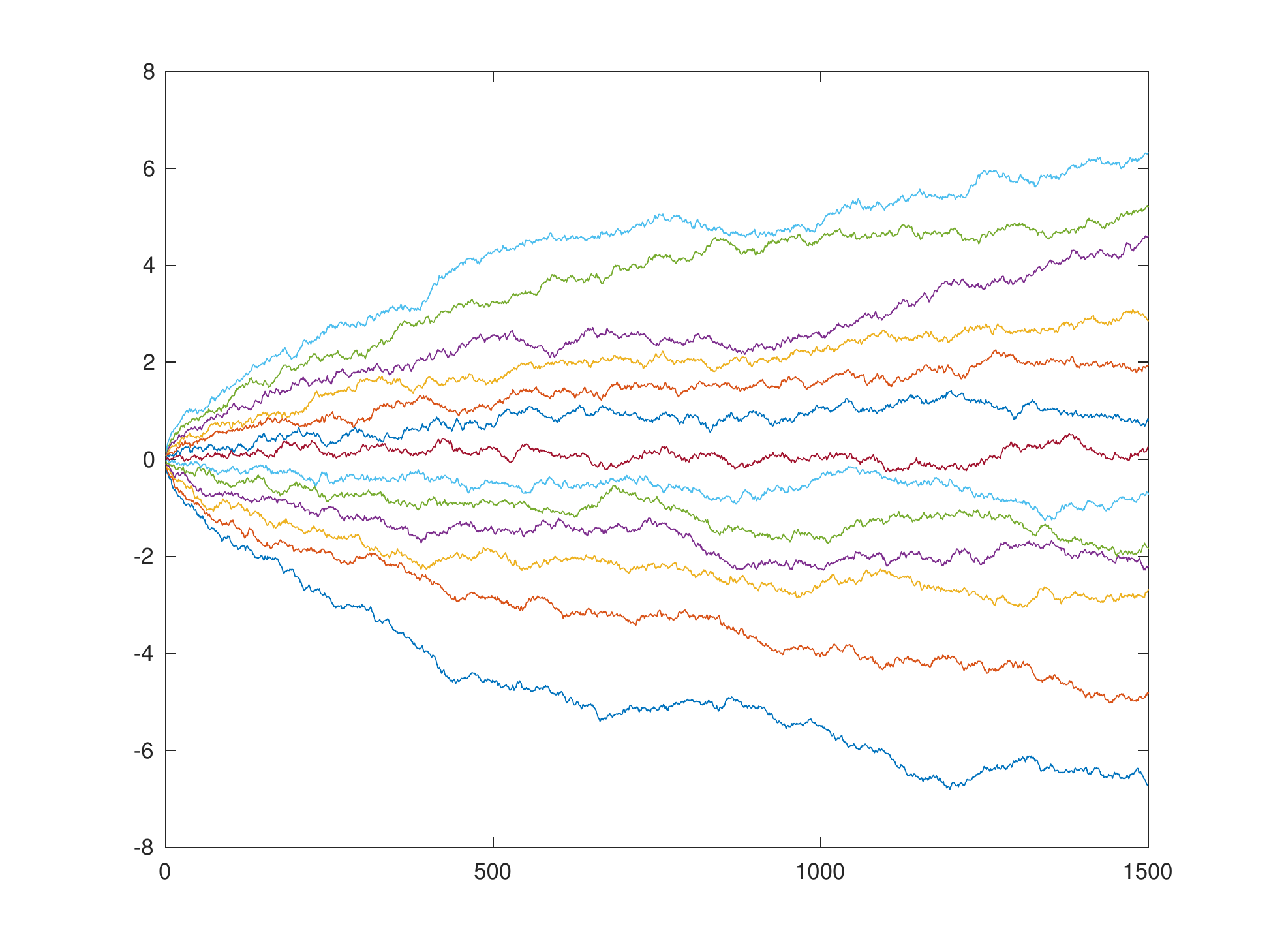}$$
$$\includegraphics[width=2.8in]{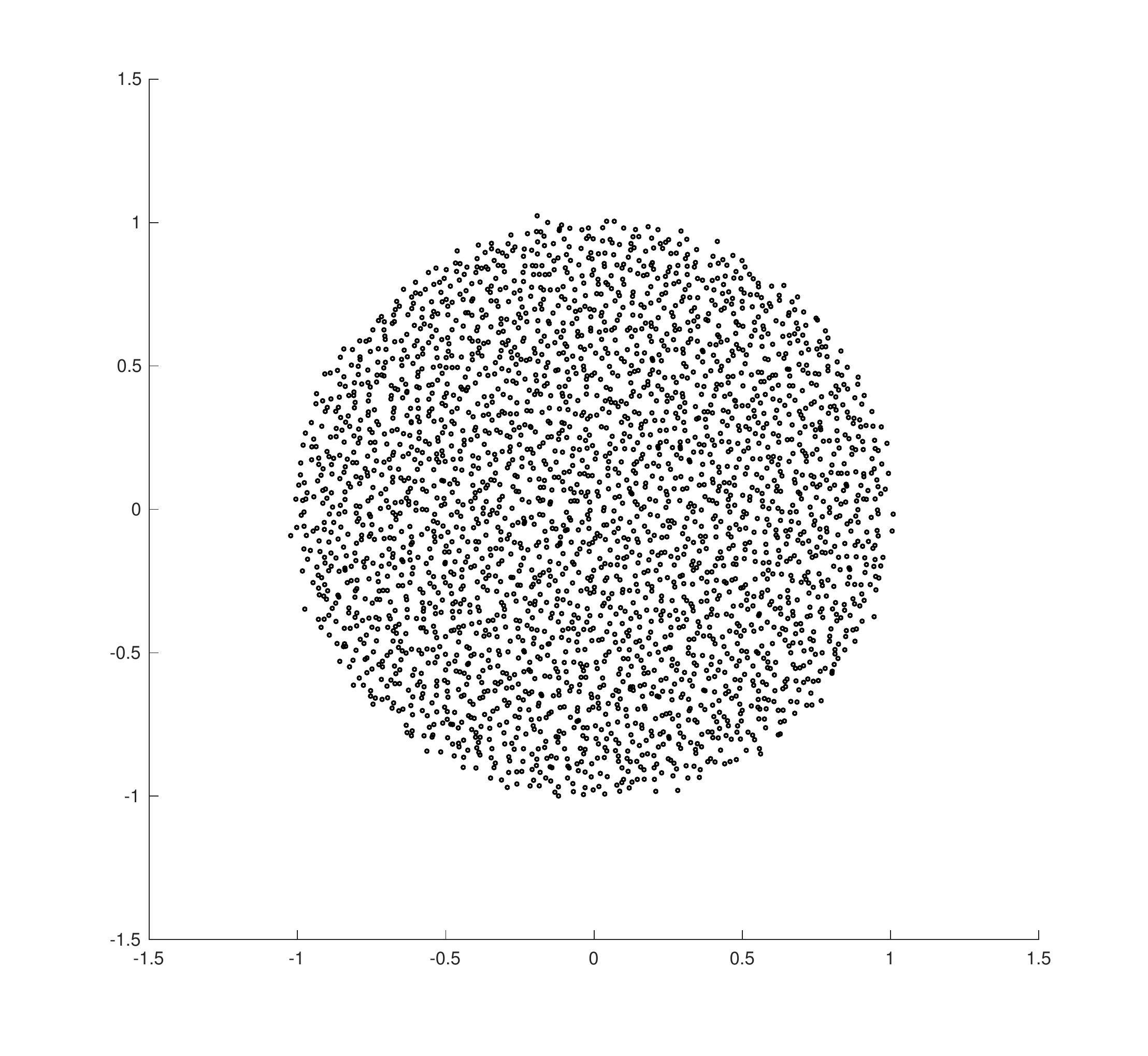}\quad
\includegraphics[width=2.8in]{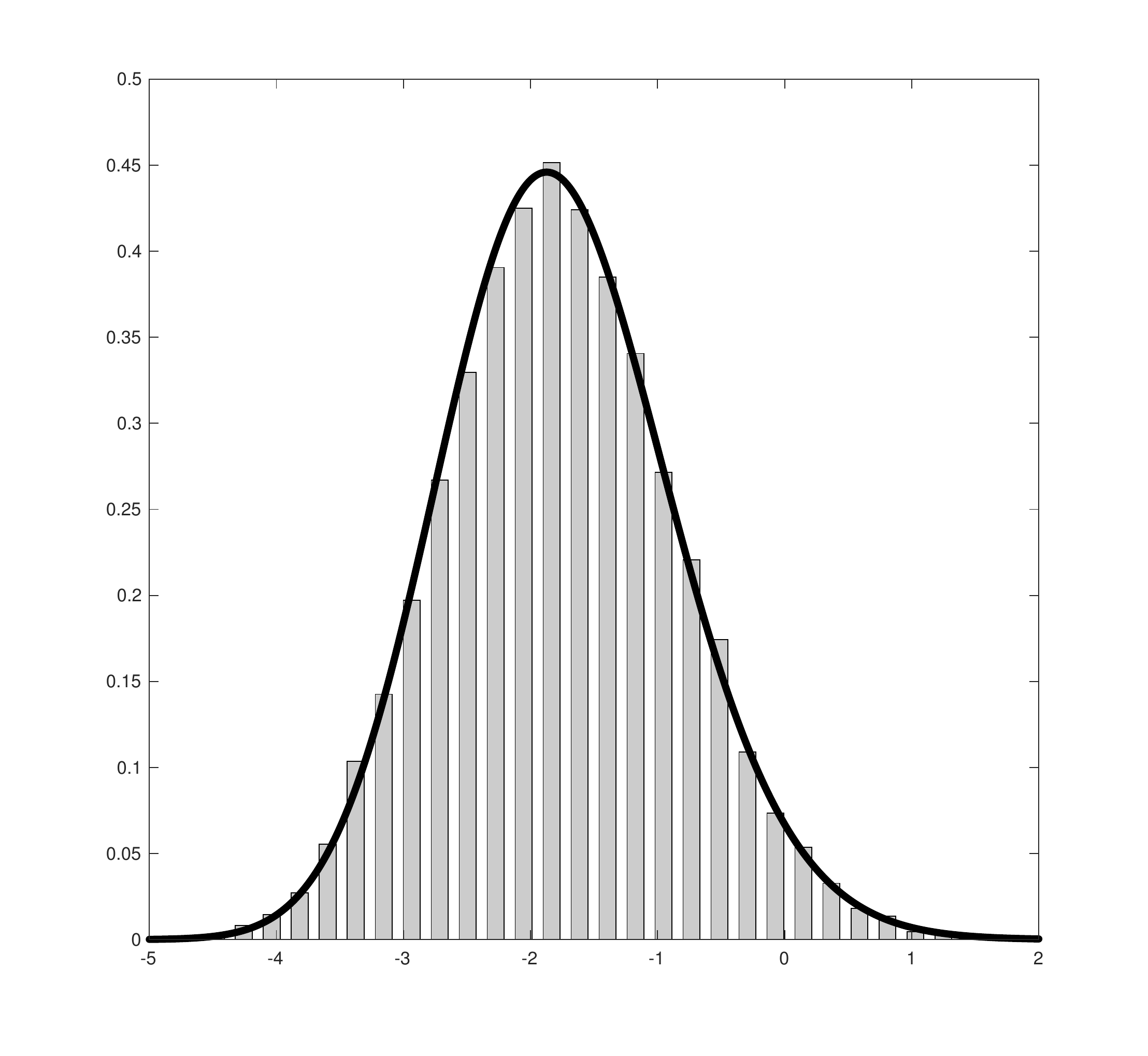}$$

\renewcommand{\contentsname}{Table of contents}
\tableofcontents

\numberwithin{theorem}{chapter}

\chapter{Introduction}

We start with giving a brief history of the subject and a feeling for some the basic objects, questions, and methods - this is just motivational and should be seen as an appetizer. Rigorous versions of statements will come in later chapters.

%\begin{remark}[History]
\section{Brief history of random matrix theory}

\begin{itemize}
	\item 1897: the paper \textit{Über die Erzeugung der Invarianten durch Integration} by Hurwitz is, according to Diaconis and Forrester, the first paper on
random matrices in mathematics;
	\item 1928: usually, the first appearance of random matrices, for fixed size $N$, is attributed to a paper of Wishart in statistics;
	\item 1955: Wigner introduced random matrices as statistical models for heavy nuclei and studied in particular the asymptotics for $N \to\infty$ (the so-called \enquote{large $N$ limit}); this set off a lot of activity around random matrices in physics;

	\item since 1960's: random matrices have become important tools in physics; in particular in the context of \textit{quantum chaos} and \textit{universality} questions; 
	important work was done by by Mehta and Dyson;
\item 1967: the first and influential book on \textit{Random Matrices} by Mehta appeared;
\item 1967: Marchenko and Pastur calculated the asymptotics $N\to\infty$ of Wishart matrices;
	\item $\sim$ 1972: a relation between the statistics of eigenvalues of random matrices and the zeros of the Riemann $\zeta$-function was conjectured by Montgomery and Dyson; with substantial evidence given by numerical calculations of Odlyzko; this made the subject more and more popular in mathematics;
	\item since 1990's: random matrices are studied more and more extensively in mathematics, in the context of quite different topics, like
\begin{itemize}
\item
Tracy-Widom distribution of largest eigenvalue 
\item
free probability theory 
\item
universality of fluctuations
\item
\enquote{circular law}
\item
and many more
\end{itemize}
\end{itemize}
%\end{remark}

%\begin{remark}[What is a random matrix?]
\section{What are random matrices and what do we want to know about them?}

A \emph{random matrix} is a matrix $A=(a_{ij})_{i,j=1}^N$ where the entries $a_{ij}$ are chosen randomly and
we are mainly interested in the \emph{eigenvalues} of the matrices. 
Often we require $A$ to be selfadjoint, which guarantees that its eigenvalues are real.

\begin{example}
	Choose $a_{ij} \in \{-1,+1\}$ with $a_{ij} = a_{ji}$ for all $i,j$. We consider all such matrices and ask for typical or generic behaviour of the eigenvalues. 
In a more probabilistic language we declare all allowed matrices to have the same probability and we ask for probabilities of properties of the eigenvalues. We can do this for 
different sizes $N$.
To get a feeling, let us look at different $N$.

\begin{itemize}
\item
For $N=1$ we have two matrices.

\medskip
\begin{tabular}{cccc}
& matrix & eigenvalues & probability of the matrix\\ 
%&&&\\
	& $(1)$ & $+1$ & $\frac{1}{2}$ \\ 
&&&\\
	& $(-1)$ & $-1$ & $\frac{1}{2}$ 
	\end{tabular}
\medskip
\item
For $N=2$ we have eight matrices.
\begin{longtable}{cccc}
&matrix & eigenvalues & probability of the matrix\\ 
%&&&\\
	& $\begin{pmatrix} 1 & 1 \\ 1 & 1 \end{pmatrix}$ & $0,2$  & $\frac{1}{8}$ \\ 
&&&\\
	& $\begin{pmatrix} 1 & 1 \\ 1 & -1 \end{pmatrix}$ & $-\sqrt{2}, \sqrt{2}$ & $\frac{1}{8}$ \\ 
&&&\\
	& $\begin{pmatrix} 1 & -1 \\ -1 & 1 \end{pmatrix}$ \qquad& $0,2$ & $\frac{1}{8}$ \\ 
&&&\\
& $\begin{pmatrix} -1 & 1 \\ 1 & -1 \end{pmatrix}$ \qquad& $-\sqrt{2}, \sqrt{2}$ & $\frac{1}{8}$ \\ 
&&&\\
& $\begin{pmatrix} 1 & -1 \\ -1 & -1 \end{pmatrix}$ \qquad& $-\sqrt{2}, \sqrt{2}$ & $\frac{1}{8}$ \\ 
&&&\\
& $\begin{pmatrix} -1 & 1 \\ 1 & -1 \end{pmatrix}$ \qquad& $-{2}, 0$ & $\frac{1}{8}$ \\ 
&&&\\
& $\begin{pmatrix} -1 & -1 \\ -1 & 1 \end{pmatrix}$ \qquad& $-\sqrt{2}, \sqrt{2}$ & $\frac{1}{8}$ \\ 
&&&\\
	& $\begin{pmatrix} -1 & -1 \\ -1 & -1 \end{pmatrix}$ & $-2,0$ & $\frac{1}{8}$
\end{longtable} 

\item
For
general $N$, we have $2^{N(N+1) /2}$ matrices, each counting with probability $2^{-N(N+1) /2}$.
Of course, there are always very special matrices such as
\[ A = \begin{pmatrix}
1 & \cdots & 1\\
\vdots & \ddots & \vdots \\
1 & \cdots & 1
\end{pmatrix},
\]
where all entries are $+1$. Such a matrix has an atypical eigenvalue behaviour (namely, only two eigenvalues, $N$ and $0$, the latter with multiplicity $N-1$); however,  the probability of such atypical behaviours will become small if we increase $N$. 
What we are interested in is the behaviour of most of the matrices for large $N$.

\begin{question*}
	What is the \enquote{typical} behaviour of the eigenvalues?
\end{question*}	
%\end{remark}

\item
Here are two ``randomly generated'' (by throwing a coin for each of the entries on and above the diagonal) $8\times 8$ symmetric matrices with $\pm 1$ entries ...

\begin{scriptsize}
{$$\qquad\begin{pmatrix}
 -1&    -1&     1&     1&     1&    -1&     1&    -1\\
   -1&    -1&     1&    -1&    -1&    -1&     1&    -1\\
    1&     1&     1&     1&    -1&    -1&     1&    -1\\
    1&    -1&     1&    -1&    -1&    -1&     1&     1\\
    1&    -1&    -1&    -1&     1&     1&    -1&    -1\\
   -1&    -1&    -1&    -1&     1&     1&    -1&     1\\
    1&     1&     1&     1&    -1&    -1&    -1&    -1\\
   -1&    -1&    -1&     1&    -1&     1&    -1&    -1
\end{pmatrix}\qquad\qquad\begin{pmatrix}
   1&    -1&    -1&     1&    -1&    -1&     1&    -1\\
   -1&     1&    -1&     1&     1&    -1&     1&     1\\
   -1&    -1&    -1&    -1&     1&     1&     1&    -1\\
    1&     1&    -1&     1&    -1&     1&     1&     1\\
   -1&     1&     1&    -1&    -1&    -1&    -1&    -1\\
   -1&    -1&     1&     1&    -1&     1&    -1&    -1\\
    1&     1&     1&     1&    -1&    -1&    -1&    -1\\
   -1&     1&    -1&     1&    -1&    -1&    -1&    -1
\end{pmatrix}$$}
\end{scriptsize}

... and their corresponding histograms of the 8 eigenvalues

{\includegraphics[width=3.in]{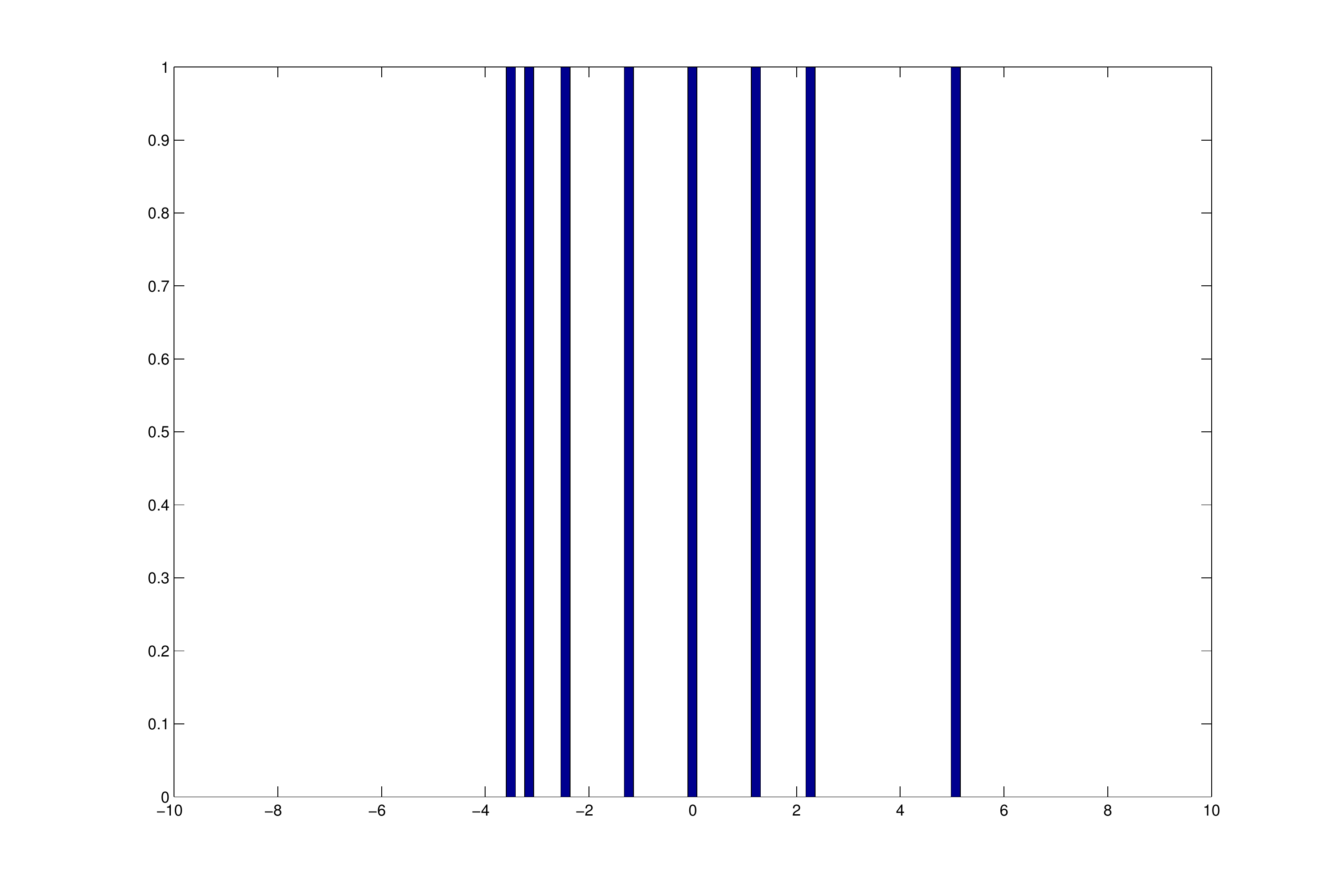}} {\includegraphics[width=3.in]{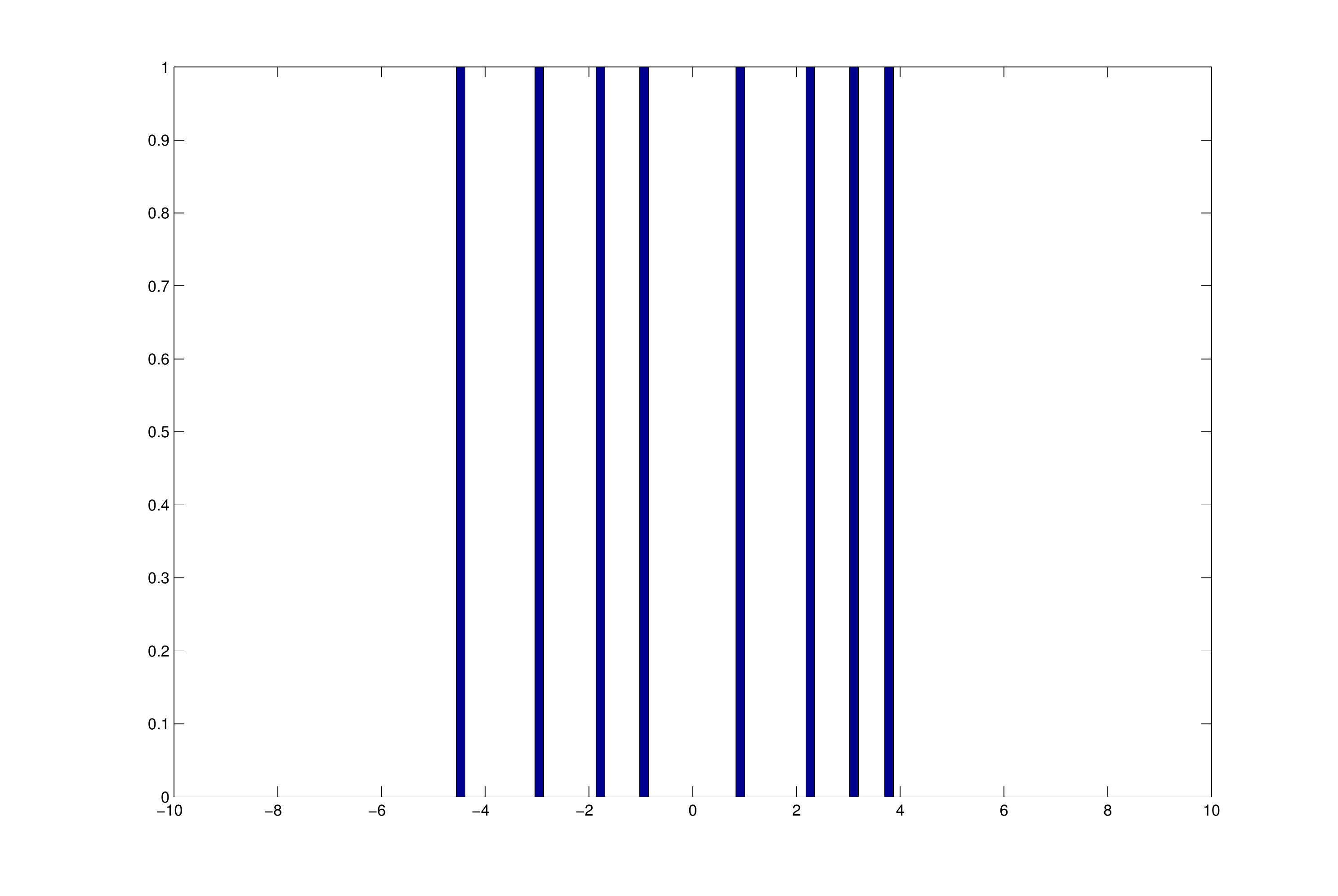}}

The message of those histograms is not so clear - apart from maybe that degeneration of eigenvalues is atypical. However, if we increase $N$ further then there will appear much more structure.

\item
Here are the eigenvalue histograms for two ``random'' $100\times 100$ matrices ...

$$\includegraphics[width=2.5in]{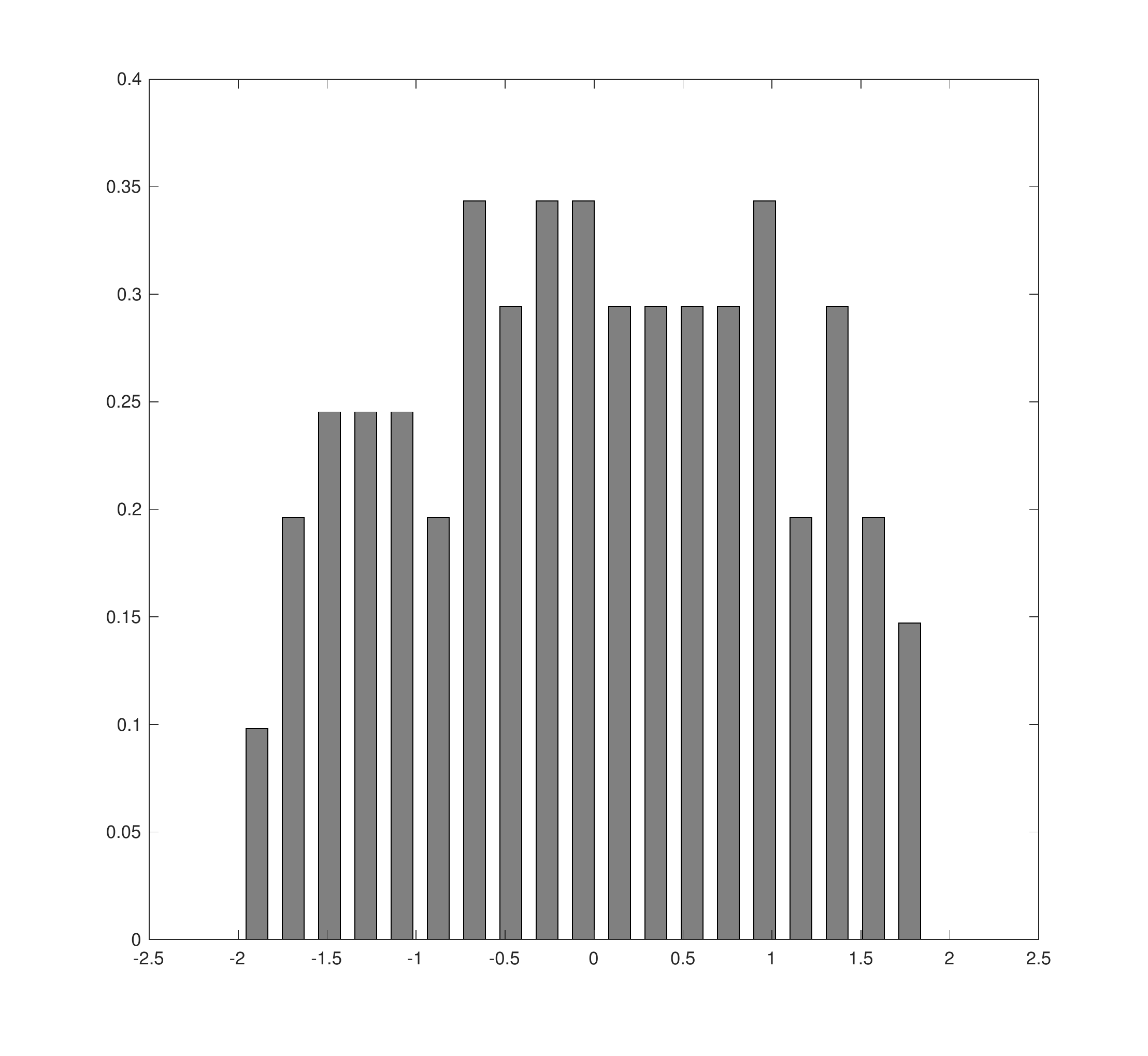}
\includegraphics[width=2.5in]{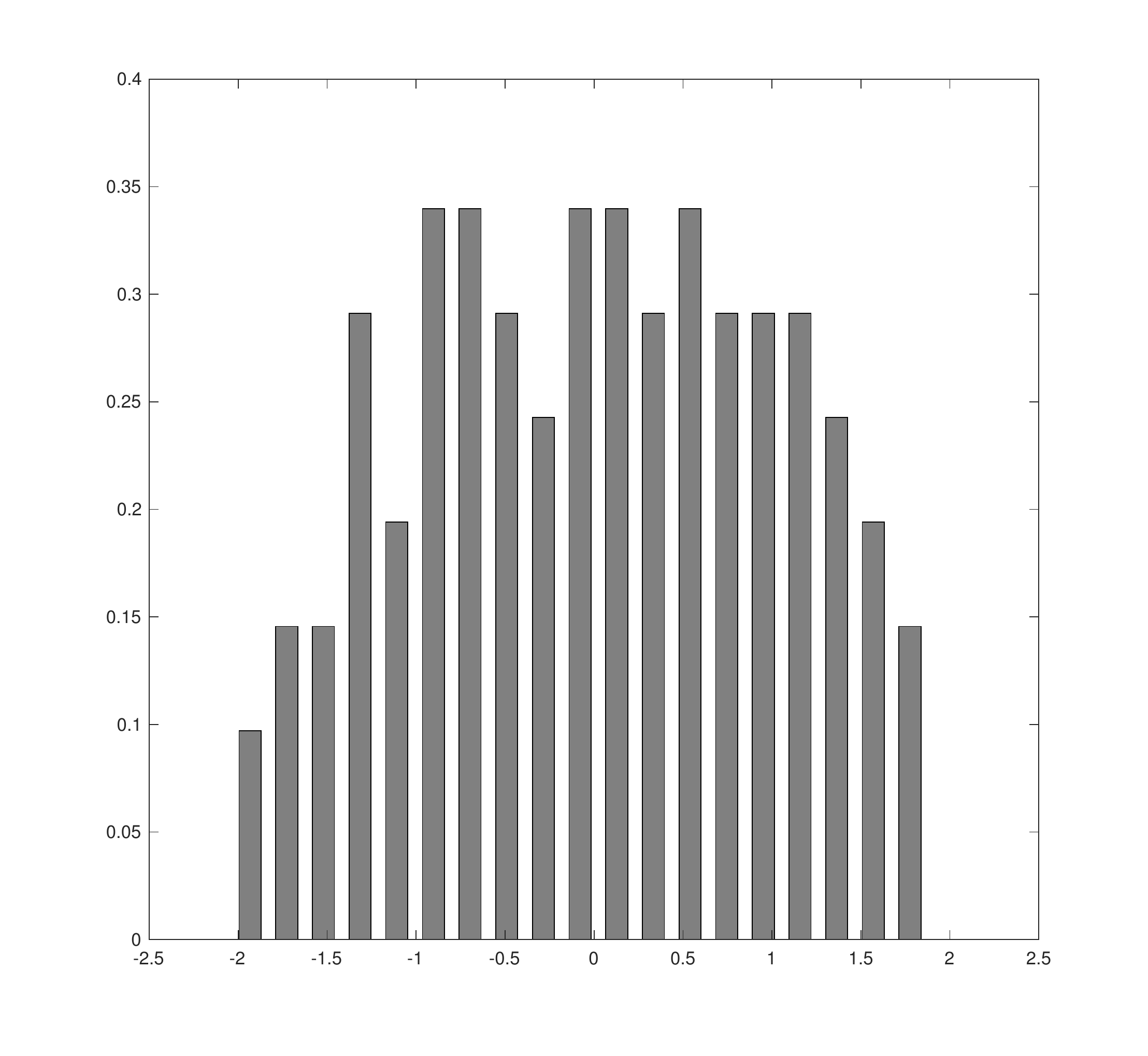}$$

\item
... and here for two ``random'' $3000\times 3000$ matrices ...

$$\includegraphics[width=2.5in]{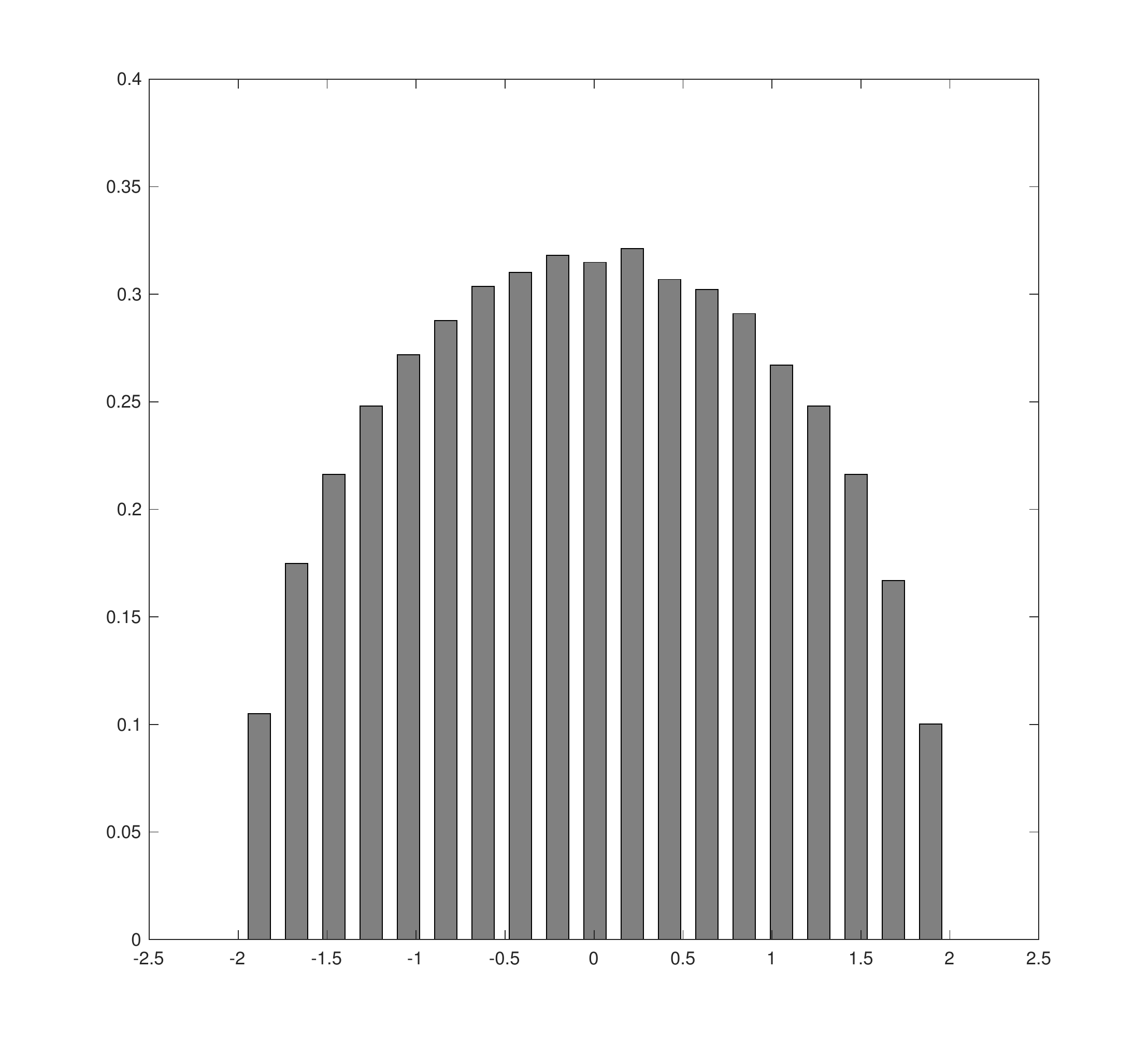}
\includegraphics[width=2.5in]{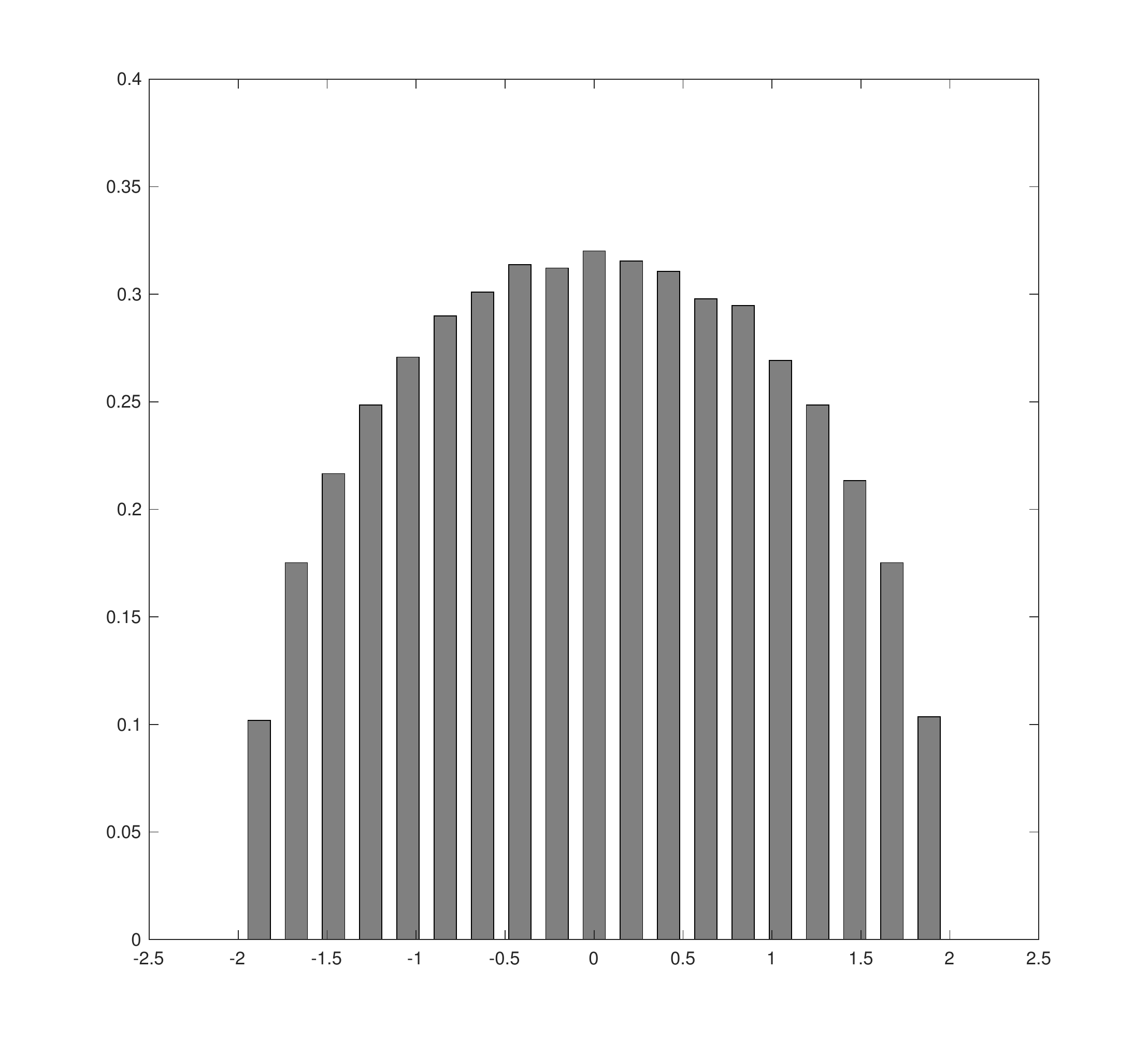}$$

\end{itemize}
\end{example}

To be clear, no coins were thrown for producing those matrices, but we relied on the MATLAB procedure for creating random matrices. Note that we also rescaled our matrices, as we will address in Section \ref{sec:scaling}.

%\begin{remark}[Wigner's smicircle law]
\section{Wigner's semicircle law}

What we see in the above figures is the most basic and important result of random matrix theory, the so-called \emph{Wigner's semicirlce law} \dots

$$\includegraphics[width=4.in]{semi-mit}$$

\dots which says that typically the eigenvalue distribution of such a random matrix converges to Wigner's semicircle for $N\to\infty$.

$$\text{eigenvalue distribution} \qquad \to \qquad\qquad \text{semicircle}\qquad\qquad$$
$$\includegraphics[width=2.5in]{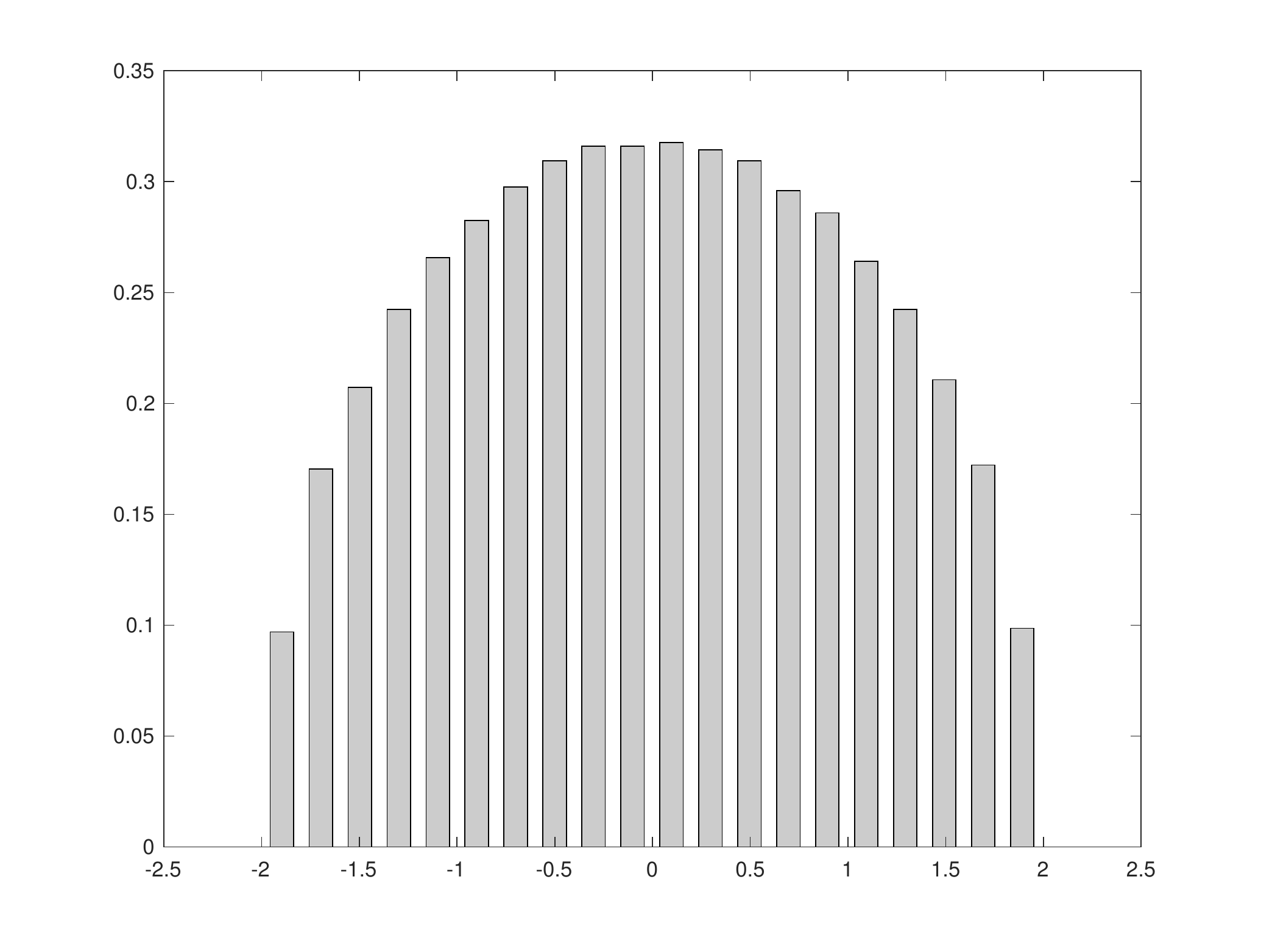}  
\includegraphics[width=2.5in]{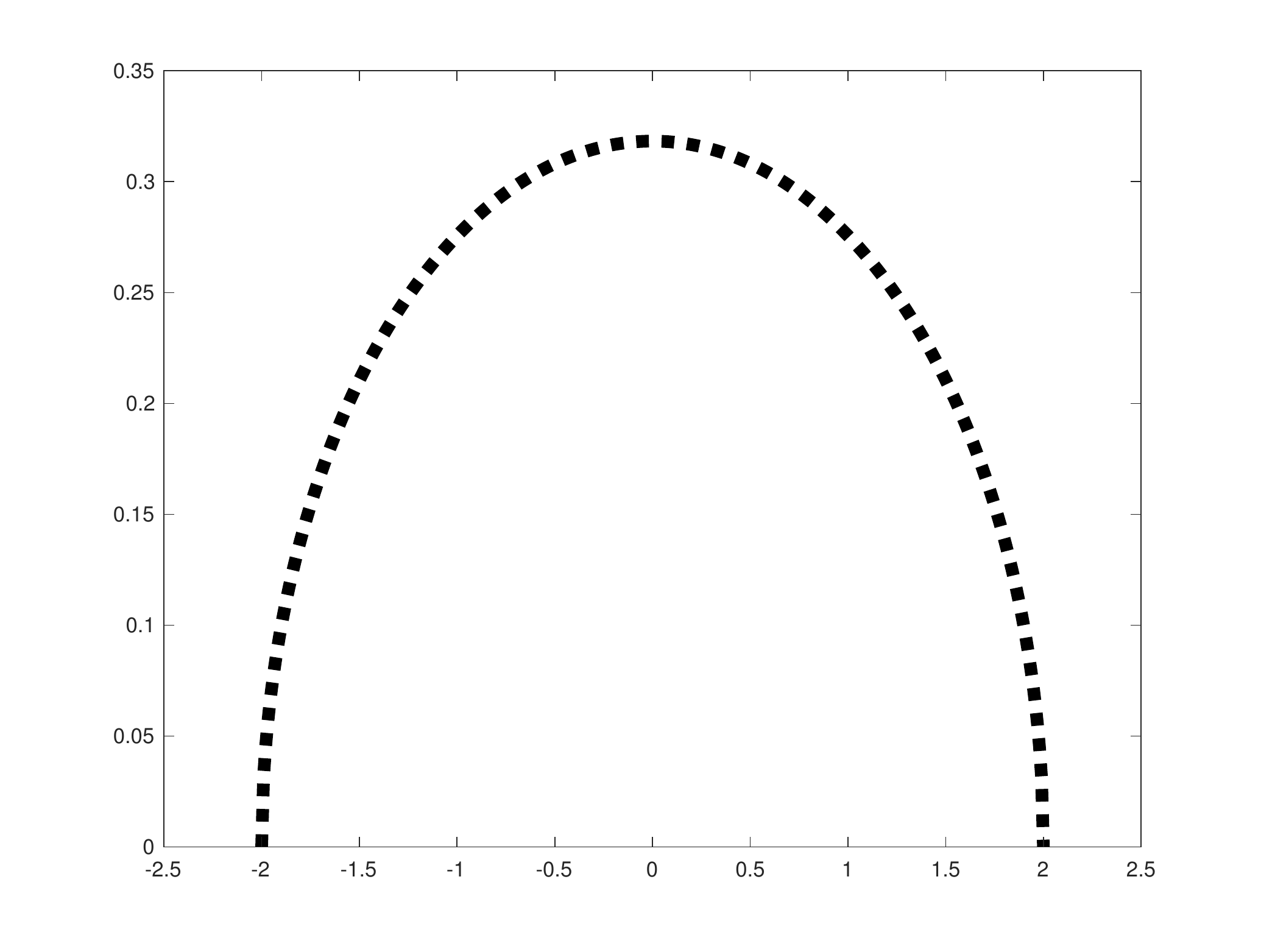}$$

Note the quite surprising feature that the limit of the random eigenvalue distribution for $N\to\infty$ is a deterministic object - the semicircle distribution.  The randomness disappears for large $N$.
%\end{remark}

%\begin{remark}[Universality]
\section{Universality}

This statement is valid much more generally.
Choose the $a_{ij}$ not just from $\{-1,+1\}$ but, for example, 
\begin{itemize}
	\item $a_{ij} \in \{1,2,3,4,5,6\}$,
	\item $a_{ij}$ normally (Gauß) distributed,
	\item $a_{ij}$ distributed according to your favorite distribution,
\end{itemize}
but still independent (apart from symmetry), then we still have the same result:
The eigenvalue distribution typically converges to a semicircle for $N \to \infty$.	
%\end{remark}

%\begin{remark}[Concentration phenomena]
\section{Concentration phenomena}

The (quite amazing) fact that the a priori random eigenvalue distribution is, for $N\to\infty$, not random anymore, but concentrated on one deterministic distribution (namely the semicircle) is an example of the general high-dimensional phenomenon of \enquote{measure concentration}.	
%\end{remark}

\begin{example}
To illustrate this let us give an easy but illustrating example of such a concentration in high dimensions; namely that the volume of a ball is essentially sitting in the surface.
	
	Denote by $B_r(0)$ the ball of radius $r$ about $0$ in $\R^n$ and for $0 <\varepsilon<1$ consider the $\varepsilon$-neighborhood of the surface inside the ball,
	$B:= \left\{ x \in \R^n \, | \, 1- \varepsilon \leq \norm{x} \leq 1 \right\}$.
	As we know the volume of balls
	\[ \vol(B_r(0)) = r^n \frac{\pi^{\frac{n}{2}}}{\left( \frac {n}{2} -1 \right)!},
	\]
	we can calculate the volume of $B$ as
	\[ \vol(B) = \vol(B_1(0)) - \vol(B_{1-\varepsilon}(0))
	= \frac{\pi^{\frac{n}{2}}}{\left( \frac{n}{2} -1 \right)!} \left( 1- (1-\varepsilon)^n \right).
	\]
	Thus,
	\[ \frac{\vol(B)}{\vol(B_1(0))} = 1- (1-\varepsilon)^n \overset{n\to\infty}\longrightarrow 1.
	\]
This says that in high dimensions the volume of a ball is concentrated in an arbitrarily small neighborhood of the surface. This is, of course, not true in small dimension - hence from our usual 3-dimensional perspective this appears quite counter-intuitive.
%\end{remark}
\end{example}

%\begin{remark}[From histograms to moments]
\section{From histograms to moments}

Let $A_N = A = (a_{ij})_{i,j=1}^N$ be our selfadjoint matrix with $a_{ij} = \pm 1$ randomly chosen.
Then we typically see for the eigenvalues of $A$:
\[ \begin{tikzpicture}[scale=0.50]
\tikzset{dot/.style={circle,fill=#1,inner sep=0,minimum size=4pt}}

\node[dot=black] at (4, 0)   (1)  {};
\node[dot=black] at (5.5, 0)   (2)  {};

\draw (4,-0.6) -- node {$s$} (4,-0.6);
\draw (5.5,-0.6) -- node {$t$} (5.5,-0.6);

\draw (0,0) -- (11,0);
\draw (0.5,0) |- (1.5,1.5) -- (1.5,0);
\draw (2,0) |- (3,2.5) -- (3,0);
\draw (3.5,0) |- (4.5,3.5) -- (4.5,0);
\draw (5,0) |- (6, 4.5) |- (6,0);
\draw (6.5,0) |- (7.5,2.5) -- (7.5,0);
\draw (8,0) |- (9,2.5) -- (9,0);
\draw (9.5,0) |- (10.5,1) -- (10.5,0);

\draw [->] (11.5,1.25) -- node [midway,above,align=center] {$N \to \infty$} (15, 1.25);

\node[dot=black] at (20, 0)   (3)  {};
\node[dot=black] at (21.5, 0)   (4)  {};

\draw (20,-0.6) -- node {$s$} (20,-0.6);
\draw (21.5,-0.6) -- node {$t$} (21.5,-0.6);

\draw (16,0) -- (27,0);

\draw (16.5,0) -- (26.5,0) arc(0:180:5) --cycle;

\begin{scope}
\clip (16.5,0) -- (26.5,0) arc(0:180:5) --cycle;
\draw [pattern=north east lines, pattern color=black] (20,0) rectangle (21.5,5);
\end{scope}
\end{tikzpicture}
\]	

This convergence means
\[ \frac{\#\{ \text{eigenvalues in } [s,t] \}}{N} \xrightarrow{N \to \infty}
\int_s^t \td \mu_W = \int_s^t p_W(x) \td x,
\]
where $\mu_W$ is the semicircle distribution, with density $p_W$.

The left-hand side of this is difficult to calculate directly, but we note that the above statement is the same as
\begin{align}
\frac{1}{N} \sum_{i=1}^{N} 1_{[s,t]} (\lambda_i) \xrightarrow{N \to\infty} \int_\R 1_{[s,t]} (x) \td \mu_W(x), \tag{$\star$}
\end{align}
where $\lambda_1, \dots, \lambda_N$ are the eigenvalues of $A$ counted with multiplicity and
$1_{[s,t]}$ is the characteristic function of $[s,t]$, i.e,
\[1_{[s,t]}(x) = \begin{cases}
1, &\mbox{} x \in [s,t], \\
0, &\mbox{} x \not \in [s,t].
\end{cases}
\]
Hence in ($\star$) we are claiming that 
\[ \frac{1}{N} \sum_{i=1}^{N} f(\lambda_i) \xrightarrow{N \to\infty} \int_\R f(x) \td \mu_W(x)
\]
for all $f=1_{[s,t]}$.
It is easier to calculate this for other functions $f$, in particular, for $f$ of the form $f(x)=x^n$, i.e.,
\begin{align}
\frac{1}{N} \sum_{i=1}^{N} \lambda_i^n \xrightarrow{N \to\infty} \int_\R x^n \td \mu_W(x); \tag{$\star\star$}
\end{align}
the latter are the \emph{moments} of $\mu_W$. (Note that $\mu_W$ must necessarily be a probability measure.)
 
We will see later that in our case the validity of ($\star$) for all $s<t$ is equivalent to 
the validity of ($\star\star$) for all $n$. Hence we want to show ($\star\star$) for all $n$.	
%\end{remark}

\begin{remark}
The above raises of course the question:
What is the advantage of $(\star\star)$ over $(\star)$, or of
$x^n$ over $1_{[s,t]}$?

Note that $A=A^*$ is selfadjoint and hence can be diagonalized, i.e., $A=UDU^*$, where
$U$ is unitary and $D$ is diagonal with $d_{ii} = \lambda_i$ for all $i$ (where $\lambda_1,\dots,\lambda_N$ are the eigenvalues of $A$, counted with multiplicity). Moreover, we have
\[ A^n = (UDU^*)^n = UD^nU^*
\qquad
\text{with}\qquad
 D^n = \begin{pmatrix}
\lambda_1^n & & \\
& \ddots & \\
& & \lambda_N^n
\end{pmatrix},
\]
hence
\[ \sum_{i=1}^{N} \lambda_i^n
= \Tr(D^n)
=\Tr(UD^nU^*)
=\Tr(A^n)
\]
and thus
\[ \frac{1}{N} \sum_{i=1}^{N} \lambda_i^n
= \frac{1}{N} \Tr(A^n).
\]	
\end{remark}

%\begin{notation}
\begin{notation}
	We denote by $\tr = \frac{1}{N} \Tr$ the \emph{normalized trace} of matrices, i.e.,
	\[ \tr \left( (a_{ij})_{i,j=1}^N  \right)
	= \frac{1}{N} \sum_{i=1}^{N} a_{ii}.
	\]
\end{notation}

	So we are claiming that for our matrices we typically have that
	\[ \tr(A_N^n) \xrightarrow{N\to\infty} \int x^n \td \mu_W(x).
	\]
The advantage in this formulation is that the quantity $\tr(A_N^n)$ can be expressed in terms of the entries of the matrix, without actually having to calculate the eigenvalues.

%\begin{remark}[Choice of scaling]
\section{Choice of scaling}\label{sec:scaling}

	Note that we need to choose the right scaling in $N$ for the existence of the limit $N\to\infty$.
	For the case $a_{ij} \in \{\pm 1 \}$ with $A_N=A_N^*$ we have
	\[ \tr(A_N^2) = \frac{1}{N} \sum_{i,j=1}^{N} a_{ij}a_{ji}=\frac 1N\sum_{i,j=1} a_{ij}^2
	=\frac{1}{N} N^2
	=N.
	\]
	Since this has to converge for $N\to\infty$ we should rescale our matrices
	\[ A_N \rightsquigarrow \frac{1}{\sqrt{N}} A_N,
	\]
	i.e., we consider matrices $A_N = (a_{ij})_{i,j=1}^N$, where $a_{ij} = \pm \frac{1}{\sqrt{N}}$.
	For this scaling we claim that we typically have that
	\[ \tr(A_N^n) \xrightarrow{N\to\infty} \int x^n \td \mu_W(x)
	\]
	for a deterministic probability measure $\mu_W$.
%\end{remark}

\section{The semicircle and its moments}

It's now probably time to give the precise definition of the semicircle distribution and, in the light of what we have to prove, also its moments.

\begin{definition}
\begin{itemize}
	\item[(1)] The (standard) \emph{semicircular distribution} $\mu_W$ is the measure on $[-2,2]$ with density 
	\[ \td\mu_W(x) = \frac{1}{2\pi} \sqrt{4-x^2} \td x.
	\]
	\[ \begin{tikzpicture}	
	\tikzset{dot/.style={circle,fill=#1,inner sep=0,minimum size=4pt}}

	\node[dot=black] at (-2, 0)   (1) {};
	\node[dot=black] at (2, 0)   (2) {};
	\node[dot=black] at (0, 2)   (3) {};
		
	\draw (-2,-0.3) -- node {$-2$} (-2,-0.3);
	\draw (2,-0.3) -- node {$2$} (2,-0.3);
	\draw (0.3,2.3) -- node {$\frac{1}{\pi}$} (0.3,2.3);

	\draw [->] (-2.5,0) -- (2.5,0);	
	\draw [->] (0,-0.5) -- (0,2.5);	
	\draw (-2,0) -- (2,0) arc(0:180:2) --cycle;
	\end{tikzpicture}
	\]	
	
	\item[(2)] The \emph{Catalan numbers} $(C_k)_{k\geq 0}$ are given by
	\[ C_k = \frac{1}{k+1} \binom{2k}{k}.
	\]
	They look like this: $1,1,2,5,14,42,132,\dots$
\end{itemize}
\end{definition}

\begin{theorem}
	\begin{itemize}\phantomsection{\label{thm:1.4}}
		\item[(1)] 
The Catalan numbers have the following properties.
		\begin{itemize}
			\item[(i)] The Catalan numbers satisfy the following recursion:
			\[ C_k = \sum_{l=0}^{k-1} C_lC_{k-l-1} \qquad (k \geq1)
			\]
			\item[(ii)] The Catalan numbers are uniquely determined by this recursion and by the initial value $C_0=1$.
		\end{itemize}
		\item[(2)] The semicircular distribution $\mu_W$ is a probability measure, i.e., 
		\[\frac 1{2\pi} \int_{-2}^{2} \sqrt{4-x^2} \td x = 1
		\]
		and its moments are given by
		\[ \frac 1{2\pi}\int_{-2}^{2} x^n \sqrt{4-x^2} \td x = \begin{cases}
		0, &\mbox{} n \text{ odd,} \\
		C_k, &\mbox{} n=2k \text{ even.}
		\end{cases}
		\]
	\end{itemize}
\end{theorem}

Exercises \ref{exercise:2} and \ref{exercise:3} address the proof of Theorem \ref{thm:1.4}.

%\begin{remark}[Type of convergence]
\section{Types of convergence}

So we are claiming that typically 
$$
\tr(A_N^2) \rightarrow 1,\qquad
\tr(A_N^4) \rightarrow 2, \qquad
\tr(A_N^6) \rightarrow 5, \quad
\tr(A_N^8) \rightarrow 14, \qquad
\tr(A_N^{10}) \rightarrow 42,
$$
and so forth. But what do we mean by \enquote{typically}?
The mathematical expression for this is \enquote{almost surely}, but for now let us look on the more intuitive \enquote{convergence in probability} for
$\tr(A_N^{2k}) \rightarrow C_k$.

Denote by $\Omega_N$ the set of our considered matrices, that is
\[ \Omega_N= \left\{  A_N = \frac{1}{\sqrt{N}} (a_{ij})_{i,j=1}^N \, | \,
A_N=A_N^* \text{ and } a_{ij} \in \{\pm 1\}
\right\}.
\]
Then convergence in probability means that for all $\varepsilon > 0$ we have
\begin{align}
\frac{\#\left\{
	A_N \in \Omega_N \, \big| \, \abs{ \tr\left(A_N^{2k}\right) -C_k }>\varepsilon
	\right\}}{\#\Omega_N}=
P\left(A_N \, \big| \, \abs{ \tr\left(A_N^{2k}\right) -C_k }>\varepsilon \right)
\xrightarrow{N\to \infty} 0.
\tag{$\star$}
\end{align}
How can we show ($\star$)? Usually, one proceeds as follows.
\begin{itemize}
	\item[(1)] First we show the weaker form of convergence in average, i.e., 
	\[ 
	\frac{ \sum_{A_N \in \Omega_N} \tr\left(A_N^{2k} \right)}{\#\Omega_N}
	=\ev{\tr\left(A_N^{2k} \right)} \xrightarrow{N \to \infty} C_k.
	\]
	\item[(2)] Then we show that with high probability the derivation from the average will become small as $N \rightarrow \infty$. 
\end{itemize}

We will first consider step (1); (2) is a concentration phenomenon and will be treated later.	

Note that step (1) is giving us the insight into why the semicircle shows up. Step (2) is more of a theoretical nature, adding nothing to our understanding of the semicircle, but making the (very interesting!) statement that in high dimensions the typical behaviour is close to the average behaviour.
%\end{remark}

\begin{comment}
\begin{remark}
Note that
\begin{align}
\frac{1}{N} \sum_{i=1}^{N} f(\lambda_i) \xrightarrow{N \to\infty} \int_\R f(x) \td \mu_W(x)
\tag{$\star$}
\end{align}
is actually a statement on convergence of measures, since
\[ \frac{1}{N} \sum_{i=1}^{N} f(\lambda_i) 
= \int_\R f(x) \td \mu_N(x)
\]
for the \emph{empirical spectral measure} or \emph{density of states}
\[ \mu_N = \frac{1}{N} \left( \delta_{\lambda_1} + \cdots+ \delta_{\lambda_N}   \right),
\]
where $\delta_\lambda$ is the \textbf{Dirac measure}
\[ \delta_\lambda(E) = \begin{cases}
0, &\mbox{} \lambda \not \in E, \\
1, &\mbox{} \lambda \in E.
\end{cases}
\]
Hence ($\star$) says that
\[ \int_\R f(x) \td \mu_N(x) \xrightarrow{N \to\infty} \int_\R f(x) \td \mu_W(x).
\]
If we require this for sufficiently many $f$, this is a kind of convergence $\mu_N \to \mu_W$
of measures.
We will need to understand such convergence better and develop tools (Stieltjes transform)
to deal with them.	
\end{remark}
\end{comment}

\chapter{Gaussian Random Matrices: Wick Formula and Combinatorial Proof of Wigner's Semicircle}

We want to prove convergence of our random matrices to the semicircle by showing
\[ \ev{\tr A_N^{2k}} \xrightarrow{N\to\infty} C_k,
\]
where the $C_k$ are the Catalan numbers.

Up to now our matrices were of the form
$A_N = \frac{1}{\sqrt{N}} (a_{ij})_{i,j=1}^N$ with $a_{ij} \in \{-1,1\}$.
From an analytic and also combinatorial point of view it is easier to deal with another choice for the $a_{ij}$, namely we will take them as Gaussian (aka normal) random variables; different $a_{ij}$ will still be, up to symmetry, independent. If we want to calculate the expectations $\tr(A^{2k})$, then we should of course understand how to calculate moments of independent Gaussian random variables. 

\section{Gaussian random variables and Wick formula}

\begin{definition}\label{def:2.1}
A \emph{standard Gaussian} (or \emph{normal}) random variable $X$ is a real-valued Gaussian random variable with mean $0$ and variance $1$, i.e., it has distribution
\[ \prob{t_1 \leq X \leq t_2} = \frac{1}{\sqrt{2\pi}} \int_{t_1}^{t_2} e^{- \frac{t^2}{2}} \td t
\]
and hence its moments are given by
\[ \ev{X^n} = \frac{1}{\sqrt{2\pi}} \int_\R t^n e^{- \frac{t^2}{2}} \td t.
\]
\end{definition}

\begin{prop}\label{prop:2.2}
The moments of a standard Gaussian random variable are of the form
\[  \frac{1}{\sqrt{2\pi}} \int_{-\infty}^\infty t^n e^{- \frac{t^2}{2}} \td t
= \begin{cases}
0, &\mbox{} n \text{ odd,} \\
(n-1)!!, &\mbox{} n \text{ even,}
\end{cases}
\]
where the ``double factorial'' is defined, for $m$ odd, as
\[ m!! = m(m-2)(m-4)\cdots 5\cdot3\cdot1. 
\]
\end{prop}

You are asked to prove this in Exercise \ref{exercise:5}.

\begin{remark}
From an analytic point of view
it is surprising that those integrals evaluate to natural numbers.
They actually count interesting combinatorial objects,
\[ \ev{X^{2k}} = \#\{ \text{pairings of } 2k \text{ elements} \}.
\]
\end{remark}

\begin{definition}\label{def:2.4}
\begin{enumerate}
\item For a natural number $n \in\N$ we put $[n] = \{1,\dots, n \}$.
\item A \emph{pairing} $\pi$ of $[n]$ is a decomposition of $[n]$ into disjoints subsets of size $2$, i.e., $\pi = \{ V_1, \dots, V_k \}$ such that for all $i,j=1,\dots, k$ with $i\neq j$, we have:
\begin{itemize}
\item $V_i \subset [n]$
\item $\# V_i = 2$
\item $V_i \cap V_j = \emptyset$
\item $\bigcup_{i=1}^k V_i = [n]$
\end{itemize}
Note that necessarily $k = \frac{n}{2}$.
\item The set of all pairings of $[n]$ is denoted by
\[ \pair(n) = \{\pi\,|\, \pi \text{ is a pairing of } [n]  \}.
\]
\end{enumerate}
\end{definition}

\begin{prop}
\begin{enumerate}
\item We have 
\[\#\pair(n) = \begin{cases}
0, &\mbox{} n \text{ odd,} \\
(n-1)!!, &\mbox{} n \text{ even.}
\end{cases}
\]
\item Hence for a standard Gaussian variable $X$ we have
\[ \ev{X^n} = \#\pair(n).
\]
\end{enumerate}
\end{prop}

\begin{proof}
\begin{enumerate}
\item
Count elements in $\pair(n)$ in a recursive way. Choose the pair which contains the element $1$, for this we have $n-1$ possibilities. Then we are left with choosing a pairing of the remaining $n-2$ numbers. Hence we have
\[ \#\pair(n) = (n-1) \cdot \#\pair(n-2).
\]
Iterating this and noting that $\#\pair(1)=0$ and $\#\pair(2)=1$ gives the desired result.

\item
Follows from (i) and Proposition \ref{prop:2.2}.
\end{enumerate}
\end{proof}

\begin{example}
Usually we draw our partitions by connecting the elements in each pair.
Then $\ev{X^2} = 1$ corresponds to the single partition
\[ \begin{tikzpicture}[thick,font=\small, node distance=1cm and 5mm, baseline={([yshift=-.5ex]current bounding box.center)}]
\tikzset{dot/.style={circle,fill=#1,inner sep=0,minimum size=4pt}}

\node (1) 					{1};
\node (2)	[right=of 1] 	{2};

\node [dot=black] at ($ (1) + (0,-.3) $) {};	
\node [dot=black] at ($ (2) + (0,-.3) $) {};

\draw (1)	-- ++(0,-.7)		-| (2);	
\end{tikzpicture}
\]
and $\ev{X^4} = 3$ corresponds to the three partitions
\[ \begin{tikzpicture}[thick,font=\small, node distance=1cm and 5mm, baseline={([yshift=-.5ex]current bounding box.center)}]
\tikzset{dot/.style={circle,fill=#1,inner sep=0,minimum size=4pt}}

\node (1) 					{1};
\node (2)	[right=of 1] 	{2};
\node (3) 	[right=of 2]	{3};
\node (4)	[right=of 3] 	{4};

\node [dot=black] at ($ (1) + (0,-.3) $) {};	
\node [dot=black] at ($ (2) + (0,-.3) $) {};	
\node [dot=black] at ($ (3) + (0,-.3) $) {};	
\node [dot=black] at ($ (4) + (0,-.3) $) {};

\draw (1)	-- ++(0,-.7)		-| (2);	
\draw (3)	-- ++(0,-.7)		-| (4);	
\end{tikzpicture}, \qquad
\begin{tikzpicture}[thick,font=\small, node distance=1cm and 5mm, baseline={([yshift=-.5ex]current bounding box.center)}]
\tikzset{dot/.style={circle,fill=#1,inner sep=0,minimum size=4pt}}

\node (1) 					{1};
\node (2)	[right=of 1] 	{2};
\node (3) 	[right=of 2]	{3};
\node (4)	[right=of 3] 	{4};

\node [dot=black] at ($ (1) + (0,-.3) $) {};	
\node [dot=black] at ($ (2) + (0,-.3) $) {};	
\node [dot=black] at ($ (3) + (0,-.3) $) {};	
\node [dot=black] at ($ (4) + (0,-.3) $) {};

\draw (1)	-- ++(0,-.7)		-| (3);	
\draw (2)	-- ++(0,-1.1)		-| (4);	
\end{tikzpicture}, \qquad
\begin{tikzpicture}[thick,font=\small, node distance=1cm and 5mm, baseline={([yshift=-.5ex]current bounding box.center)}]
\tikzset{dot/.style={circle,fill=#1,inner sep=0,minimum size=4pt}}

\node (1) 					{1};
\node (2)	[right=of 1] 	{2};
\node (3) 	[right=of 2]	{3};
\node (4)	[right=of 3] 	{4};

\node [dot=black] at ($ (1) + (0,-.3) $) {};	
\node [dot=black] at ($ (2) + (0,-.3) $) {};	
\node [dot=black] at ($ (3) + (0,-.3) $) {};	
\node [dot=black] at ($ (4) + (0,-.3) $) {};

\draw (1)	-- ++(0,-1.1)		-| (4);	
\draw (2)	-- ++(0,-.7)		-| (3);	
\end{tikzpicture}.
\]
\end{example}

\begin{remark}[Independent Gaussian random variables]
We will have several, say two, Gaussian random variables $X,Y$ and have to calculate their joint moments.
The random variables are independent; this means that their joint distribution is the product measure of the single distributions,
\[ \prob{t_1\leq X\leq t_2, s_1\leq Y\leq s_2}
= \prob{t_1\leq X\leq t_2}\cdot \prob{s_1\leq Y\leq s_2},
\]
so in particular, for the moments we have
$$\ev{X^nY^m} 
= \ev{X^n}\cdot  \ev{Y^m}.$$
This gives then also a combinatorial description for their mixed moments:
\begin{align*}
\ev{X^nY^m} 
&= \ev{X^n} \cdot\ev{Y^m} \\
&=\#\{\text{pairings of } \underbrace{X \cdots X}_{n}  \}
	\cdot \#\{\text{pairings of } \underbrace{Y \cdots Y}_{m}  \} \\
&=\#\{\text{pairings of } \underbrace{X \cdots X}_{n}\underbrace{Y \cdots Y}_{m} \text{ which connect $X$ with $X$ and $Y$ with $Y$}  \}.
\end{align*}

\begin{example*} We have $\ev{XXYY}=1$ since the only possible pairing is
\[ \begin{tikzpicture}[thick,font=\small, node distance=1cm and 5mm, baseline={([yshift=-.5ex]current bounding box.center)}]
\tikzset{dot/.style={circle,fill=#1,inner sep=0,minimum size=4pt}}

\node (1) 					{X};
\node (2)	[right=of 1] 	{X};
\node (3) 	[right=of 2]	{Y};
\node (4)	[right=of 3] 	{Y};

\node [dot=black] at ($ (1) + (0,-.3) $) {};	
\node [dot=black] at ($ (2) + (0,-.3) $) {};	
\node [dot=black] at ($ (3) + (0,-.3) $) {};	
\node [dot=black] at ($ (4) + (0,-.3) $) {};

\draw (1)	-- ++(0,-.7)		-| (2);	
\draw (3)	-- ++(0,-.7)		-| (4);	
\end{tikzpicture}.
\]
On the other hand, $\ev{XXXYXY}=3$ since we have the following three possible pairings:

\resizebox{\textwidth}{!}{%
\begin{tikzpicture}[thick,font=\small, node distance=1cm and 5mm, baseline={([yshift=-.5ex]current bounding box.center)}]
\tikzset{dot/.style={circle,fill=#1,inner sep=0,minimum size=4pt}}

\node (1) 					{X};
\node (2)	[right=of 1] 	{X};
\node (3) 	[right=of 2]	{X};
\node (4)	[right=of 3] 	{Y};
\node (5)	[right=of 4] 	{X};
\node (6)	[right=of 5] 	{Y};

\node [dot=black] at ($ (1) + (0,-.3) $) {};	
\node [dot=black] at ($ (2) + (0,-.3) $) {};	
\node [dot=black] at ($ (3) + (0,-.3) $) {};	
\node [dot=black] at ($ (4) + (0,-.3) $) {};
\node [dot=black] at ($ (5) + (0,-.3) $) {};	
\node [dot=black] at ($ (6) + (0,-.3) $) {};

\draw (1)	-- ++(0,-.7)		-| (2);	
\draw (3)	-- ++(0,-.7)		-| (5);	
\draw (4)	-- ++(0,-1.1)		-| (6);	
\end{tikzpicture}
\qquad\qquad
\begin{tikzpicture}[thick,font=\small, node distance=1cm and 5mm, baseline={([yshift=-.5ex]current bounding box.center)}]
\tikzset{dot/.style={circle,fill=#1,inner sep=0,minimum size=4pt}}

\node (1) 					{X};
\node (2)	[right=of 1] 	{X};
\node (3) 	[right=of 2]	{X};
\node (4)	[right=of 3] 	{Y};
\node (5)	[right=of 4] 	{X};
\node (6)	[right=of 5] 	{Y};

\node [dot=black] at ($ (1) + (0,-.3) $) {};	
\node [dot=black] at ($ (2) + (0,-.3) $) {};	
\node [dot=black] at ($ (3) + (0,-.3) $) {};	
\node [dot=black] at ($ (4) + (0,-.3) $) {};
\node [dot=black] at ($ (5) + (0,-.3) $) {};	
\node [dot=black] at ($ (6) + (0,-.3) $) {};

\draw (1)	-- ++(0,-1.1)		-| (3);	
\draw (2)	-- ++(0,-.7)		-| (5);	
\draw (4)	-- ++(0,-1.1)		-| (6);	
\end{tikzpicture}
\qquad\qquad
\begin{tikzpicture}[thick,font=\small, node distance=1cm and 5mm, baseline={([yshift=-.5ex]current bounding box.center)}]
\tikzset{dot/.style={circle,fill=#1,inner sep=0,minimum size=4pt}}

\node (1) 					{X};
\node (2)	[right=of 1] 	{X};
\node (3) 	[right=of 2]	{X};
\node (4)	[right=of 3] 	{Y};
\node (5)	[right=of 4] 	{X};
\node (6)	[right=of 5] 	{Y};

\node [dot=black] at ($ (1) + (0,-.3) $) {};	
\node [dot=black] at ($ (2) + (0,-.3) $) {};	
\node [dot=black] at ($ (3) + (0,-.3) $) {};	
\node [dot=black] at ($ (4) + (0,-.3) $) {};
\node [dot=black] at ($ (5) + (0,-.3) $) {};	
\node [dot=black] at ($ (6) + (0,-.3) $) {};

\draw (1)	-- ++(0,-1.1)		-| (5);	
\draw (2)	-- ++(0,-.7)		-| (3);	
\draw (4)	-- ++(0,-.7)		-| (6);	
\end{tikzpicture}
}
\end{example*}

Consider now $x_1,\dots, x_n \in \{X,Y\}$. Then we still have 
\[ \ev{x_1 \dots x_n} =\#\{ \text{pairings which connect $X$ with $X$ and $Y$ with $Y$}  \}.
\]
Can we decide in a more abstract way whether $x_i=x_j$ or $x_i \neq x_j$? Yes, we can read this from the corresponding second moment, since
\[ \ev{x_ix_j} = \begin{cases}
\ev{x_i^2} = 1, &\mbox{} x_i = x_j, \\
\ev{x_i}\ev{x_j} = 0, &\mbox{} x_i \neq x_j. \\
\end{cases}
\]
Hence we have:
\begin{align*}
\ev{x_1 \cdots x_n}
&= \sum_{\pi\in \pair(n) }  \prod_{(i,j) \in \pi} \ev{x_ix_j}
\end{align*}
\end{remark}

\begin{theorem}[Wick 1950, physics; Isserlis 1918, statistics]\label{thm:2.8}
Let $Y_1,\dots, Y_p$ be independent standard Gaussian random variables and consider 
$x_1,\dots, x_n \in\{Y_1,\dots, Y_p \}$. Then we have the \emph{Wick formula}
\[ \ev{x_1 \cdots x_n} = \sum_{\pi\in \pair(n) }  \evpartition{\pi}{x_1, \dots ,x_n},
\]
where, for $\pi \in \pair(n)$, we use the notation
\[\evpartition{\pi}{x_1, \dots, x_n} = \prod_{(i,j) \in \pi} \ev{x_ix_j}.
\]
\end{theorem}

Note that the Wick formula is linear in the $x_i$, hence it remains valid if we replace the $x_i$ by linear combinations of the $x_j$. In particular, we can go over to complex Gaussian variables.

\begin{definition}
A \emph{standard complex Gaussian} random variable $Z$ is of the form
\[ Z = \frac{X+iY}{\sqrt{2}},
\]
where $X$ and $Y$ are independent standard real Gaussian variables.
\end{definition}

\begin{remark}
Let $Z$ be a standard complex Gaussian, i.e., $Z = \frac{X+iY}{\sqrt{2}}$.
Then we have $\bar Z= \frac{X-iY}{\sqrt{2}}$ and the first and second moments are given by
\begin{itemize}
\item
$\ev{Z} 
= 0=
\ev{\bar{Z}}$
\item
$\ev{Z^2}
= \ev{ZZ}
= \frac{1}{2} \left[\ev{XX} - \ev{YY} + i \left( \ev{XY} + \ev{YX}  \right)  \right] = 0$
\item
$\ev{\bar{Z}^2} =0$
\item
$\ev{\vert {Z}\vert^2}
= \ev{Z\bar{Z}}
= \frac{1}{2} \left[ \ev{XX} + \ev{YY}+i\left( \ev{YX} - \ev{YX}  \right)   \right]
=1$
\end{itemize}
Hence, for $z_1,z_2 \in \{ Z,\bar{Z} \}$ and 
$ \pi = \begin{tikzpicture}[thick,font=\small, node distance=1cm and 5mm, baseline={([yshift=-.5ex]current bounding box.center)}]
\tikzset{dot/.style={circle,fill=#1,inner sep=0,minimum size=4pt}}

\node (1) 					{$z_1$};
\node (2)	[right=of 1] 	{$z_2$};

\node [dot=black] at ($ (1) + (0,-.3) $) {};	
\node [dot=black] at ($ (2) + (0,-.3) $) {};

\draw (1)	-- ++(0,-.7)		-| (2);	
\end{tikzpicture}
$
we have
\[ \ev{z_1z_2} = \begin{cases}
1, &\mbox{} \text{$\pi$ connects $Z$ with $\bar{Z}$,} \\
0, &\mbox{} \text{$\pi$ connects ($Z$ with $Z$) or ($\bar{Z}$ with $\overline{Z}$).}
\end{cases}
\]
\end{remark}

\begin{theorem}\label{thm:2.11}
Let $Z_1, \dots, Z_p$ be independent standard complex Gaussian random variables and consider $z_1, \dots, z_n \in \{ Z_1, \bar{Z_1}, \dots, Z_p, \bar{Z_p} \}$. Then we have the Wick formula
\begin{align*}
\ev{z_1 \cdots z_n}
&= \sum_{\pi \in \pair(n)}  \evpartition{\pi}{z_1,\dots,z_n} \\
&= \#\{\text{pairings of $[n]$ which connect $Z_i$ with $\bar{Z_i}$}\}.
\end{align*}
\end{theorem}

\section{Gaussian random matrices and genus expansion}

Now we are ready to consider random matrices with such complex Gaussians as entries.

\begin{definition}
A \emph{Gaussian random matrix} is of the form $A_N = \frac{1}{\sqrt{N}} \left(a_{ij}\right)_{i,j=1}^N$, where
\begin{itemize}
\item $A_N =A_N^*$, i.e., $a_{ij} = {\bar a_{ji}}$ for all $i,j$,
\item $\{ a_{ij} \,|\, i \geq j \}$ are independent,
\item each $a_{ij}$ is a standard Gaussian random variable, which is complex for $i\neq j$ and real for $i=j$.
\end{itemize}
\end{definition}

\begin{remark}
\begin{enumerate}
\item More precisely, we should address the above as \emph{selfadjoint} Gaussian random matrices.
\item Another common name for those random matrices is \emph{\GUE}, which stands for \emph{Gaussian unitary ensemble}. \enquote{Unitary} corresponds here to the fact that the entries are complex, since such matrices are invariant under unitary transformations. 
With $\GUEN$ we denote the \GUE\ of size $N\times N$.
There are also real and quaternionic versions, \emph{Gaussian orthogonal ensembles} \GOE, and \emph{Gaussian symplectic ensembles} \GSE.
\item Note that we can also express this definition in terms of the Wick formula \ref{thm:2.11} as 
\begin{align*}
\ev{a_{i(1)j(1)} \cdots a_{i(n)j(n)} } 
&= \sum_{\pi \in \pair(n)}   \evpartition{\pi}{a_{i(1)j(1)}, \dots, a_{i(n)j(n)}},
\end{align*}
for all $n$ and $1 \leq i(1),j(1), \dots, i(n),j(n) \leq N$, and where the second moments are given by
\begin{align*}
\ev{a_{ij}a_{kl}}
&= \delta_{il} \delta_{jk}.
\end{align*}
So we have for example for the fourth moment
\begin{align*}
\ev{a_{i(1)j(1)}a_{i(2)j(2)}a_{i(3)j(3)}a_{i(4)j(4)}}=
&\delta_{i(1)j(2)}\delta_{j(1)i(2)}\delta_{i(3)j(4)}\delta_{j(3)i(4)}\\
&+ \delta_{i(1)j(3)}\delta_{j(1)i(3)}\delta_{i(2)j(4)}\delta_{j(2)i(4)}\\
&+ \delta_{i(1)j(4)}\delta_{j(1)i(4)}\delta_{i(2)j(3)}\delta_{j(2)i(3)},
\end{align*}
and more concretely,
$\ev{a_{12}a_{21}a_{11}a_{11}}=1$ and $\ev{a_{12}a_{12}a_{21}a_{21}}=2$.

\end{enumerate}
\end{remark}

\begin{remark}[Calculation of $\ev{\tr(A_N^m)}$]\label{rem:2.14}
For our Gaussian random matrix we want to calculate their moments
\begin{align*}
\ev{\tr(A_N^m)}
&= \frac{1}{N}  \frac{1}{N^{m/2}} \sum_{i(1), \dots, i(m) = 1}^N \ev{a_{i(1)i(2)}a_{i(2)i(3)} \cdots a_{i(m)i(1)} }.
\end{align*}
Let us first consider small examples before we treat the general case:
\begin{enumerate}
\item $$ \ev{\tr(A_N^2)}
= \frac{1}{N^2}   \sum_{i,j = 1}^N \underbrace{\ev{a_{ij}a_{ji} }}_{=1} 
=\frac{1}{N^2} N^2 = 1, $$
and hence: $\ev{\tr(A_N^2)}=1=C_1$ for all $N$.
\item We consider the partitions

\resizebox{.9\textwidth}{!}{%  
$\pi_1=$  \begin{tikzpicture}[thick,font=\small, node distance=1cm and 5mm, baseline={([yshift=-.5ex]current bounding box.center)}, scale=0.8]
\tikzset{dot/.style={circle,fill=#1,inner sep=0,minimum size=4pt}}

\node (1) 					{$1$};
\node (2)	[right=of 1] 	{$2$};
\node (3) 	[right=of 2]	{$3$};
\node (4)	[right=of 3] 	{$4$};

\node [dot=black] at ($ (1) + (0,-.3) $) {};	
\node [dot=black] at ($ (2) + (0,-.3) $) {};	
\node [dot=black] at ($ (3) + (0,-.3) $) {};	
\node [dot=black] at ($ (4) + (0,-.3) $) {};

\draw (1)	-- ++(0,-.7)		-| (2);	
\draw (3)	-- ++(0,-.7)		-| (4);	
\end{tikzpicture},
$\quad\pi_2 =$ \begin{tikzpicture}[thick,font=\small, node distance=1cm and 5mm, baseline={([yshift=-.5ex]current bounding box.center)},scale=0.8]
\tikzset{dot/.style={circle,fill=#1,inner sep=0,minimum size=4pt}}

\node (1) 					{$1$};
\node (2)	[right=of 1] 	{$2$};
\node (3) 	[right=of 2]	{$3$};
\node (4)	[right=of 3] 	{$4$};

\node [dot=black] at ($ (1) + (0,-.3) $) {};	
\node [dot=black] at ($ (2) + (0,-.3) $) {};	
\node [dot=black] at ($ (3) + (0,-.3) $) {};	
\node [dot=black] at ($ (4) + (0,-.3) $) {};

\draw (1)	-- ++(0,-.7)		-| (3);	
\draw (2)	-- ++(0,-1.1)		-| (4);	
\end{tikzpicture},
$\quad\pi_3 =$ \begin{tikzpicture}[thick,font=\small, node distance=1cm and 5mm, baseline={([yshift=-.5ex]current bounding box.center)},scale=0.8]
\tikzset{dot/.style={circle,fill=#1,inner sep=0,minimum size=4pt}}

\node (1) 					{$1$};
\node (2)	[right=of 1] 	{$2$};
\node (3) 	[right=of 2]	{$3$};
\node (4)	[right=of 3] 	{$4$};

\node [dot=black] at ($ (1) + (0,-.3) $) {};	
\node [dot=black] at ($ (2) + (0,-.3) $) {};	
\node [dot=black] at ($ (3) + (0,-.3) $) {};	
\node [dot=black] at ($ (4) + (0,-.3) $) {};

\draw (1)	-- ++(0,-1.1)		-| (4);	
\draw (2)	-- ++(0,-.7)		-| (3);	
\end{tikzpicture}.
}

With this, we have
\[\ev{\tr(A_N^4)}
= \frac{1}{N^3}  \sum_{i,j,k,l = 1}^N  \underbrace{\ev{a_{ij}a_{jk}a_{kl}a_{li} }}_{=\evpartition{\pi_1}{\dots}+\evpartition{\pi_2}{\dots}+\evpartition{\pi_3}{\dots}}
\]
and calculate
\begin{align*}
\sum_{i,j,k,l = 1}^N\evpartition{\pi_1}{a_{ij},a_{jk},a_{kl},a_{li}}
&= \sum_{\stackrel{i,j,k,l = 1}{i=k}}^N 1 = N^3, \\
\sum_{i,j,k,l = 1}^N\evpartition{\pi_2}{a_{ij},a_{jk},a_{kl},a_{li}}
&= \sum_{\stackrel{i,j,k,l = 1}{j=l}}^N 1 = N^3, \\
\sum_{i,j,k,l = 1}^N\evpartition{\pi_3}{a_{ij},a_{jk},a_{kl},a_{li}}
&= \sum_{\stackrel{i,j,k,l = 1}{i=l, j=k, j=i, k=l}}^N 1 
= \sum_{i=1}^N 1 = N,
\end{align*}
hence 
\[ \ev{\tr(A_N^4)}
= \frac{1}{N^3}  \left( N^3+N^3+N  \right)
= 2+ \frac{1}{N^2}.
\]
So we have
\[\lim\limits_{N\to\infty}  \ev{\tr(A_N^4)} = 2 = C_2.
\]
\item
Let us now do the general case.
\begin{align*}
\ev{a_{i(1)i(2)}a_{i(2)i(3)} \cdots a_{i(m)i(1)} }
&= \sum_{\pi\in\pair(m)} \evpartition{\pi}{a_{i(1)i(2)},a_{i(2)i(3)}, \dots, a_{i(m)i(1)} } \\
&= \sum_{\pi\in\pair(m)} \prod_{(k,l) \in \pi} \ev{a_{i(k)i(k+1)}a_{i(l)i(l+1)} }. 
\end{align*}
We use the notation $[i=j] = \delta_{ij}$ and, by identifying a pairing $\pi$ with a permutation $\pi \in S_m$ via
\[ (k,l) \in \pi \leftrightarrow \pi(k)=l, \pi(l) =k,
\]
find that
\begin{align*}
\ev{\tr(A_N^m)}
&= \frac{1}{N^{m/2+1}} \sum_{i(1), \dots, i(m) = 1}^N
	\sum_{\pi\in\pair(m)} \prod_{(k,l) \in \pi} \ev{a_{i(k)i(k+1)}a_{i(l)i(l+1)} } \\
&= \frac{1}{N^{m/2+1}} \sum_{\pi\in\pair(m)} \sum_{i(1), \dots, i(m) = 1}^N
	 \prod_{k} \bigl[ i(k) = i(\underbrace{\pi(k)+1}_{\gamma \pi(k)})  \bigr],
\end{align*}
where $\gamma = (1,2, \dots, m) \in S_m$ is the shift by $1$ modulo $m$. The above product is different from $0$ if and only if $ i \colon [m] \to [N]$
is constant on the cycles of $\gamma\pi \in S_m$. Thus we get finally
\begin{align*}
\ev{\tr(A_N^m)}
&= \frac{1}{N^{m/2+1}} \sum_{\pi\in\pair(m)} N^{\#(\gamma\pi)},
\end{align*}
where $\#(\gamma\pi)$ is the number of cycles of the permutation $\gamma\pi$.
\end{enumerate}
\end{remark}

Hence we have proved the following.

\begin{theorem}\label{thm:2.15}
Let $A_N$ be a \GUEN\ random matrix. Then we have for all $m\in\N$,
\begin{align*}
\ev{\tr(A_N^m)}
&=  \sum_{\pi\in\pair(m)} N^{\#(\gamma\pi)- \frac{m}{2} -1}.
\end{align*}
\end{theorem}

\begin{example}\label{ex:2.16}
\begin{enumerate}
\item This says in particular that all odd moments are zero, since $\pair(2k+1)  = \emptyset$.
\item Let $m=2$, then $\gamma = (1,2)$ and we have only one $\pi = (1,2)$; then $\gamma\pi = \id = (1) (2)$, and thus
$\#(\gamma\pi) =2$ and
\[ \#(\gamma\pi) - \frac{m}{2} - 1 = 0.
\]
Thus,
\[ \ev{ \tr \left( A_N^2 \right) } = N^0 =1.
\]
\item Let $m=4$ and $\gamma =(1,2,3,4)$. Then there are three $\pi\in\pair(4)$ with the following contributions:
\[\begin{array}{l|l|l|l}
\pi & \gamma\pi & \# (\gamma\pi) -3 & \text{contribution} \\
\hline
(1,2)(34) & (1,3)(2)(4) & 0& N^0 = 1\\
(13)(24)  & (1,4,3,2) & -2& N^{-2} = \frac{1}{N^2} \\
(14)(23) & (1)(2,4)(3)& 0& N^0=1\\
\end{array}
\]
so that
\[ \ev{ \tr \left( A_N^4 \right) } = 2+ \frac{1}{N^2}.
\]
\item In the same way one can calculate that
\begin{align*}
\ev{ \tr \left( A_N^6 \right) } &= 5+ 10\frac{1}{N^2}, \\
\ev{ \tr \left( A_N^8 \right) } &= 14+ 70\frac{1}{N^2} + 21 \frac{1}{N^4}.
\end{align*}
\item For $m=6$ the following $5$ pairings give contribution $N^0$:
\[ \begin{tikzpicture}[thick,font=\small, node distance=1cm and 5mm, baseline={([yshift=-.5ex]current bounding box.center)}]
\tikzset{dot/.style={circle,fill=#1,inner sep=0,minimum size=4pt}}

\node (1) 					{$1$};
\node (2)	[right=of 1] 	{$2$};
\node (3) 	[right=of 2]	{$3$};
\node (4)	[right=of 3] 	{$4$};
\node (5)	[right=of 4] 	{$5$};
\node (6)	[right=of 5] 	{$6$};

\node [dot=black] at ($ (1) + (0,-.3) $) {};	
\node [dot=black] at ($ (2) + (0,-.3) $) {};	
\node [dot=black] at ($ (3) + (0,-.3) $) {};	
\node [dot=black] at ($ (4) + (0,-.3) $) {};	
\node [dot=black] at ($ (5) + (0,-.3) $) {};	
\node [dot=black] at ($ (6) + (0,-.3) $) {};

\draw (1)	-- ++(0,-.7)		-| (2);	
\draw (3)	-- ++(0,-.7)		-| (4);	
\draw (5)	-- ++(0,-.7)		-| (6);	
\end{tikzpicture}
\]
\[ \begin{tikzpicture}[thick,font=\small, node distance=1cm and 5mm, baseline={([yshift=-.5ex]current bounding box.center)}]
\tikzset{dot/.style={circle,fill=#1,inner sep=0,minimum size=4pt}}

\node (1) 					{$1$};
\node (2)	[right=of 1] 	{$2$};
\node (3) 	[right=of 2]	{$3$};
\node (4)	[right=of 3] 	{$4$};
\node (5)	[right=of 4] 	{$5$};
\node (6)	[right=of 5] 	{$6$};

\node [dot=black] at ($ (1) + (0,-.3) $) {};	
\node [dot=black] at ($ (2) + (0,-.3) $) {};	
\node [dot=black] at ($ (3) + (0,-.3) $) {};	
\node [dot=black] at ($ (4) + (0,-.3) $) {};	
\node [dot=black] at ($ (5) + (0,-.3) $) {};	
\node [dot=black] at ($ (6) + (0,-.3) $) {};

\draw (1)	-- ++(0,-1.1)		-| (4);	
\draw (2)	-- ++(0,-.7)		-| (3);	
\draw (5)	-- ++(0,-.7)		-| (6);	
\end{tikzpicture}
\]
\[ \begin{tikzpicture}[thick,font=\small, node distance=1cm and 5mm, baseline={([yshift=-.5ex]current bounding box.center)}]
\tikzset{dot/.style={circle,fill=#1,inner sep=0,minimum size=4pt}}

\node (1) 					{$1$};
\node (2)	[right=of 1] 	{$2$};
\node (3) 	[right=of 2]	{$3$};
\node (4)	[right=of 3] 	{$4$};
\node (5)	[right=of 4] 	{$5$};
\node (6)	[right=of 5] 	{$6$};

\node [dot=black] at ($ (1) + (0,-.3) $) {};	
\node [dot=black] at ($ (2) + (0,-.3) $) {};	
\node [dot=black] at ($ (3) + (0,-.3) $) {};	
\node [dot=black] at ($ (4) + (0,-.3) $) {};	
\node [dot=black] at ($ (5) + (0,-.3) $) {};	
\node [dot=black] at ($ (6) + (0,-.3) $) {};

\draw (1)	-- ++(0,-.7)		-| (2);	
\draw (3)	-- ++(0,-1.1)		-| (6);	
\draw (4)	-- ++(0,-.7)		-| (5);	
\end{tikzpicture}
\]
\[ \begin{tikzpicture}[thick,font=\small, node distance=1cm and 5mm, baseline={([yshift=-.5ex]current bounding box.center)}]
\tikzset{dot/.style={circle,fill=#1,inner sep=0,minimum size=4pt}}

\node (1) 					{$1$};
\node (2)	[right=of 1] 	{$2$};
\node (3) 	[right=of 2]	{$3$};
\node (4)	[right=of 3] 	{$4$};
\node (5)	[right=of 4] 	{$5$};
\node (6)	[right=of 5] 	{$6$};

\node [dot=black] at ($ (1) + (0,-.3) $) {};	
\node [dot=black] at ($ (2) + (0,-.3) $) {};	
\node [dot=black] at ($ (3) + (0,-.3) $) {};	
\node [dot=black] at ($ (4) + (0,-.3) $) {};	
\node [dot=black] at ($ (5) + (0,-.3) $) {};	
\node [dot=black] at ($ (6) + (0,-.3) $) {};

\draw (1)	-- ++(0,-1.1)		-| (6);	
\draw (2)	-- ++(0,-.7)		-| (3);	
\draw (4)	-- ++(0,-.7)		-| (5);	
\end{tikzpicture}
\]
\[ \begin{tikzpicture}[thick,font=\small, node distance=1cm and 5mm, baseline={([yshift=-.5ex]current bounding box.center)}]
\tikzset{dot/.style={circle,fill=#1,inner sep=0,minimum size=4pt}}

\node (1) 					{$1$};
\node (2)	[right=of 1] 	{$2$};
\node (3) 	[right=of 2]	{$3$};
\node (4)	[right=of 3] 	{$4$};
\node (5)	[right=of 4] 	{$5$};
\node (6)	[right=of 5] 	{$6$};

\node [dot=black] at ($ (1) + (0,-.3) $) {};	
\node [dot=black] at ($ (2) + (0,-.3) $) {};	
\node [dot=black] at ($ (3) + (0,-.3) $) {};	
\node [dot=black] at ($ (4) + (0,-.3) $) {};	
\node [dot=black] at ($ (5) + (0,-.3) $) {};	
\node [dot=black] at ($ (6) + (0,-.3) $) {};

\draw (1)	-- ++(0,-1.5)		-| (6);	
\draw (2)	-- ++(0,-1.1)		-| (5);	
\draw (3)	-- ++(0,-.7)		-| (4);	
\end{tikzpicture}
\]
Those are non-crossing pairings, all other pairings $\pi\in\pair(6)$ have a crossing, e.g.:
\[ \begin{tikzpicture}[thick,font=\small, node distance=1cm and 5mm, baseline={([yshift=-.5ex]current bounding box.center)}]
\tikzset{dot/.style={circle,fill=#1,inner sep=0,minimum size=4pt}}

\node (1) 					{$1$};
\node (2)	[right=of 1] 	{$2$};
\node (3) 	[right=of 2]	{$3$};
\node (4)	[right=of 3] 	{$4$};
\node (5)	[right=of 4] 	{$5$};
\node (6)	[right=of 5] 	{$6$};

\node [dot=black] at ($ (1) + (0,-.3) $) {};	
\node [dot=black] at ($ (2) + (0,-.3) $) {};	
\node [dot=black] at ($ (3) + (0,-.3) $) {};	
\node [dot=black] at ($ (4) + (0,-.3) $) {};	
\node [dot=black] at ($ (5) + (0,-.3) $) {};	
\node [dot=black] at ($ (6) + (0,-.3) $) {};

\draw (1)	-- ++(0,-1.5)		-| (5);	
\draw (2)	-- ++(0,-1.1)		-| (6);	
\draw (3)	-- ++(0,-.7)		-| (4);	
\end{tikzpicture}
\]
\end{enumerate}
\end{example}

\section{Non-crossing pairings}

\begin{definition}
A pairing $\pi \in\pair(m)$ is \emph{non-crossing} (\emph{NC}) if there are no pairs $(i,k)$ and $(j,l)$ in $\pi$ with $i < j < k <l$, i.e., we don't have a crossing in $\pi$.
\[ \begin{tikzpicture}[thick,font=\small, node distance=1cm and 5mm, baseline={([yshift=-.5ex]current bounding box.center)}]
\tikzset{dot/.style={circle,fill=#1,inner sep=0,minimum size=4pt}}

\node (1) 					{$i$};
\node (2)	[right=of 1] 	{$j$};
\node (3) 	[right=of 2]	{$k$};
\node (4)	[right=of 3] 	{$k$};

\node [dot=black] at ($ (1) + (0,-.3) $) {};	
\node [dot=black] at ($ (2) + (0,-.3) $) {};	
\node [dot=black] at ($ (3) + (0,-.3) $) {};	
\node [dot=black] at ($ (4) + (0,-.3) $) {};

\draw (1)	-- ++(0,-.7)		-| (3);	
\draw (2)	-- ++(0,-1.1)		-| (4);	
\end{tikzpicture} \text{ is not allowed!}
\]
We put
\[ \ncpair(m) = \left\{  \pi \in \pair(m) \, | \, \pi \text{ is non-crossing}  \right\}.
\]
\end{definition}

\begin{example}
\begin{enumerate}
\item $\ncpair(2)=\pair(2) = \{   
\begin{tikzpicture}[thick,font=\small, node distance=6mm and 3mm, baseline={([yshift=-1.5ex]current bounding box.center)}]

\node (1) 					{};
\node (2)	[right=of 1] 	{};

\draw (1)	-- ++(0,-.4)	-| (2);	
\end{tikzpicture}
\}$
\item $\ncpair(4)= \{   
\begin{tikzpicture}[thick,font=\small, node distance=6mm and 3mm, baseline={([yshift=-1.5ex]current bounding box.center)}]

\node (1) 					{};
\node (2)	[right=of 1] 	{};
\node (3) 	[right=of 2]	{};
\node (4)	[right=of 3] 	{};

\draw (1)	-- ++(0,-.4)	-| (2);	
\draw (3)	-- ++(0,-.4)	-| (4);	
\end{tikzpicture}, 
\begin{tikzpicture}[thick,font=\small, node distance=6mm and 3mm, baseline={([yshift=-1.5ex]current bounding box.center)}]

\node (1) 					{};
\node (2)	[right=of 1] 	{};
\node (3) 	[right=of 2]	{};
\node (4)	[right=of 3] 	{};

\draw (1)	-- ++(0,-.7)	-| (4);	
\draw (2)	-- ++(0,-.4)	-| (3);	
\end{tikzpicture}
\}$ 
and $\pair(4) \backslash \ncpair(4) = \{\begin{tikzpicture}[thick,font=\small, node distance=6mm and 3mm, baseline={([yshift=-1.5ex]current bounding box.center)}]

\node (1) 					{};
\node (2)	[right=of 1] 	{};
\node (3) 	[right=of 2]	{};
\node (4)	[right=of 3] 	{};

\draw (1)	-- ++(0,-.4)	-| (3);	
\draw (2)	-- ++(0,-.7)	-| (4);	
\end{tikzpicture}
\}$
\item The 5 elements of $\ncpair(6)$ are given in Example \ref{ex:2.16} (v), $\pair(6)$ contains 15 elements; thus there are $15-5=10$ more elements in $\pair(6)$ with crossings.
\end{enumerate}
\end{example}

\begin{remark}\label{rem:2.19}
Note that NC-pairings have a recursive structure, which usually is crucial for dealing with them.
\begin{enumerate}
\item The first pair of $\pi \in \ncpair(2k)$ must necessarily be of the form $(1,2l)$
		and the remaining pairs can only pair within $\{2, \dots, 2l-1 \}$ or within
		$\{ 2l+1, \dots, 2l \}$.
\[ \begin{tikzpicture}[thick,font=\small, node distance=1cm and 5mm, baseline={([yshift=-.5ex]current bounding box.center)}]
\tikzset{dot/.style={circle,fill=#1,inner sep=0,minimum size=4pt}}

\node (1) 					{$1$};
\node (2) 	[right=of 1] 	{};
\node (3)	[right=of 2] 	{};
\node (4) 	[right=of 3]	{};
\node (5)	[right=of 4] 	{};
\node (6)	[right=of 5] 	{$2l$};
\node (7)	[right=of 6] 	{};
\node (8)	[right=of 7] 	{};
\node (9)	[right=of 8] 	{};
\node (10)	[right=of 9] 	{};

\node [dot=black] at ($ (1) + (0,-.3) $) {};	
\node [dot=black] at ($ (2) + (0,-.3) $) {};
\node [] at ($ (3) + (0,-.3) $) {$\dots$};		
\node [] at ($ (4) + (0,-.3) $) {$\dots$};	
\node [dot=black] at ($ (5) + (0,-.3) $) {};
\node [dot=black] at ($ (6) + (0,-.3) $) {};
\node [dot=black] at ($ (7) + (0,-.3) $) {};	
\node [] at ($ (8) + (0,-.3) $) {$\dots$};		
\node [] at ($ (9) + (0,-.3) $) {$\dots$};	
\node [dot=black] at ($ (10) + (0,-.3) $) {};

\draw (1)	-- ++(0,-1.1)		-| (6);	
\end{tikzpicture}
\]
\item Iterating this shows that we must find in any $\pi\in\ncpair(2k)$ at least one pair of the form $(i,i+1)$ with $1\leq i \leq 2k-1$. Removing this pair gives a NC-pairing of $2k-2$ points.
This characterizes the NC-pairings as those pairings, which can be reduced to the empty set by iterated removal of pairs, which consist of neighbors. 
\end{enumerate}

An example for the reduction of a non-crossing pairing is the following.
\[ \begin{tikzpicture}[thick,font=\small, node distance=1cm and 5mm, baseline={([yshift=-.5ex]current bounding box.center)}]
\tikzset{dot/.style={circle,fill=#1,inner sep=0,minimum size=4pt}}

\node (1) 					{$1$};
\node (2)	[right=of 1] 	{$2$};
\node (3) 	[right=of 2]	{$3$};
\node (4)	[right=of 3] 	{$4$};
\node (5)	[right=of 4] 	{$5$};
\node (6)	[right=of 5] 	{$6$};
\node (7)	[right=of 6] 	{$7$};
\node (8)	[right=of 7] 	{$8$};

\node [dot=black] at ($ (1) + (0,-.3) $) {};	
\node [dot=black] at ($ (2) + (0,-.3) $) {};	
\node [dot=black] at ($ (3) + (0,-.3) $) {};	
\node [dot=black] at ($ (4) + (0,-.3) $) {};	
\node [dot=black] at ($ (5) + (0,-.3) $) {};	
\node [dot=black] at ($ (6) + (0,-.3) $) {};
\node [dot=black] at ($ (7) + (0,-.3) $) {};
\node [dot=black] at ($ (8) + (0,-.3) $) {};

\draw (1)	-- ++(0,-1.5)		-| (8);	
\draw (2)	-- ++(0,-1.1)		-| (5);	
\draw (3)	-- ++(0,-.7)		-| (4);	
\draw (6)	-- ++(0,-.7)		-| (7);	
\end{tikzpicture}
\to
\begin{tikzpicture}[thick,font=\small, node distance=1cm and 5mm, baseline={([yshift=-.5ex]current bounding box.center)}]
\tikzset{dot/.style={circle,fill=#1,inner sep=0,minimum size=4pt}}

\node (1) 					{$1$};
\node (2)	[right=of 1] 	{$2$};
\node (3) 	[right=of 2]	{$5$};
\node (4)	[right=of 3] 	{$8$};

\node [dot=black] at ($ (1) + (0,-.3) $) {};	
\node [dot=black] at ($ (2) + (0,-.3) $) {};	
\node [dot=black] at ($ (3) + (0,-.3) $) {};	
\node [dot=black] at ($ (4) + (0,-.3) $) {};

\draw (1)	-- ++(0,-1.1)		-| (4);	
\draw (2)	-- ++(0,-.7)		-| (3);	
\end{tikzpicture}
\to
\begin{tikzpicture}[thick,font=\small, node distance=1cm and 5mm, baseline={([yshift=-.5ex]current bounding box.center)}]
\tikzset{dot/.style={circle,fill=#1,inner sep=0,minimum size=4pt}}

\node (1) 					{$1$};
\node (2)	[right=of 1] 	{$8$};

\node [dot=black] at ($ (1) + (0,-.3) $) {};	
\node [dot=black] at ($ (2) + (0,-.3) $) {};

\draw (1)	-- ++(0,-.7)		-| (2);	
\end{tikzpicture}
\to \emptyset
\]

In the case of a crossing pairing, some reductions might be possible, but eventually one arrives at a point, where no further reduction can be done.
\[ \begin{tikzpicture}[thick,font=\small, node distance=1cm and 5mm, baseline={([yshift=-.5ex]current bounding box.center)}]
\tikzset{dot/.style={circle,fill=#1,inner sep=0,minimum size=4pt}}

\node (1) 					{$1$};
\node (2)	[right=of 1] 	{$2$};
\node (3) 	[right=of 2]	{$3$};
\node (4)	[right=of 3] 	{$4$};
\node (5)	[right=of 4] 	{$5$};
\node (6)	[right=of 5] 	{$6$};

\node [dot=black] at ($ (1) + (0,-.3) $) {};	
\node [dot=black] at ($ (2) + (0,-.3) $) {};	
\node [dot=black] at ($ (3) + (0,-.3) $) {};	
\node [dot=black] at ($ (4) + (0,-.3) $) {};	
\node [dot=black] at ($ (5) + (0,-.3) $) {};	
\node [dot=black] at ($ (6) + (0,-.3) $) {};

\draw (1)	-- ++(0,-1.)		-| (5);	
\draw (2)	-- ++(0,-.7)		-| (3);	
\draw (4)	-- ++(0,-.7)		-| (6);	
\end{tikzpicture}
\to
\begin{tikzpicture}[thick,font=\small, node distance=1cm and 5mm, baseline={([yshift=-.5ex]current bounding box.center)}]
\tikzset{dot/.style={circle,fill=#1,inner sep=0,minimum size=4pt}}

\node (1) 					{$1$};
\node (2)	[right=of 1] 	{$4$};
\node (3) 	[right=of 2]	{$5$};
\node (4)	[right=of 3] 	{$6$};

\node [dot=black] at ($ (1) + (0,-.3) $) {};	
\node [dot=black] at ($ (2) + (0,-.3) $) {};	
\node [dot=black] at ($ (3) + (0,-.3) $) {};	
\node [dot=black] at ($ (4) + (0,-.3) $) {};

\draw (1)	-- ++(0,-1.1)		-| (3);	
\draw (2)	-- ++(0,-.7)		-| (4);	
\end{tikzpicture}
\text{ no further reduction possible!}
\]

\end{remark}

\begin{prop}\label{prop:2.20}
Consider $m$ even and let $\pi\in\pair(m)$, which we identify with a permutation $\pi\in S_m$.
As before, $\gamma =(1,2, \dots, m) \in S_m$. Then we have:
\begin{enumerate}
\item $\#(\gamma\pi) - \frac{m}{2} -1 \leq 0$ for all $\pi \in \pair(m)$.
\item $\#(\gamma\pi) - \frac{m}{2}-1  = 0$  if and only if  $\pi \in \ncpair(m)$.
\end{enumerate}
\end{prop}

\begin{proof}
First we note that a pair $(i,i+1)$ in $\pi$ corresponds to a fixed point of $\gamma\pi$. More precisely, in such a situation we have $ i+1 \xrightarrow{\pi} i \xrightarrow{\gamma} i+1$ and $ i \xrightarrow{\pi} i+1 \xrightarrow{\gamma} i+2$. Hence $\gamma \pi$ contains the cycles $(i+1)$ and $(\dots, i ,i +2, \dots)$.

This implication also goes in the other direction:
If $\gamma\pi(i+1) = i+1$ then $\pi(i+1) = \gamma^{-1}(i+1) = i$. Since $\pi$ is a pairing we must then also have $\pi(i)=i+1$ and hence we have the pair $(i,i+1)$ in $\pi$.

If we have $(i,i+1)$ in $\pi$, we can remove the points $i$ and $i+1$, yielding another pairing $\tilde{\pi}$. By doing so, we remove in $\gamma\pi$ the cycle $(i+1)$ and we remove in the cycle  $(\dots, i ,i +2, \dots)$  the point $i$, yielding $\gamma \tilde{\pi}$. We reduce thus $m$ by $2$ and $\#(\gamma\pi)$ by $1$.

If $\pi$ is NC we can iterate this until we arrive at $\tilde{\pi}$ with $m=2$. Then we have $\tilde{\pi} = (1,2)$ and $\gamma = (1,2)$ such that $\gamma\tilde{\pi} = (1)(2)$ and $\#(\gamma\tilde{\pi}) = 2$.
If $m=2k$ we did $k-1$ reductions where we reduced in each step the number of cycles by $1$ and at the end we remain with $2$ cycles, hence
\[ \#(\gamma{\pi}) = (k-1) \cdot 1 +2 = k+1 = \frac{m}{2} + 1.
\]
Here is an example for this:

\resizebox{\textwidth}{!}{%  
$\pi =$ 
\begin{tikzpicture}[thick,font=\small, node distance=1cm and 5mm, baseline={([yshift=-.5ex]current bounding box.center)},scale=1]
\tikzset{dot/.style={circle,fill=#1,inner sep=0,minimum size=4pt}}

\node (1) 					{$1$};
\node (2)	[right=of 1] 	{$2$};
\node (3) 	[right=of 2]	{$3$};
\node (4)	[right=of 3] 	{$4$};
\node (5)	[right=of 4] 	{$5$};
\node (6)	[right=of 5] 	{$6$};
\node (7)	[right=of 6] 	{$7$};
\node (8)	[right=of 7] 	{$8$};

\node [dot=black] at ($ (1) + (0,-.3) $) {};	
\node [dot=black] at ($ (2) + (0,-.3) $) {};	
\node [dot=black] at ($ (3) + (0,-.3) $) {};	
\node [dot=black] at ($ (4) + (0,-.3) $) {};	
\node [dot=black] at ($ (5) + (0,-.3) $) {};	
\node [dot=black] at ($ (6) + (0,-.3) $) {};
\node [dot=black] at ($ (7) + (0,-.3) $) {};
\node [dot=black] at ($ (8) + (0,-.3) $) {};

\draw (1)	-- ++(0,-1.5)		-| (8);	
\draw (2)	-- ++(0,-1.1)		-| (5);	
\draw (3)	-- ++(0,-.7)		-| (4);	
\draw (6)	-- ++(0,-.7)		-| (7);	
\end{tikzpicture}
$\underset{(3,4), (6,7)}{\xrightarrow{\text{remove}}}$
\begin{tikzpicture}[thick,font=\small, node distance=1cm and 5mm, baseline={([yshift=-.5ex]current bounding box.center)}]
\tikzset{dot/.style={circle,fill=#1,inner sep=0,minimum size=4pt}}

\node (1) 					{$1$};
\node (2)	[right=of 1] 	{$2$};
\node (3) 	[right=of 2]	{$5$};
\node (4)	[right=of 3] 	{$8$};

\node [dot=black] at ($ (1) + (0,-.3) $) {};	
\node [dot=black] at ($ (2) + (0,-.3) $) {};	
\node [dot=black] at ($ (3) + (0,-.3) $) {};	
\node [dot=black] at ($ (4) + (0,-.3) $) {};

\draw (1)	-- ++(0,-1.1)		-| (4);	
\draw (2)	-- ++(0,-.7)		-| (3);	
\end{tikzpicture}
$\underset{(2,5)}{\xrightarrow{\text{remove}}}$
\begin{tikzpicture}[thick,font=\small, node distance=1cm and 5mm, baseline={([yshift=-.5ex]current bounding box.center)}]
\tikzset{dot/.style={circle,fill=#1,inner sep=0,minimum size=4pt}}

\node (1) 					{$1$};
\node (2)	[right=of 1] 	{$8$};

\node [dot=black] at ($ (1) + (0,-.3) $) {};	
\node [dot=black] at ($ (2) + (0,-.3) $) {};

\draw (1)	-- ++(0,-.7)		-| (2);	
\end{tikzpicture}
}

\[ \qquad\qquad \quad\gamma \pi = (1)(268)(35)(4)(7) \quad\qquad\longrightarrow \qquad (1)(28)(5) \qquad \longrightarrow \quad (1)(8)
\]

\bigskip For a general $\pi\in\pair(m)$ we remove cycles $(i,i+1)$ as long as possible. If $\pi$ is crossing we arrive at a pairing $\tilde{\pi}$, where this is not possible anymore.
It suffices to show that such a $\tilde{\pi}\in\pair(m)$ satisfies $\#(\gamma\tilde{\pi}) - \frac{m}{2} -1< 0$.
But since $\tilde{\pi}$ has no cycle $(i,i+1)$, $\gamma\tilde{\pi}$ has no fixed point.
Hence each cycle has at least $2$ elements, thus 
\[ \#(\gamma\tilde{\pi})  \leq \frac{m}{2} < \frac{m}{2}  +1.
\]
Note that in the above arguments, 
with $(i,i+1)$ we actually mean $(i,\gamma(i))$; thus also 
$(1,m)$ counts as a pair of neighbors for a $\pi\in \pair(m)$, in order to have the characterization of fixed points right. Hence, when reducing a general pairing to one without fixed points we have also to remove such cyclic neighbors as long as possible.
\end{proof}

\section{Semicircle law for GUE}

\begin{theorem}[Wigner's semicircle law for GUE, averaged version]\label{thm:2.21}
Let $A_N$ be a Gaussian (GUE) $N\times N$ random matrix. Then we have for all $m\in \N$:
\[ \lim\limits_{N\to\infty} \ev{\tr \left( A_N^m \right) }
= \frac{1}{2\pi} \int_{-2}^{2} x^m \sqrt{4-x^2} \td x.
\]
\end{theorem}

\begin{proof}
This is true for $m$ odd, since then both sides are equal to zero. Consider $m=2k$ even.
Then Theorem \ref{thm:2.15} and Proposition \ref{prop:2.20} show that
\begin{align*}
\lim\limits_{\N\to\infty} \ev{\tr \left( A_N^m \right) }
&= \sum_{\pi\in\pair(m) } \lim\limits_{N\to\infty} N^{ \#(\gamma\pi) - \frac{m}{2} -1 }
= \sum_{\pi\in\ncpair(m) } 1
= \#\ncpair(m).
\end{align*}
Since the moments of the semicircle are given by the Catalan numbers, it remains to see that
$\#\ncpair(2k)$ is equal to the Catalan number $C_k$. To see this,
we now count $d_k := \#\ncpair(2k)$ according to the recursive structure of NC-pairings as in \ref{rem:2.19} (i). 
\[ \begin{tikzpicture}[thick,font=\small, node distance=1cm and 5mm, baseline={([yshift=-.5ex]current bounding box.center)}]
\tikzset{dot/.style={circle,fill=#1,inner sep=0,minimum size=4pt}}

\node (1) 					{$1$};
\node (2) 	[right=of 1] 	{};
\node (3)	[right=of 2] 	{};
\node (4) 	[right=of 3]	{};
\node (5)	[right=of 4] 	{};
\node (6)	[right=of 5] 	{$2l$};
\node (7)	[right=of 6] 	{};
\node (8)	[right=of 7] 	{};
\node (9)	[right=of 8] 	{};
\node (10)	[right=of 9] 	{2k};

\node [dot=black] at ($ (1) + (0,-.3) $) {};	
\node [dot=black] at ($ (2) + (0,-.3) $) {};
\node [] at ($ (3) + (0,-.3) $) {$\dots$};		
\node [] at ($ (4) + (0,-.3) $) {$\dots$};	
\node [dot=black] at ($ (5) + (0,-.3) $) {};
\node [dot=black] at ($ (6) + (0,-.3) $) {};
\node [dot=black] at ($ (7) + (0,-.3) $) {};	
\node [] at ($ (8) + (0,-.3) $) {$\dots$};		
\node [] at ($ (9) + (0,-.3) $) {$\dots$};	
\node [dot=black] at ($ (10) + (0,-.3) $) {};

\draw (1)	-- ++(0,-1.1)		-| (6);	
\end{tikzpicture}
\]

\quad

Namely, we can identify $\pi \in \ncpair(2k)$ with $\{ (1,2l) \} \cup \pi_0 \cup \pi_1$, where
$l \in \{1, \dots, k \}$, $\pi_0 \in \ncpair(2(l-1))$ and $\pi_1 \in \ncpair(2(k-l))$.
Hence we have 
\begin{align*}
d_k
&= \sum_{l=1}^{k} d_{l-1} d_{k-l}, \qquad \text{where } d_0 = 1.
\end{align*}
This is the recursion for the Catalan numbers, whence $d_k = C_k$ for all $k \in \N$.
\end{proof}

\begin{remark}
\begin{enumerate}
\item One can refine 
\[ \#(\gamma \pi) - \frac{m}{2} - 1 \leq 0
\qquad\text{to}\qquad
 \#(\gamma \pi) - \frac{m}{2} - 1 =-2g(\pi)
\]
for $g(\pi) \in \N_0$. This $g$ has the meaning that it is the minimal genus of a surface on which $\pi$ can be drawn without crossings. NC pairings are also called \emph{planar}, they correspond to $g=0$. Theorem \ref{thm:2.15} is usually addressed as \emph{genus expansion},
\begin{align*}
\ev{\tr(A_N^m)}
&=  \sum_{\pi\in\pair(m)} N^{-2g(\pi)}.
\end{align*}
\item For example, $(1,2)(3,4) \in \ncpair(4)$ has $g=0$,
	but the crossing pairing $(1,3)(2,4)\in\pair(4)$ has genus $g=1$. It has a crossing in the plane but this can be avoided on a torus.	
\[ \begin{tikzpicture}	
\tikzset{dot/.style={circle,fill=#1,inner sep=0,minimum size=4pt}}

\node[dot=black] at (-1, 0)   (4) {};
\node[dot=black] at (1, 0)   (2) {};
\node[dot=black] at (0, -1)   (3) {};
\node[dot=black] at (0, 1)   (1) {};

\draw (-1.3,0) -- node {$4$} (-1.3,0);
\draw (1.3,0) -- node {$2$} (1.3,0);
\draw (0,-1.3) -- node {$3$} (0,-1.3);
\draw (0,1.3) -- node {$1$} (0,1.3);

\draw (0,0) circle (1);

\path [bend right] (1) edge (2);
\path [bend right] (3) edge (4);
\end{tikzpicture}
\qquad \qquad
\begin{tikzpicture}	
\tikzset{dot/.style={circle,fill=#1,inner sep=0,minimum size=4pt}}

\node[dot=black] at (-1, 0)   (4) {};
\node[dot=black] at (1, 0)   (2) {};
\node[dot=black] at (0, -1)   (3) {};
\node[dot=black] at (0, 1)   (1) {};

\draw (-1.3,0) -- node {$4$} (-1.3,0);
\draw (1.3,0) -- node {$2$} (1.3,0);
\draw (0,-1.3) -- node {$3$} (0,-1.3);
\draw (0,1.3) -- node {$1$} (0,1.3);

\draw (0,0) circle (1);

\draw (1) -- (3);
\draw (4) -- (-0.3,0);
\draw (2) -- (0.3,0);
\draw (0.3,0) arc(0:180:0.3);
\end{tikzpicture}
\]	
\item If we denote
\[ \varepsilon_g(k) = \#\left\{
\pi\in\pair(2k) \, | \, \pi \text{ has genus } g
\right\}
\]
then the genus expansion \ref{thm:2.15} can be written as
\begin{align*}
\ev{\tr(A_N^{2k})}
&=  \sum_{g \geq 0} \varepsilon_g(k) N^{-2g}.
\end{align*}
We know that 
\[ \varepsilon_g(0) = C_k = \frac{1}{k+1} \binom{2k}{k},
\]
but what about the $\varepsilon_g(k)$ for $g>0$?
There does not exist an explicit formula for them, but Harer and Zagier have shown in 1986 that
\[ \varepsilon_g(k) =  \frac{(2k)!}{(k+1)!(k-2g)!} \cdot \lambda_g(k),
\]
where $\lambda_g(k)$ is the coefficient of $x^{2g}$ in 
\[ \left( \frac{  \frac{x}{2}  }{ \tanh \frac{x}{2} }  \right)^{k+1}.
\]
We will come back later to this statement of Harer and Zagier; see Theorem \ref{thm:9.2}
\end{enumerate}
\end{remark}

\chapter{Wigner Matrices: Combinatorial Proof of Wigner's Semicircle Law}\label{ch:3}

Wigner's semicircle law does not only hold for Gaussian random matrices, but more general for so-called Wigner matrices; there we keep the independence and identical distribution of the entries, but allow arbitrary distribution instead of Gaussian. As there is no Wick formula any more, there is no clear advantage of the complex over the real case any more, hence we will consider in the following the real one.

\section{Wigner matrices}

\begin{definition}
Let $\mu$ be a probability distribution on $\R$. A corresponding \emph{Wigner random matrix} is of the form $A_N = \frac{1}{\sqrt{N}} \left( a_{ij} \right)_{i,j=1}^N$, where
\begin{itemize}
\item $A_N=A_N^*$, i.e., $a_{ij} = a_{ji}$ for all $i,j$,
\item $\{ a_{ij} \, | \, i \geq j \}$ are independent,
\item each $a_{ij}$ has distribution $\mu$.
\end{itemize}
\end{definition}

\begin{remark}
\begin{enumerate}
\item In our combinatorial setting we will assume that all moments of $\mu$ exist; that the first moment is $0$; and the second moment will be normalized to $1$. In an analytic setting one can deal with more general situations: usually only the existence of the second moment is needed; and one can also allow non-vanishing mean.
\item Often one also allows different distributions for the diagonal and the off-diagonal entries.
\item Even more general, one can give up the assumption of identical distribution of all entries and replace this by uniform bounds on their moments.
\item 
We will now try to imitate our combinatorial proof from the Gaussian case also in  this more general situation. Without a precise Wick formula for the higher moments of the entries, we will not aim at a precise genus expansion; it suffices to see that the leading contributions are still given by the Catalan numbers.
\end{enumerate}
\end{remark}

\section{Combinatorial description of moments of Wigner matrices}

Consider a Wigner matrix $A_N = \frac{1}{\sqrt{N}} \left( a_{ij} \right)_{i,j=1}^N$, where $\mu$ has all moments and
\[ \int_\R x \td \mu(x) = 0, \qquad \int_\R x^2 \td \mu(x) = 1.
\]
Then
$$
\ev{\tr(A_N^m)}
=  \frac{1}{N^{1+\frac{m}{2}}} \sum_{i_1,\dots,i_m=1}^{N} \ev{a_{i_1i_2}a_{i_2i_3} \cdots a_{i_mi_1}}
=  \frac{1}{N^{1+\frac{m}{2}}}  \sum_{\sigma \in \PP(m)} \, \sum_{\stackrel{i \colon [m] \to [N]}{\ker i = \sigma}} \ev{\sigma},
$$
where we group the appearing indices $(i_1,\dots,i_m)$ according to their ``kernel'', which is a ``partition $\sigma$ of $\{1,\dots,m\}$.

\begin{definition}
\begin{enumerate}
\item A \emph{partition} $\sigma$ of $[n]$ is a decomposition of $[n]$ into disjoint, non-empty subsets (of arbitrary size), i.e., $\sigma = \{ V_1, \dots, V_k \}$, where
\begin{itemize}
\item $V_i \subset [n]$ for all $i$,
\item $V_i \neq \emptyset$ for all $i$,
\item $V_i \cap V_j = \emptyset$ for all $i\neq j$,
\item $\bigcup_{i=1}^k V_i = [n]$.
\end{itemize}
The $V_i$ are called \emph{blocks} of $\sigma$. The set of all partitions of $[n]$ is denoted by
\[ \PP(n) := \left\{
\sigma \, | \, \sigma \text{ is a partition of } [n]
\right\}.
\]
\item For a multi-index $i=(i_1,\dots,i_m)$ we denote by $\ker i$ its \emph{kernel}; this is the partition $\sigma \in \PP(m)$ such that we have $i_k = i_l$ if and only if $k$ and $l$ are in the same block of $\sigma$.
If we identify $i$ with a function $i\colon[m] \to [N]$ via $i(k)=i_k$ then we can also write
\[ \ker i = \left\{
i^{-1}(1), i^{-1}(2), \dots, i^{-1}(N)
\right\},
\]
where we discard all empty sets.
\end{enumerate}
\end{definition}

\begin{example}
For $i=(1,2,1,3,2,4,2)$ we have

\[
\begin{tikzpicture}[thick,font=\small, node distance=1cm and 5mm, baseline={([yshift=-.5ex]current bounding box.center)}]
\tikzset{dot/.style={circle,fill=#1,inner sep=0,minimum size=4pt}}

\node (0)                      {$i_k$};
\node (1) 	[right=of 0]			{$1$};
\node (2)	[right=of 1] 	{$2$};
\node (3) 	[right=of 2]	{$1$};
\node (4)	[right=of 3] 	{$3$};
\node (5)	[right=of 4] 	{$2$};
\node (6)	[right=of 5] 	{$4$};
\node (7)	[right=of 6] 	{$2$};

\node (0) 	[below=.1cm of 0]				{$k$};
\node (1) 	[right=of 0]				{$1$};
\node (2)	[right=of 1] 	{$2$};
\node (3) 	[right=of 2]	{$3$};
\node (4)	[right=of 3] 	{$4$};
\node (5)	[right=of 4] 	{$5$};
\node (6)	[right=of 5] 	{$6$};
\node (7)	[right=of 6] 	{$7$};

\node [dot=black] at ($ (1) + (0,-.3) $) {};	
\node [dot=black] at ($ (2) + (0,-.3) $) {};	
\node [dot=black] at ($ (3) + (0,-.3) $) {};	
\node [dot=black] at ($ (4) + (0,-.3) $) {};	
\node [dot=black] at ($ (5) + (0,-.3) $) {};	
\node [dot=black] at ($ (6) + (0,-.3) $) {};
\node [dot=black] at ($ (7) + (0,-.3) $) {};

\draw (1)	-- ++(0,-1.1)		-| (3);	
\draw (2)	-- ++(0,-.7)		-| (5);	
\draw (2)	-- ++(0,-.7)		-| (7);	
\end{tikzpicture}
\]
such that
\[
 \ker i = \left\{
 (1,3), (2,5,7), (4), (6)
 \right\} \in \PP(7).
\]
\end{example}

\begin{remark}
The relevance of this kernel in our setting is the following:

For $i=(i_1,\dots,i_m)$ and $j=(j_1,\dots,j_m)$ with $\ker i = \ker j$ we have
\[ \ev{a_{i_1i_2}a_{i_2i_3} \cdots a_{i_mi_1}} = \ev{a_{j_1j_2}a_{j_2j_3} \cdots a_{j_mj_1}}.
\]
For example, for $i=(1,1,2,1,1,2)$ and $j=(2,2,7,2,2,7)$ we have
\[
\ker i =
\begin{tikzpicture}[thick,font=\small, node distance=1cm and 5mm, baseline={([yshift=-.5ex]current bounding box.center)}]
\tikzset{dot/.style={circle,fill=#1,inner sep=0,minimum size=4pt}}

\node (1) 					{$1$};
\node (2)	[right=of 1] 	{$2$};
\node (3) 	[right=of 2]	{$3$};
\node (4)	[right=of 3] 	{$4$};
\node (5)	[right=of 4] 	{$5$};
\node (6)	[right=of 5] 	{$6$};

\node [dot=black] at ($ (1) + (0,-.3) $) {};	
\node [dot=black] at ($ (2) + (0,-.3) $) {};	
\node [dot=black] at ($ (3) + (0,-.3) $) {};	
\node [dot=black] at ($ (4) + (0,-.3) $) {};	
\node [dot=black] at ($ (5) + (0,-.3) $) {};	
\node [dot=black] at ($ (6) + (0,-.3) $) {};

\draw (1)	-- ++(0,-.7)		-| (2);	
\draw (1)	-- ++(0,-.7)		-| (4);	
\draw (1)	-- ++(0,-.7)		-| (5);	
\draw (3)	-- ++(0,-1.1)		-| (6);	
\end{tikzpicture}
= \ker j
\]
and
\begin{align*}
\ev{a_{11}a_{12}a_{21}a_{11}a_{12}a_{21}}
&=\ev{a_{11}^2} \ev{a_{12}^4}
=\ev{a_{22}^2} \ev{a_{27}^4}
=\ev{a_{22}a_{27}a_{72}a_{22}a_{27}a_{72}}.
\end{align*}
We denote this common value by
\[ \ev{\sigma} := \ev{a_{i_1i_2}a_{i_2i_3} \cdots a_{i_mi_1}}
\qquad\text{
if $\ker i = \sigma$.}
\]
 Thus we get:
\begin{align*}
\ev{\tr(A_N^m)}
&=  \frac{1}{N^{1+\frac{m}{2}}}  \sum_{\sigma \in \PP(m)} \ev{\sigma} \cdot \# \left\{
i\colon [m] \to [N] \, | \, \ker i = \sigma
\right\}.
\end{align*}
To understand the contribution corresponding to a $\sigma\in\PP(m)$ we associate to $\sigma$ a graph $\G_\sigma$.
\end{remark}

\begin{definition}
For $\sigma = \{ V_1,\dots,V_k \} \in\PP(m)$ we define a corresponding graph $\G_\sigma$ as follows.
The vertices of $\G_\sigma$ are given by the blocks $V_p$ of $\sigma$ and there is an edge betwees $V_p$ and $V_q$ if there is an $r\in[m]$ such that $r\in V_p$ and $r+1 (\text{mod } m) \in V_q$.

Another way of saying this is that we start with a graph with vertices $1,2,\dots, m$ and edges 
$(1,2), (2,3), (3,4), \dots, (m-1,m),(m,1)$ and then identify vertices according to the blocks of $\sigma$. We keep loops, but erase multiple edges.
\end{definition}

\begin{example}
\begin{enumerate}
\item For $$\sigma = \{ (1,3), (2,5), (4) \} = \begin{tikzpicture}[thick,font=\small, node distance=6mm and 3mm, baseline={([yshift=-1.5ex]current bounding box.center)}]

\node (1) 					{};
\node (2)	[right=of 1] 	{};
\node (3) 	[right=of 2]	{};
\node (4)	[right=of 3] 	{};
\node (5)	[right=of 4] 	{};

\draw (1)	-- ++(0,-.4)	-| (3);	
\draw (2)	-- ++(0,-.7)	-| (5);	
\draw (4)	-- ++(0,-.4);	
\end{tikzpicture}$$
we have
\[\begin{tikzpicture}[node distance=1cm]
\tikzset{bubble/.style={rectangle, rounded corners, minimum width=1.4cm, minimum height=0.9cm,text centered, draw=black}};

\node [bubble] (1) {$1=3$};
\node [bubble, below right = of 1] (2) {$2=5$};
\node [bubble, below left = of 1] (3) {$4$};

\draw (1) -- (2);
\draw (2) -- (3);
\draw (3) -- (1);

\node [below left = 0.2cm and 2.4cm of 1] {$\G_\sigma =$};
\end{tikzpicture}
\]

\item For $$\sigma = \{ (1,5), (2,4), (3) \} = \begin{tikzpicture}[thick,font=\small, node distance=6mm and 3mm, baseline={([yshift=-1.5ex]current bounding box.center)}]

\node (1) 					{};
\node (2)	[right=of 1] 	{};
\node (3) 	[right=of 2]	{};
\node (4)	[right=of 3] 	{};
\node (5)	[right=of 4] 	{};

\draw (1)	-- ++(0,-1.0)	-| (5);	
\draw (2)	-- ++(0,-.7)	-| (4);	
\draw (3)	-- ++(0,-.4);	
\end{tikzpicture}$$
we have
\[\begin{tikzpicture}[node distance=1cm]
\tikzset{bubble/.style={rectangle, rounded corners, minimum width=1.4cm, minimum height=0.9cm,text centered, draw=black}};

\node [bubble] (1) {$1=5$};
\node [bubble, below right = of 1] (2) {$2=4$};
\node [bubble, below left = of 1] (3) {$3$};

\draw (1) -- (2);
\draw (2) -- (3);

\begin{scope}
\clip (-0.7, 0.45) rectangle (0.7,1.2);
\draw (0,0.75) circle(0.4);
\end{scope}

\node [below left = 0.2cm and 2.4cm of 1] {$\G_\sigma =$};
\end{tikzpicture}
\]
\item For $$\sigma = \{ (1,3), (2), (4) \} = \begin{tikzpicture}[thick,font=\small, node distance=6mm and 3mm, baseline={([yshift=-1.5ex]current bounding box.center)}]

\node (1) 					{};
\node (2)	[right=of 1] 	{};
\node (3) 	[right=of 2]	{};
\node (4)	[right=of 3] 	{};

\draw (1)	-- ++(0,-.7)	-| (3);	
\draw (2)	-- ++(0,-.4);
\draw (4)	-- ++(0,-.4);	
\end{tikzpicture}$$
we have
$$\begin{tikzpicture}[node distance=1cm]
\tikzset{bubble/.style={rectangle, rounded corners, minimum width=1.4cm, minimum height=0.9cm,text centered, draw=black}};

\node [bubble] (1) {$1=3$};
\node [bubble, below right = of 1] (2) {$4$};
\node [bubble, below left = of 1] (3) {$2$};

\draw (1) -- (2);
\draw (3) -- (1);

\node [below left = 0.2cm and 2.4cm of 1] {$\G_\sigma =$};
\end{tikzpicture}
$$
\end{enumerate}
\end{example}

The term $ \ev{a_{i_1i_2}a_{i_2i_3} \cdots a_{i_mi_1}}$ corresponds now to a walk in $\G_\sigma$, with $\sigma = \ker i$, along the edges with steps
\[ i_1 \to i_2 \to i_3 \to \cdots \to i_m \to i_1.
\]
Hence we are using each edge in $\G_\sigma$ at least once. Note that different edges in $\G_\sigma$ correspond to independent random variables. Hence, if we use an edge only once in our walk, then $\ev{\sigma} = 0$, because the expectation factorizes into a product with one factor being the first moment of $a_{ij}$, which is assumed to be zero. Thus, every edge must be used at least twice, but this implies 
\begin{align*}
\# \text{ edges in } \G_\sigma  
\leq \frac{\# \text{ steps in the walk} }{2}
= \frac{m}{2}.
\end{align*}
Since the number of $i$ with the same kernel is
\[ \# \left\{
i \colon [m] \to [N] \,|\, \ker i = \sigma
\right\}
=N(N-1)(N-2) \cdots (N- \#\sigma +1),
\]
where $\# \sigma$ is the number of blocks in $\sigma$, we finally get
\begin{align*}\tag{$\star$}
\ev{\tr\left(A_N^m\right)}
&= \frac{1}{N^{1+\frac{m}{2}}} \sum_{\stackrel{\sigma\in\PP(m)}{ \# \text{edges}\left( \G_\sigma \right) \leq \frac{m}{2} }}  \ev{\sigma}
\underbrace{N(N-1)(N-2) \cdots (N- \#\sigma +1)}_{\sim N^{\#\sigma} \text{ for } N \to \infty }.
\end{align*}

We have to understand what the constraint on the number of edges in $\G_\sigma$ gives us for the number of vertices in $\G_\sigma$ (which is the same as $\#\sigma)$. For this,
we will now use the following well-known basic result from graph theory.

\begin{prop}\label{prop:3.8}
Let $\G=(V,E)$ be a connected finite graph with vertices $V$ and edges $E$. (We allow loops and multi-edges.) Then we have that
\[ \# V \leq \#E +1
\]
and we have equality if and only if $\G$ is a \emph{tree}, i.e., a connected graph without cycles.
\end{prop}

\section{Semicircle law for Wigner matrices}

\begin{theorem}[Wigner's semicircle law for Wigner matrices, averaged version]\label{thm:3.9}
Let $A_N$ be a Wigner matrix corresponding to $\mu$, which has all moments, with mean $0$ and second moment $1$. Then we have for all $m\in\N$:
\begin{align*}
\lim\limits_{N\to\infty} \ev{\tr\left(A_N^m\right)}
&=\frac{1}{2\pi} \int_{-2}^{2} x^m \sqrt{4-x^2} \td x.
\end{align*}
\end{theorem}

\begin{proof}
From ($\star$) we get
\begin{align*}
\lim\limits_{N\to\infty} \ev{\tr\left(A_N^m\right)}
&=  \sum_{\sigma\in\PP(m)} \ev{\sigma} \lim_{N\to\infty} N^{\#V(\G_\sigma)-\frac{m}{2}-1}.
\end{align*}
In order to have $\ev{\sigma} \neq 0$, we can restrict to $\sigma$ with $\#E(\G_\sigma) \leq\frac{m}{2}$, which by Proposition \ref{prop:3.8} implies that
\[ \#V(\G_\sigma) \leq \#E(\G_\sigma) +1 \leq\frac{m}{2}+1.
\]
Hence all terms converge and the only contribution in the limit $N\to\infty$ comes from those $\sigma$, where we have equality, i.e.,
\[ \#V(\G_\sigma) = \#E(\G_\sigma) +1 = \frac{m}{2}+1.
\]
Thus, $\G_\sigma$ must be a tree and in our walk we use each edge exactly twice (necessarily in opposite directions). For such a $\sigma$ we have $\ev{\sigma} =1$; thus
\begin{align*}
\lim\limits_{N\to\infty} \ev{\tr\left(A_N^m\right)}
&=  \# \left\{
\sigma \in \PP(m) \, | \, \G_\sigma \text{ is a tree}
\right\}.
\end{align*}
We will check in Exercise \ref{exercise:9} that the latter number is also counted by the Catalan numbers.
\end{proof}

\begin{remark}
Note that our $\G_\sigma$ are not just abstract trees, but they are coming with the walks, which encode
\begin{itemize}
\item
a starting point, i.e., the $\G_\sigma$ are rooted trees
\item
a cyclic order of the outgoing edges at a vertex, which gives a planar drawing of our graph
\end{itemize}
Hence what we have to count are rooted planar trees.

Note also that a rooted planar tree determines uniquely the corresponding walk.
\end{remark}

\chapter{Analytic Tools: Stieltjes Transform and Convergence of Measures}

%\begin{remark}
Let us recall our setting and goal. We have, for each $N\in\N$, selfadjoint $N\times N$ random matrices, which are given by a probability measure $\P_N$ on the entries of the matrices; the prescription of $\P_N$ should be kind of uniform in $N$.

For example, for the \GUEN\ we have $A=(a_{ij})_{i,j=1}^N$ with the complex variables $a_{ij}=x_{ij}+\sqrt{-1} y_{ij}$ having real part $x_{ij}$ and imaginary part $y_{ij}$. Since $a_{ij}=\bar a_{ji}$, we have $y_{ii}=0$ for all $i$ and we remain with the $N^2$ many ``free variables''
$x_{ii}$ ($i=1,\dots,N$) and 
 $x_{ij},y_{ij}$ ($1\leq i<j\leq N$).
All those are independent and Gaussian distributed, which can be written in the compact form
$$d\P(A)=c_N \exp(-N \frac {\Tr(A^2)}2) dA,$$
where $dA$ is the product of all differentials of the $N^2$ variables and $c_N$ is a normalization constant, to make $\P_N$ a probability measure.

We want now statements about our matrices with respect to this measure $\P_N$, either in average or in probability. Let us be a bit more specific on this.

Denote by $\Omega _N$ the space of our selfadjoint $N\times N$ matrices, i.e., 
$$\Omega_N:=\{A=(x_{ij}+\sqrt{-1} y_{ij})_{i,j=1}^N\mid
x_{ii} \in\R (i=1,\dots,N), x_{ij},y_{ij}\in\R (i<j)\}
\hat = \R^{N^2},
$$
then, for each $N\in\N$, $\P_N$ is a probability measure on $\Omega_N$. 

For $A\in\Omega_N$ we consider its $N$ eigenvalues $\lambda_1,\dots,\lambda_N$, counted with multiplicity. We encode those eigenvalues in a probability measure $\mu_A$ on $\R$:
$$\mu_A:=\frac 1N(\delta_{\lambda_1}+\cdots+ \delta_{\lambda_N}),$$
which we call the \emph{eigenvalue distribution of $A$}. Our claim is now that $\mu_A$ converges under $\P_N$, for $N\to\infty$ to the semicircle distribution $\mu_W$,
\begin{itemize}
\item
in average, i.e.,
$$\mu_N:=\int_{\Omega_N} \mu_A d\P_N(A)=\ev{\mu_a} \overset{N\to\infty}{\longrightarrow} \mu_W$$
\item
and, stronger, in probability or almost surely.
\end{itemize}
So what we have to understand now is:
\begin{itemize}
\item
What kind of convergence $\mu_N\to \mu$ do we have here, for probability measures on $\R$?
\item
How can we describe probability measures (on $\R$) and their convergence with analytic tools?
\end{itemize}
%\end{remark}

The relevant notions of convergence are the ``vague'' and ''weak'' convergence and our analytic tool will be the Stieltjes transform. We start with describing the latter.

\section{Stieltjes transform}

\begin{definition}
Let $\mu$ be a Borel measure on $\R$.
\begin{enumerate}
\item 
$\mu$ is \emph{finite} if $\mu(\R)<\infty$.
\item 
$\mu$ is a \emph{probability measure} if $\mu(\R)=1$.
\item
For a finite measure $\mu$ on $\R$ we define its \emph{Stieltjes transform} $S_\mu$ on $\C\backslash\R$ by
$$S_\mu(z)=\int_{\R} \frac 1{t-z} d\mu(t)\qquad (z\in \C\backslash\R).$$
\item
$-S_\mu=G_\mu$ is also called the \emph{Cauchy transform}.
\end{enumerate}
\end{definition}

\begin{theorem}\label{thm:4.3}
The Stieltjes transform has the following properties.
\begin{enumerate}
\item
Let $\mu$ be a finite measure on $\R$ and $S=S_\mu$ its Stieltjes transform. Then one has:
\begin{enumerate}
\item
$S:\C^+\to\C^+$, where
$\C^+:=\{z\in\C\mid \Im(z) >0\}$;
\item
$S$ is analytic on $\C^+$;
\item
$\lim_{y\to\infty} iy S(iy)=-\mu(\R)$.
\end{enumerate}
\item 
$\mu$ can be recovered from $S_\mu$ via the \emph{Stieltjes inversion formula}: for $a<b$ we have
$$\lim_{\varepsilon \searrow 0} \frac 1\pi \int_a^b \Im S_\mu(x+i\varepsilon)dx=
\mu((a,b))+\frac 12 \mu(\{a,b\}).$$
\item
In particular, we have for two finite measures $\mu$ and $\nu$: $S_\mu=S_\nu$ implies that $\mu=\nu$.
\end{enumerate}
\end{theorem}

\begin{proof}
\begin{enumerate}
\item
This is Exercise \ref{exercise:10}.
\item
We have
$$\Im S_\mu(x+ i\ee)=\int_\R \Im\left(\frac 1{t-x-i\ee}\right)d\mu(t)=
\int_\R \frac \ee{(t-x)^2+\ee} d\mu(t)$$
and thus
$$\int_a^b\Im S_\mu(x+i\ee)dx=\int_\R\int_a^b \frac \ee{(t-x)^2+\ee^2}dx d\mu(t).$$
For the inner integral we have
\begin{align*}
\int_a^b \frac \ee{(t-x)^2+\ee^2}dx=\int_{(a-b)/\ee}^{(b-t)/\ee} \frac 1{x^2+1} dx
&=\tan^{-1}\left(\frac {b-t}\ee\right) -\tan^{-1}\left(\frac {a-t}\ee\right)\\
&\overset{\ee\searrow 0}\longrightarrow
\begin{cases}
0,& t\not\in [a,b]\\
\pi/2,& t\in\{a,b\}\\
\pi,& t\in (a,b)
\end{cases}
\end{align*}
From this the assertion follows.
\item
Now assume that $S_\mu=S_\nu$. By the Stieltjes inversion formula it follows then that $\mu((a,b))=\nu((a,b))$ for all open intervals such that $a$ and $b$ are atoms neither of $\mu$ nor of $\nu$. Since there can only be countably many atoms we can write any interval as
$$(a,b)=\bigcup_{n=1}^\infty (a+\ee_n,b-\ee_n),$$
where the sequence $\ee_n\searrow 0$ is chosen such that all $a+\ee_n$ and $b-\ee_n$ are no atoms of $\mu$ nor $\nu$. By monotone convergence for measures we get then
$$\mu((a,b)=\lim_{n\to\infty}\mu((a+\ee_n,b-\ee_n))=
\lim_{n\to\infty} \nu((a+\ee_n,b-\ee_n))=\nu((a,b)).$$
\end{enumerate}
\end{proof}

\begin{remark}
If we put $\mu_\ee=p_\ee \lambda$ (where $\lambda$ is Lebesgue measure) with density
$$p_\ee(x):=\frac 1\pi\Im S_\mu(x+i\ee)=\frac 1\pi\int_\R \frac \ee{(t-x)^2+\ee^2}d\mu(t),$$
then $\mu_\ee=\gamma_\ee *\mu$, where $\gamma_\ee$ is the Cauchy distribution, and we have checked explicitly in our proof of the Stieltjes inversion formula the well-known fact that $\gamma_\ee *\mu$ converges weakly to $\mu$ for $\ee\searrow 0$.
We will talk about weak convergence later, see Definition \ref{def:4.8}.
\end{remark}

\begin{prop}\label{prop:4.5}
Let $\mu$ be a compactly supported probability measure, say
$\mu([-r,r])=1$ for some $r>0$. Then $S_\mu$ has a power series expansion (about $\infty$) as follows
$$S_\mu(z)=-\sum_{n=0}^\infty \frac {m_n}{z^{n+1}}\qquad \text{for $\vert z\vert>r$,}$$
where $m_n:=\int_\R t^n d\mu(t)$ are the moments of $\mu$.
\end{prop}

\begin{proof}
For $\vert z\vert >r$ we can expand
$$\frac 1{t-z}=-\frac 1{z(1-\frac tz)}=-\frac 1z \sum_{n=0}^\infty \left(\frac tz\right)^n$$
for all $t\in[-r,r]$; the convergence on $[-r,r]$ is uniform, hence
$$S_\mu(z)=\int_{-r}^r \frac 1{t-z}d\mu(t)=-\sum_{n=0}^\infty \int_{-r}^r \frac {t^n}{z^{n+1}}d\mu(t)=-\sum_{n=0}^\infty \frac {m_n}{z^{n+1}}.$$
\end{proof}

\begin{prop}\label{prop:4.6}
The Stieltjes transform $S(z)$ of the semicircle distribution, $d\mu_W(t)=\frac 1{2\pi}\sqrt{t^2-4}dt$, is, for $z\in\C^+$, uniquely determined by 
\begin{itemize}
\item
$S(z)\in\C^+$;
\item
$S(z)$ is the solution of the equation
$S(z)^2+z S(z)+1=0$.
\end{itemize}
Explicitly, this means
$$S(z)=\frac {-z+\sqrt{z^2-4}}2\qquad (z\in\C^+).$$
\end{prop}

\begin{proof}
By Proposition \ref{prop:4.5}, we know that for large $\vert z\vert$:
$$S(z)=-\sum_{k=0}^\infty \frac {C_k}{z^{2k+1}},$$
where $C_k$ are the Catalan numbers. By using the recursion for the Catalan numbers (see Theorem \ref{thm:1.4}), this implies that for large $\vert z\vert$ we have $S(z)^2+zS(z)+1=0$. Since we know that $S$ is analytic on $\C^+$, this equation is, by analytic extension, then valid for all $z\in\C^+$.

This equation has two solutions,
$(-z\pm \sqrt{z^2-4})/2$, 
and only the one with the $+$-sign is in $\C^+$.
\end{proof}

\begin{remark}
Proposition \ref{prop:4.6} gave us the Stieltjes transform of $\mu_W$ just from the knowledge of the moments. From $S(z)=(-z-\sqrt{z^2-4})/2$ we can then get the density of $\mu_W$ via the Stieltjes inversion formula:
$$
\frac 1\pi \Im S(x+i\ee)=\frac 1{2\pi}\Im\sqrt{(x+i\ee)^2-4}
\overset{\ee\searrow 0}{\longrightarrow}\frac 1{2\pi}\Im\sqrt{x^4-4}
=\begin{cases}
0,& \vert x\vert >2\\\frac 1{2\pi} \sqrt{4-x^2},& \vert x\vert\leq 2.
\end{cases}
$$
Thus this analytic machinery gives an effective way to calculate a distribution from its moments (without having to know the density in advance).
\end{remark}

\section{Convergence of probability measures}

%\begin{remark}
Now we want to consider the convergence $\mu_N\to \mu$. We can consider (probability) measures from two equivalent perspectives:
\begin{enumerate}
\item
measure theoretic perspective:
$\mu$ gives us the measure (probability) of sets, i.e., $\mu(B)$ for measurable sets $B$, or just $\mu(\text{intervals})$
\item
functional analytic perspective:
$\mu$ allows to integrate continuous functions, i.e, it gives us
$\int fd\mu$ for continuous $f$
\end{enumerate}
According to this there are two canonical choices for a notion of convergence:
\begin{enumerate}
\item
$\mu_N(B)\to \mu(B)$ for all measurable sets $B$, or maybe for all intervals $B$
\item
$\int fd\mu_N\to \int f d\mu$ for all continuous $f$
\end{enumerate}
The first possibility is problematic in this generality, as it does treat atoms too restrictive.

Example: Take $\mu_N=\delta_{1-1/N}$ and $\mu=\delta_1$. Then we surely want that $\mu_N\to\mu$, but for $B=[1,2]$ we have
$\mu_N([1,2])=0$ for all $N$, but $\mu([1,2])=1.$

Thus the second possibility above is the better definition. But we have to be careful about which class of continuous functions we allow; we need bounded ones, otherwise $\int fd\mu$ might not exist in general; and, for compactness reasons, it is sometimes better to ignore the behaviour of the measures at infinity.

%\end{remark}

\begin{definition}\label{def:4.8}
\begin{enumerate}
\item
We use the notations
\begin{enumerate}
\item 
$C_0(\R):=\{f\in C(\R)\mid \lim_{\vert t\vert \to\infty} f(t) =0\}$ are the 
continuous functions on $\R$ vanishing at infinity
\item
$C_b(\R):=\{f\in C(\R)\mid \exists M>0: \vert f(t)\vert \leq M \quad\forall t\in \R\}$
are the continuous bounded functions on $\R$
\end{enumerate}
\item
Let $\mu$ and $\mu_N$ ($N\in N$) be finite measures. Then we say that
\begin{enumerate}
\item
$\mu_N$ \emph{converges vaguely} to $\mu$, denoted by $\mu_N \tov \mu$,
if 
$$\int f(t)d\mu(t)\to \int f(t) d\mu(t)\qquad\text{for all $f\in C_0(\R)$};$$
\item
$\mu_N$ \emph{converges weakly} to $\mu$, denoted by $\mu_N\tow \mu$, if
$$\int f(t) d\mu_N(t) \to \int f(t) d\mu(t) \qquad \text{for all $f\in C_b(\R)$}.$$
\end{enumerate}
\end{enumerate}
\end{definition}

\begin{remark}
\begin{enumerate}
\item
Note that weak convergence includes in particular that
$$\mu_N(\R)=\int 1 d\mu_N(t)\to \int 1d\mu(t)=\mu(\R),$$
and thus the weak limit of probability measures must again be a probability measure. For the vague convergence this is not true; there we can loose mass at infinity.

Example: Consider $\mu_N=\frac 12 \delta_1 +\frac 12 \delta_N$ and
$\mu=\frac 12 \delta_1$; then
$$\int f(t) d\mu_N(t)=\frac 12 f(1)+\frac 12 f(N) \to \frac 12 f(1)=\int f(t) d\mu(t)$$
for all $f\in C_0(\R)$. Thus the sequence of probability measures $\frac 12 \delta_1 +\frac 12\delta _N$ converges, for $N\to \infty$, to the finite measure $\frac 12 \delta_1$ with total mass $1/2$.

\item
The relevance of the vague convergence, even if we are only interested in probability measures, is that the probability measures are precompact in the vague topology, but not in the weak topology. E.g., in the above example, $\mu_N=\frac 12 \delta_1+\frac 12 \delta_N$ has no subsequence which converges weakly (but it has a subsequence, namely itself, which converges vaguely).
\end{enumerate}
\end{remark}

\begin{theorem}\label{thm:4.11}
The space of probability measures on $\R$ is precompact in the vague topology: every sequence $(\mu_N)_{N\in\N}$ of probability measures on $\R$ has a subsequence which converges vaguely to a finite measure $\mu$, with $\mu(\R)\leq 1$.
\end{theorem}

\begin{proof}
\begin{enumerate}
\item
From a functional analytic perspective this is a special case of the Banach-Alaoglu theorem; since complex measures on $\R$ are the dual space of the Banach space $C_0(\R)$, and its weak$^*$ topology is exactly the vague topology.
\item
From a measure theory perspective this is known as Helly's (Selection) Theorem. Here are the main ideas for the proof in this setting.
\begin{enumerate}
\item
We describe a finite measure $\mu$ by its distribution function $F_\mu$, given by
$$F_\mu:\R\to\R; \quad F_\mu(t):=\mu((-\infty,t]).$$
Such distribution functions can be characterized as functions $F$ with the properties:
\begin{itemize}
\item
$t\mapsto F(t)$ is non-decreasing
\item
$F(-\infty):=\lim_{t\to-\infty}F(t)=0$ and $F(+\infty):=\lim_{t\to\infty}F(t)<\infty$
\item
$F$ is continuous on the right
\end{itemize}
\item
The vague convergence of $\mu_N\tov \mu$ can also be described in terms of their distribution functions $F_N$, $F$; $\mu_N\tov \mu$ is equivalent to:
$$F_N(t)\to F(t)\qquad \text{for all $t\in\R$ at which $F$ is continuous.}$$
\item
Let now a sequence $(\mu_N)_N$ of probability measures be given. We consider the corresponding distribution functions $(F_N)_N$ and want to find a convergent subsequence (in the sense of (ii)) for those.

For this choose a countable dense subset $T=\{t_1,t_2,\dots\}$ of $\R$. Then, by choosing subsequences of subsequences and taking the ``diagonal'' subsequence, we get convergence for all $t\in T$.  More precisely: Choose subsequence $(F_{N_1(m)})_m$ such that
$$F_{N_1(m)}(t_1) \overset{m\to\infty}\longrightarrow F_T(t_1),$$
choose then a subsequence $(F_{N_2(m)})_m$ of this such that
$$F_{N_2(m)}(t_1) \overset{m\to\infty}\longrightarrow F_T(t_1),\qquad
F_{N_2(m)}(t_2) \overset{m\to\infty}\longrightarrow F_T(t_2);$$
iterating this gives subsequences $(F_{N_k(m)})_m$ such that
$$F_{N_k(m)}(t_i) \overset{m\to\infty}\longrightarrow F_T(t_i)\qquad
\text{for all $i=1,\dots,k$.}$$
The diagonal subsequence $(F_{N_m(m)})_m$ converges then at all $t\in T$ to $F_T(t)$. We improve now $F_T$ to the wanted $F$ by
$$F(t):=\inf\{F_T(s)\mid s\in T, s>t\}$$
and show that 
\begin{itemize}
\item
$F$ is a distribution function;
\item
$F_{N_m(m)}(t) \overset{m\to\infty}\longrightarrow F(t)$ at all continuity points of $F$.
\end{itemize}
According to (ii) this gives then the convergence $\mu_{N_m(m)} \overset{m\to\infty}\longrightarrow \mu$, where $\mu$ is the finite measure corresponding to the distribution function $F$.
Note that $F_N(+\infty)=1$ for all $N\in\N$ gives $F(+\infty)\leq 1$, but we cannot guarantee $F(+\infty)=1$ in general.
\end{enumerate}
\end{enumerate}

\end{proof}

\begin{remark}
If we want compactness in the weak topology, then we must control the mass at $\infty$ in a uniform way. This is given by the notion of tightness. A sequence $(\mu_N)_N$ of probability measures is \emph{tight} if: for all $\ee>0$ there exists an interval $I=[-R,R]$ such that $\mu_N(I^c)<\ee$ for all $N$.

Then one has: Any tight sequence of probability measures has a subsequence which converges weakly; the limit is then necessarily a probability measure.

\end{remark}

\section{Probability measures determined by moments}

We can now also relate weak convergence to convergence of moments; which shows that our combinatorial approach (using moments) and analytic approach (using Stieltjes transforms) for proving the semicircle law are essentially equivalent. We want to make this more precise in the following.

\begin{definition}
A probability measure $\mu$ on $\R$ is \emph{determined by its moments} if
\begin{enumerate}
\item[(i)]
all moments $\int t^kd\mu(t)<\infty$ ($k\in\N$) exist;
\item[(ii)]
$\mu$ is the only probability measure with those moments: if $\nu$ is a probability measure and $\int t^kd\nu(t)=\int t^kd\mu(t)$ for all $k\in \N$, then $\nu=\mu$.
\end{enumerate}
\end{definition}

\begin{theorem}\label{thm:4.12}
Let $\mu$ and $\mu_N$ ($N\in\N$) be probability measures for which all moments exist. Assume that $\mu$ is determined by its moments. 
Assume furthermore that we have convergence of moments, i.e., 
$$\lim_{N\to\infty} \int t^k d\mu_N(t)=\int t^k d\mu(t)\qquad \text{for all $k\in\N$.}$$
Then we have weak convergence: 
$\mu_N\tow \mu$.
\end{theorem}

\begin{proof}[Rough idea of proof] 
One has to note that convergence of moments implies tightness, which implies the  existence of a weakly convergent subsequence, $\mu_{N_m} \to \nu$. Furthermore, the assumption that the moments converge implies that they are uniformly integrable, which implies then that the moments of this subsequence converge to the moments of $\nu$. (These are kind of standard measure theoretic arguments, though a bit involved; for details see the book of Billingsley, in particular, his Theorem 25.12 and its Corollary.) However, the moments of the subsequence converge, as the moments of the whole sequence, by assumption to the moments of $\mu$; this means that $\mu$ and $\nu$ have the same moments and hence, by our assumption that $\mu$ is determined by its moments, we have that $\nu=\mu$.
\\
In the same way all weakly convergent subsequences of $(\nu_N)_N$ must converge to the same $\mu$, and thus the whole sequence must converge weakly to $\mu$.

\end{proof}

\begin{remark}
\begin{enumerate}
\item[(0)]
Note that in the first version of these notes (and also in the recorded lectures) it was claimed that, under the assumption that the limit is determined by its moments, convergence in moments is equivalent to weak convergence. This is clearly not true as the following simple example shows. Consider
$$\mu_N=(1-\frac 1N)\delta_0+\frac 1N \delta_N\qquad\text{and}\qquad
\mu=\delta_0.$$
Then it is clear that $\mu_N\tow \mu$, and $\mu$ is also determined by its moments. But there is no convergence of moments. For example,
the first moment converges, but to the wrong limit
$$\int t d\mu_N(t)=\frac 1N N=1 \to 1\not= 0=\int td\mu(t),$$
and the other moments explode
$$\int t^k d\mu_N(t)=\frac 1N N^k=N^{k-1}\to\infty\qquad\text{for $k\geq 2$}.$$
In order to have convergence of moments one needs a uniform integrability assumption; see Billingsley, in particular, his Theorem 25.12 and its Corollary.
\item
Note that there exist measures for which all moments exist but which, however, are not determined by their moments. Weak convergence to them cannot be checked by just looking on convergence of moments.

Example: The log-normal distribution with density
$$d\mu(t)=\frac 1{\sqrt{2\pi}} \frac 1x e^{-(\log x)^2/2} dt\qquad \text{on $[0,\infty)$}$$
(which is the distribution of $e^X$ for $X$ Gaussian) is not determined by its moments.
\item
Compactly supported measures (like the semicircle) or also the Gaussian distribution are determined by their moments.
\end{enumerate}
\end{remark}

\section{Description of weak convergence via the Stieltjes transform}

\begin{theorem}\label{thm:4.15}
Let $\mu$ and $\mu_N$ ($N\in\N$) be probability measures on $\R$. Then the following are equivalent.
\begin{enumerate}
\item[(i)]
$\mu_N\tow\mu$.
\item[(ii)]
For all $z\in\C^+$ we have: $\lim_{N\to\infty} S_{\mu_N}(z)=S_\mu(z)$.
\item[(iii)]
There exists a set $D\subset \C^+$, which has an accumulation point in $\C^+$, such that: $\lim_{N\to\infty} S_{\mu_N}(z)=S_\mu(z)$ for all $z\in D$.
\end{enumerate}
\end{theorem}

\begin{proof}
\begin{itemize}
\item
(i)$\implies$(ii): 
Assume that $\mu_N\tow \mu$. For $z\in\C^+$ we consider
$$f_z:\R\to\C \qquad\text{with}\qquad f_z(t)=\frac 1{t-z}.$$
Since $\lim_{\vert t\vert \to\infty} f_z(t)=0$, we have
$f_z\in C_0(\R)\subset C_b(\R)$ and thus, by definition of weak convergence:
$$S_{\mu_N}(z)=\int f_z(t)d\mu_N(t)\to \int f_z(t) d\mu(t)=S_\mu(z).$$
\item
(ii)$\implies$(iii): clear
\item
(iii)$\implies$(i): By Theorem \ref{thm:4.11}, we know that $(\mu_N)_N$
has a subsequence $(\mu_{N(m)})_m$ which converges vaguely to some finite measure $\nu$ with $\nu(\R)\leq 1$. Then, as above, we have for all $z\in D$:
$$S_\mu(z)=\lim_{m\to\infty} S_{\mu_N(m)}(z) =S_\nu(z).$$
Thus the analytic functions $S_\mu$ and $S_\nu$ agree on $D$ and hence, but the identity therem for analytic functions, also on $\C^+$, i.e., $S_\mu=S_\nu$. But this implies, by Theorem \ref{thm:4.3}, that $\nu=\mu$.

Thus the subsequence $(\mu_{N(m)})_m$ converges vaguely to the probability measure $\mu$ (and thus also weakly, see Exercise \ref{exercise:12}). In the same way, any weak cluster point of $(\mu_N)_N$ must be equal to $\mu$, and thus the whole sequence must converge weakly to $\mu$.
\end{itemize}

\end{proof}

\begin{remark}
If we only assume that $S_{\mu_N}(z)$ converges to a limit function $S(z)$, then $S$ must be the Stieltjes transform of a measure $\nu$ with $\nu(\R)\leq 1$ and we have the vague convergence $\mu_N\tov \nu$.
\end{remark}

\chapter{Analytic Proof of Wigner's Semicircle Law for Gaussian Random Matrices}

Now we are ready to give an analytic proof of Wigner's semicircle law, relying on the analytic tools we developed in the last chapter. As for the combinatorial approach, the Gaussian case is easier compared to the general Wigner case and thus we will restrict to this. The difference between real and complex is here not really relevant, instead of \GUE\ we will treat the \GOE\ case.

\section{GOE random matrices}

\begin{definition}\label{def:5.1}
\emph{Real Gaussian random matrices (\GOE)} are of the form
$A_N=(x_{ij})_{i,j=1}^N$, where $x_{ij}=x_{ji}$ for all $i,j$ and $\{x_{ij}\mid i\leq j\}$ are i.i.d. (independent identically distributed) random variables with Gaussian distribution of mean zero and variance $1/N$.
More formaly, on the space of symmetric $N\times N$ matrices
$$\Omega_N=\{A_N=(x_{ij})_{i,j=1}^N\mid x_{ij}\in\R, x_{ij}=x_{ji}\,\forall i,j\}$$
we consider the probability measure
$$d\P_N(A_N)=c_N\exp(-\frac N4 \Tr(A_N^2))\prod_{i\leq j} dx_{ij},$$
with a normalization constant $c_N$ such that $\P_N$ is a probability measure.
\end{definition}

\begin{remark}\label{rem:5.2}
\begin{enumerate}
\item
Note that with this choice of $\P_N$, which is invariant under orthogonal rotations, we have actually different variances on and off the diagonal:
$$\ev{x_{ij}^2}=\frac 1N \quad (i\not=j)\qquad\text{and}\qquad
\ev{x_{ii}^2}=\frac 2N.$$
\item
We consider now, for each $N\in \N$, the averaged eigenvalue distribution
$$\mu_N:=\ev{\mu_{A_N}}=\int_{\Omega_N} \mu_A dP_N(A).$$
We want to prove that $\mu_N\tow \mu_W$. 
According to Theorem \ref{thm:4.15} we can prove this by showing
$\lim_{N\to\infty} S_{\mu_N}(z)=S_{\mu_W}(z)$ for all $z\in\C^+$.

\item
Note that
$$
S_{\mu_N}(z)=\int_\R \frac 1{t-z} d\mu_N(t)
=\ev{\int_\R \frac 1{t-z} d\mu_{A_N}(t)}
=\ev{\tr[(A_N-z1)^{-1}]},
$$
since, by Assignement ...., $S_{\mu_{A_N}}(t)=\tr[(A_N-z1)^{-1}]$.
So what we have to see, is for all $z\in\C^+$:
$$\lim_{N\to\infty} \ev{\tr[(A_N-z1)^{1-}]}=S_{\mu_W}(z).$$
For this, we want to see that $S_{\mu_N}(z)$ satisfies approximately the quadratic equation for $S_{\mu_W}(z)$, from \ref{prop:4.6}.
\item\label{rem:5.2ii}
Let us use for the resolvents of our matrices $A$ the notation
$$R_A(z)=\frac 1{A-z1},\qquad\text{so that}\qquad
S_{\mu_N}(z)=\ev{\tr(R_{A_N}(z))}.
$$
In the following we will usually suppress the index $N$ at our matrices; thus write just $A$ instead of $A_N$, as long as the $N$ is fixed and clear.

We have then $(A-z1)R_A(z)=1$, or
$A\cdot R_A(z)-z R_A(z)=1$, thus
$$R_A(z)=-\frac 1z 1 +\frac 1z A R_A(z).
$$
Taking the normalized trace and expectation of this yields
$$\ev{\tr(R_{A}(z))}=-\frac 1z +\frac 1z \ev{\tr(A R_{A}(z))}.$$ 
The left hand side is our Stieltjes transform, but what about the right hand side; can we relate this also to the Stieltjes transform? Note that the function under the expectation is 
$\frac 1N\sum_{k,l} x_{kl} [R_A(z)]_{lk}$; thus a sum of terms which are the product of one of our Gaussian variables times a function of all the independent Gaussian variables. There exists actually a very nice and important formula to deal with such expectations of independent Gaussian variables. In a sense, this is the analytic version of the combinatorial Wick formula. 

\end{enumerate}
\end{remark}

\section{Stein's identity for independent Gaussian variables}

\begin{prop}[Stein's identity]\label{prop:5.3}
Let $X_1,\dots,X_k$ be independent random variables with Gaussian distribution, with mean zero and variances $\ev{X_i}=\sigma_i^2$. Let 
$h:\R^k\to\C$ be continuously differentiable such that $h$ and all partial derivatives are of polynomial growth. Then we have for $i=1,\dots,k$:
$$\ev{X_i h(X_1,\dots,X_k)}=\sigma_i^2 \ev{\frac{\partial h}{\partial x_i}
(X_1,\dots,X_k)}.$$
More explicitly,
\begin{multline*}
\int_{\R^k} x_i h(x_1,\dots,x_k)
\exp\left(-\frac{x_1^2}{2\sigma_1^2}-\cdots -\frac{x_k^2}{2\sigma_k^2}\right)
dx_1\dots dx_k=\\
\sigma_i^2 \int_{\R^k} \frac{\partial h}{\partial x_i}(x_1,\dots,x_k)
\exp\left(-\frac{x_1^2}{2\sigma_1^2}-\cdots -\frac{x_k^2}{2\sigma_k^2}\right)
dx_1\dots dx_k
\end{multline*}

\end{prop}

\begin{proof}

The main argument happens for $k=1$. Since 
$xe^{-{x^2}/{(2\sigma^2)}}=[-\sigma^2 e^{- {x^2}/{(2\sigma^2})}]'$
we get by partial integration
$$\int_\R xh(x)  e^{-{x^2}/{(2\sigma^2)}} dx=\int_\R h(x) [-\sigma^2 e^{- {x^2}/{(2\sigma^2})}]'dx =
\int_\R h'(x)\sigma^2 e^{-{x^2}/{(2\sigma^2)}}dx;$$
our assumptions on $h$ are just such that the boundary terms vanish.

For general $k$, we just do partial integration for the $i$-th coordinate.

\end{proof}

We want to apply this now to our Gaussian random matrices, with Gaussian random variables $x_{ij}$ ($1\leq i\leq j\leq N$) of variance
$$\sigma_{ij}^2=\begin{cases}
\frac 1N, & i\not= j\\
\frac 2N, & i=j
\end{cases}$$
and for the function
$$h(x_{ij}\mid i\leq j)=h(A)=[R_A(z)]_{lk},\qquad\text{where}\qquad
R_A(z)=\frac 1{A-z1}.$$
To use  Stein's identitiy \ref{prop:5.3} in this case we need the partial derivatives of the resolvents.

\begin{lemma}\label{lem:5.4}
For $A=(x_{ij})_{i,j=1}^N$ with $x_{ij}=x_{ji}$ for all $i,j$, we have for all $i,j,k,l$:
$$\frac\partial{\partial x_{ij}} [R_A(z)]_{lk}=
\begin{cases}
-[R_A(z)]_{li}\cdot [R_A(z)]_{ik},& i=j,\\
-[R_A(z)]_{li}\cdot [R_A(z)]_{jk}- [R_A(z)]_{lj}\cdot [R_A(z)]_{ik},& i\not= j.
\end{cases}
$$

\end{lemma}

\begin{proof}
Note first that
$$\frac {\partial A}{\partial x_{ij}}=
\begin{cases}
E_{ii},& i=j\\
E_{ij}+E_{ji},& i\not=j
\end{cases}
$$
where $E_{ij}$ is a matrix unit with 1 at position $(i,j)$ and 0 elsewhere.

We have $R_A(z)\cdot (A-z1)=1$, which yields by differentiating
$$\frac {\partial R_A(z)}{\partial x_{ij}} \cdot (A-z1)+
R_a(z)\cdot \frac{\partial A}{\partial x_{ij}}=0,$$
and thus
$$\frac {\partial R_A(z)}{\partial x_{ij}}=-R_A(z)\cdot \frac{\partial A}{\partial x_{ij}} \cdot R_A(z).$$
This gives, for $i=j$,
$$ \frac\partial{\partial x_{ii}} [R_A(z)]_{lk}=-[R_A(z)\cdot E_{ii}\cdot R_A(z)]_{lk}
=-[R_A(z)]_{li}\cdot [R_A(z)]_{ik},$$
and, for $i\not= j$,
\begin{align*}
\frac\partial{\partial x_{ij}} [R_A(z)]_{lk}&=-[R_A(z)\cdot E_{ij} \cdot R_A(z)]_{lk}
-[R_A(z)\cdot E_{ji}\cdot R_A(z)]_{lk}\\
&=-[R_A(z)]_{li}\cdot [R_A(z)]_{jk}-[R_A(z)]_{lj}\cdot [R_A(z)]_{ik}.
\end{align*}
\end{proof}

\section{Semicircle law for GOE}

\begin{theorem}\label{thm:5.5}
Let $A_N$ be \GOE\ random matrices as in \ref{def:5.1}. Then its averaged eigenvalue distribution $\mu_N:=\ev{\mu_{A_N}}$ converges weakly to the semicircle distribution: $\mu_N\tow \mu_W$.
\end{theorem}

\begin{proof}
By Theorem \ref{thm:4.15}, it suffices to show $\lim_{N\to\infty} S_{\mu_N}(z)=S_{\mu_W}(z)$ for all $z\in \C^+$. 

In Remark \ref{rem:5.2}\ref{rem:5.2ii} we have seen that (where we write $A$ instead of $A_N$)
$$S_{\mu_N}(z)=\ev{\tr(R_A(z)]}=-\frac 1z+\frac 1z\ev{\tr[AR_A(z)]}.$$
Now we calculate, with $A=(x_{ij})_{i,j=1}^N$,
\begin{align*}
\ev{\tr[AR_A(z)]}&=
\frac 1N\sum_{k,l=1}^N\ev{x_{kl}\cdot [R_A(z)]_{lk}}\\
&=\frac 1N\sum_{k,l=1}^N \sigma_{kl}^2\cdot \ev{\frac \partial{\partial x_{kl}}
[R_A(z)]_{lk}}\\
&=-\frac 1N\sum_{k,l=1}^N \frac 1N\cdot\bigl(
[R_A(z)]_{lk}\cdot [R_A(z)]_{lk} + [R_A(z)]_{ll}\cdot [R_A(z)]_{kk}\bigr).
\end{align*}
(Note that the combination of the different covariances and of the different form of the formula in Lemma \ref{lem:5.4} for on-diagonal and for off-diagonal entries gives in the end the same result for all pairs of $(k,l)$.)

Now note that $(A-z1)$ is symmetric, hence the same is true for its inverse
$R_A(z)=(A-z1)^{-1}$ and thus: $[R_A(z)]_{lk}=[R_A(z)]_{kl}$. Thus we get finally
$$\ev{\tr[AR_A(z)]}=-\frac 1N\ev{\tr[R_A(z)^2]}-
\ev{\tr[R_A(z)]\cdot \tr[R_A(z)]}.$$
To proceed further we need to deal with the two summands on the right hand side; we expect
\begin{itemize}
\item
the first term, $\frac 1N\ev{\tr[R_A(z)^2]}$, should go to zero, for $N\to\infty$
\item
the second term, $\ev{\tr[R_A(z)]\cdot \tr[R_A(z)]}$, should be close to its factorized version
$\ev{\tr[R_A(z)]}\cdot \ev{\tr[R_A(z)]}=S_{\mu_N}(z)^2$
\end{itemize}
Both these ideas are correct; let us try to make them rigorous.
\begin{itemize}
\item
$A$ as a symmetric matrix can be diagonalized by an orthogonal matrix $U$,
$$A=U 
\begin{pmatrix}
\lambda_1& \dots& 0\\
\vdots&\ddots&\vdots\\
0&\dots&\lambda_N
\end{pmatrix} U^*,\quad\text{and thus}\quad
R_A(z)^2=\begin{pmatrix}
\frac 1{(\lambda_1-z)^2}& \dots& 0\\
\vdots&\ddots&\vdots\\
0&\dots&\frac 1{(\lambda_N-z)^2}
\end{pmatrix} U^*,
$$
which yields
$$\vert \tr[R_A(z)^2]\vert \leq \frac 1N \sum_{i=1}^N\vert \frac 1{(\lambda_i-z)^2}\vert.$$
Note that for all $\lambda\in\R$ and all $z\in\C^+$
$$\vert \frac 1{\lambda-z}\vert \leq \frac 1{\Im z}$$
and hence
$$\frac 1N \vert \ev{\tr(R_A(z)^2]}\vert\leq \frac 1N \ev{\vert \tr[R_A(z)^2]\vert}\leq
\frac 1N \frac 1{(\Im z)^2} \to 0 \qquad\text{for $N\to\infty$.}
$$
\item
By definition of the variance we have
$$\var{X}:=\ev{(X-\ev{X})^2}=\ev{X^2}-\ev{X}^2,$$
and thus
$$\ev{X^2}=\ev{X}^2+\var{X}.$$
Hence we can replace
$\ev{\tr[R_A(z)]\cdot \tr[R_A(z)]}$
by 
$$\ev{\tr[R_A(z)]}^2+\var{\tr[R_A(z)]}=S_{\mu_N}(z)^2+\var{S_{\mu_A}(z)}.$$
In the next chapter we will show that we have concentration, i.e., the variance 
$\var{S_{\mu_A}(z)}$ goes to zero for $N\to\infty$.
\end{itemize}
With those two ingredients we have then
$$S_{\mu_N}(z)=-\frac 1z+\frac 1z\ev{\tr[AR_A(z)]}=-\frac 1z -\frac 1z S_{\mu_N}(z)^2 +\ee_N,$$
where $\ee_N\to$ for $N\to\infty$.

Note that, as above, for any Stieltjes transform $S_\nu$ we have
$$\vert S_\nu(z)\vert=\vert\int \frac 1{t-z} d\nu(t)\vert
\leq \int \vert\frac 1{t-z}\vert d\nu(t) \leq \frac 1{\Im z},$$
and thus $(S_{\mu_N}(z))_N$ is a bounded sequence of complex numbers. Hence, by compactness, there exists a convergent subsequence 
$(S_{\mu_{N(m)}}(z))_m$, which converges to some $S(z)$. This $S(z)$ must then satisfy the limit $N\to\infty$ of the above equation, thus
$$S(z)=-\frac 1z -\frac 1z S(z)^2.$$
Since all $S_{\mu_N}(z)$ are in $\C^+$, the limit $S(z)$ must be in $\overline{\C^+}$, which leaves for $S(z)$ only the possibility that
$$S(z)=\frac{-z+\sqrt{z^2-4}}2 = S_{\mu_W}(z)$$
(as the other solution is in $\C^-$).

In the same way, it follows that any subsequence of $(S_{\mu_N}(z))_N$ has a convergent subsequence which converges to $S_{\mu_W}(z)$; this forces all cluster points of $(S_{\mu_N}(z))_N$ to be $S_{\mu_W}(z)$. Thus the whole sequence converges to $S_{\mu_N}(z)$. This holds for any $z\in\C^+$, and thus implies that $\mu_N\tow \mu_W$.
\end{proof}

To complete the proof we still have to see the concentration to get the asymptotic vanishing of the variance. We will address such concentration questions in the next chapter.
\chapter{Concentration Phenomena and Stronger Forms of Convergence for the Semicircle Law}

\section{Forms of convergence}
\begin{remark}
\begin{enumerate}
\item
Recall that our random matrix ensemble is given by probability measures $\P_N$ on sets $\Omega_N$ of $N\times N$ matrices and we want to see that $\mu_{A_N}$ converges weakly to $\mu_W$, or, equivalently, that, for all $z\in\C^+$, $S_{\mu_{A_N}}(z)$ converges to $S_{\mu_W}(z)$. There are different levels of this convergence with respect to $\P_N$:
\begin{enumerate}
\item[(i)]
\emph{convergence in average}, i.e.,
$$\ev{S_{\mu_{A_N}}(z)} \Nto S_{\mu_W}(z)$$
\item[(ii)]
\emph{convergence in probability}, i.e.,
$$\P_N\{ A_N \mid \vert S_{\mu_{A_N}}(z)- S_{\mu_W}(z)\vert \geq \ee\} \Nto 0\qquad \text{for all $\ee>0$}$$
\item[(iii)]
\emph{almost sure convergence}, i.e.,
$$\P\{ (A_N)_N\mid \text{$S_{\mu_{A_N}}(z)$ does not converge to $S_{\mu_W}(z)$}\}=0;$$
instead of making this more precise, let us just point out that this almost sure convergence is guaranteed, by the Borel-Cantelli Lemma, if the convergence in (ii) to zero is sufficiently fast in $N$, so that for all $\ee>0$
$$\sum_{N=1}^\infty \P_N\{ A_N \mid \vert S_{\mu_{A_N}}(z)- S_{\mu_W}(z)\vert \geq \ee\} <\infty.$$
\end{enumerate}
\item
Note that we have here convergence of probabilistic quantities to a deterministic limit, thus (ii) and (iii) are saying that for large $N$ the eigenvalue distribution of $A_N$ concentrates in a small neighborhood of $\mu_W$. This is an instance of a quite general ``concentration of measure'' phenomenon; according to a dictum of M. Talagrand:
``A random variable that depends (in a smooth way) on the influence of many independent variables (but not too much on any of them) is essentially constant.''  
\item
Note also that many classical results in probability theory (like law of large numbers) can be seen as instances of this, dealing with linear functions. However, this principle also applies to non-linear functions - like in our case, to $\tr[(A-z1)^{-1}]$, considered as function of the entries of $A$.
\item
Often control of the variance of the considered function is a good way to get concentration estimates. We develop in the following some of the basics for this.
\end{enumerate}
\end{remark}

\section{Markov's and Chebyshev's inequality}

\begin{notation}
A \emph{probability space} $(\Omega,\P)$ consists of a set $\Omega$, equipped with a $\sigma$-algebra of measurable sets, and a probability measure $\P$ on the measurable sets of $\Omega$. A \emph{random variable} $X$ is a measurable function $X:\Omega\to\R$; its \emph{expectation} or \emph{mean} is given by
$$\ev{X}:=\int_\Omega X(\omega) d\P(\omega),$$
and its \emph{variance} is given by
$$\var{X}:=\ev{(X-\ev{X})^2}=\ev{X^2}-\ev{X}^2=\int_\Omega(X(\omega)-\ev(X))^2d\P(\omega).$$
\end{notation}

\begin{theorem}[Markov's Inequality]\label{thm:6.3}
Let $X$ be a random variable taking non-negative values. Then, for any $t>0$,
$$\P\{\omega\mid X(\omega)\geq t\} \leq \frac {\ev{X}} t.$$
\end{theorem}
$\ev{X}$ could here also be $\infty$, but then the statement is not very useful. The Markov inequality only gives useful information if $X$ has finite mean, and then only for $t>\ev{X}$.

\begin{proof}
Since $X(\omega)\geq 0$ for all $\omega\in\Omega$ we can estimate as follows:
\begin{align*}
\ev{X}&=\int_\Omega X(\omega)d\P(\omega)\\
&= \int_{\{X(\omega)\geq t\}} X(\omega)d\P(\omega)+\int_{\{X(\omega)< t\}} X(\omega) d\P(\omega) \\
&\geq
\int_{\{X(\omega)\geq t\}} X(\omega)d\P(\omega)\\
&\geq
\int_{\{X(\omega)\geq t\}}  t d\P(\omega)\\
&=t\cdot \P\{X(\omega)\geq t\}.
\end{align*}

\end{proof}

\begin{theorem}[Chebyshev's Inequality]\label{thm:6.4}
Let $X$ be a random variable with finite mean $\ev{X}$ and variance $\var{X}$. Then, for any $\ee>0$,
$$\P\{\omega\mid \vert X(\omega) -\ev{X}\vert\geq \ee\}\leq \frac {\var{X}}{\ee^2}.$$

\begin{proof}
We use Markov's inequality \ref{thm:6.3} for the positive random variable
$Y:=(X-\ev{X})^2$. Note that
$$\ev{Y}=\ev{(X-\ev{X})^2}=\var{X}.$$
Thus we have for $\ee>0$
\begin{align*}
\P\{\omega\mid \vert X(\omega-\ev(X)\vert\geq\ee\}&=
\P\{\omega\mid (X(\omega)-\ev{X})^2\geq \ee^2\}\\
&=\P[\{\omega
\mid Y(\omega)\geq \ee^2\}\\
&\leq \frac {\ev{Y}}{\ee^2}\\
&=\frac {\var{X}}{\ee^2}.
\end{align*}
\end{proof}

\end{theorem}

\begin{remark}
Our goal will thus be to control the variance of $X=f(X_1,\dots,X_n)$ for $X_1,\dots,X_n$ independent random variables. (In our case, the $X_i$ will be the entries of the GOE matrix $A$ and $f$ will be the function $f=\tr[(A-z1)^{-1}]$.)
A main idea in this context is to have estimates which go over from separate control of each variable to control of all variables together; i.e., which are stable under tensorization. There are two prominent types of such estimates, namely
\begin{enumerate}
\item[(i)]
Poincar\'e inequality
\item[(ii)]
LSI=logarithmic Sobolev inequality
\end{enumerate}
We will focus here on (i) and say a few words on (ii) later.
\end{remark}

\section{Poincar\'e inequality}

\begin{definition}
A random variable
$X=(X_1,\dots,X_n):\Omega\to\R^n$
satisfies a \emph{Poincar\'e inequality with constant $c>0$} if for any differentiable function $f:\R^n\to\R$ with $\ev{f(X)^2}<\infty$ we have
$$\var{f(X)}\leq c \cdot \ev{\Vert\nabla f(X)\Vert_2^2} \qquad
\text{where}\qquad
\Vert \nabla f\Vert_2^2=\sum_{i=1}^n (\frac {\partial f}{\partial x_i} )^2.$$
\end{definition}

\begin{remark}
Let us write this also ``explicitly'' in terms of the distribution $\mu$ of the random variable $X:\Omega\to\R^n$; recall that $\mu$ is the push-forward of the probability measure $\P$ under the map $X$ to a probability measure on $\R^n$. In terms of $\mu$ we have then
$$\ev{f(X)}=\int_{\R^n} f(x_1,\dots,x_n) d\mu(x_1,\dots,x_n)$$
and the Poincar\'e inequality asks then for 
\begin{multline*}
\int_{\R^n} \bigl(f(x_1,\dots,x_n)-\ev{f(X)}\bigr)^2 d\mu(x_1,\dots,x_n)\\
\leq c \cdot \sum_{i=1}^n\int_{\R^n} \left( \frac {\partial f}{\partial x_i} (x_1,\dots,x_n)\right)^2 d\mu(x_1,\dots,x_n).
\end{multline*}
\end{remark}

\begin{theorem}[Efron-Stein Inequality]\label{thm:6.8}
Let $X_1,\dots,X_n$ be independent random variables and let $f(X_1,\dots,X_n)$ be a square-integrable function of $X=(X_1,\dots,X_n)$. Then we have
$$\var{f(X)}\leq\sum_{i=1}^n \ev{\vari{f(X)}},$$
where $\mathbb{V}\text{ar}^{(i)}$ denotes taking the variance in the $i$-th variable, keeping all the other variables fixed, and the expectation is then integrating over all the other variables.
\end{theorem}

\begin{proof}
We denote the distribution of $X_i$ by $\mu_i$; this is, for each $i$, a probability measure on $\R$. Since $X_1,\dots,X_n$ are independent, the distribution of $X=(X_1,\dots,X_n)$ is given by the product measure
$\mu_1\times\cdots\times \mu_n$ on $\R^n$.

Putting $Z=f(X_1,\dots,X_n)$, we have
$$\ev{Z}=\int_{\R^n} f(x_1,\dots,x_n)d\mu_1\dots d\mu_n(x_n)$$
and 
$$\var{Z}=\int_{\R^n}\bigl( f(x_1,\dots,x_n)-\ev{Z}\bigr)^2 d\mu_1\dots d\mu_n.$$
We will now do the integration $\mathbb{E}$ by integrating one variable at a time and control each step. For this we write
\begin{align*}
Z-\ev{Z}&=\quad Z-\mathbb{E}_1[Z]\\
&\quad + \mathbb{E}_1[Z] - \mathbb{E}_{1,2}[Z]\\
&\quad + \mathbb{E}_{1,2}[Z] - \mathbb{E}_{1,2,3}[Z]\\
&\qquad \vdots\\
&\quad + \mathbb{E}_{1,2,\dots,n-1}[Z] -\ev{Z},
\end{align*}
where $\mathbb{E}_{1,\dots,k}$ denotes integration over the variables $x_1,\dots,x_k$, leaving a function of the variables $x_{k+1},\dots,x_n$.
Thus, with
$$\Delta_i:=\mathbb{E}_{1,\dots,i-1}[Z]-\mathbb{E}_{1,\dots,i-1,i}[Z]$$
(which is a function of the variables $x_i,x_{i+1},\dots,x_n$), we have
$Z-\ev{Z}=\sum_{i=1}^n \Delta_i$,
and thus
\begin{align*}
\var{Z}&=\var{(Z-\ev{Z})^2}\\
&=\var{\left(\sum_{i=1}^n \Delta_i\right)^2}\\
&=\sum_{i=1}^n \ev{\Delta_i^2}+\sum_{i\not= j}\ev{\Delta_i \Delta_j}.
\end{align*}
Now observe that for all $i\not= j$ we have $\ev{\Delta_i\Delta_j}=0$. Indeed, consider, for example, $n=2$ and $i=1$, $j=2$:
\begin{align*}
&\ev{\Delta_1 \Delta_2}=\ev{(Z-\mathbb{E}_1[Z])\cdot (\mathbb{E}_1[Z]-\mathbb{E}_{1,2}[Z])}\\
&=\int \left[ f(x_1,x_2)-\int f(\tilde x_1,x_2)d\mu_1(\tilde x_1)\right]\cdot \\
&\qquad\qquad\qquad\cdot
\left[ \int f(\tilde x_1,x_2)d\mu_1(\tilde x_1) - \int f(\tilde x_1,\tilde x_2)
d\mu_1(\tilde \mu_1) d\mu_2(\tilde x_2)\right] d\mu_1(x_1) d\mu_2(x_2)
\end{align*}
Integration with respect to $x_1$ now affects only the first factor and integrating this gives zero. The general case $i\not= j$ works in the same way. Thus we get
$$\var{Z}=\sum_{i=1}^n\ev{\Delta_i^2}.$$
We denote now with $\mathbb{E}^{(i)}$ integration with respect to the variable $x_i$, leaving a function of the other variables $x_1,\dots,x_{i-1},x_{i+1},\dots, x_n$, and
$$\vari{Z}:=\mathbb{E}^{(i)}[(Z-\mathbb{E}^{(i)}[Z])^2].$$
Then we have
$$
\Delta_i=\mathbb{E}_{1,\dots,i-1}[Z]-\mathbb{E}_{1,\dots,i}[Z]
=\mathbb{E}_{1,\dots,i-1}[Z-\mathbb{E}^{(i)}[Z]],
$$
and thus by Jensen's inequality (which is here just the fact that variances are non-negative),
$$\Delta_i^2\leq \mathbb{E}_{1,\dots,i-1}[(Z-\mathbb{E}^{(i)}[Z])^2].$$
This gives finally
$$\var{Z}=\sum_{i=1}^n \ev{\Delta_i^2}\leq \sum_{i=1}^n
\ev{\mathbb{E}^{(i)}[(Z-\mathbb{E}^{(i)}[Z])^2]},$$
which is the assertion.
\end{proof}

\begin{theorem}\label{thm:6.9}
Let $X_1,\dots,X_n$ be independent random variables in $\R$, such that each $X_i$ satisfies a Poincar\'e inequality with constant $c_i$. Then $X=(X_1,\dots,X_n)$ satisfies a Poicar\'e inequality in $\R^n$ iwth constant 
$c=\max(c_1,\dots,c_n)$.
\end{theorem}

\begin{proof}
By the Efron-Stein inequality \ref{thm:6.8}, we have
\begin{align*}
\var{f(X)}&\leq \sum_{i=1}^n \ev{\vari{f(X)}}\\
&\leq \sum_{i=1}^n \ev{c_i\cdot \mathbb{E}^{(i)}\left[\left(\frac{\partial f}{\partial x_i}\right)^2\right]}\\
&\leq c\cdot \sum_{i=1}^n \ev{\left(\frac {\partial f}{\partial x_i}\right)^2}\\
&=c \cdot\ev{\Vert \nabla f(X)\Vert^2_2}.
\end{align*}
In the step from the first to the second line we have used, for fixed $i$, the Poincar\'e inequality for $X_i$ and the function
$x_i\mapsto f(x_1,\dots,x_{i-1},x_i,x_{i+1},\dots,x_n)$, for each fixed $x_1,\dots,x_{i-1},x_{i+1},\dots,x_n$.
\end{proof}

\begin{theorem}[Gaussian Poincar\'e Inequality]\label{thm:6.10}
Let $X_1,\dots,X_n$ be independent standard Gaussian random variables, $\ev{X_i}=0$ and $\ev{X_i^2}=1$. Then $X=(X_1,\dots,X_n)$ satisfies a Poincar\'e inequality with constant 1; i.e., for each continuously differentiable $f:\R^n\to\R$ we have
$$\var{f(X)}\leq \ev{\Vert \nabla f(X)\Vert^2}.$$
\end{theorem}

\begin{remark}\label{rem:6.11}
\begin{enumerate}
\item
Note the independence of $n$!
\item
By Theorem \ref{thm:6.9} it suffices to prove the statement for $n=1$. But even in this one-dimensional case the statement is not obvious. Let us see what we are actually claiming in this case: $X$ is a standard Gaussian random variable and $f:\R\to\R$, and the Poincar\'e inequality says
$$\var{f(X)}\leq \ev{f'(X)^2}.$$
We might also assume that $\ev{f(X)}=0$, then this means explicitly:
$$\int_\R f(x)^2 e^{-x^2/2}dx\leq \int_\R f'(x)^2 e^{-x^2/2}dx.$$
\end{enumerate}
\end{remark}

\begin{proof}
As remarked in Remark \ref{rem:6.11}, the general case can, by Theorem \ref{thm:6.9}, be reduced to the one-dimensional case and, by shifting our function $f$ by a constant, we can also assume that $f(X)$ has mean zero. One possible proof is to approximate $X$ via a central limit theorem by independent Bernoulli variables $Y_i$. 

So let $Y_1,Y_2, \dots$ be independent Bernoulli variables, i.e., $\prob{Y_i=1}=1/2=\prob{Y_i=-1}$ and put
$$S_n=\frac {Y_1+\cdots+Y_n}{\sqrt{n}} .$$
Then, by the central limit theorem, the distribution of $S_n$ converges weakly, for $n\to\infty$, to a standard Gaussian distribution. So we can approximate $f(X)$ by
$$g(Y_1,\dots,Y_n)=f(S_n)=f\left(\frac 1{\sqrt{n}}(Y_1+\cdots+ Y_n)\right).$$
By the Efron-Stein inequality \ref{thm:6.8}, we have
\begin{align*}
\var{f(S_n)}&=\var{g(Y_1,\dots,Y_n)}\\&\leq \sum_{i=1}^n \ev{\vari{g(Y_1,\dots,Y_n)}}\\
&=\sum_{i=1}^n \ev{\vari{f(S_n)}}.
\end{align*}
Put
$$S_n^{[i]}:=S_n-\frac 1{\sqrt{n}} Y_i=\frac 1{\sqrt{n}}(Y_1+\cdots Y_{i-1}+
Y_{i+1}+\cdots+Y_n).$$
Then
$$\mathbb{E}^{(i)}[f(S_n)]=\frac 12\left( f\bigl(S_n^{[i]}+\frac 1{\sqrt n}\bigr) +
f\bigl(S_n^{[i]}-\frac 1{\sqrt n}\bigr)\right)$$
and
\begin{align*}
&\vari{f(S_n)}\\&=
\frac 12\left\{\left( f\bigl(S_n^{[i]}+\frac 1{\sqrt n}\bigr)-\mathbb{E}^{(i)}[f(S_n)]\right)^2+
\left( f\bigl(S_n^{[i]}-\frac 1{\sqrt n}\bigr)-\mathbb{E}^{(i)}[f(S_n)]\right)^2\right\}\\
&=\frac 14  \left(f\bigl(S_n^{[i]}+\frac 1{\sqrt n}\bigr) - f\bigl(S_n^{[i]}-\frac 1{\sqrt n}\bigr)\right)^2,
\end{align*}
and thus
$$\var{f(S_n)}\leq \frac 14\sum_{i=1}^n\ev{\left(f\bigl(S_n^{[i]}+\frac 1{\sqrt n}\bigr) - f\bigl(S_n^{[i]}-\frac 1{\sqrt n}\bigr)\right)^2}.$$
By Taylor's theorem we have now
\begin{align*}
f\bigl(S_n^{[i]}+\frac 1{\sqrt n}\bigr)&=f\bigl(S_n^{[i]}\bigr)+\frac 1{\sqrt n} f'\bigl(S_n^{[i]}\bigr)+\frac 1{2n}f''(\xi_+)\\
f\bigl(S_n^{[i]}-\frac 1{\sqrt n}\bigr)&=f\bigl(S_n^{[i]}\bigr)-\frac 1{\sqrt n} f'\bigl(S_n^{[i]}\bigr)+\frac 1{2n}f''(\xi_+)
\end{align*}
We assume that $f$ is twice differentiable and $f'$ and $f''$ are bounded: $\vert f'(\xi)\vert \leq K$ and
$\vert f''(\xi)\vert \leq K$ for all $\xi\in \R$. (The general situation can be approximated by this.)
Then we have
\begin{align*}
\left(f\bigl(S_n^{[i]}+\frac 1{\sqrt n}\bigr) - f\bigl(S_n^{[i]}-\frac 1{\sqrt n}\bigr)\right)^2&=
\left(\frac 2{\sqrt n} f'\bigl(S_n^{[i]}\bigr) + \frac 1{2n} \bigl(f''(\xi_+)-f''(\xi_-)\bigr)\right)^2\\
&=\frac 4n f'\bigl(S_n^{[i]}\bigr)^2 + \frac 2{n^{3/2}} R_1R_2 + \frac 1{4n^2}R_2^2,
\end{align*}
where we have put $R_1:=f'\bigl(S_n^{[i]}\bigr)$ and $R_2:=f''(\xi_+)-f''(\xi_-)$. Note that $\vert R_1\vert \leq K$ and $\vert R_2\vert\leq 2K$, and thus
$$\var{f(S_n)}\leq \frac 14 n\left\{ \frac 4n \ev{f'\bigl(S_n^{[i]}\bigr)^2}+\frac 2{n^{3/2}}2K^2 +
\frac 1{4n^2} 4K^2\right\}.$$
Note that the first term containing $S_n^{[i]}$ is actually independent of $i$. Now we take the limit $n\to\infty$ in this inequality; since both $S_n$ and $S_n^{[i]}$ converge to our standard Gaussian variable $X$ we obtain finally the wanted
$$\var{f(X)}\leq \ev{f'(X)^2}.$$

\end{proof}

\section{Concentration for $\tr[R_A(z)]$ via Poincar\'e inequality}

We apply this Gaussian Poincar\'e inequality now to our random matrix setting $A=(x_{ij})_{i,j=1}^N$, where $\{x_{ijj}\mid i\leq j\}$ are independent Gaussian random variables with 
$\ev{x_{ij}}=0$ and $\ev{x_{ii}}=2/N$ on the diagonal and $\ev{x_{ij}}=1/N$ ($i\not=j$) off the diagonal. Note that by a change of variable the constant in the Poincare inequality for this variances is given by $\max\{\sigma_{ij}^2\mid i\leq j\}=2/N$. Thus we have in our setting for nice real-valued $f$:
$$\var{f(A)}\leq \frac 2N\cdot  \ev{\Vert f(A)\Vert_2^2}.$$
We take now
$$g(A):=\tr[(A-z1)^{-1}]=\tr[R_A(z)]$$
and want to control $\var{g(A_N}$ for $N\to\infty$. Note that $g$ is complex-valued (since $z\in\C^+$), but we can estimate
$$\vert\var{g(A)}\vert=\vert \var{\Re g(A)+\sqrt{-1} \Im g(A)}\vert
\leq 2\bigl( \var{\Re g(A)} + \var{\Im g(A)}\bigr).$$
Thus it suffices to estimate the variance of real and imaginary part of $g(A)$.

We have, for $i<j$,
\begin{align*}
\frac{\partial g(A)}{\partial x_{ij}}&=\frac {\partial}{\partial x_{ij}} \tr[R_A(z)]\\
&=\frac 1N \sum_{k=1}^N \frac {\partial [R_A(z)]_{kk}}{\partial x_{ij}}\\
&=-\frac 1N\sum_{k=1}^N \bigl([R_A(z)]_{ki}\cdot [R_A(z)]_{jk} + [R_A(z)]_{kj}\cdot [R_A(z)]_{ik}\bigr)
\quad\qquad\text{by Lemma \ref{lem:5.4}}\\
&=-\frac 2N \sum_{k=1}^N [R_A(z)]_{ik}\cdot [R_A(z)]_{kj} \quad\quad\text{since $R_A(z)$ is symmetric, see proof of \ref{thm:5.5}}\\
&=-\frac 2N [R_A(z)^2]_{ij},
\end{align*}
and the same for $i=j$ with $2/N$ replaced by $1/N$. 

Thus we get for $f(A):=\Re g(A)=\Re \tr[R_A(z)]$:
$$\vert \frac{\partial f(A)}{\partial x_{ij}}\vert=\vert \Re \frac{\partial g(A)}{\partial x_{ij}}\vert \leq
\frac 2N \vert [R_A(z)^2]_{ij}\vert \leq\frac 2N \Vert R_A(z)^2\Vert \leq \frac 2{N\cdot (\Im z)^2} ,
$$
where in the last step we used the usual estimate for resolvents as in the proof of Theorem \ref{thm:5.5}.
Hence we have
$$\vert \frac {\partial f(A)}{\partial x_{ij}}\vert^2\leq \frac 4{N^2\cdot (\Im z)^4},$$
and thus our Gaussian Poincar\'e inequality \ref{thm:6.10} (with constant $2/N$) yields
$$\var{f(A)}\leq \frac 2N \cdot\sum_{i\leq j} \vert \frac {\partial f(A)}{\partial x_{ij}}\vert^2\leq
\frac 8{N\cdot (\Im z)^4}.
$$
The same estimate holds for the imaginary part and thus, finally, we have for the variance of the trace of the resolvent:
$$\var{\tr[R_A(z)]}\leq \frac {32}{N\cdot (\Im z)^4}.$$

The fact that $\var{\tr[R_A(z)]}$ goes to zero for $N\to\infty$ closes the gap in our proof of Theorem \ref{thm:5.5}. Furthermore, it also improves the type of convergence in Wigner's semicircle law.

\begin{theorem}
Let $A_N$ be \GOE\ random matrices as in \ref{def:5.1}. Then the eigenvalue distribution ${\mu_{A_N}}$ converges in probability to the semicircle distribution. Namely, for each $z\in\C^+$ and all $\ee>0$ we have
$$\lim_{N\to\infty}\P_N\{A_N\mid \vert S_{\mu_{A_N}}(z)-S_{\mu_W}(z)\vert\geq\ee\}=0.$$
\end{theorem}

\begin{proof}
By the Chebyshev inequality \ref{thm:6.4}, our above estimate for the variance implies for any $\ee>0$ that

\begin{align*}
\P_N\{ A_N\mid \vert
\tr[R_{A_N}(z)]-\ev{\tr[R_{A_N}(z)]}\vert \geq\ee\}&\leq \frac {\var{\tr[R_{A_N}]}}{\ee^2}\\&\leq 
\frac {32}{N\cdot (\Im z)^4\cdot \ee^2}
\Nto 0.
\end{align*}
Since we already know, by Theorem \ref{thm:5.5}, that $\lim_{N\to\infty}\ev{\tr[R_{A_N}(z)]}=S_{\mu_W}(z)$, this gives the assertion.
\end{proof}

\begin{remark}
Note that our estimate $\var{...}\sim 1/N$ is not strong enough to get almost sure convergence; one can, however, improve our arguments to get $\var{...}\sim 1/{N^2}$, which implies then also almost sure convergence.
\end{remark}

\section{Logarithmic Sobolev inequalities}

One actually has typcially even exponential convergence in $N$. Such stronger concentration estimates rely usually on so called logarithmic Sobolev inequalities

\begin{definition}
A probability measure on $\R^n$ satisfies a \emph{logarithmic Sobolev inequality (LSI)} with constant $c>0$, if for all nice $f$:
$$\text{Ent}_\mu(f^2)\leq 2c \int_{\R^n} \Vert \nabla f\Vert_2^2 d\mu,$$
where 
$$\text{Ent}_\mu(f):=\int_{\R^n} f\log f d\mu-\int_{\R^n} f d\mu \cdot \log \int_{\R^n} f d\mu$$
is an entropy like quantity.
\end{definition}

\begin{remark}
\begin{enumerate}
\item
As for Poincar\'e inequalities, logarithmic Sobolev inqualities are stable under tensorization and Gaussian measures satisfy LSI.
\item
From a logarithmic Sobolev inequality one can then derive a concentration inequality for our random matrices of the form
$$P_N\{ A_N\mid \vert \tr[R_{A_N}(z)]-\ev{\tr[R_{A_N}(z)]}\vert\geq \ee\}\leq const\cdot \exp\left(-\frac {N^2 \ee^2}2\cdot (\Im z)^4\right).$$
\end{enumerate}
\end{remark}
\chapter{Analytic Description of the Eigenvalue Distribution of Gaussian Random Matrices}

In Exercise \ref{exercise:7} we showed that the joint distribution of the entries $a_{ij} = x_{ij} + \sqrt{-1}y_{ij}$ of a \GUE\ $A= \left( a_{ij} \right)_{i,j=1}^N$ has density
\[ c \cdot\exp \left( - \frac{N}{2}  \Tr A^2 \right) \td A.
\]
This clearly shows the invariance of the distribution under unitary transformations: Let $U$ be a unitary $N\times N$ matrix and let $B=U^*AU =\left( b_{ij} \right)_{i,j=1}^N$. Then we have $\Tr B^2 = \Tr A^2$ and the volume element is invariant under unitary transformations, $\td B = \td A$.
Therefore, for the joint distributions of entries of $A$ and of $B$, respectively, we have
\[ c\cdot \exp \left( - \frac{N}{2}  \Tr B^2 \right) \td B = c \cdot\exp \left( - \frac{N}{2}  \Tr A^2 \right) \td A.
\]
Thus the joint distribution of entries of a \GUE\ is invariant under unitary transformations, which explains the name \textbf{G}aussian \textbf{U}nitary \textbf{E}nsemble. What we are interested in, however, are not the entries but the eigenvalues of our matrices. Thus we
should transform this density from entries to eigenvalues. Instead of \GUE, we will mainly consider the real case, i.e., \GOE. 

\section{Joint eigenvalue distribution for GOE and GUE}

Let us recall the definition of \GOE, see also
Definition \ref{def:5.1}.

\begin{definition}
	A \emph{Gaussian orthogonal random matrix} (\GOE) $A= \left( x_{ij} \right)_{i,j=1}^N$
	is given by real-valued entries $x_{ij}$ with $x_{ij} = x_{ji}$ for all $i,j=1,\dots, N$ and joint distribution
	\[ c_N \exp \left( - \frac{N}{4}  \Tr A^2 \right) \prod_{i \geq j}\td x_{ij}.
	\]
With $\GOEN$ we denote the \GOE\ of size $N\times N$.
\end{definition}

\begin{remark}
	\begin{enumerate}
		\item This is clearly invariant under orthogonal transformation of the entries.
		\item This is equivalent to independent real Gaussian random variables. Note, however, that the variance for the diagonal entries has to be chosen differently from the off-diagonals; see Remark \ref{rem:5.2}. Let us check this 
	for $N=2$ with 
		\[ A = 
		\begin{pmatrix}
			x_{11} & x_{12} \\
			x_{12} & x_{22}.
		\end{pmatrix}
		\]
		Then
		\begin{align*}
		\exp \left( - \frac{N}{4}  \Tr \begin{pmatrix}
		x_{11} & x_{12} \\
		x_{12} & x_{22}.
		\end{pmatrix}^2 \right) 
		&= \exp \left( - \frac{N}{4}  \left( x_{11}^2 + 2x_{12}^2 +x_{22}^2  \right)  \right)  \\
		&= \exp \left( - \frac{N}{4} x_{11}^2  \right)   
			 \exp \left( - \frac{N}{2} x_{12}^2  \right)    
			 \exp \left( - \frac{N}{4}x_{22}^2  \right);
		\end{align*}
those give the density of a Gaussian of variance $1/N$ for $x_{11}$ and $x_{22}$ and of variance $1/N$ for $x_{12}$.
		\item From this one can easily determine the normalization constant $c_N$ (as a function of $N$). 
	\end{enumerate}
\end{remark}

Since we are usually interested in functions of the eigenvalues, we will now transform this density to eigenvalues.

\begin{example}
	As a warmup, let us consider the \GOE(2) case,
	\[ A = 
	\begin{pmatrix}
	x_{11} & x_{12} \\
	x_{12} & x_{22}
	\end{pmatrix}
\qquad\text{with density}\qquad
 p(A) = c_2 \exp \left( - \frac{N}{4}  \Tr A^2 \right).
	\]
	We parametrize $A$ by its eigenvalues $\lambda_1$ and $\lambda_2$ and an angle $\theta$ by diagonalization $A=O^TDO$, where
	\[ D = 
	\begin{pmatrix}
	\lambda_1 & 0 \\
	0 & \lambda_2
	\end{pmatrix}
	\qquad \text{and} \qquad
	O = 
	\begin{pmatrix}
	\cos \theta & -\sin\theta \\
	\sin \theta & \cos \theta
	\end{pmatrix};
	\]
	explicityly
	\begin{align*}
	x_{11} &= \lambda_1 \cos^2\theta + \lambda_2 \sin^2 \theta, \\
	x_{12} &= (\lambda_1-\lambda_2)\cos\theta\sin\theta, \\
	x_{22} &= \lambda_1 \sin^2 \theta + \lambda_2 \cos^2\theta.
	\end{align*}
	Note that $O$ and $D$ are not uniquely determined by $A$. In particular, if $\lambda_1 = \lambda_2$ then any orthogonal $O$ works. However, this case has probability zero and thus can be ignored (see Remark \ref{rem:7.4}).
	If $\lambda_1\not=\lambda_2$, then we can choose
$\lambda_1 < \lambda_2$; $O$ contains then the normalized eigenvectors for $\lambda_1$ and $\lambda_2$. Those are unique up to a sign, which can be fixed by requiring that $\cos \theta \geq 0$. Hence $\theta$ is not running from $-\pi$ to $\pi$, but instead it can be restricted
	to $\left[ -{\pi}/{2}, {\pi}/{2} \right]$.
	We will now transform
	\[ p(x_{11}, x_{22}, x_{12}) \td x_{11} \td x_{22} \td x_{12}
	\to q(\lambda_1,\lambda_2,\theta) \td \lambda_1 \td \lambda_2 \td \theta
	\]
	by the change of variable formula
$ q = p \abs{\det \text{D} F}$,
	where $\text{D} F$ is the Jacobian of
	\[ F \colon  (x_{11}, x_{22}, x_{12})  \mapsto (\lambda_1,\lambda_2,\theta).
	\]
	We calculate
	\begin{align*}
	\det \text{D} F
	&= \det \begin{pmatrix}
		\cos^2\theta &\sin^2\theta&-2(\lambda_1-\lambda_2)\sin\theta\cos\theta \\
		\cos\theta \sin\theta & - \cos\theta\sin\theta &(\lambda_1-\lambda_2) \left( -\sin^2\theta+\cos^2\theta\right) \\
		\sin^2\theta &\cos^2\theta&2(\lambda_1-\lambda_2)\sin\theta\cos\theta \\
	\end{pmatrix}
	= -(\lambda_1 - \lambda_2),
	\end{align*}
and hence
	 $ \abs{\det \text{D} F}  = \abs{\lambda_1 - \lambda_2}$.  Thus,
	\[ q(\lambda_1,\lambda_2,\theta) 
	= c_2 e^{-\frac{N}{4}(\Tr(A^2))} \abs{\lambda_1 - \lambda_2}=	
c_2 e^{-\frac{N}{4}(\lambda_1^2 + \lambda_2^2)} \abs{\lambda_1 - \lambda_2}.			
	\]
	Note that $q$ is independent of $\theta$, i.e., we have a uniform distribution in $\theta$.
	Consider a function $f=f(\lambda_1,\lambda_2)$ of the eigenvalues. Then
	\begin{align*}
	\ev{f(\lambda_1,\lambda_2)}
	&= \int\int\int  q(\lambda_1,\lambda_2,\theta) f(\lambda_1,\lambda_2) \td \lambda_1 \td \lambda_2 \td \theta \\
	&= \int_{-\frac{\pi}{2}}^{\frac{\pi}{2}} \, \underset{\lambda_1 < \lambda_2}{\int\int}  f(\lambda_1,\lambda_2)  c_2 e^{-\frac{N}{4}(\lambda_1^2 + \lambda_2^2)} \abs{\lambda_1 - \lambda_2} \td \lambda_1 \td \lambda_2 \td \theta \\
	&= \pi c_2 \underset{\lambda_1 < \lambda_2}{\int\int} 
	  f(\lambda_1,\lambda_2)  e^{-\frac{N}{4}(\lambda_1^2 + \lambda_2^2)} \abs{\lambda_1 - \lambda_2} \td \lambda_1 \td \lambda_2.
	\end{align*}
	Thus, the density for the joint distribution of the eigenvalues on 
	$\{ (\lambda_1, \lambda_2); \, \lambda_1 < \lambda_2  \}$ is given by
	\[ \tilde{c}_2  \cdot e^{-\frac{N}{4}(\lambda_1^2 + \lambda_2^2)} \abs{\lambda_1 - \lambda_2} 
	\]
	with $\tilde{c}_2 = \pi c_2$.
\end{example}

\begin{remark}\label{rem:7.4}
	Let us check that the probability of $\lambda_1 = \lambda_2$ is zero.
	
	$\lambda_1, \lambda_2$ are the solutions of the characteristic equation
	\begin{align*}
	0 
	= \det (\lambda I -A)
 	&= (\lambda - x_{11})(\lambda-x_{22}) - x_{12}^2 \\
	&= \lambda^2 - (x_{11} + x_{22}) \lambda + (x_{11}x_{22}-x_{12}^2) \\
	&= \lambda^2 -b\lambda + c.
	\end{align*}
	Then there is only one solution if and only if the discriminant $d = b^2-4ac$ is zero.
	However,
	\[ \left\{ (x_{11}, x_{22}, x_{12}); \, d(x_{11}, x_{22}, x_{12}) = 0 \right\}
	\]
	is a two-dimensional surface in $\R^3$,
	i.e., its Lebesgue measure is zero. 
\end{remark}

Now we consider general \GOEN.

\begin{theorem}
	The joint distribution of the eigenvalues of a \GOEN\ is given by a density
	\[ \tilde{c}_N e^{-\frac{N}{4} (\lambda_1^2 + \cdots + \lambda_N^2) }
			\prod_{k<l} (\lambda_l - \lambda_k)
	\]
	restricted on $\lambda_1 < \cdots < \lambda_N$.
\end{theorem}

\begin{proof}
	In terms of the entries of the \GOE\ matrix $A$ we have density
	\[ p\left(x_{kl} \, | \, k \geq l  \right)
	= c_N e^{-\frac{N}{4}\Tr A^2},
	\]
	where $A=(x_{kl})_{k,l=1}^N$ with $x_{kl}$ real and $x_{kl} = x_{lk}$ for all $l,k$.
	Again we diagonalize $A=O^TDO$ with $O$ orthogonal and $D=\diag(\lambda_1, \dots, \lambda_N)$ with $\lambda_1 \leq \cdots \leq \lambda_N$.
	As before, degenerated eigenvalues have probability zero, hence this case can be neglected and we assume $\lambda_1 < \cdots < \lambda_N$. We parametrize $O$ via $O=e^{-H}$ by a \emph{skew-symmetric} matrix $H$, that is, $H^T=-H$, i.e., $H=(h_{ij})_{i,j=1}^N$ with $h_{ij}\in \R$ and
	$h_{ij} = -h_{ji}$ for all $i,j$. In particular, $h_{ii} = 0$ for all $i$.
	We have
	\[ O^T = \left( e^{-H} \right)^T = e^{-H^T} = e^H
	\]
	and thus $O$ is indeed orthogonal:
	\[ O^TO = e^{H}e^{-H} = e^{H-H} = e^0 = I = OO^T.
	\]
$O=e^{-H}$ is actually a parametrization of the Lie group $\text{SO}(N)$ by the Lie algebra $\text{so}(N)$ of skew-symmetric matrices.

	Note that our parametrization $A=e^HDe^{-H}$ has the right number of parameters. For $A$ we have the variables $\{ x_{ij} ; \, j \leq i \} $
	and for $e^HDe^{-H}$ we have the $N$ eigenvalues
	$\{ \lambda_1, \dots, \lambda_N  \} $ and the $\frac 12(N^2-N)$  many parameters $\{ h_{ij}; \, i >j  \}$.
In both cases we have $\frac 12N (N+1)$ many variables.
	This parametrization is locally bijective; 
	so we need to compute the Jacobian of the map
 $S \colon A \mapsto e^HDe^{-H}$.
	We have
	\begin{align*}
	\td A
	&= (\td e^{H}) D e^{-H} + e^H (\td D) e^{-H} + e^HD(\td e^{-H}) \\
	&= e^H \left[e^{-H}(\td e^{H}) D  + \td D -  D(\td e^{-H}) e^H \right] e^{-H}.
	\end{align*}
	This transports the calculation of the derivative at any arbitrary point $e^H$ to the identity element $I=e^0$ in the Lie group. Since the Jacobian is preserved under this transformation, it suffices to calculate the Jacobian at $H=0$, i.e., for $e^H=I$ and $\td e^H = \td H$. Then
	\[ \td A = \td H \cdot D - D\cdot \td H + \td D,
	\]
	i.e.,
	\[ \td x_{ij} = \td h_{ij} \lambda_j - \lambda_i \td h_{ij} + \td \lambda_i \delta_{ij}
	\]
	This means that we have
	\[ \frac{\partial x_{ij}}{\partial \lambda_k} = \delta_{ij} \delta_{ik}
\qquad\text{and}\qquad
\frac{\partial x_{ij}}{\partial h_{kl}} = \delta_{ik} \delta_{jl} (\lambda_l - \lambda_k).
	\]
	Hence the Jacobian is given by
	\[ J = \det DS 
= \prod_{k<l} (\lambda_l - \lambda_k).
	\]
	Thus,
	\begin{align*}
	q(\lambda_1,\dots, \lambda_N, h_{kl})
	= p(x_{ij} \, | \, i \geq j) J
	&= c_N e^{-\frac{N}{4}\Tr A^2} \prod_{k<l} (\lambda_l - \lambda_k)\\
&= c_N e^{-\frac{N}{4}(\lambda_1^2+\cdots+\lambda_N^2)} \prod_{k<l} (\lambda_l - \lambda_k).
	\end{align*}
	This is independent of the \enquote{angles} $h_{kl}$, so integrating over those variables just changes the constant $c_N$ into another constant $\tilde{c}_N$.
\end{proof}

In a similar way, the complex case can be treated; see Exercise \ref{exercise:19}. One gets the following.

\begin{theorem}\label{thm:7.6}
	The joint distribution of the eigenvalues of a \GUEN\ is given by a density
	\[ \hat{c}_N e^{-\frac{N}{2} (\lambda_1^2 + \cdots + \lambda_N^2) }
	\prod_{k<l} (\lambda_l - \lambda_k)^2
	\]
	restricted on $\lambda_1 < \cdots < \lambda_N$.
\end{theorem}

\section{Rewriting the Vandermonde}

\begin{definition}
	The function
	\[ \Delta(\lambda_1, \dots, \lambda_N) = \prod_{\substack{k,l=1 \\ k <l}}^{N} (\lambda_l -\lambda_k)
 	\]
 	is called the \emph{Vandermonde determinant}.
\end{definition}

\begin{prop}
	For $\lambda_1, \dots, \lambda_N \in \R$ we have that
	\[\Delta(\lambda_1, \dots, \lambda_N) = \det \left( \lambda_j^{i-1} \right)_{i,j=1}^N
	= \det \begin{pmatrix}
	1 & 1 & \cdots & 1\\
	\lambda_1 & \lambda_2 & \dots &\lambda_N\\
	\vdots & \vdots & \ddots & \vdots \\
	\lambda_1^{N-1} & \lambda_2^{N-1}  & \dots &\lambda_N^{N-1} 
	\end{pmatrix}.
	\]
\end{prop}

\begin{proof}
	$\det ( \lambda_j^{i-1})$ is a polynomial in $\lambda_1, \dots, \lambda_N$. If $\lambda_l = \lambda_k$ for some $l,k \in \{ 1, \dots, N \}$ then 
$\det ( \lambda_j^{i-1})=0$.
Thus $ \det ( \lambda_j^{i-1})$ contains a factor $\lambda_l-\lambda_k$ for each $k<l$, hence $\Delta(\lambda_1,\dots,\lambda_N)$ divides
$\det ( \lambda_j^{i-1})$.

Since $\det ( \lambda_j^{i-1})$ is a sum of products with one factor from each row, we have that the degree of $\det ( \lambda_j^{i-1})$ is equal to
$$0+1+2+\cdots+(N-1)=\frac 12 N(N-1),$$
which is the same as the degree of $\Delta(\lambda_1,\dots,\lambda_N)$. This shows that 
$$\Delta(\lambda_1,\dots,\lambda_N)=c\cdot \det ( \lambda_j^{i-1})\qquad\text{for some $c\in\R$.}$$
By comparing the coefficient of $1\cdot \lambda_2\cdot \lambda_3^2\cdots \lambda_N^{N-1}$ on both sides one can check that $c=1$.
	
\end{proof}

The advantage of being able to write our density in terms of a determinant comes from the following observation: In $\det ( \lambda_j^{i-1} )$ we can add arbitrary linear combinations of smaller rows to the $k$-th row without changing the value of the determinant, i.e.,
we can replace $\lambda^k$ by any arbitrary monic polynomial 
$p_k(\lambda) = \lambda^k + \alpha_{k-1} \lambda^{k-1} + \cdots + \alpha_1 \lambda + \alpha_0$ 
of degree $k$.
Hence we have the following statement.

\begin{prop}\label{prop:7.9}
	Let $p_0, \dots, p_{N-1}$ be monic polynomials with $\deg p_k = k$. Then we have
	\[ \det \left( p_{i-1}\left(\lambda_j\right)\right)_{i,j=1}^N
	= \Delta (\lambda_1, \dots, \lambda_N)
	= \prod_{\substack{k,l=1 \\ k <l}}^{N} (\lambda_l -\lambda_k).
	\]
\end{prop}

\section{Rewriting the GUE density in terms of Hermite kernels}

In the following, we will make a special choice for the $p_k$. We will choose them as the Hermite polynomials, which are orthogonal with respect to the Gaussian distribution $\frac{1}{c} e^{-\frac{1}{2} \lambda^2}$.

\begin{definition}\label{def:7.10}
	The \emph{Hermite polynomials} $H_n$ are defined by the following requirements.
	\begin{enumerate}
		\item[(i)] $H_n$ is a monic polynomial of degree $n$.
		\item[(ii)] For all $n,m \geq 0$:
		\[ \int_{\R} H_n(x)\overline{H_m(x)}\frac{1}{\sqrt{2\pi}} e^{-\frac{1}{2} x^2} \td x 
		= \delta_{nm} n!
		\]
	\end{enumerate}
\end{definition}

\begin{remark}
	\begin{enumerate}
		\item One can get the $H_n(x)$ from the monomials $1,x,x^2,\dots$ via Gram-Schmidt orthogonalization as follows.
		\begin{itemize}
			\item We define an inner product on the polynomials by
			\[ \langle f,g \rangle  
			= \frac{1}{\sqrt{2\pi}}\int_{\R} f(x) \overline{g(x)}  e^{-\frac{1}{2} x^2} \td x.
			\]
			\item We put $H_0(x) =1$. This is monic of degree $0$ with
			\[ \langle H_0,H_0 \rangle 
			= \frac{1}{\sqrt{2\pi}}\int_{\R}   e^{-\frac{1}{2} x^2} \td x
			= 1 = 0!.
			\]
			\item We put $H_1(x) =x$. This is monic of degree $1$ with
			\[ \langle H_1,H_0 \rangle 
			= \frac{1}{\sqrt{2\pi}}\int_{\R} x  e^{-\frac{1}{2} x^2} \td x
			= 0
			\]
			and
			\[ \langle H_1,H_1 \rangle 
			= \frac{1}{\sqrt{2\pi}}\int_{\R} x^2  e^{-\frac{1}{2} x^2} \td x
			= 1 = 1!.
			\]
			\item For $H_2$, note that
			\[ \langle x^2,H_1 \rangle 
			= \frac{1}{\sqrt{2\pi}}\int_{\R} x^3  e^{-\frac{1}{2} x^2} \td x
			= 0
			\]
			and
			\[ \langle x^2,H_0 \rangle 
			= \frac{1}{\sqrt{2\pi}}\int_{\R} x^2  e^{-\frac{1}{2} x^2} \td x
			= 1.
			\]
			Hence we set $H_2(x) := x^2 - H_0(x) = x^2-1$.
			Then we have
			\[ \langle H_2, H_0 \rangle = 0 = \langle H_2, H_1 \rangle 
			\]
			and
			\begin{align*} \langle H_2, H_2 \rangle 
			&= \frac{1}{\sqrt{2\pi}}\int_{\R} (x^2-1)^2  e^{-\frac{1}{2} x^2} \td x\\
   			&=\frac{1}{\sqrt{2\pi}}\int_{\R} (x^4-2x^2+1)  e^{-\frac{1}{2} x^2} \td x
			= 3-2+1=2!
			\end{align*}
			\item Continue in this way.
		\end{itemize}
		Note that the $H_n$ are uniquely determined by the requirements that $H_n$ is monic and that
		$\langle H_m, H_n \rangle  = 0$ for all $m \neq n$. That we have 
		$\langle H_n, H_n \rangle  = n!$, is then a statement which has to be proved.
		\item The Hermite polynomials satisfy many explicit relations; important is the three-term recurrence relation
		\[ xH_n(x) = H_{n+1}(x) + n H_{n-1}(x)
		\]
		for all $n \geq 1$; see Exercise \ref{exercise:22}.
		\item The first few $H_n$ are
		\begin{align*}
		H_0(x) &= 1, \\
		H_1(x) &= x, \\
		H_2(x) &= x^2-1, \\
		H_3(x) &= x^3-3x, \\
		H_4(x) &= x^4-6x^2+3.
		\end{align*}
		\item By Proposition \ref{prop:7.9}, we can now use the $H_n$ for writing our Vandermonde determinant as
		\[ \Delta (\lambda_1, \dots, \lambda_N)
		= \det \left( H_{i-1}\left(\lambda_j\right)\right)_{i,j=1}^N.
		\]
		We want to use this for our \GUEN\ density
		\begin{align*}
		q(\lambda_1, \dots, \lambda_N)
		&= \hat{c}_N e^{ -\frac{N}{2} ( \lambda_1^2 + \dots + \lambda_N^2) } \Delta (\lambda_1, \dots, \lambda_N)^2 \\
		&= \hat{c}_N e^{ -\frac{1}{2} ( \mu_1^2 + \dots + \mu_N^2) } 
			\Delta \left(\frac{\mu_1}{\sqrt{N}}, \dots, \frac{\mu_N}{\sqrt{N}}  \right)^2. \\
		&= \hat{c}_N e^{ -\frac{1}{2} ( \mu_1^2 + \dots + \mu_N^2) } 
			\Delta (\mu_1, \dots, \mu_N)^2  \left( \frac{1}{\sqrt{N}} \right)^{N(N-1)},
		\end{align*}
		where the $\mu_i = \sqrt{N} \lambda_i$ are the eigenvalues of the \enquote{unnormalized} \GUE\ matrix $\sqrt{N}A_N$. It will be easier to deal with those. We now will also go over from ordered eigenvalues $\lambda_1 < \lambda_2 < \cdots < \lambda_N$ to unordered eigenvalues 
		$(\mu_1, \dots, \mu_N) \in \R^N$. Since in the latter case each ordered tuple shows up $N!$ times, this gives an additional factor $N!$ in our density. We collect all these $N$-dependent factors in our constant $\tilde{c}_N$. So we now have the density
		\begin{align*}
		p(\mu_1, \dots, \mu_N)
		&= \tilde{c}_N e^{ -\frac{1}{2} ( \mu_1^2 + \dots + \mu_N^2) }
				\Delta (\mu_1, \dots, \mu_N)^2 \\
		&= \tilde{c}_N e^{ -\frac{1}{2} ( \mu_1^2 + \dots + \mu_N^2) }
		\left[ \det \left( H_{i-1}\left(\mu_j\right)\right)_{i,j=1}^N \right]^2 \\
		&= \tilde{c}_N 	\left[ \det \left( e^{ -\frac{1}{4} \mu_j^2} H_{i-1}\left(\mu_j\right)\right)_{i,j=1}^N \right]^2.
		\end{align*}
	\end{enumerate}
\end{remark}

\begin{definition}\label{def:7.12}
	The \emph{Hermite functions} $\Psi_n$ are defined by
	\[ \Psi_n(x) = (2\pi)^{-\frac{1}{4}} (n!)^{-\frac{1}{2}} e^{ -\frac{1}{4}x^2}H_n(x).
	\]
\end{definition}

\begin{remark}
	\begin{enumerate}
		\item We have
		\begin{align*}
		\int_\R \Psi_n(x)\Psi_m(x) \td x
		&= \frac{1}{\sqrt{2\pi}} \frac{1}{\sqrt{n!m!}} \int_\R e^{ -\frac{1}{4}x^2}H_n(x) H_m(x) \td x
		= \delta_{nm},
		\end{align*}
		i.e., the $\Psi_n$ are orthonormal with respect to the Lebesgue measure.
		Actually, they form an orthonormal Hilbert space basis of $L^2(\R)$.
		\item Now we can continue the calculation
		\begin{align*}
		p(\mu_1, \dots, \mu_N)
		&= c_N 	\left[ \det \left( \Psi_{i-1}\left(\mu_j\right)\right)_{i,j=1}^N \right]^2
		\end{align*}
		with a new constant $c_N$. Denote $V_{ij} = \Psi_{i-1}(\mu_j)$. Then we have
		\begin{align*}
		(\det V)^2
		&= \det V^T \det V
		= \det (V^TV)
		\end{align*}
		such that
		\begin{align*}
		(V^TV)_{ij}
		&= \sum_{k=1}^{N} V_{ki}V_{kj}
		= \sum_{k=1}^{N} \Psi_{k-1}(\mu_i) \Psi_{k-1}(\mu_j).
		\end{align*}
	\end{enumerate}
\end{remark}

\begin{definition}
	The $N$-th \emph{Hermite kernel} $K_N$ is defined by
	\[ K_N(x,y) = \sum_{k=0}^{N-1} \Psi_k(x) \Psi_k(y).
	\]
\end{definition}

Collecting all our notations and calculations we have thus proved the following.

\begin{theorem}\label{thm:7.15}
	The unordered joint eigenvalue distribution of an unnormalized \GUEN\ is given by the density
	\[ p(\mu_1, \dots, \mu_N)
	= c_N \det \left( K_N(\mu_i, \mu_j) \right)_{i,j=1}^N.
	\]
\end{theorem}

\begin{prop}
	$K_N$ is a \emph{reproducing kernel}, i.e.,
	\[ \int_\R K_N(x,u) K_N(u,y) \td u = K_N(x,y).
	\]
\end{prop}

\begin{proof}
	We calculate
	\begin{align*}
	\int_\R K_N(x,u) K_N(u,y) \td u
	&= \int_\R \left(\, \sum_{k=0}^{N-1} \Psi_k(x) \Psi_k(u) \right)
			\left(\, \sum_{l=0}^{N-1} \Psi_l(u) \Psi_l(y) \right) \td u \\
	&= \sum_{k,l=0}^{N-1} \Psi_k(x) \Psi_l(y) \int_\R \Psi_k(u) \Psi_l(u) \td u \\
&= \sum_{k,l=0}^{N-1} \Psi_k(x) \Psi_l(y) \delta_{kl} \\
	&= \sum_{k=0}^{N-1} \Psi_k(x) \Psi_k(y) \\
	&=K_N(x,y).
	\end{align*}
\end{proof}

\begin{lemma}\label{lem:7.17}
	Let $K\colon \R^2 \to \R$ be a reproducing kernel, i.e.,
	\[ \int_\R K(x,u) K(u,y) \td u = K(x,y).
	\]
	Put $d = \int_\R K(x,x) \td x$. Then, for all $n \geq 2$,
	\[ \int_\R \det \left( K(\mu_i, \mu_j) \right)_{i,j=1}^n \td \mu_n
	= (d-n+1) \cdot\det \left( K(\mu_i, \mu_j) \right)_{i,j=1}^{n-1}.
	\]
\end{lemma}
We assume that all those integrals make sense, as it is the case for our Hermite kernels.

\begin{proof}
	Consider the case $n=2$. Then
	\begin{align*}
	\int_\R 
	&\det \begin{pmatrix}
	K(\mu_1, \mu_1) & K(\mu_1, \mu_2) \\ 
	K(\mu_2, \mu_1) & K(\mu_2, \mu_2)
	\end{pmatrix}
	\td \mu_2\\
	&\qquad\qquad= K(\mu_1, \mu_1) \int_\R K(\mu_2, \mu_2) \td \mu_2 
		- \int_\R K(\mu_1, \mu_2)K(\mu_2, \mu_1) \td \mu_2 \\
	&\qquad\qquad= (d-1) K(\mu_1, \mu_1) \\
	&\qquad\qquad=  (d-1) K(\mu_1, \mu_1) \det \left(K(\mu_1, \mu_1) \right).
	\end{align*}
	For $n=3$,
	\begin{align*}
	&\det \begin{pmatrix}
	K(\mu_1, \mu_1) & K(\mu_1, \mu_2) & K(\mu_1, \mu_3) \\ 
	K(\mu_2, \mu_1) & K(\mu_2, \mu_2) & K(\mu_2, \mu_3) \\ 
	K(\mu_3, \mu_1) & K(\mu_3, \mu_2) & K(\mu_3, \mu_3)
	\end{pmatrix}\\
	&= \det \begin{pmatrix}
	K(\mu_2, \mu_1) & K(\mu_2, \mu_2) \\ 
	K(\mu_3, \mu_1) & K(\mu_3, \mu_2)
	\end{pmatrix} K(\mu_1, \mu_3) 
	-\det \begin{pmatrix}
	K(\mu_1, \mu_1) & K(\mu_1, \mu_2) \\ 
	K(\mu_3, \mu_1) & K(\mu_3, \mu_2)
	\end{pmatrix} K(\mu_2, \mu_3) \\
	&\qquad\qquad\qquad\qquad\qquad\qquad\qquad\qquad\qquad+ \det \begin{pmatrix}
	K(\mu_1, \mu_1) & K(\mu_1, \mu_2) \\
	K(\mu_2, \mu_1) & K(\mu_2, \mu_2)
	\end{pmatrix} K(\mu_3, \mu_3),	
	\end{align*}
	with
	\begin{align*}
	\int_\R  \det \begin{pmatrix}
	K(\mu_1, \mu_1) & K(\mu_1, \mu_2) \\
	K(\mu_2, \mu_1) & K(\mu_2, \mu_2)
	\end{pmatrix} K(\mu_3, \mu_3) \td \mu_3
	&=\det \begin{pmatrix}
	K(\mu_1, \mu_1) & K(\mu_1, \mu_2) \\
	K(\mu_2, \mu_1) & K(\mu_2, \mu_2)
	\end{pmatrix}\cdot  d,
	\end{align*}
	and
	\begin{align*}
	- \int_\R \det & \begin{pmatrix}
	K(\mu_1, \mu_1) & K(\mu_1, \mu_2) \\ 
	K(\mu_3, \mu_1) & K(\mu_3, \mu_2)
	\end{pmatrix} K(\mu_2, \mu_3) \td \mu_3 \\
	&\qquad\qquad\qquad= 	- \int_\R \det \begin{pmatrix}
	K(\mu_1, \mu_1) & K(\mu_1, \mu_2) \\ 
	K(\mu_2, \mu_3)  K(\mu_3, \mu_1) & K(\mu_2, \mu_3)  K(\mu_3, \mu_2) 
	\end{pmatrix}  \td \mu_3 \\
	&\qquad\qquad\qquad= - \det\begin{pmatrix}
	K(\mu_1, \mu_1) & K(\mu_1, \mu_2) \\ 
	K(\mu_2, \mu_1) & K(\mu_2, \mu_2)
	\end{pmatrix},
	\end{align*}
	and
	\begin{align*}
	\int_\R  \det &\begin{pmatrix}
	K(\mu_2, \mu_1) & K(\mu_2, \mu_2) \\ 
	K(\mu_3, \mu_1) & K(\mu_3, \mu_2)
	\end{pmatrix} K(\mu_1, \mu_3) \td \mu_3 \\
	&\qquad\qquad\qquad= \int_\R  \det \begin{pmatrix}
	K(\mu_2, \mu_1) & K(\mu_2, \mu_2) \\ 
	K(\mu_1, \mu_3)K(\mu_3, \mu_1) & K(\mu_1, \mu_3)K(\mu_3, \mu_2)
	\end{pmatrix}  \td \mu_3 \\
	&\qquad\qquad\qquad= \det\begin{pmatrix}
	K(\mu_2, \mu_1) & K(\mu_2, \mu_2) \\ 
	K(\mu_1, \mu_1) & K(\mu_1, \mu_2)
	\end{pmatrix} \\
	&\qquad\qquad\qquad= -\det\begin{pmatrix}
	K(\mu_1, \mu_1) & K(\mu_1, \mu_2) \\
	K(\mu_2, \mu_1) & K(\mu_2, \mu_2) 
	\end{pmatrix}.
	\end{align*}
Putting all terms together gives
$$\int_\R \det (K(\mu_i,\mu_j))_{i,j=1}^3\td \mu_3=
(d-2)\det (K(\mu_i,\mu_j))_{i,j=1}^2.$$

	The general case works in the same way.
\end{proof}

Iteration of Lemma \ref{lem:7.17} gives then the following.

\begin{cor}
	Under the assumptions of Lemma \ref{lem:7.17} we have
	\begin{align*}
	\int_\R \cdots \int_\R \det \left(  K(\mu_i,\mu_j) \right)_{i,j =1}^n \td \mu_1 \cdots \td \mu_n
	&= (d-n+1) (d-n+2) \cdots (d-1) d.
	\end{align*}
\end{cor}

\begin{remark}
	We want to apply this to the Hermite kernel $K=K_N$. In this case we have
	\begin{align*}
	d
	&= \int_\R K_N(x,x) \td x \\
	&= \int_\R  \sum_{k=0}^{N-1} \Psi_k(x) \Psi_k(x)  \td x \\
	&=  \sum_{k=0}^{N-1}  \int_\R  \Psi_k(x) \Psi_k(x)  \td x \\
	&= N,
	\end{align*}
	and thus, since now $d=N=n$,
	\begin{align*}
	\int_\R \cdots \int_\R \det \left(  K_N(\mu_i,\mu_j) \right)_{i,j =1}^N \td \mu_1 \cdots \td \mu_n
	&= N!.
	\end{align*}
	This now allows us to determine the constant $c_N$ in the density $p(\mu_1, \dots, \mu_n)$ in Theorem \ref{thm:7.15}.
	Since $p$ is a probability density on $\R^N$, we have
	\begin{align*}
	1
	&= \int_{\R^N} p(\mu_1, \dots, \mu_n) \td\mu_1 \cdots \td \mu_N \\
	&= c_N \int_\R \cdots \int_\R \det \left(  K_N(\mu_i,\mu_j) \right)_{i,j =1}^N \td \mu_1 \cdots \td \mu_N\\
	&= c_N N!,
	\end{align*}
	and thus $c_N = \frac{1}{N!}$.
\end{remark}

\begin{theorem}\label{thm:7.20}
	The unordered joint eigenvalue distribution of an unnormalized \GUEN\ is given by a density
	\[p(\mu_1, \dots, \mu_N) = \frac{1}{N!} \det \left(  K_N(\mu_i,\mu_j) \right)_{i,j =1}^N,
	\]
	where $K_N$ is the Hermite kernel
	\[ K_N(x,y) = \sum_{k=0}^{N-1} \Psi_k(x) \Psi_k(x) .
	\]
\end{theorem}

\begin{theorem}\label{thm:7.21}
	The averaged eigenvalue density of an unnormalized \GUEN\ is given by
	\begin{align*}
	p_N(\mu) 
	&= \frac{1}{N} K_N(\mu, \mu)=\frac 1N \sum_{k=0}^{N-1}\Psi_k(\mu)^2
	= \frac{1}{\sqrt{2\pi}} \frac{1}{N} \sum_{k=0}^{N-1} \frac{1}{k!} H_k(\mu)^2e^{-\frac{\mu^2}{2}}.
	\end{align*}
\end{theorem}

\begin{proof}
Note that $p(\mu_1,\dots,\mu_N)$ is the probability density to have
$N$ eigenvalues at the positions $\mu_1,\dots,\mu_N$. If we are integrating out $N-1$ variables we are left with the probability for one eigenvalue (without caring about the others).
	With the notation $\mu_N = \mu$ we get
	\begin{align*}
	p_N(\mu)
	&= \int_{\R^{N-1}} p(\mu_1, \dots, \mu_{N-1}, \mu) \td \mu_1 \cdots \td\mu_{N-1} \\
	&= \frac{1}{N!} \int_{\R^{N-1}} \det \left(  K_N(\mu_i,\mu_j) \right)_{i,j =1}^N \td \mu_1 \cdots \td\mu_{N-1} \\
	&= \frac{1}{N!} (N-1)! \det (K_N(\mu,\mu)) \\
	&= \frac{1}{N} K_N(\mu,\mu).
	\end{align*}
\end{proof}
\chapter{Determinantal Processes and Non-Crossing Paths: Karlin--McGregor and Gessel--Viennot}

	Our probability distributions for the eigenvalues of \GUE\ have a determinantal structure, i.e., are of the form
	\begin{align*} \label{\star}
	p(\mu_1, \dots, \mu_n) 
	&= \frac{1}{N!} \det \left(  K_N(\mu_i,\mu_j) \right)_{i,j =1}^N.
	\end{align*}
	They describe $N$ eigenvalues which repel each other (via the factor $(\mu_i-\mu_j)^2$).
	If we consider corresponding processes, then the paths of the eigenvalues should not cross; for this see also Section \ref{section:8.3}. There is a quite general relation between determinants as above and non-crossing paths. This appeared in fundamental papers in different contexts:
	\begin{itemize}
		\item in a paper by Karlin and McGregor, 1958, in the context of Markov chains and Brownian motion
		\item in a paper of Lindström, 1973, in the context of matroids
		\item in a paper of Gessel and Viennot, 1985, in combinatorics
	\end{itemize}

\section{Stochastic version à la Karlin--McGregor}
	Consider a random walk on the integers $\Z$:
	\begin{itemize}
		\item $Y_k$: position at time $k$
		\item $\Z$: possible positions
		\item Transition probability (to the two neighbors) might depend on position:
		\[ i-1 \xleftarrow{q_i} i \xrightarrow{p_i} i+1, \qquad q_i + p_1 = 1
		\]
	\end{itemize}
	We now consider $n$ copies of such a random walk, which at time $k=0$ start at different positions $x_i$.	We are interested in the probability that the paths don't cross.
	Let $x_i$ be such that all distances are even, i.e., if two paths cross they have to meet.

\begin{theorem}[Karlin--McGregor]\label{thm:8.1}
	Consider $n$ copies of $Y_k$, i.e., $(Y_k^{(1)}, \dots, Y_k^{(n)})$ with 
	$Y_0^{(i)} =x_i$, where $x_1 > x_2 > \cdots > x_n$. Consider now $t \in \N$ and  $y_1 > y_2 > \cdots > y_n$. Denote by 
	\[ P_t(x_i,y_j)
	= \prob{Y_t = y_j \, | \, Y_0 = x_i}
	\]
	the probability of one random walk to get from $x_i$ to $y_j$ in $t$ steps. Then we have
	\begin{align*}
	\prob{Y_t^{(i)}  = y_i \text{ for all } i, Y_s^{(1)} > Y_s^{(2)} > \cdots > Y_s^{(n)}
		\text{ for all } 0 \leq s \leq t  }
	&= \det \left( P_t(x_i,y_j) \right)_{i,j=1}^n. 
	\end{align*}
\end{theorem}

\begin{example}
	For one symmetric random walk $Y_t$ we have the following probabilities to go in two steps from 0 to -2,0,2:
	\[ \begin{tikzpicture}[thick,font=\small, node distance=.5cm and 1.0cm, baseline={([yshift=0ex]current bounding box.center)}]
	\tikzset{dot/.style={circle,fill=#1,inner sep=0,minimum size=4pt}}
	
	\node (1) 	[dot=black]				{};
	\node (2)	[dot=black, above right=of 1] 	{};
	\node (3)	[dot=black, above right=of 2] 	{};
	\node (4)	[dot=black, below right=of 2] 	{};
	\node (5)	[dot=black, below right=of 1] 	{};
	\node (6)	[dot=black, below right=of 5] 	{};
	
	\draw [->] (1) -- node [midway,above,align=center] {$p_0$} (2);
	\draw [->] (2) -- node [midway,above,align=center] {$p_1$} (3);
	\draw [->] (2) -- node [midway,above,align=center] {$q_1$} (4);
	\draw [->] (1) -- node [midway,below,align=center] {$q_0$} (5);
	\draw [->] (5) -- node [midway,below,align=center] {$\quad p_{-1}$} (4);
	\draw [->] (5) -- node [midway,below,align=center] {$q_{-1}$} (6);
	
	\node (7) [right=of 3] {$p_0p_1 = \frac{1}{4}$};
	\node (8) [right=of 4] {$p_0q_1 + q_0p_{-1} = \frac{1}{2}$};
	\node (9) [right=of 6] {$q_0q_{-1} = \frac{1}{4}$};
	\end{tikzpicture}
	\]
	Now consider two such symmetric random walks and set $x_1 = 2 = y_1$, $x_2 = 0 = y_2$.
	Then
	\[\prob{ Y_2^{(1)} = 2 = Y_0^{(1)}, Y_2^{(2)} = 0 = Y_0^{(2)}, Y_1^{(1)} > Y_1^{(2)} } 
	\]
	\[ = \prob{ \left\{
		\begin{tikzpicture}[thick,font=\small, node distance=.5cm and 1.0cm, baseline={([yshift=0ex]current bounding box.center)}]
		\tikzset{dot/.style={circle,fill=#1,inner sep=0,minimum size=4pt}}
		
		\node (1) 	[dot=black]						{};
		\node (2)	[dot=black, above right=of 1] 	{};
		\node (3)	[dot=black, below right=of 2] 	{};
		\node (4)	[dot=black, below right=of 1] 	{};
		\node (5)	[dot=black, below left=of 4] 	{};
		\node (6)	[dot=black, below right=of 4] 	{};
		\node (7)	[below=of 5] 	{};
		
		\draw [->] (1) -- (2);
		\draw [->] (2) -- (3);
		\draw [->] (5) -- (4);
		\draw [->] (4) -- (6);
		\end{tikzpicture}, \quad
		\begin{tikzpicture}[thick,font=\small, node distance=.5cm and 1.0cm, baseline={([yshift=0ex]current bounding box.center)}]
		\tikzset{dot/.style={circle,fill=#1,inner sep=0,minimum size=4pt}}
		
		\node (1) 	[dot=black]						{};
		\node (2)	[dot=black, above right=of 1] 	{};
		\node (3)	[dot=black, below right=of 2] 	{};
		\node (7)	[below right=of 1] 				{};
		\node (4)	[dot=black, below left=of 4] 	{};
		\node (5)	[dot=black, below right=of 4] 	{};
		\node (6)	[dot=black, above right=of 5] 	{};
		
		\draw [->] (1) -- (2);
		\draw [->] (2) -- (3);
		\draw [->] (4) -- (5);
		\draw [->] (5) -- (6);
		\end{tikzpicture}, \quad
		\begin{tikzpicture}[thick,font=\small, node distance=.5cm and 1.0cm, baseline={([yshift=0ex]current bounding box.center)}]
		\tikzset{dot/.style={circle,fill=#1,inner sep=0,minimum size=4pt}}
		
		\node (1) 	[dot=black]						{};
		\node (7)	[above right=of 1]				{};
		\node (2)	[dot=black, below right=of 1]	{};
		\node (3)	[dot=black, above right=of 2] 	{};
		\node (4)	[dot=black, below left=of 2]	{};
		\node (5)	[dot=black, below right=of 4] 	{};
		\node (6)	[dot=black, above right=of 5] 	{};
		
		\draw [->] (1) -- (2);
		\draw [->] (2) -- (3);
		\draw [->] (4) -- (5);
		\draw [->] (5) -- (6);
		\end{tikzpicture}
		\right\} }
	= \frac{3}{16}.
	\]
	Note that $
	\begin{tikzpicture}[thick,font=\small, node distance=.5cm and 1.0cm, baseline={([yshift=0ex]current bounding box.center)}]
	\tikzset{dot/.style={circle,fill=#1,inner sep=0,minimum size=4pt}}
	
	\node (1) 	[dot=black]						{};
	\node (2)	[dot=black, below right=of 1]	{};
	\node (3)	[dot=black, above right=of 2] 	{};
	\node (4)	[dot=black, below left=of 2] 	{};
	\node (5)	[dot=black, below right=of 2] 	{};
	
	\draw [->] (1) -- (2);
	\draw [->] (2) -- (3);
	\draw [->] (4) -- (2);
	\draw [->] (2) -- (5);
	\end{tikzpicture}
	$ is not allowed. 
	
	\bigskip
	Theorem \ref{thm:8.1} says that we also obtain this probability from the transition probabilities of one random walk as
	\[ \det \begin{pmatrix}
	1/2 & {1}/{4} \\ {1}/{4} & {1}/{2}
	\end{pmatrix}
	= \frac{1}{4} - \frac{1}{16}
	= \frac{3}{16}.
	\]
\end{example}

\begin{proof}[Proof of Theorem \ref{thm:8.1}]
	Let $\Omega_{ij}$ be the set of all possible paths in $t$ steps from $x_i$ to $y_j$. Denote by $\prob{\pi}$ the probability for such a path $\pi\in\Omega_{ij}$. Then we have
	\[ P_t(x_i,y_j) = \sum_{\pi\in\Omega_{ij}} \prob{\pi}
	\]
	and we have to consider the determinant
	\[\det\left( P_t(x_i,y_j) \right)_{i,j=1}^n
	= \det\left(\,\sum_{\pi\in\Omega_{ij}} \prob{\pi}\right)_{i,j=1}^n.
	\]
	Let us consider the case $n=2$:
	\begin{align*}
	\det \begin{pmatrix}
	\sum_{\pi\in\Omega_{11}} \prob{\pi}
	& \sum_{\pi\in\Omega_{12}} \prob{\pi} \\
	\sum_{\pi\in\Omega_{21}} \prob{\pi}
	& \sum_{\pi\in\Omega_{22}} \prob{\pi}
	\end{pmatrix}
	&= \sum_{\pi\in\Omega_{11}} \prob{\pi} \cdot \sum_{\sigma\in\Omega_{22}} \prob{\sigma} 
		- \sum_{\pi\in\Omega_{12}} \prob{\pi} \cdot\sum_{\sigma\in\Omega_{21}} \prob{\sigma}
	\end{align*}
	Here, the first term counts all pairs of paths $x_1 \to y_1$ and $x_2 \to y_2$; hence non-crossing ones, but also crossing ones.
	However, such a crossing pair of paths is, via the \enquote{reflection principle}  (where we exchange the parts of the two paths after their first crossing), in bijection with a pair of paths from $x_1 \to y_2$ and $x_2\to y_1$; this bijection also preserves the probabilities.
	
	Those paths,  $x_1 \to y_2$ and $x_2\to y_1$, are counted by the second term in the determinant. Hence the second term cancels out all the crossing terms in the first term, leaving only the non-crossing paths.
	
For general $n$ it works in a similar way.
\end{proof}

\section{Combinatorial version à la Gessel--Viennot}
	Let $G$ be a weighted directed graph without directed cycles, e.g.
	\[\begin{tikzpicture}[thick,font=\small, node distance=1.0cm and 1.0cm, baseline={([yshift=0ex]current bounding box.center)}]
	\tikzset{dot/.style={circle,fill=#1,inner sep=0,minimum size=4pt}}
	
	\node (1) 	[dot=black]						{};
	\node (2)	[dot=black, below right=of 1]	{};
	\node (3)	[dot=black, below =of 2] 		{};
	\node (5)	[dot=black, below left=of 1] 	{};
	\node (4)	[dot=black, below =of 5]		{};
	
	\draw [->] (2) -- (1);
	\draw [->] (5) -- (1);
	\draw [->] (4) -- (5);
	\draw [->] (4) -- (3);
	\draw [->] (3) -- (5);
	\draw [->] (3) -- (2);
	\end{tikzpicture}
	\]
	where we have weights $m_{ij} = m_e$ on each edge $i \xrightarrow{e} j$. This gives weights for directed paths
	\[ P = \begin{tikzpicture}[thick,font=\small, node distance=.5cm and 1.0cm, baseline={([yshift=0ex]current bounding box.center)}]
	\tikzset{dot/.style={circle,fill=#1,inner sep=0,minimum size=4pt}}
	
	\node (1) 	[dot=black]						{};
	\node (2)	[dot=black, above right=of 1]	{};
	\node (3)	[dot=black, below right=of 2] 	{};
	\node (4)	[dot=black, above right=of 3] 	{};
	\node (5)	[dot=black, below right=of 4]	{};
	
	\draw [->] (1) -- (2);
	\draw [->] (2) -- (3);
	\draw [->] (3) -- (4);
	\draw [->] (4) -- (5);
	\end{tikzpicture}
	\qquad\text{via}\qquad
	 m(P) = \prod_{e\in P} m_e,
	\]
	and then also a weight for connecting two vertices $a,b$,
	\[ m(a,b) = \sum_{P \colon a \to b} m(P),
	\]
where we sum over all directed paths from $a$ to $b$. Note that this is a finite sum, because we do not have directed cycles in our graph.

\begin{definition}
	Consider two $n$-tuples of vertices $A=(a_1,\dots, a_n)$ and $B=(b_1,\dots,b_n)$.
	A \emph{path system} $P\colon A\to B$ is given by a permutation $\sigma \in S_n$ and paths $P_i \colon a_i \to b_{\sigma(i)}$ for $i=1,\dots, n$. We also put $\sigma(P) = \sigma$ and $\sgn P = \sgn \sigma$.
	A \emph{vertex-disjoint} path system is a path system $(P_1, \dots, P_n)$, where the paths 
	$P_1, \dots, P_n$ do not have a common vertex.

\end{definition}

\begin{lemma}[Gessel--Viennot]
	Let $G$ be a finite acyclic weighted directed graph and let $A=(a_1,\dots, a_n)$ and $B=(b_1,\dots,b_n)$ be two $n$-sets of vertices. Then we have
	\[ \det \left( m(a_i,b_j) \right)_{i,j=1}^n
	= \sum_{\substack{P\colon A \to B \\ \text{vertex-disjoint}}} \sgn \sigma(P) \prod_{i=1}^{n} m(P_i).
	\]
\end{lemma}

\begin{proof}
	Similar as the proof of Theorem \ref{thm:8.1}; the crossing paths cancel each other out in the determinant.
\end{proof}

This lemma can be useful in two directions. Whereas in the stochastic setting one uses mainly the determinant to count non-crossing paths, one can also count vertex-disjoint path systems to calculate determinants. The following is an example of this.

\begin{example}
	Let $C_n$ be the Catalan numbers
	\[ C_0 =1, C_1=1, C_2=2,C_3=5,C_4=14, \dots
	\]
	and consider
	\[ M_n = \begin{pmatrix}
	C_0 & C_1 & \cdots & C_n \\
	C_1 & C_2 & \cdots & C_{n+1} \\
	\vdots & & \ddots & \vdots \\
	C_n & C_{n+1} & \cdots & C_{2n}
	\end{pmatrix}.
	\]
	Then we have
	\begin{align*}
	\det M_0 &= \det (1)=1, \\
	\det M_1 &= 
\det \begin{pmatrix} 1&1\\1&2\end{pmatrix}=
2-1 = 1, \\
	\det M_2 &=
\det\begin{pmatrix}
1&1&2\\1&2&5\\2&5&14
\end{pmatrix}=
 28 + 10 +10 -8-14-25 = 1. 
	\end{align*}
	This is actually true for all $n$: $\det M_n =1$. This is not obvious directly, but follows easily from Gessel--Viennot, if one chooses the right setting.

	Let us show it for $M_2$. For this, consider the graph
	\[\begin{tikzpicture}[thick,font=\small, node distance=1.0cm and 1.0cm, baseline={([yshift=0ex]current bounding box.center)}]
	\tikzset{dot/.style={circle,fill=#1,inner sep=0,minimum size=4pt}}
	
	\node (11) 	[dot=black]							{};
	\node (21)	[dot=black, above=of 11]			{};
	\node (31)	[dot=black, above=of 21]			{};
	\node (41)	[dot=black, above=of 31]			{};
	\node (51)	[dot=black, above=of 41]			{};
	\node (22)	[dot=black, right=of 21]			{};
	\node (32)	[dot=black, above=of 22]			{};
	\node (42)	[dot=black, above=of 32]			{};
	\node (52)	[dot=black, above=of 42]			{};
	\node (33)	[dot=black, right=of 32]			{};
	\node (43)	[dot=black, above=of 33]			{};
	\node (53)	[dot=black, above=of 43]			{};
	\node (44)	[dot=black, right=of 43]			{};
	\node (54)	[dot=black, above=of 44]			{};
	\node (55)	[dot=black, right=of 54]			{};
	
	\draw [->] (11) edge (21) edge (31) edge (41) edge (51);
	\draw [->] (22) edge (32) edge (42) edge (52);
	\draw [->] (33) edge (43) edge (53);
	\draw [->] (44) edge (54);
	
	\draw [->] (51) edge (52) edge (53) edge (54) edge (55);
	\draw [->] (41) edge (42) edge (43) edge (44);
	\draw [->] (31) edge (32) edge (33);
	\draw [->] (21) edge (22);
	
	\draw (11) -- node [midway,right,align=center] {$a_2$} (11);
	\draw (22) -- node [midway,right,align=center] {$a_1$} (22);
	\draw (33) -- node [midway,right,align=center] {$a_0=b_0$} (33);
	\draw (44) -- node [midway,right,align=center] {$b_1$} (44);
	\draw (55) -- node [midway,right,align=center] {$b_2$} (55);
	\end{tikzpicture}
	\]
The possible directions in the graph are up and right, and all weights are chosen as 1.
	Paths in this graph correspond to Dyck graphs, and thus the weights for connecting the $a$'s with the $b$'s are counted by Catalan numbers; e.g.,
	\begin{align*}
	m(a_0,b_0) &= C_0, \\
	m(a_0,b_1) &= C_1, \\
	m(a_0,b_2) &= C_2, \\
	m(a_2,b_2) &= C_4.
	\end{align*}
Thus
$$M_2=\det\begin{pmatrix}
m(a_0,b_0)& m(a_0,b_1) & m(a_0,b_2)\\
m(a_1,b_0) & m(a_1,b_1) & m(a_1,b_2)\\
m(a_2,b_0) & m(a_2,b_1) & m(a_2,b_2)
\end{pmatrix}$$
and hence, by Gessel-Viennot,
	\begin{align*}
	\det M_2
	&= \det \left( m(a_i,b_j) \right)_{i,j=0}^2
	= \sum_{\substack{P\colon (a_0,a_1,a_2) \to (b_0,b_1,b_2) \\ \text{vertex-disjoint}}} 1
	=1,
	\end{align*}
	since there is only one such vertex-disjoint system of three paths, corresponding to $\sigma = \id$. This is given as follows; note that the path from $a_0$ to $b_0$ is actually a path with 0 steps.

\[\begin{tikzpicture}[thick,font=\small, node distance=1.0cm and 1.0cm, baseline={([yshift=0ex]current bounding box.center)}]
	\tikzset{dot/.style={circle,fill=#1,inner sep=0,minimum size=4pt}}
	
	\node (11) 	[dot=black]							{};
	\node (21)	[dot=black, above=of 11]			{};
	\node (31)	[dot=black, above=of 21]			{};
	\node (41)	[dot=black, above=of 31]			{};
	\node (51)	[dot=black, above=of 41]			{};
	\node (22)	[dot=black, right=of 21]			{};
	\node (32)	[dot=black, above=of 22]			{};
	\node (42)	[dot=black, above=of 32]			{};
	\node (52)	[dot=black, above=of 42]			{};
	\node (33)	[dot=black, right=of 32]			{};
	\node (43)	[dot=black, above=of 33]			{};
	\node (53)	[dot=black, above=of 43]			{};
	\node (44)	[dot=black, right=of 43]			{};
	\node (54)	[dot=black, above=of 44]			{};
	\node (55)	[dot=black, right=of 54]			{};
	
	\draw [->] (11) edge (21) edge (31) edge (41) edge (51);
	\draw [->] (22) edge (32) edge (42);
	%\draw [->] (33) edge (43) edge (53);
	%\draw [->] (44) edge (54);
	
	\draw [->] (51) edge (52) edge (53) edge (54) edge (55);
	\draw [->] (42) edge (43) edge (44);
	%\draw [->] (31) edge (32) edge (33);
	%\draw [->] (21) edge (22);
	
	\draw (11) -- node [midway,right,align=center] {$a_2$} (11);
	\draw (22) -- node [midway,right,align=center] {$a_1$} (22);
	\draw (33) -- node [midway,right,align=center] {$a_0=b_0$} (33);
	\draw (44) -- node [midway,right,align=center] {$b_1$} (44);
	\draw (55) -- node [midway,right,align=center] {$b_2$} (55);
	\end{tikzpicture}
	\]
\end{example}

\section{Dyson Brownian motion and non-intersecting paths}\label{section:8.3}

We have seen that the eigenvalues of random matrices repel each other. This becomes even more apparent when we consider process versions of our random matrices, where the eigenvalue processes yield then non-intersecting paths. Those process versions of our Gaussian ensembles are called \emph{Dyson Brownian motions}. They are defined as 
$A_N(t):=(a_{ij}(t))_{i,j=1}^N$ ($t\geq 0$), where each $a_{ij}(t)$ is a classical Brownian motion (complex or real) and they are independent, apart from the symmetry condition $a_{ij}(t)=\bar a_{ji}(t)$ for all $t\geq 0$ and all $i,j=1,\dots,N$. The eigenvalues $\lambda_1(t),\dots,\lambda_N(t)$ of $A_N(t)$ give then $N$ non-intersecting Brownian motions.

Here are plots for discretized random walk versions of the Dyson Brownian motion, corresponding to
 $\GOE(13)$, $\GUE(13)$ and, for comparision, also 13 independent Brownian motions; see also Exercise \ref{exercise:24}. Guess which is which! 
$$\includegraphics[width=3in]{dyson-13-complex}
\qquad\includegraphics[width=3in]{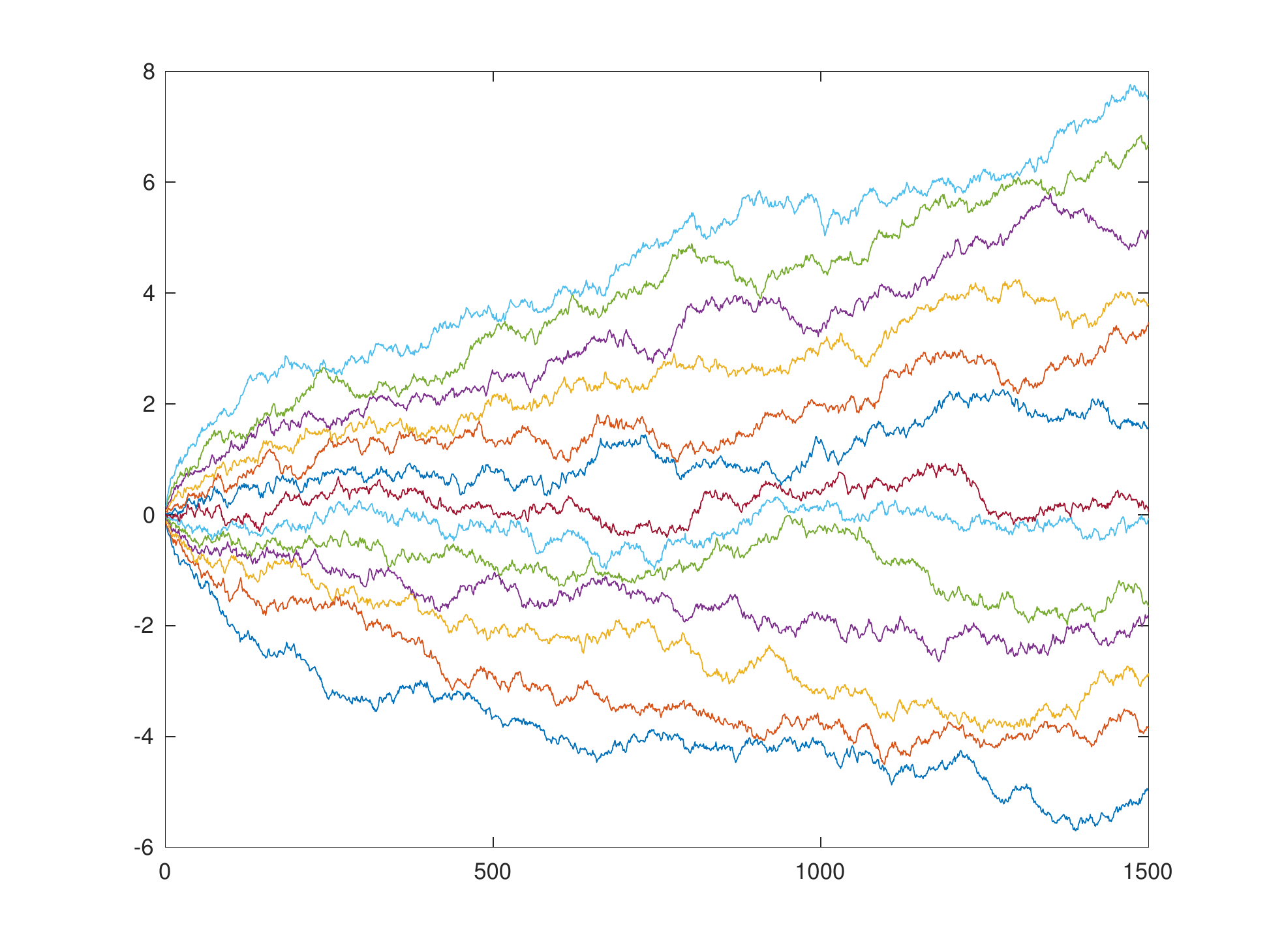}$$

$$\includegraphics[width=3in]{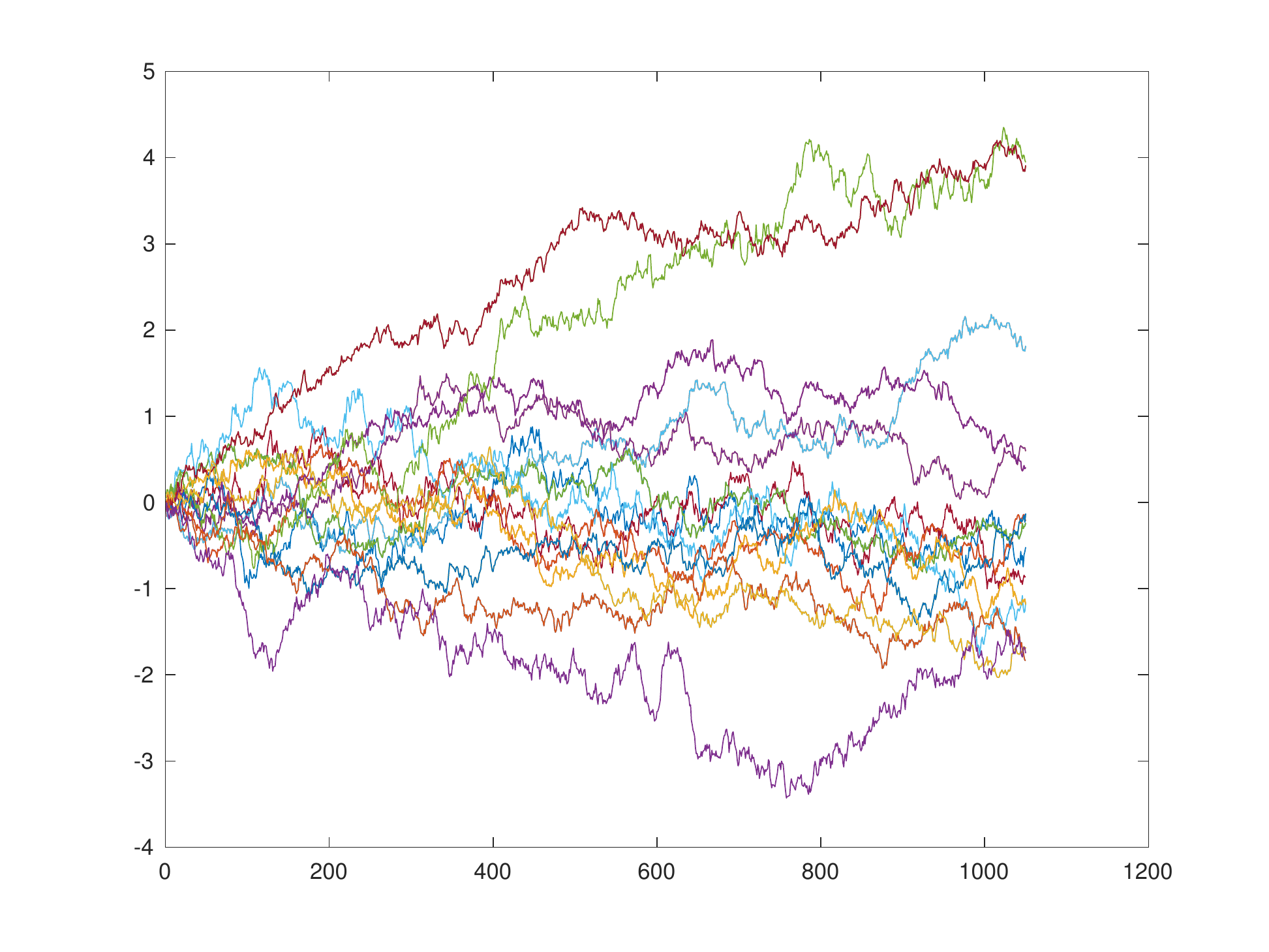}$$

\chapter{Statistics of the Largest Eigenvalue and Tracy--Widom Distribution}

Consider \GUEN\ or \GOEN. For large $N$, the eigenvalue distribution is close to a semircircle with density 
\[ p(x)= \frac{1}{2\pi} \sqrt{4-x^2}.
\]

\begin{tikzpicture}[scale=0.50]
\tikzset{dot/.style={circle,fill=#1,inner sep=0,minimum size=4pt}}

%\node[dot=black] at (20, 0)   (3)  {};
%\node[dot=black] at (21.5, 0)   (4)  {};

%\draw (20,-0.6) -- node {$s$} (20,-0.6);
\draw (21.5,-0.6) -- node {bulk} (21.5,-0.6);
\draw (26.5,-0.6) -- node {edge} (26.5,-0.6);

\draw (16,0) -- (27,0);

\draw (16.5,0) -- (26.5,0) arc(0:180:5) --cycle;

\begin{scope}
\clip (16.5,0) -- (26.5,0) arc(0:180:5) --cycle;
\draw [pattern=north east lines, pattern color=black] (21,0) rectangle (21.5,5);
\draw [pattern=north east lines, pattern color=black] (26,0) rectangle (26.5,5);
\end{scope}
\end{tikzpicture}

We will now zoom to a microscopic level and try to understand the behaviour of a single eigenvalue.
The behaviour in the \textit{bulk} and at the \textit{edge} is different.
We are particularly interested in the largest eigenvalue.
Note that at the moment we do not even know whether the largest eigenvalue sticks close to $2$ with high probability. Wigner's semicircle law implies that it cannot go far below $2$, but it does not prevent it from being very large. We will in particular see that this cannot happen.

\section{Some heuristics on single eigenvalues}
	Let us first check heuristically what we expect as typical order of fluctuations of the eigenvalues. For this we assume (without justification) that the semicircle predicts the behaviour of eigenvalues down to the microscopic level.

	\bigskip
	\textbf{Behaviour in the bulk:} 
	In $[\lambda, \lambda +t ]$ there should be $\sim tp(\lambda)N$ eigenvalues. This is of order $1$ if we choose $t \sim {1}/{N}$. This means that eigenvalues in the bulk have for their position an interval of size $\sim {1}/{N}$, so this is a good guess for the order of fluctuations for an eigenvalue in the bulk.

	\bigskip
	\textbf{Behaviour at the edge:} In $[2-t,2]$ there should be roughly 
	\[ N \int_{2-t}^{2} p(x) \td x 
	= \frac{N}{2\pi} \int_{2-t}^{2} \sqrt{(2-x)(2+x)} \td x
	\]
	many eigenvalues. To have this of order $1$, we should choose $t$ as follows:
	$$
		1
		\approx \frac{N}{2\pi} \int_{2-t}^{2} \sqrt{(2-x)(2+x)} \td x 
		\approx \frac{N}{\pi} \int_{2-t}^{2} \sqrt{2-x} \td x 
		= \frac{N}{\pi} \frac{2}{3} t^\frac{3}{2}
	$$
	Thus $1 \sim t^{{3}/{2}} N$, i.e., $t \sim N^{-{2}/{3}}$.
	Hence we expect for the largest eigenvalue an interval or fluctuations of size 
	$N^{-{2}/{3}}$. Very optimistically, we might expect
	\[ \lambda_{\max} \approx 2 + N^{-{2}/{3}}X,
	\]
	where $X$ has $N$-independent distribution.

\section{Tracy--Widom distribution}
	This heuristics (at least its implication) is indeed true and one has that the limit
	\begin{align*}
		F_\beta(x)
		&:= \lim\limits_{N\to\infty} \prob{  N^{{2}/{3}} (\lambda_{\max} - 2) \leq x }
	\end{align*}
	exists. It is called the \emph{Tracy--Widom distribution}.

\begin{remark}
	\begin{enumerate}
		\item Note the parameter $\beta$! This corresponds to:
		\begin{center}
			\begin{tabular}{c|c|c}
{ensemble}& $\beta$& {repulsion}\\
				\GOE& $1$& $(\lambda_i - \lambda_j)^1$\\ 
				\GUE& $2$& $(\lambda_i - \lambda_j)^2$\\ 
				\GSE& $4$& $(\lambda_i - \lambda_j)^4$
			\end{tabular}
		\end{center}
		It turns out that the statistics of the largest eigenvalue is different for real, complex, quaternionic Gaussian random matrices. The behaviour on the microscopic level is more sensitive to the underlying symmetry than the macroscopic behaviour; note that we get the semicircle as the macroscopic limit for all three ensembles. (In models in physics the choice of $\beta$ corresponds often to underlying physical symmetries; e.g., \GOE\ is used to describe systems which have a time-reversal symmetry.)
		\item On the other hand, when $\beta$ is fixed, there is a large universality class for the corresponding Tracy--Widom distribution. $F_2$ shows up as limiting fluctuations for
		\begin{itemize}
			\item[(a)] largest eigenvalue of \GUE\ (Tracy, Widom, 1993),
			\item[(b)] largest eigenvalue of more general Wigner matrices (Soshnikov, 1999),
			\item[(c)] largest eigenvalue of general unitarily invariant matrix ensembles (Deift et al., 1994-2000),
			\item[(d)] length of the longest increasing subsequence of random permutations
						(Baik, Deift, Johansson, 1999; Okounkov, 2000),
			\item[(e)] arctic cicle for Aztec diamond (Johansson, 2005),
			\item[(f)] various growth processes like ASEP (``asymmetric single exclusion process''), TASEP (``totally asymmetric ...'').
		\end{itemize}
		\item There is still no uniform explanation for this universality. The feeling is that the Tracy-Widom distribution is somehow the analogue of the normal distribution for a kind of central limit theorem, where independence is replaced by some kind of dependence.
		But no one can make this precise at the moment.
		\item Proving Tracy--Widom for \GUE\ is out of reach for us, but we will give some ideas.
		In particular, we try to derive rigorous estimates which show that our $N^{-{2}/{3}}$-heuristic is of the right order and, in particular, we will prove that the largest eigenvalue converges almost surely to $2$.
	\end{enumerate}
\end{remark}

\section{Convergence of the largest eigenvalue to 2}

	We want to derive an estimate, in the \GUE\ case, for the probability $\prob{\lambda_{\max}\geq 2+ \varepsilon}$, which is compatible with our heuristic that $\varepsilon$ should be of the order $N^{-{2}/{3}}$.
	We will refine our moment method for this.
Let	$A_N$ be our normalized \GUEN. We have for all $k\in\N$:
	\begin{align*}
		\prob{\lambda_{\max}\geq 2+ \varepsilon}
		&= \prob{\lambda_{\max}^{2k}\geq (2+ \varepsilon)^{2k}} \\[0.5ex]
		&\leq \prob{ \, \sum_{j=1}^{N} \lambda_j^{2k}\geq (2+ \varepsilon)^{2k}} \\[0.5ex]
		&= \prob{ \tr A_N^{2k}  \geq  \frac{(2+ \varepsilon)^{2k}}{N} } \\[0.5ex]
		&\leq \frac{N}{(2+ \varepsilon)^{2k}} \ev{ \tr A_N^{2k}  }.
	\end{align*}
In the last step we used Markov's inequality \ref{thm:6.3}; note that we have even powers, and hence the random variable $\tr(A_N^{2k})$ is positive.

	In Theorem \ref{thm:2.15} we calculated the expectation in terms of a genus expansion as
	$$
		\ev{\tr(A_N^{2k})}
		=  \sum_{\pi\in\pair(2k)} N^{\#(\gamma\pi)- k -1} 
		= \sum_{g\geq 0} \varepsilon_g(k) N^{-2g},
	$$
	where
	\begin{align*}
		\varepsilon_g(k)
		&= \# \left\{ \pi\in\pair(2k) \mid  \text{$\pi$ has genus $g$}  \right\}.
	\end{align*}
	The inequality
	\begin{align*}
		\prob{\lambda_{\max}\geq 2+ \varepsilon}
		&\leq \frac{N}{(2+ \varepsilon)^{2k}} \ev{ \tr A_N^{2k}  }
	\end{align*}
	is useless if $k$ is fixed for $N\to \infty$, because then the right hand side goes to $\infty$.
	Hence we also have to scale $k$ with $N$ (we will use $k \sim N^{{2}/{3}}$), but then the sub-leading terms in the genus expansion become important.
	Up to now we only know that $\varepsilon_0(k) = C_k$, but now we need some information on the other $\varepsilon_g(k)$. This is provided by a theorem of Harer and Zagier.

\begin{theorem}[Harer--Zagier, 1986]\label{thm:9.2}
	Let us define $b_k$ by
	\begin{align*}
		\sum_{g\geq 0} \varepsilon_g(k) N^{-2g} = C_k b_k,
	\end{align*}
	where $C_k$ are the Catalan numbers. (Note that the $b_k$ depend also on $N$, but we suppress this dependency in the notation.) Then we have the recursion formula
	\begin{align*}
		b_{k+1}
		&= b_k + \frac{k(k+1)}{4N^2} b_{k-1}
	\end{align*}
	for all $k \geq 2$.
\end{theorem}

We will prove this later; see Section \ref{section:9.6}. For now, let us just check it for small examples.

\begin{example}
	
	From Remark \ref{ex:2.16} we know
	\begin{align*}
	C_1 b_1 
	&= \ev{\tr A_N^2}
	= 1, \\[0.5ex]
	C_2 b_2 
	&= \ev{\tr A_N^4}
	= 2 + \frac{1}{N^2}, \\[0.5ex]
	C_3 b_3
	&= \ev{\tr A_N^6}
	= 5 + \frac{10}{N^2}, \\[0.5ex]
	C_4 b_4
	&= \ev{\tr A_N^8}
	= 14 + \frac{70}{N^2}+ \frac{21}{N^4}, 
	\end{align*}
	which gives
$$
	b_1 = 1, \qquad
	b_2 = 1 + \frac{1}{2N^2}, \qquad
	b_3 = 1+ \frac{2}{N^2}, \qquad
	b_4 = 1 + \frac{5}{N^2} + \frac{3}{2N^4}.
$$
	We now check the recursion from Theorem \ref{thm:9.2} for $k=3$:
	\begin{align*}
	b_3 + \frac{k(k+1)}{4N^2} b_2
	&= 1+ \frac{2}{N^2} + \frac{12}{4N^2} \left( 1 + \frac{1}{2N^2} \right) \\[0.5ex]
	&= 1+ \frac{5}{N^2} + \frac{3}{2N^4} \\[0.5ex]
	&= b_4
	\end{align*}
\end{example}

\begin{cor}
	For all $N,k \in \N$ we have for a \GUEN\ matrix $A_N$ that
	\[ \ev{\tr A_N^{2k}} \leq C_k \exp\left(\frac{k^3}{2N^2}\right).
	\]
\end{cor}

\begin{proof}
	Note that, by definition, $b_k > 0$ for all $k \in \N$ and hence, by Theorem \ref{thm:9.2}, $b_{k+1} > b_k$. Thus,
	\begin{align*}
	b_{k+1}
	= b_k + \frac{k(k+1)}{4N^2} b_{k-1} 
	\leq b_k \left( 1 +  \frac{k(k+1)}{4N^2}  \right) 
	\leq b_k \left( 1 +  \frac{k^2}{2N^2}  \right);
	\end{align*}
iteration of this yields
	\begin{align*}
	b_k
	&\leq \left( 1 +  \frac{(k-1)^2}{2N^2}  \right) \left( 1 +  \frac{(k-2)^2}{2N^2}  \right) 
		\cdots \left( 1 +  \frac{1^2}{2N^2}  \right) \\[0.5ex]
	& \leq \left( 1 +  \frac{k^2}{2N^2}  \right)^k \\[0.5ex]
	& \leq \exp \left( \frac{k^2}{2N^2}  \right)^k \qquad\qquad\qquad\qquad\text{since $1+x \leq e^x$}\\[0.5ex]
	&= \exp \left( \frac{k^3}{2N^2}  \right)
	\end{align*}
\end{proof}

We can now continue our estimate for the largest eigenvalue.
\begin{align*}
\prob{\lambda_{\max}\geq 2+ \varepsilon}
&\leq \frac{N}{(2+ \varepsilon)^{2k}} \ev{ \tr A_N^{2k}  } \\[0.5ex]
&=\frac{N}{(2+ \varepsilon)^{2k}} C_k b_k\\[0.5ex]
&\leq \frac{N}{(2+ \varepsilon)^{2k}} C_k \exp\left(\frac{k^3}{2N^2}\right)\\[0.5ex]
&\leq \frac{N}{(2+ \varepsilon)^{2k}} \frac{4^k}{k^{{3}/{2}}} \exp\left(\frac{k^3}{2N^2}\right).
\end{align*}
For the last estimate, we used (see Exercise \ref{exercise:26})
\[ C_k \leq \frac{4^k}{\sqrt{\pi}k^{{3}/{2}}}\leq \frac{4^k}{k^{{3}/{2}}}.
\]
Let us record our main estimate in the following proposition.

\begin{prop}\label{prop:9.5}
For a normalized \GUEN\ matrix\ $A_N$ we have for all $N,k\in\N$ and all $\ee>0$
$$\prob{\lambda_{\max}(A_N)\geq 2+ \varepsilon}
\leq \frac{N}{(2+ \varepsilon)^{2k}} \frac{4^k}{k^{{3}/{2}}} \exp\left(\frac{k^3}{2N^2}\right).$$
\end{prop}

This estimate is now strong enough to see that the largest eigenvalue has actually to converge to 2. For this, let us fix $\varepsilon > 0$ and choose $k$ depending on $N$ as
$k_N := \lfloor N^{{2}/{3}} \rfloor$, where $\floor x$ denotes the smallest integer $\geq x$.
Then
\[ \frac{N}{k_N^{{3}/{2}}} \xrightarrow{N \to \infty} 1
\qquad\text{and}\qquad
\frac{k_N^3}{2N^2}  \xrightarrow{N \to \infty} \frac{1}{2}.
\]
Hence
\begin{align*}
\limsup_{N\to \infty} \prob{\lambda_{\max}\geq 2+ \varepsilon}
&\leq \lim_{N\to \infty} \left( \frac{2}{2+\varepsilon} \right)^{2k_N}\cdot 1\cdot  e^{{1}/{2}}
= 0,
\end{align*}
and thus for all $\varepsilon >0$,
\[ \lim_{N\to \infty} \prob{\lambda_{\max}\geq 2+ \varepsilon} = 0.
\]
This says that the largest eigenvalue $\lambda_{\max}$ of a \GUE\ converges in probability to $2$.

\begin{cor}\label{cor:9.6}
	For a normalized \GUEN\ matrix $A_N$ we have that its largest eigenvalue converges in probability, and also almost surely, to $2$, i.e.,
	\[ \lambda_{\max} (A_N) \xrightarrow{N \to \infty} 2
	\qquad \text{almost surely.}
	\]
\end{cor}

\begin{proof}
The convergence in probability was shown above. For the strenghtening to almost sure convergence one has to use Borel--Cantelli and the fact that
	\[ \sum_N  \left( \frac{2}{2+\varepsilon} \right)^{2k_N}  < \infty.
	\]
See Exercise \ref{exercise:28}.
\end{proof}

\section{Estimate for fluctuations}

Our estimate from Proposition \ref{prop:9.5} gives also some information about the fluctuations of $\lambda_{\max}$ about $2$, if we choose $\varepsilon$ also depending on $N$. Let us use there now
\[ k_N = \lfloor N^{{2}/{3}} r \rfloor
\qquad\text{and}\qquad
\varepsilon_N =  N^{-{2}/{3}} t.
\]
Then
$$\frac{N}{k_N^{{3}/{2}}}  \xrightarrow{N \to \infty}  \frac{1}{r^{{3}/{2}}} \qquad\text{and}\qquad
\frac{k_N^3}{2N^2}  \xrightarrow{N \to \infty}  \frac{r^3}{2} ,$$
and
$$\frac{4^{k_N}}{(2+\varepsilon_N)^{2k_N}}
= \left(  \frac{1}{1+ \frac{1}{2N^{{2}/{3}}} t }  \right)^{2\floor{N^{{2}/{3}}r}}
\xrightarrow{N \to \infty} e^{-rt},
$$
and thus
\begin{align*}
\limsup_{N\to \infty} \prob{\lambda_{\max}\geq 2+ tN^{-{2}/{3}}}
&\leq \frac{1}{r^{{3}/{2}}} e^{{r^3}/{2}} e^{-rt}
\end{align*}
for arbitrary $r>0$. We optimize this now by choosing $r = \sqrt{t}$ for $t > 0$ and get

\begin{prop}
For a normalized \GUEN\ matrix $A_N$ we have for all $t>0$
\begin{align*}
\limsup_{N\to \infty} \prob{\lambda_{\max}(A_N)\geq 2+ tN^{-\frac{2}{3}}}
&\leq t^{-{3}/{4}} \exp\left({-\frac{1}{2}t^{{3}/{2}}}\right).
\end{align*}
\end{prop}

Although this estimate does not prove the existence of the limit on the left hand side, it turns out that the right hand side is quite sharp and captures the tail behaviour of the Tracy--Widom distribution quite well.

\section{Non-rigorous derivation of Tracy--Widom distribution}

	For determining the Tracy--Widom fluctuations in the limit $N\to\infty$ one has to use the analytic description of the \GUE\ joint density. Recall from Theorem \ref{thm:7.20} that the joint density of the unordered eigenvalues of an unnormalized \GUEN\ is given by
	\[p(\mu_1, \dots, \mu_N) = \frac{1}{N!} \det \left(  K_N(\mu_i,\mu_j) \right)_{i,j =1}^N,
	\]
	where $K_N$ is the Hermite kernel
	\[ K_N(x,y) = \sum_{k=0}^{N-1} \Psi_k(x) \Psi_k(y)
	\]
	with the Hermite functions $\Psi_k$ from Definition \ref{def:7.12}.
	Because $K_N$ is a reproducing kernel, we can integrate out some of the eigenvalues and get a density of the same form. If we are integrate out all but $r$ eigenvalues we get, by Lemma \ref{lem:7.17},
	\begin{align*}
	\int_\R \cdots \int_\R p(\mu_1, \dots, \mu_N) \td \mu_{r+1} \cdots \td \mu_N
	&= \frac{1}{N!} \cdot 1 \cdot 2 \cdots (N-r) \cdot \det \left(  K_N(\mu_i,\mu_j) \right)_{i,j =1}^r \\
	&= \frac{(N-r)!}{N!} \det \left(  K_N(\mu_i,\mu_j) \right)_{i,j =1}^r \\
	&=: p_N(\mu_1, \dots, \mu_r).
	\end{align*} 
	Now consider
	\begin{align*}
	&\prob{ \mu^{(N)}_{\max} \leq t }
	= \prob{\text{there is no eigenvalue in } (t,\infty)} \\[0.5ex]
	&= 1 - \prob{\text{there is an eigenvalue in } (t,\infty)} \\[0.5ex]
	&= 1- \Big[  N \prob{\mu_1 \in (t,\infty)} 
			- \binom{N}{2} \prob{\mu_1, \mu_2 \in (t,\infty)} 
	+ \text{$\binom{N}{3}$} \prob{\mu_1, \mu_2, \mu_3 \in (t,\infty)} 
			-\cdots
			\Big] \\[0.5ex]
	&= 1 + \sum_{r=1}^{N} (-1)^r \binom{N}{r} \int_{t}^{\infty} \cdots \int_{t}^{\infty}
		p_N(\mu_1, \dots, \mu_r) \td \mu_1 \cdots \td \mu_r \\
	&= 1 + \sum_{r=1}^{N} (-1)^r \frac{1}{r!}\int_{t}^{\infty} \cdots \int_{t}^{\infty}
		\det \left(  K(\mu_i,\mu_j) \right)_{i,j =1}^r \td \mu_1 \cdots \td \mu_r.
	\end{align*}
	Does this have a limit for $N\to \infty$?
	
	\bigskip
	Note that $p$ is the distribution for a \GUEN\ without normalization, i.e.,
	$\mu_{\max}^{(N)} \approx 2\sqrt{N}$.
	More precisely, we expect fluctutations
	\[ \mu_{\max}^{(N)} \approx \sqrt{N} \left( 2+tN^{-{2}/{3}} \right)
	= 2 \sqrt{N} +tN^{-{1}/{6}}.
	\]
	We put
	\[ \tilde{K}_N(x,y) = N^{-{1}/{6}}\cdot  K_N \left( 2\sqrt{N} + x N^{-{1}/{6}}, 2\sqrt{N} + y N^{-{1}/{6}}  \right)
	\]
	so that we have
	\begin{align*}
	\prob{N^{{2}/{3}} \left( \frac{ \mu_{\max}^{(N)} -2 }{\sqrt{N}} -2 \right) \leq t }
	&= \sum_{r=0}^{N} \frac{(-1)^r}{r!} \int_{t}^{\infty} \cdots \int_{t}^{\infty}
	\det \left(  \tilde{K}(x_i,x_j) \right)_{i,j =1}^r \td x_1 \cdots \td x_r.
	\end{align*}
	We expect that the limit
	\begin{align*}
	F_2(t)
	&:= \lim\limits_{N\to\infty} \prob{N^{{2}/{3}} \left( \frac{ \mu_{\max}^{(N)} -2 }{\sqrt{N}} -2 \right) \leq t }
	\end{align*}
	exists. For this, we need the limit $\lim\limits_{N\to\infty}\tilde{K}_N(x,y)$.
	Recall that
	\begin{align*}
	K_N(x,y)
	&= \sum_{k=0}^{N-1} \Psi_k(x) \Psi_k(y).
	\end{align*}
As this involves $\Psi_k$ for all $k=0,1,\dots,N-1$ this is not amenable to taking the limit $N\to\infty$. However, by the Christoffel--Darboux identity
	for the Hermite functions (see Exercise \ref{exercise:27})) 
	\begin{align*}
	\sum_{k=0}^{n-1} \frac {H_k(x)H_k(y)}{k!}=\frac {H_n(x)H_{n-1}(y)-H_{n-1}(x)H_n(y)}{(x-y)\  (n-1)! }
	\end{align*}
	and with
	\[ \Psi_k(x) = (2\pi)^{-{1}/{4}} (k!)^{-{1}/{2}} e^{ -\frac{1}{4}x^2}H_k(x),
	\]
	as defined in Definition \ref{def:7.12}, we can rewrite $K_N$ in the form
	\begin{align*}
	K_N(x,y)
	&= \frac{1}{\sqrt{2\pi}} \sum_{k=0}^{N-1} \frac{1}{k!} e^{ -\frac{1}{4} \left( x^2 + y^2 \right)}
		H_k(x)H_k(y) \\[0.5ex]
	&= \frac{1}{\sqrt{2\pi}} e^{ -\frac{1}{4} \left( x^2 + y^2 \right)}
	\frac {H_N(x)H_{N-1}(y)-H_{N-1}(x)H_N(y)}{(x-y)\  (N-1)! } \\[0.5ex]
	&= \sqrt{N} \cdot \frac {\Psi_N(x)\Psi_{N-1}(y)-\Psi_{N-1}(x)\Psi_N(y)}{x-y}.
	\end{align*}
	Note that the $\Psi_N$ satisfy the differential equation (see Exercise \ref{exercise:30})
	\[\Psi_N'(x)=-\frac{x}{2} \Psi_N(x)+\sqrt{N} \Psi_{N-1}(x),
	\]
and thus
	\begin{align*}
	K_N(x,y)
	&=\frac{
		\Psi_N(x) \left[ \Psi_N'(y) + \frac{y}{2} \Psi_N(y) \right]		
		-  \left[ \Psi_N'(x) + \frac{x}{2} \Psi_N(x) \right] \Psi_N(y)
		}{x-y} \\[0.9ex]
	&=\frac{\Psi_N(x) \Psi_N'(y) - \Psi_N'(x)  \Psi_N(y)}{x-y}
		- \frac{1}{2} \Psi_N(x)\Psi_N(y).
	\end{align*}
	Now put
	\[ \widetilde{\Psi}_N(x): = N^{{1}/{12}} \cdot\Psi_N \left( 2 \sqrt{N} + xN^{-{1}/{6}} \right),
	\]
	thus
	\[ \widetilde{\Psi}_N'(x) = N^{{1}/{12}}\cdot \Psi_N' \left( 2 \sqrt{N} + xN^{-{1}/{6}} \right)\cdot N^{-{1}/{6}}=
N^{-{1}/{12}}\cdot \Psi_N' \left( 2 \sqrt{N} + xN^{-{1}/{6}} \right).
	\]
	Then
	\begin{align*}
	\tilde{K}(x,y)
	&= \frac{ \widetilde{\Psi}_N(x)\widetilde{\Psi}_N'(y) - \widetilde{\Psi}_N'(x)\widetilde{\Psi}_N(y)  }{x-y} - \frac{1}{2N^{{1}/{3}}} \widetilde{\Psi}_N'(x)\widetilde{\Psi}_N'(y).
	\end{align*}
	One can show, by a quite non-trivial steepest descent method, that $\widetilde{\Psi}_N(x)$ converges to a limit. Let us call this limit the \emph{Airy function}
	\[ \Ai(x) = \lim\limits_{N\to\infty}\widetilde{\Psi}_N(x).
	\]
	The convergence is actually so strong that also
	\[ \Ai'(x) = \lim\limits_{N\to\infty}\widetilde{\Psi}_N'(x),
	\]
	and hence
	\begin{align*}
	\lim\limits_{N\to\infty}\tilde{K}(x,y)
	&= \frac{ \Ai(x)\Ai'(y)-\Ai'(x)\Ai(y) }{x-y} 
	=: \text{A}(x,y).
	\end{align*}
	A is called the \emph{Airy kernel}.

Let us try, again non-rigorously, to characterize this limit function $\Ai$.
	For the Hermite functions we have (see Exercise \ref{exercise:30})
	\[\Psi_N''(x)+\left(N+\frac{1}{2} -\frac{x^2}{4}\right)\Psi_N(x)=0.
	\]
	For the $\widetilde{\Psi}_N$ we have
	\[ \widetilde{\Psi}_N'(x) = N^{-{1}/{12}} \cdot \Psi_N' \left( 2\sqrt{N} + x N^{-{1}/{6}} \right)
	\]
	and
	\[ \widetilde{\Psi}_N''(x) = N^{-{1}/{4}} \cdot \Psi_N'' \left( 2\sqrt{N} + x N^{-{1}/{6}} \right).
	\]
	Thus,
	\begin{align*}
	\widetilde{\Psi}_N''(x) 
	&= - N^{-{1}/{4}} \left[ N + \frac{1}{2} - \frac{\left(2\sqrt{N}+xN^{-{1}/{6}}\right)^2}{4}  \right]
	\Psi_N \left( 2\sqrt{N} + x N^{-{1}/{6}} \right) \\[0.5ex]
	&= - N^{-{1}/{3}} \left[ N + \frac{1}{2} - \frac{  4N + 4 \sqrt{N}xN^{-{1}/{6}} + x^2 N^{-{1}/{3}} }{4}  \right]	\widetilde{\Psi}_N(x) \\[0.5ex]
	&= - N^{-{1}/{3}} \left[ \frac{1}{2} - \frac{ 4 xN^{{1}/{3}} + x^2 N^{-{1}/{3}} }{4}  \right]	\widetilde{\Psi}_N(x) \\[0.5ex]
	&\approx x \widetilde{\Psi}_N(x)\qquad\qquad\qquad\text{for large $N$}.
	\end{align*}
	Hence we expect that $\Ai$ should satisfy the differential equation
	\[ \Ai''(x)-x\Ai(x) = 0.
	\]
This is indeed the case, but the proof is again beyond our tools. Let us just give the formal definition of the Airy function and formulate the final result.

\begin{definition}
	The \emph{Airy function} $\Ai \colon \R \to \R$ is a solution of the \emph{Airy ODE}
	\[ u''(x) = xu(x)
	\]
	determined by the following asymptotics as $x \to \infty:$
	\[ \Ai(x) \sim \frac{1}{2\sqrt{\pi}} x^{-{1}/{4}} \exp\left(-\frac{2}{3} x^{{3}/{2}}\right).
	\]
	The \emph{Airy kernel} is defined by
	\[ \text{A}(x,y) = \frac{ \Ai(x)\Ai'(y)-\Ai'(x)\Ai(y) }{x-y}.
	\]
\end{definition}

\begin{theorem}\label{thm:9.9}
	The random variable $N^{{2}/{3}}(\lambda_{\max}(A_N) -2)$ of a normalized \GUEN\ has a limiting distribution as $N\to\infty$. Its limiting distribution function is
	\begin{align*}
	F_2(t):
	&= \lim\limits_{N\to\infty} \prob{N^\frac{2}{3} \left( \lambda_{\max} -2 \right) \leq t }\\[0.5ex]
	&= \sum_{r=0}^{N} \frac{(-1)^r}{r!} \int_{t}^{\infty} \cdots \int_{t}^{\infty}
	\det \left(  \text{A}(x_i,x_j) \right)_{i,j =1}^r \td x_1 \cdots \td x_r.
	\end{align*}
\end{theorem}

The form of $F_2$ from Theorem \ref{thm:9.9} is more of a theoretical nature and not very convenient for calculations. A main contribution of Tracy--Widom in this context was that they were able to derive another, quite astonishing, representation of $F_2$.

\begin{theorem}[Tracy--Widom, 1994]\label{thm:9.10}
	The distribution function $F_2$ satisfies
	\begin{align*}
	F_2(t)
	&= \exp \left( - \int_{t}^{\infty} (x-t) q(x)^2 \td x  \right),
	\end{align*}
	where $q$ is a solution of the Painlevé II equation
	$ q''(x) -xq(x) + 2q(x)^3 = 0$
	with $q(x) \sim \Ai(x)$ as $x \to \infty$.
\end{theorem}

Here is a plot of the 
Tracy--Widom distribution $F_2$, via solving the Painlev\'e II equation from above, and a comparision with the histogram for the rescaled largest eigenvalue of 5000 \GUE(200); see also Exercise \ref{exercise:31}.

\includegraphics[width=10cm]{tw-200}

\section{Proof of the Harer--Zagier recursion}\label{section:9.6}

We still have to prove the recursion of Harer--Zagier, Theorem \ref{thm:9.2}.

Let us denote
	\[ T(k,N) := \ev{\tr A_N^{2k}} = \sum_{g\geq 0} \varepsilon_g(k) N^{-2g}.
	\]
	The genus expansion shows that $T(k,N)$ is, for fixed $k$, a polynomial in $N^{-1}$.
	Expressing it in terms of integrating over eigenvalues reveals the surprising fact that, up to a Gaussian factor, it is also a polynomial in $k$ for fixed $N$.
We show this in the next lemma.
This is actually the only place where we need
the random matrix interpretation of this quantity.

\begin{lemma}\label{lem:10.1}
	The expression
	\[ N^k \frac{1}{(2k-1)!!} T(k,N)
	\]
	is a polynomial of degree $N-1$ in $k$.
\end{lemma}

\begin{proof}
	First check the easy case $N=1$:
	$T(k,1)=(2k-1)!!$ is the $2k$-th moment of a normal variable and
	\[ \frac{T(k,1)}{(2k-1)!!} = 1
	\]
	is a polynomial of degree $0$ in $k$.
	
	\bigskip For general $N$ we have
	\begin{align*}
	T(k,N) 
	&= \ev{\tr A_N^{2k}} \\
	&= c_N \int_{\R^N} \left( \lambda_1^{2k} + \cdots + \lambda_N^{2k} \right)
		e^{-\frac{N}{2} \left( \lambda_1^{2} + \cdots + \lambda_N^{2} \right) }
		\prod_{i < j} (\lambda_i - \lambda_j)^2
		\td \lambda_1 \cdots \td \lambda_N \\
	&= N c_N \int_{\R^N} \lambda_1^{2k} 
		e^{-\frac{N}{2} \left( \lambda_1^{2} + \cdots + \lambda_N^{2} \right) }
		\prod_{i \neq j} \abs{\lambda_i - \lambda_j}
		\td \lambda_1 \cdots \td \lambda_N \\
	&= N c_N \int_{\R} \lambda_1^{2k} e^{-\frac{N}{2} \lambda_1^{2} } p_N(\lambda_1)
	\td \lambda_1,
	\end{align*}
	where $p_N$ is the result of integrating the Vandermonde over $\lambda_2, \dots, \lambda_N$.
	It is an even polynomial in $\lambda_1$ of degree $2(N-1)$, whose coefficients depend only on $N$ and not on $k$. Hence
	\[ p_N(\lambda_1) = \sum_{l=0}^{N-1} \alpha_l \lambda_1^{2l}
	\]
	with $\alpha_l$ possibly depending on $N$. Thus,
	\begin{align*}
	T(k,N) 
	&= N  c_N  \sum_{l=0}^{N-1} \alpha_l \int_{\R} \lambda_1^{2k + 2l} e^{-\frac{N}{2} \lambda_1^{2} }
	\td \lambda_1\\
	&= N  c_N  \sum_{l=0}^{N-1} \alpha_l \cdot k_N\cdot (2k+2l-1)!! \cdot N^{-(k+l)},
	\end{align*}
since the integral over $\lambda_1$ gives the $(2k+2l)$-th moment of a Gauss variable of variance $N^{-1}$,
where $k_N$ contains the $N$-dependent normalization constants of the Gaussian measure;
	hence
	\[ \frac{N^kT(k,N)}{(2k-1)!!}
	\]
	is a linear combination (with $N$-dependent coefficients) of terms of the form
	\[ \frac{(2k+2l-1)!!}{(2k-1)!!}.
	\]
	These terms are polynomials in $k$ of degree $l$.
\end{proof}

	We now have that
	\begin{align*}
	N^k \frac{1}{(2k-1)!!} T(k,N)
	&= N^k \frac{1}{(2k-1)!!} \sum_{\pi\in \pair(2k)} N^{\# (\gamma\pi) -k -1 } \\
	&= \frac{1}{N} \frac{1}{(2k-1)!!} \sum_{\pi\in \pair(2k)} N^{\# (\gamma\pi)} \\
	&= \frac{1}{N} \frac{1}{(2k-1)!!} t(k,N),
	\end{align*}
where the last equality defines $t(k,N)$.

	By Lemma \ref{lem:10.1}, ${t(k,N)}/{(2k-1)!!} $ is a polynomial of degree $N-1$ in $k$.
	We interpret it as follows:
	\begin{align*}
	t(k,N) 
	&= \sum_{\pi\in \pair(2k)} \# \left\{\text{coloring cycles of $\gamma\pi$ with at most $N$ different colors}\right\}.
	\end{align*}
	Let us introduce
	\begin{align*}
	\tilde{t}(k,L) 
	&= \sum_{\pi\in \pair(2k)} \# \left\{\text{coloring cycles of $\gamma\pi$ with exactly $L$ different colors}\right\},
	\end{align*}	
	then we have
	\begin{align*}
	t(k,N) 
	&= \sum_{L=1}^{N} \tilde{t}(k,L)  \binom{N}{L},
	\end{align*}	
because if we want to use at most $N$ different colors, then we can do this by using exactly $L$ different colors (for any $L$ between 
$1$ and $N$), and after fixing $L$ we have $\binom NL$ many possibilities to choose the $L$ colors among the $N$ colors.

	This relation can be inverted by
	\[ \tilde{t}(k,N) = \sum_{L=1}^{N} (-1)^{N-L}  \binom{N}{L} t(k,L)
	\]
	and hence ${\tilde{t}(k,N)}/{(2k-1)!!}$ is also a polynomial in $k$ of degree $N-1$. 
	But now we have
	\begin{align*}
	0
	&= \tilde{t}(0,N)
	= \tilde{t}(1,N)
	= \cdots
	= \tilde{t}(N-2,N),
	\end{align*}
	since $\gamma\pi$ has, by Proposition \ref{prop:2.20}, at most $k+1$ cycles for $\pi\in\pair(2k)$; and thus
	$\tilde{t}(k+1,N)=0$ if $k+1<N$, as we need at least $N$ cycles if we want to use $N$ different colors.

 So, ${\tilde{t}(k,N)}/{(2k-1)!!}$ is a polynomial in $k$ of degree $N-1$ and we know $N-1$ zeros; hence it must be of the form
	\begin{align*}
	\frac{\tilde{t}(k,N)}{(2k-1)!!}
	&= \alpha_N k (k-1) \cdots (k-N+2)
	= \alpha_N \binom{k}{N-1}(N-1)!.
	\end{align*}
	Hence,
	\begin{align*}
	t(k,N) 
	&= \sum_{L=1}^{N}  \binom{N}{L} \binom{k}{L-1} (L-1)! \alpha_L (2k-1)!!.
	\end{align*}
	To identify $\alpha_N$ we look at
	\begin{align*}
	\alpha_{N+1} \binom{N}{N}N!(2N-1)!!
	&=\tilde{t}(N,N+1)
	=C_N (N+1)!.
	\end{align*}
	Note that only the NC pairings can be colored with exactly $N+1$ colors, and for each such $\pi$ there are $(N+1)!$ ways of doing so. We conclude
	\begin{align*}
	\alpha_{N+1}
	&= \frac{C_N(N+1)!}{N!(2N-1)!!} \\[0.5ex]
	&= \frac{C_N(N+1)}{(2N-1)!!} \\[0.5ex]
	&=  \frac{1}{N+1} \binom{2N}{N} \frac{N+1}{(2N-1)!!} \\[0.5ex]
	&=  \frac{(2N)!}{N!N!(2N-1)!!} \\[0.5ex]
	&= \frac{2^N}{N!}.
	\end{align*}
	Thus we have
	\begin{align*}
	T(k,N)
	&= \frac{1}{N^{k+1}} t(k,N) \\[0.5ex]
	&= \frac{1}{N^{k+1}} \sum_{L=1}^{N} \binom{N}{L} \binom{k}{L-1} (L-1)! \frac{2^{L-1}}{(L-1)!}(2k-1)!! \\[0.5ex]
	&= (2k-1)!! \frac{1}{N^{k+1}} \sum_{L=1}^{N} \binom{N}{L} \binom{k}{L-1} 2^{L-1}.
	\end{align*}
	To get information from this on how this changes in $k$ we consider a generating function in $k$,
	\begin{align*}
	\mathcal{T}(s,N)
	&= 1 + 2\sum_{k=0}^{\infty} \frac{T(k,N)}{(2k-1)!!}(Ns)^{k+1} \\[0.5ex]
	&= 1 + 2\sum_{k=0}^{\infty}  \sum_{L=1}^{N} \binom{N}{L} \binom{k}{L-1} 2^{L-1} s^{k+1} \\[0.5ex]
	&= \sum_{L=0}^{N} \binom{N}{L} 2^{L}  \sum_{k=L-1}^{\infty} \binom{k}{L-1} s^{k+1} \\[0.5ex]
	&= \sum_{L=0}^{N} \binom{N}{L} 2^{L}  \left( \frac{s}{1-s} \right)^L \\[0.5ex]
	&= \sum_{L=0}^{N} \binom{N}{L} \left( \frac{2s}{1-s} \right)^L \\[0.5ex]
	&= \left(1 + \frac{2s}{1-s} \right)^N \\[0.5ex]
	&= \left( \frac{1+s}{1-s} \right)^N.
	\end{align*}
	Note that (as in our calculation for the $\alpha_N$)
	\begin{align*}
	\frac{1}{(2k-1)!!} 
	&= \frac{2^k}{k!C_k(k+1)}
	\end{align*}
	and hence $T(s,N)$ can also be rewritten as a generating function in our main quantity of interest,
	\begin{align*}
	b_k^{(N)}
	&= \frac{T(k,N)}{C_k}\qquad\text{(we make now the dependence of $b_k$ on $N$ explicit)}
	\end{align*}
	as
	\begin{align*}
	\mathcal{T}(s,N)
	&= 1 + 2\sum_{k=0}^{\infty} \frac{T(k,N)}{(k+1)! C_k} 2^k (Ns)^{k+1} \\[0.5ex]
	&= 1 + \sum_{k=0}^{\infty} \frac{b_k^{(N)}}{(k+1)! }  (2Ns)^{k+1}.
	\end{align*}
	In order to get a recursion for the $b_k^{(N)}$, we need some functional relation for $\mathcal{T}(s,N)$.  Note that the recursion in Harer--Zagier involves $b_k$, $b_{k+1}$, $b_{k-1}$ for the same $N$, thus we need a relation which does not change the $N$. For this we look on the derivative with respect to $s$. From
$$ \mathcal{T}(s,N)=\left( \frac{1+s}{1-s} \right)^N$$
we get
	\begin{align*}
	\frac{\td}{\td s} \mathcal{T}(s,N)
	&= N \left( \frac{1+s}{1-s} \right)^{N-1} \frac{(1-s)+(1+s)}{(1-s)^2} \\[0.5ex]
	&= 2N \left( \frac{1+s}{1-s} \right)^{N} \frac{1}{(1-s)(1+s)}\\[0.5ex]
&=2N \cdot\mathcal{T}(s,N)\frac 1{1-s^2},
	\end{align*}
	and thus
	\[ (1-s^2) \frac{\td}{\td s} \mathcal{T}(s,N) = 2N \cdot\mathcal{T}(s,N).
	\]
	Note that we have
	\begin{align*}
	\frac{\td}{\td s} \mathcal{T}(s,N)
	&= \sum_{k=0}^{\infty} \frac{b_k^{(N)}}{k! }  (2Ns)^{k} 2N.
	\end{align*}
	Thus, by comparing coefficients of $s^{k+1}$ in our differential equation from above, we conclude
	\begin{align*}
	\frac{b_{k+1}^{(N)}}{(k+1)!} (2N)^{k+2} - \frac{b_{k-1}^{(N)}}{(k-1)!} (2N)^{k}
	&= 2N \frac{b_{k+1}^{(N)}}{(k+1)!}  (2N)^{k+1},
	\end{align*}
	and thus, finally,
	\[b_{k+1}^{(N)} = b_{k}^{(N)} + b_{k-1}^{(N)} \frac{(k+1)k}{(2N)^2}.
	\]

\chapter{Statistics of the Longest Increasing Subsequence}

\section{Complete order is impossible}
\begin{definition}
	A permutation $\sigma \in S_n$ is said to have an \emph{increasing subsequence of length $k$} if there exist indices $1 \leq i_1 < \cdots < i_k \leq n$ such that
	$\sigma(i_1) < \cdots < \sigma(i_k)$.
	For a \emph{decreasing subsequence of length k} the above holds with the second set of inequalities reversed. For a given $\sigma\in S_n$ we denote the length of an increasing subsequence of maximal length by $L_n(\sigma)$.
\end{definition}

\begin{example}
	\begin{enumerate}
		\item 
Maximal length is achieved for  the identity permutation
$$\sigma=\id=
\begin{pmatrix}
		1 & 2 & \cdots & n-1 & n \\
		1&  2 & \cdots&  n-1&  n
		\end{pmatrix};$$
this has an increasing subsequence of length $n$, hence $L_n(\id) = n$. In this case,
		all decreasing subsequences have length $1$.
		\item 
Minimal length is achieved for the permutation
$$\sigma = \begin{pmatrix}
		1 & 2 & \cdots & n-1 & n \\
		n&  n-1 & \cdots&  2&  1
		\end{pmatrix};$$
in this case all increasing subsequences have length 1, hence $L_n(\sigma)=1$; but there is a decreasing subsequence of length $n$.
		\item Consider a more ``typical'' permutation
$$\sigma = \begin{pmatrix}
		1 & 2 & 3 & 4 & 5 & 6 & 7 \\
		4 & 2 & 3 & 1 & 6 & 5 & 7
		\end{pmatrix};$$
		this has $(2,3,5,7)$ and $(2,3,6,7)$ as longest increasing subsequences, thus $L_7(\sigma) = 4$.
		Its longest decreasing subsequences are $(4,2,1)$ and $(4,3,1)$ with length $3$.
		In the graphical representation		
		\[ \begin{tikzpicture}
		\tikzset{dot/.style={circle,fill=#1,inner sep=0,minimum size=6pt},scale=0.8}
		
		\tkzInit[xmax=7,ymax=7,xmin=0,ymin=0]
		\tkzGrid
		\tkzAxeXY
		
		\node at (0, 0)   (0) {};
		\node[dot=black] at (1, 4)   (1) {};
		\node[dot=black] at (2, 2)   (2) {};
		\node[dot=black] at (3, 3)   (3) {};
		\node[dot=black] at (4, 1)   (4) {};
		\node[dot=black] at (5, 6)   (5) {};
		\node[dot=black] at (6, 5)   (6) {};
		\node[dot=black] at (7, 7)   (7) {};
		
		\draw[thick] (0) -- (2); 
		\draw[thick] (2) -- (3); 
		\draw[thick] (3) -- (6); 
		\draw[thick] (6) -- (7); 
		\end{tikzpicture}
		\]		
		an increasing subsequence corresponds to a path that always goes up.
	\end{enumerate}
\end{example}

\begin{remark}
\begin{enumerate}
\item
		Longest increasing subsequences are relevant for sorting algorithms.
		Consider a library of $n$ books, labeled bijectively with numbers $1,\dots, n$,
		arranged somehow on a single long bookshelf.
		The configuration of the books corresponds to a permutation $\sigma \in S_n$.
		How many operations does one need to sort the books in a canonical ascending order $1,2, \dots, n$?
		It turns out that the minimum number is $n-L_n(\sigma)$. One can sort around an increasing subsequence.
		\begin{example*} Around the longest increasing subsequence $(1,2,6,8)$ we sort
			\[\begin{matrix}
			&4&1&9&3&2&7&6&8&5\\
			\to&4&1&9&2&3&7&6&8&5\\
			\to&1&9&2&3&4&7&6&8&5\\
			\to&1&9&2&3&4&5&7&6&8\\
			\to&1&9&2&3&4&5&6&7&8\\
			\to&1&2&3&4&5&6&7&8&9
			\end{matrix}
			\]
			in $9-4=5$ operations.
		\end{example*}

\item
	One has situations with only small increasing subsequences,
	but then one has long decreasing subsequences. This is true in general; one cannot avoid both long decreasing and long increasing subsequences at the same time. According to the slogan
	\begin{center}
		\textit{\enquote{Complete order is impossible.}} (Motzkin)
	\end{center}
\end{enumerate}
\end{remark}

\begin{theorem}[Erdős, Szekeres, 1935]
	Every permutation $\sigma\in S_{n^2+1}$ has a monotone subsequence of length more than $n$.
\end{theorem}

\begin{proof}
	Write $\sigma = a_1 a_2 \cdots a_{n^2+1}$. Assign labels $(x_k,y_k)$, where $x_k$ is the length of a longest increasing subsequence ending at $a_k$; and $y_k$ is the length of a longest decreasing subsequence ending at $a_k$.
	Assume now that there is no monotone subsequence of length $n+1$. Hence we have for all $k$: $ 1 \leq x_k, y_k \leq n$;
	i.e., there are only $n^2$ possible labels. By the pigeonhole principle there are $i < j$ with $(x_i,y_i) = (x_j, y_j)$.
	If $a_i < a_j$ we can append $a_j$ to a longest increasing subsequence ending at $a_i$, but then $x_j > x_i$.
	If $a_i > a_j$ we can append $a_j$ to a longest decreasing subsequence ending at $a_i$, but then $y_j > y_i$.
	In both cases we have a contradiction.
\end{proof}

\section{Tracy--Widom for the asymptotic distribution of $L_n$}
	We are now interested in the distribution of $L_n(\sigma)$ for $n\to\infty$. This means, we put the uniform distribution on permutations, i.e.,
$ \prob{\sigma} = {1}/{n!}$
	for all $\sigma\in S_n$, and consider $L_n\colon S_n\to \R$ as a random variable.
	What is the asymptotic distribution of $L_n$? This question is called Ulan's problem and was raised in the 1960's.
	In 1972, Hammersley showed that the limit
	\[ \Lambda = \lim\limits_{n\to\infty} \frac{\ev{L_n}}{\sqrt{n}}
	\]
	 exists and that ${L_n}/\sqrt{n}$ converges to  $\Lambda$
	in probability. In 1977, Vershik--Kerov and Logan--Shepp showed independently that $\Lambda = 2$. Then in 1998, Baik, Deift and Johansson proved the asymptotic behaviour of the fluctuations of $L_n$; quite surprisingly, this is also captured by the Tracy--Widom distribution: 
	\[ \lim\limits_{n\to\infty} \prob{ \frac{L_n - 2\sqrt{n}}{n^{{1}/{6}}} \leq t } = F_2(t).
	\]

\section{Very rough sketch of the proof of the Baik, Deift, Johansson theorem}
Again, we have no chance of giving a rigorous proof of the BDJ theorem. Let us give at least a possible route for a proof, which gives also an idea why the statistics of the length of the longest subsequence could be related to the statistics of the largest eigenvalue.
	\begin{enumerate}
		\item The RSK correspondence relates permutations to Young diagrams. $L_n$ goes under this mapping to the length of the first row of the diagram.
		\item These Young diagrams correspond to non-intersecting paths.
		\item Via Gessel--Viennot the relevant quantities in terms of NC paths have a determinantal form.
		\item Then one has to show that the involved kernel, suitably rescaled, converges to the Airy kernel.
	\end{enumerate}

In the following we want to give some idea of the first two items in the above list; the main (and very hard part of the proof) is to show the convergence to the Airy kernel.

\subsection{RSK correspondence}
	RSK stands for Robinson--Schensted--Knuth after papers from 1938, 1961 and 1973. It gives a bijection
	\[ S_n \longleftrightarrow\bigcup_{\substack{\lambda \\ \text{Young diagram} \\ \text{of size $n$}}} \left( \text{Tab} \, \lambda \times \text{Tab} \, \lambda \right),
	\]
	where $\text{Tab} \, \lambda$ is the set of Young tableaux of shape $\lambda$.

\begin{definition}
	\begin{enumerate}
		\item Let $n\geq 1$. A \emph{partition} of $n$ is a sequence of natural numbers
		$\lambda = (\lambda_1,\dots,\lambda_r)$ such that
		\[ \lambda_1 \geq \lambda_2 \geq \cdots \geq \lambda_r
		\qquad \text{and} \qquad
		\sum_{i=1}^{r} \lambda_i =n.
		\]
		We denote this by $\lambda \vdash n$. Graphically, a partition $\lambda\vdash n$ is represented by a \emph{Young diagram} with $n$ boxes.
		\begin{center}
			\begin{tabular}{l}
				\yng(5,3) \\
				$\vdots$ \\
				\yng(1)
			\end{tabular}
		\end{center}
		\item A \emph{Young tableau} of shape $\lambda$ is the Young diagram $\lambda$ filled with numbers $1,\dots,n$ such that in any row the numbers are increasing from left to right and in any column the numbers are increasing from top to bottom.
		We denote the set of all Young tableaux of shape $\lambda$ by $\text{Tab} \, \lambda$.
	\end{enumerate}
\end{definition}

\begin{example}
	\begin{enumerate}
		\item 
		For $n=1$ there is only one Young diagram, $\yng(1)$, and one corresponding Young tableau: $ \young(1)$.

For $n=2$, there are two Young diagrams, 
$$ \yng(2)\qquad\text{and}\qquad \yng(1,1)\, ,$$
each of them having one corresponding Young tableau
		$$ \young(12)  \qquad\text{and}\qquad
		  \young(1,2)\, . $$
For $n=3$, there are three Young diagrams
		$$\yng(3) \qquad\yng(2,1) \qquad\yng(1,1,1)\, ;$$
the first and the third have only one tableau, but the middle one has two:
$$ \young(12,3) \qquad\young(13,2)  $$

		\item Note that a tableau of shape $\lambda$ corresponds to a walk from $\emptyset$ to $\lambda$ by adding one box in each step and only visiting Young diagrams. For example, the Young tableau
		
		$$\young(1248,37,5,6)$$ corresponds to the walk
$$\young(1)
\to \young(12)
\to \young(12,3)
\to \young(124,3)
\to \young(124,3,5)
\to \young(124,3,5,6)
\to \young(124,37,5,6)
\to \young(1248,37,5,6)$$

\end{enumerate}
\end{example}

\begin{remark}
		
		Those objects are extremely important since they parametrize the irreducible representations of $S_n$:
		\[ \lambda \vdash n \longleftrightarrow \text{irreducible representation $\pi_\lambda$ of $S_n$}.
		\]
Furthermore, the dimension of such a representation $\pi_\lambda$ is given by the number of tableaux of shape $\lambda$.
If one recalls that for any finite group one has the general statement that the sum of the squares of the dimensions over all irreducible representations of 
the group gives the number of elements in the group, then one has for the symmetric group the statement that
$$\sum_{\lambda\vdash n} (\# Tab\lambda)^2=\# S_n=n!.$$
This shows that there is a bijection between elements in $S_n$ and pairs of tableaux of the same shape $\lambda\vdash n$.
The RSK correspondence is such a concrete bijection, given by an explicit algorithm. It has the property, that $L_n$ goes under this bijection over to the length of 
the first row of the corresponding Young diagram $\lambda$.
\end{remark}

\begin{example}\label{ex:10.8}

For example,
under the RSK correspondence, the permutation
		\[ \sigma = \begin{pmatrix}
		1&2&3&4&5&6&7\\
		4&2&3&6&5&1&7
		\end{pmatrix}
		\]
		corresponds to the pair of Young tableaux
		\[ \young(1357,26,4), \quad \young(1347,25,6) .
		\]
		Note that $L_7(\sigma) = 4$ is the length of the first row.
	
\end{example}

\subsection{Relation to non-intersecting paths}
	Pairs $(Q,P) \in \Tab \lambda \times \Tab \lambda$ can be identified with
	$r = \# \text{rows}(\lambda)$ paths. $Q$ gives the positions of where to go up and $P$ of where to go down; the conditions on the Young tableau 
guarantee that the paths will be non-intersecting. For example,
the pair corresponding to the $\sigma$ from Example \ref{ex:10.8} above gives the following non-intersecting paths:
	\[ \begin{tikzpicture}
	\tikzset{dot/.style={circle,fill=#1,inner sep=0,minimum size=4pt},scale=.58}
	
	\tkzInit[xmax=15, ymax=7]
	\begin{scope}[dashed]
		\tkzGrid
	\end{scope}
	
	\node[dot=black] at (1, 0)   (1)  {};
	\node[dot=black] at (2, 0)   (2)  {};
	\node[dot=black] at (3, 0)   (3)  {};
	\node[dot=black] at (4, 0)   (4)  {};
	\node[dot=black] at (5, 0)   (5)  {};
	\node[dot=black] at (6, 0)   (6)  {};
	\node[dot=black] at (7, 0)   (7)  {};
	\node[dot=black] at (8, 0)   (8)  {};
	\node[dot=black] at (9, 0)   (9)  {};
	\node[dot=black] at (10, 0)  (10) {};
	\node[dot=black] at (11, 0)  (11) {};
	\node[dot=black] at (12, 0)  (12) {};
	\node[dot=black] at (13, 0)  (13) {};
	\node[dot=black] at (14, 0)  (14) {};
	
	\draw (1,-0.5) -- node {$1$} (1,-0.5);
	\draw (2,-0.5) -- node {$2$} (2,-0.5);
	\draw (3,-0.5) -- node {$3$} (3,-0.5);
	\draw (4,-0.5) -- node {$4$} (4,-0.5);
	\draw (5,-0.5) -- node {$5$} (5,-0.5);
	\draw (6,-0.5) -- node {$6$} (6,-0.5);
	\draw (7,-0.5) -- node {$7$} (7,-0.5);
	\draw (8,-0.5) -- node {$7$} (8,-0.5);
	\draw (9,-0.5) -- node {$6$} (9,-0.5);
	\draw (10,-0.5) -- node {$5$} (10,-0.5);
	\draw (11,-0.5) -- node {$4$} (11,-0.5);
	\draw (12,-0.5) -- node {$3$} (12,-0.5);
	\draw (13,-0.5) -- node {$2$} (13,-0.5);
	\draw (14,-0.5) -- node {$1$} (14,-0.5);
	
	\draw[very thick] (0,0) -- (15,0);
	\draw[very thick] (0,1) -| (4,2) -| (9,1) -- (15,1);
	\draw[very thick] (0,2) -| (2,3) -| (6,4) -| (10,3) -| (13,2) -- (15,2);
	\draw[very thick] (0,3) -| (1,4) -| (3,5) -| (5,6) -| (7,7) -| (8,6) -| (11,5) -| (12,4) -| (14,3) -- (15,3);
	\end{tikzpicture}
	\]

\chapter{The Circular Law}

\section{Circular law for Ginibre ensemble}
The non-selfadjoint analogue of \GUE\ is given by the Ginibre ensemble, where all entries are independent and complex Gaussians.
A standard complex Gaussian is of the form
\[ z = \frac{x+iy}{\sqrt{2}},
\]
where $x$ and $y$ are independent standard real Gaussians, i.e., with joint distribution
\[ p(x,y) \td x \td y = \frac{1}{2\pi} e^{-\frac{x^2}{2}}e^{-\frac{y^2}{2}}\td x \td y.
\]
If we rewrite this in terms of a density with respect to the Lebesgue measure for real and imaginary part
\[ z = \frac{x+iy}{\sqrt{2}} = t_1 + i t_2,
\qquad
\overline{z} = \frac{x-iy}{\sqrt{2}} = t_1 - i t_2,
\]
we get
\begin{align*}
p(t_1,t_2) \td t_1 \td t_2
&= \frac{1}{\pi} e^{-(t_1^2+t_2^2)} \td t_1 \td t_2
= \frac{1}{\pi} e^{-\abs{z}^2} \td^2z,
\end{align*}
where $\td^2z = \td t_1 \td t_2$.

\begin{definition}
	A \emph{(complex) unnormalized Ginibre ensemble} $A_N = \left( a_{ij} \right)_{i,j=1}^N$
	is given by complex-valued entries with joint distribution
	\[ \frac{1}{\pi^{N^2}} \exp \left(- \sum_{i,j=1}^{N} \abs{a_{ij}}^2 \right) \td A
	= \frac{1}{\pi^{N^2}} \exp \left(-\Tr(AA^*)\right) \td A,
\qquad	\text{where }
 \td A = \prod_{i,j=1}^N \td^2a_{ij}.
 	\]
\end{definition}

As for the \GUE\ case, Theorem \ref{thm:7.6}, we can rewrite the density in terms of eigenvalues. Note that the eigenvalues are now complex.

\begin{theorem}\label{thm:11.2}
	The joint distribution of the complex eigenvalues of an $N\times N$ Ginibre ensemble is given by a density
	\begin{align*}
	p(z_1,\dots,z_N)
	&= c_N \exp \left( -\sum_{k=1}^{N} \abs{z_k}^2 \right) \prod_{1\leq i<j\leq N} \abs{z_i-z_j}^2.
	\end{align*}
\end{theorem}

\begin{remark}
	\begin{enumerate}
		\item Note that typically Ginibre matrices are not normal, i.e., 
		$AA^* \neq A^*A$.
		This means that one loses the relation between functions in eigenvalues and traces of functions in the matrix.
		The latter is what we can control, the former is what we want to understand.
		\item As in the selfadjoint case the eigenvalues repel, hence there will almost surely be no multiple eigenvalues. Thus we can also in the Ginibre case diagonalize our matrix, i.e.,
		$A = VDV^{-1}$, where $D = \diag(z_1, \dots, z_N)$ contains the eigenvalues.
		However, $V$ is now not unitary anymore, i.e., eigenvectors for different eigenvalues are in general not orthogonal. We can also diagonalize $A^*$ via 
		$ A^*= (V^{-1})^*D^*V^*$,
		but since $V^{-1} \neq V^*$ (if $A$ is not normal) we cannot diagonalize $A$ and $A^*$ simultaneously.
		This means that in general, for example $\Tr(AA^*A^*A)$ has no clear relation to
		$\sum_{i=1}^{N} z_i \bar{z_i}\bar{z_i} z_i$.
		(Note that $\Tr(AA^*A^*A) \neq \Tr(AA^*AA^*)$ if $AA^*\neq A^*A$, but of course
	$\sum_{i=1}^{N} z_i \bar{z_i}\bar{z_i} z_i
		= \sum_{i=1}^{N} z_i \bar{z_i} z_1 \bar{z_i}$.)
		\item In Theorem \ref{thm:11.2} it seems that we have rewritten the density $\exp(- \Tr(AA^*))$ as
		$\exp ( -\sum_{k=1}^{N} \abs{z_k}^2)$.
However, this is more subtle.
		On can bring any matrix via a unitary conjugation in a triangular form: $A=UTU^*$, where $U$ is unitary and
		\[ T= \begin{pmatrix}
		z_1 & \star& \cdots & \star\\
		0 & \ddots & \ddots& \vdots\\
		\vdots & \ddots & \ddots & \star \\
		0& \cdots & 0 & z_n
		\end{pmatrix}
	\]
contains on the diagonal the eigenvalues $z_1,\dots,z_n$ of $A$ (this is usually called Schur decomposition).
		Then $A^*=UT^*U^*$ with
		\[ T^*= \begin{pmatrix}
		\bar{z_1} & 0& \cdots & 0\\
		\star & \ddots & \ddots& \vdots\\
		\vdots & \ddots & \ddots & 0 \\
		\star& \cdots & \star & \bar{z_n}
		\end{pmatrix}
		\]
		and
		\begin{align*}
		\Tr(AA^*)
		=\Tr(TT^*)
		= \sum_{k=1}^{N} \abs{z_k}^2+\sum_{j>i}t_{ij}\bar t_{ij}.
		\end{align*}
Integrating out the $t_{ij}$ ($j>i$) gives then the density for the $z_i$.

		\item As for the \GUE\ case (Theorem \ref{thm:7.15}) we can write the Vandermonde density in a determinantal form. The only difference is that we have to replace the Hermite polynomials $H_k(x)$, which orthogonalize the real Gaussian distribution, by monomials $z^k$, which orthogonalize the complex Gaussian distribution.
	\end{enumerate}
\end{remark}

\begin{theorem}
	The joint eigenvalue distribution of the Ginibre ensemble is of the determinantal form
	$p(z_1,\dots,z_n)
	= \frac{1}{N!} \det  (K_N(z_i,z_j))_{i,j=1}^N$
	with the kernel
	\[
	K_N(z,w)
	= \sum_{k=0}^{N-1}\ff_k(z) \bar\ff_k(w),
\qquad\text{where}\qquad
 \ff_k(z) = \frac{1}{\sqrt{\pi}} e^{-\frac{1}{2} \abs{z}^2} \frac{1}{\sqrt{k!}} z^k.
	\]
	In particular, for the averaged eigenvalue density of an unnormalized Ginibre eigenvalue matrix	
	we have the density
	\begin{align*}
	p_N(z)
	&= \frac{1}{N} K_N(z,z)
	= \frac{1}{N\pi}e^{-\abs{z}^2} \sum_{k=0}^{N-1} \frac{\abs{z}^{2k}}{k!}.
	\end{align*}
\end{theorem}

\begin{theorem}[Circular law for the Ginibre ensemble]
	The averaged eigenvalue distribution for a normalized Ginibre random matrix $\frac{1}{\sqrt{N}} A_N$ converges for $N\to\infty$ weakly to the uniform distribution on the unit disc of $\C$ with density
$\frac{1}{\pi} 1_{ \{ z \in \C \mid \abs{z} \leq 1 \} }$.
\end{theorem}

\begin{proof}
	The density $q_N$ of the normalized Ginibre is given by
	\begin{align*}
	q_N(z)
	= N\cdot  p_N( \sqrt{N} z) 
	= \frac{1}{\pi} e^{-N\abs{z}^2} \sum_{k=0}^{N-1} \frac{(N\abs{z}^2)^k}{k!}.
	\end{align*}
We have to show that this converges to the circular density.
	For $\abs{z}<1$ we have
	\begin{align*}
	e^{N\abs{z}^2} - \sum_{k=0}^{N-1} \frac{(N\abs{z}^2)^k}{k!}
	&= \sum_{k=N}^{\infty} \frac{(N\abs{z}^2)^k}{k!} \\[0.5ex]
	&\leq \frac{(N\abs{z}^2)^N}{N!} \sum_{l=0}^{\infty}  \frac{(N\abs{z}^2)^l}{(N+1)^l} \\[0.5ex]
	&\leq \frac{(N\abs{z}^2)^N}{N!} \frac{1}{1- \frac{N\abs{z}^2}{N+1}},
	\end{align*}	
	Furthermore, using the lower bound
	$ N! \geq \sqrt{2\pi} N^{N+\frac{1}{2}} e^{-N} $
	for $N!$, we calculate
	\begin{align*}
	e^{-N\abs{z}^2} \frac{(N\abs{z}^2)^N}{N!}
	&\leq e^{-N\abs{z}^2} N^N\abs{z}^{2N} \frac{1}{\sqrt{2\pi}} \frac{1}{N^{N+\frac{1}{2}}}
		e^{N} \\[0.5ex]
	&=  \frac{1}{\sqrt{2\pi}}\frac{1}{\sqrt{N}}  e^{-N\abs{z}^2} e^{N \ln \abs{z}^2}e^{N}\\[0.5ex]
	&= \frac{1}{\sqrt{2\pi}}\frac{\exp[N ( -\abs{z}^2 + \ln \abs{z}^2 +1 )]}{\sqrt{N}} 
	\Nto 0.
	\end{align*}
	Here, we used that
	 $-\abs{z}^2 + \ln \abs{z}^2 +1 < 0$
	for $\abs{z} < 1$. Hence we conclude
	\begin{align*}
	1-e^{-N\abs{z}^2} \sum_{k=0}^{N-1} \frac{(N\abs{z}^2)^k}{k!}
	&\leq e^{-N\abs{z}^2} \frac{(N\abs{z}^2)^N}{N!} \frac{1}{1- \frac{N\abs{z}^2}{N+1}}
	\Nto 0.
	\end{align*}
	Similarly, for $\abs{z}> 1$,
	\begin{align*}
	\sum_{k=0}^{N-1} \frac{(N\abs{z}^2)^k}{k!}
	\leq \frac{(N\abs{z}^2)^{N-1}}{(N-1)!} \sum_{l=0}^{N-1} \frac{(N-1)^l}{(N\abs{z}^2)^l} 
	\leq \frac{\left(N\abs{z}\right)^{N-1}}{(N-1)!} \frac{1}{1- \frac{N-1}{N\abs{z}^2}},
	\end{align*}
which shows that
$$E^{-N\vert z\vert^2} 	\sum_{k=0}^{N-1} \frac{(N\abs{z}^2)^k}{k!}\Nto 0.$$
\end{proof}

\begin{remark}
	\begin{enumerate}
		\item The convergence also holds almost surely. Here is a plot of the 3000 eigenvalues of one realization of a $3000\times 3000$ Ginibre matrix.
$$\includegraphics[width=8cm]{circular-3000}$$
		\item The circular law also holds for non-Gaussian entries, but proving this is much harder than the extension for the semicircle law from the Gaussian case to Wigner matrices.
	\end{enumerate}
\end{remark}

\section{General circular law}

\begin{theorem}[General circular law]
	Consider a complex random matrix
	$ A_N = \frac{1}{\sqrt{N}} \left( a_{ij} \right)_{i,j=1}^N$,
	where the $a_{ij}$ are independent and identically distributed complex random variables
	with variance $1$, i.e., $\mathbb{E}[\abs{a_{ij}}^2] - \ev{a_{ij}}^2 =1$.
	Then the eigenvalue distribution of $A_N$ converges weakly almost surely for $N\to \infty$ to the uniform distribution on the unit disc.
\end{theorem}

Note that only the existence of the second moment is required, higher moments don't need to be finite.

\begin{remark}
	\begin{enumerate}
\item
It took quite a while to prove this in full generality. Here is a bit about the history of the proof.
	\begin{itemize}
		\item 60's, Mehta proved it (see his book) in expectation for Ginibre ensemble;
		\item 80's, Silverstein proved almost sure convergence for Ginibre;
		\item 80's, 90's, Girko outlined the main ideas for a proof in the general case;
		\item 1997, Bai gave the first rigorous proof, under additional assumptions on the distribution
		\item papers by Tao--Vu, Götze--Tikhomirov, Pan--Zhou and others, weakening more and more the assumptions;
		\item 2010, Tao--Vu gave final version under the assumption of the existence of the second moment.
	\end{itemize}
		\item For measures on $\C$ one can use $*$-moments or the Stieltjes transform to describe them, but controlling the convergence properties is the main problem.
		\item For a matrix $A$ its $*$-moments are all expressions of the form
		$\tr ( A^{\varepsilon(1)} \cdots A^{\varepsilon(m)} )$, where $m\in\N$ and 
		$\varepsilon(1), \dots, \varepsilon(m) \in \{1,*\}$.
		The eigenvalue distribution
		\[ \mu_A = \frac{1}{N} \left( \delta_{z_1} + \cdots + \delta_{z_N} \right)
	\qquad \text{($z_1,\dots,z_n$ are complex eigenvalues of $A$)}\]
		of $A$ is uniquely determined by the knowledge of all $*$-moments of $A$, but convergence of $*$-moments does not necessarily imply convergence of the eigenvalue distribution.
		\begin{example*}
			Consider
			\[ A_N = \begin{pmatrix}
			0 & 1 & 0 & \cdots & 0 \\
			\vdots & \ddots &  \ddots & \ddots & \vdots \\
			\vdots & & \ddots & \ddots &  0 \\
			\vdots & & & \ddots & 1 \\
			0 & \cdots & \cdots & \cdots & 0
			\end{pmatrix} 
			\qquad \text{and} \qquad
			B_N = \begin{pmatrix}
			0 & 1 & 0 & \cdots & 0 \\
			\vdots & \ddots &  \ddots & \ddots & \vdots \\
			\vdots & & \ddots & \ddots &  0 \\
			0 & & & \ddots & 1 \\
			1 & 0 & \cdots & \cdots & 0
			\end{pmatrix}.
			\]
			Then $\mu_{A_N} = \delta_0$, but $\mu_{B_N}$ is the uniform distribution on the $N$-th roots of unity. Hence $\mu_{A_N} \to \delta_0$, whereas $\mu_{B_N}$ converges to the uniform distribution on the unit circle.
			However, the limits of the $*$-moments are the same for $A_N$ and $B_N$.
		\end{example*}
		\item For each measure $\mu$ on $\C$ one has the Stieltjes transform
		\begin{align*}
		S_\mu(w)
		&= \int_\C \frac{1}{z-w} \td \mu(z).
		\end{align*}
		This is almost surely defined. However, it is analytic in $w$ only outside the support of $\mu$. In order to recover $\mu$ from $S_\mu$ one also needs the information about $S_\mu$ inside the support.
		In order to determine and deal with $\mu_A$ one reduces it via Girko's \enquote{hermitization method}
		\begin{align*}
		\int_{\C} \log \abs{\lambda-z} \td\mu_A(z)
		&= \int_0^t \log t \td \mu_{ \abs{A-\lambda 1} } (t)
		\end{align*}
		to selfadjoint matrices. The left hand side for all $\lambda$ determines $\mu_A$ and the right hand side is about selfadjoint matrices
		\[ \abs{A-\lambda 1} = \sqrt{(A-\lambda 1) (A-\lambda 1)^*}.
		\]
		Note that the eigenvalues of $\abs{B}$ are related to those of
		\[ \begin{pmatrix}
		0 & B \\ B^* & 0
		\end{pmatrix}.
		\]
		In this analytic approach one still needs to control convergence properties.
		For this, estimates of probabilities of small singular values are crucial.
		
For more details on this one should have a look at the survey of Bordenave-- Chafai, \textit{Around the circular law}.
	\end{enumerate}
\end{remark}

\chapter{Several Independent GUEs and Asymptotic Freeness}

Up to now, we have only considered limits $N\to\infty$ of one random matrix $A_N$. But often one has several matrix ensembles and would like to understand the \enquote{joint} distribution; 
e.g., in order to use them as building blocks  for more complicated random matrix models.

\section{The problem of non-commutativity}

\begin{remark}
	\begin{enumerate}
		\item Consider two random matrices $A_1^{(N)}$ and $A_2^{(N)}$
%$$A^{(N)}_1 = (a_{ij}^{(1)})_{i,j=1}^N\qquad\text{and}\qquad A^{(N)}_2 = (a_{ij}^{(2)})_{i,j=1}^N,$$
such that their entries are defined on the same probability space. What is now the ``joint'' information about the two matrices which survies in the limit  $N\to\infty$?
		Note that in general our analytical approach breaks down if $A_1$ and $A_2$ do not commute, since then we cannot diagonalize them simultaneously. Hence it makes no sense to talk about a joint eigenvalue distribution of $A_1$ and $A_2$. The notion
		$\mu_{A_1,A_2}$ has no clear analytic meaning.

		What still makes sense in the multivariate case is the combinatorial approach via ``mixed'' moments with respect to the normalized trace $\tr$. Hence we consider the collection of all \emph{mixed moments} 
		$\tr ( A_{i_1}^{(N)} \cdots A_{i_m}^{(N)} )
		$
		in $A_1$ and $A_2$, with $m \in \N$, $i_1, \dots, i_m \in \{1,2\}$, 
		as the joint distribution of $A_1$ and $A_2$ and denote this by $\mu_{A_1,A_2}$.
		We want to understand, in interesting cases, the behavior of $\mu_{A_1,A_2}$ as $N\to\infty$.
		\item In the case of one selfadjoint matrix $A$, the notion $\mu_A$ has two meanings:
		
		\emph{analytic} as $\mu_A = \frac{1}{N} \left( \delta_{\lambda_1} + \cdots + \delta_{\lambda_N} \right)$, which is a probability measure on $\R$;
		
		\emph{combinatorial}, where $\mu_A$ is given by all moments $\tr[A^k]$ for all $k \geq 1$.
		
		These two points of view are the same (at least when we restrict to cases where the proabability measure $\mu$ is determined by its moments) via 
		\[ \tr (A^k) = \int t^k \td \mu_A(t).
		\]
		In the case of two matrices $A_1$, $A_2$ the notion $\mu_{A_1, A_2}$ has only one meaning, namely the collection of all mixed moments
	 $\tr[ A_{i_1} \cdots A_{i_m}]$
		with $m \in \N$ and $i_1, \dots, i_m \in \{1,2\}$.
		If $A_1$ and $A_2$ do not commute then there exists no probability measure $\mu$ on $\R^2$ such that
		\[\tr [ A_{i_1} \cdots A_{i_m} ] = \int t_{i_1} \cdots t_{i_m} \td\mu(t_1,t_2)
		\]
		for all  $m \in \N$ and $i_1, \dots, i_m \in \{1,2\}$.
	\end{enumerate}
\end{remark}

\section{Joint moments of independent GUEs}

We will now consider the simplest case of several random matrices, namely $r$ \GUE s $A_1^{(N)}, \dots, A_r^{(N)}$, which we assume to be independent of each other,
i.e., we have
$ A_i^{(N)} = \frac{1}{\sqrt{N}} ( a_{kl}^{(i)} )_{k,l=1}^N$,
where $i=1, \dots, r$, each $A_i^{(N)}$ is a \GUEN\ and
\[ \left\{ a_{kl}^{(1)}; \, k,l =1, \dots, N \right\},
\dots, 
\left\{ a_{kl}^{(r)}; \, k,l =1, \dots, N \right\}
\]
are independent sets of Gaussian random variables.
Equivalently, this can be characterized by the requirement that all entries of all matrices together form a collection of independent standard Gaussian variables (real on the diagonal, complex otherwise).
Hence we can express this again in terms of the Wick formula \ref{thm:2.8} as
\begin{align*}
\ev{ a_{k_1l_1}^{(i_1)} \cdots a_{k_ml_m}^{(i_m)}  }
&= \sum_{\pi \in \pair(m)} \evpartition{\pi}{   a_{k_1l_1}^{(i_1)}, \dots, a_{k_ml_m}^{(i_m)}   }
\end{align*}
for all $m \in \N$, $1 \leq k_1,l_1, \dots,k_m,l_m \leq N$ and $1 \leq i_1,\dots, i_m\leq r$
and where the second moments are given by
\[ \ev{ a_{pq}^{(i)} a_{kl}^{(j)} } = \delta_{pl} \delta_{qk} \delta_{ij}.
\]
Now we can essentially repeat the calculations from Remark \ref{rem:2.14}
for our mixed moments:
\begin{align*}
\ev{\tr(A_{i_1} \cdots A_{i_m} )}
&= \frac{1}{N^{1+\frac{m}{2}}} \sum_{k_1,\dots,k_m=1}^{N}
	\ev{ a_{k_1k_2}^{(i_1)} a_{k_2k_3}^{(i_2)} \cdots a_{k_mk_1}^{(i_m)} } \\
&= \frac{1}{N^{1+\frac{m}{2}}} \sum_{k_1,\dots,k_m=1}^{N} \sum_{\pi\in\pair(m)}
\evpartition{\pi}{ a_{k_1k_2}^{(i_1)}, a_{k_2k_3}^{(i_2)}, \dots, a_{k_mk_1}^{(i_m)} } \\
&= \frac{1}{N^{1+\frac{m}{2}}} \sum_{k_1,\dots,k_m=1}^{N} \sum_{\pi\in\pair(m)} \prod_{(p,q) \in \pi} \ev{ a_{k_pk_{p+1}}^{(i_p)} a_{k_qk_{q+1}}^{(i_q)}} \\
&= \frac{1}{N^{1+\frac{m}{2}}} \sum_{k_1,\dots,k_m=1}^{N} \sum_{\pi\in\pair(m)} \prod_{(p,q) \in \pi}  \left[ k_p = k_{q+1} \right] \left[ k_q = k_{p+1} \right] \left[ i_p = i_q \right]\\
&= \frac{1}{N^{1+\frac{m}{2}}} \sum_{\substack{\pi\in\pair(m) \\ (p,q) \in \pi \\ i_p = i_q}} \sum_{k_1,\dots,k_m=1}^{N}  \prod_{p}  \left[ k_p = k_{\gamma\pi(p)} \right] \\
&= \frac{1}{N^{1+\frac{m}{2}}} \sum_{\substack{\pi\in\pair(m) \\ (p,q) \in \pi \\ i_p = i_q}}
	N^{\#(\gamma\pi)},
\end{align*}
where $\gamma= (1\ 2 \dots\ m)\in S_m$ 
is the shift by $1$ modulo $m$.
Hence we get the same kind of genus expansion for several \GUE s as for one \GUE. The only difference is, that in our pairings we only allow to connect the same matrices.

\begin{notation}
	For a given $i=(i_1,\dots,i_m)$ with $1\leq i_1,\dots, i_m \leq r$ we say that $\pi\in\pair(m)$ \emph{respects} $i$ if we have $i_p=i_q$ for all $(p,q) \in\pi$. We put
	\[ \pair^{[i]}(m) := \left\{  \pi \in \pair(m) \mid \pi \text{ respects } i \right\}
	\]
	and also
	\[ \ncpair^{[i]}(m) := \left\{  \pi \in \ncpair(m) \mid \pi \text{ respects } i \right\}.
	\]
\end{notation}

\begin{theorem}[Genus expansion of independent \GUE s]
	Let $A_1,\dots, A_r$ be $r$ independent \GUEN. Then we have for all $m\in \N$ and all $i_1, \dots, i_m \in [r]$ that
	\[ \ev{ \tr (A_{i_1} \cdots A_{i_m} ) }
		= \sum_{\pi \in \pair^{[i]}(m) } N^{\#(\gamma\pi) - \frac{m}{2} -1}
	\]
	and thus
	\[ \lim\limits_{N\to\infty} \ev{ \tr (A_{i_1} \cdots A_{i_m} ) }
		= \# \ncpair^{[i]}(m).
	\]
\end{theorem}

\begin{proof}
	The genus expansion follows from our computation above. The limit for $N\to\infty$ follows as for Wigner's semicircle law \ref{thm:2.21} from the fact that
	\[ \lim\limits_{N\to\infty} N^{\#(\gamma\pi) - \frac{m}{2} -1} = \begin{cases}
	1, & \pi \in \ncpair(m), \\
	0, & \pi \not\in \ncpair(m).
	\end{cases}
	\]
	The index tuple $(i_1,\dots,i_m)$ has no say in this limit.
\end{proof}

\section{The concept of free independence}

\begin{remark}
	We would like to find some structure in those limiting moments. We prefer to talk directly about the limit instead of making asymptotic statements. In the case of one \GUE, we had the semicircle $\mu_W$ as a limiting analytic object. Now we do not have an analytic object in the limit, but we can organize our distribution as the limit of moments in a more algebraic way.
\end{remark}

\begin{definition}
	\begin{enumerate}
		\item Let $\A = \C\langle s_1, \dots, s_r\rangle$ be the algebra of polynomials in non-commuting variables $s_1,\dots, s_r$; this means $\A$ is the free unital algebra generated by $s_1,\dots,s_r$ (i.e., there are no non-trivial relations between $s_1,\dots, s_r$ and $\A$ is the linear span of the monomials $s_{i_1} \cdots s_{i_m}$ for $m \geq 0$ and $i_1,\dots, i_m \in [r]$;
		multiplication for monomials is given by concatenation).
		\item On this algebra $\A$ we define a unital linear functional $\varphi \colon\A\to\C$ by $\varphi(1) = 1$ and
		\[ \varphi (s_{i_1} \cdots s_{i_m}) 
		=  \lim\limits_{N\to\infty} \ev{ \tr (A_{i_1} \cdots A_{i_m} ) }
		= \# \ncpair^{[i]}(m).
		\]
		\item We also address $(\A,\varphi)$ as a \emph{non-commutative probability space}
		and $s_1,\dots,s_r\in\A$ as \emph{(non-commutative) random variables}.
		The\emph{ moments} of $s_1,\dots,s_r$ are the $\varphi(s_{i_1}\cdots s_{i_m})$ and the collection of those moments is the \emph{(joint) distribution} of $s_1,\dots,s_r$.
	\end{enumerate} 
\end{definition}

\begin{remark}
	\begin{enumerate}
		\item Note that if we consider only one of the $s_i$, then its distribution is just the collection of Catalan numbers, hence correspond to the semicircle, which we understand quite well.
		\item If we consider all $s_1,\dots,s_r$, then their joint distribution is a large collection of numbers. We claim that the following theorem discovers some important structure in those.
	\end{enumerate}
\end{remark}

\begin{theorem}\label{thm:12.7}
	Let $\A = \C\langle s_1, \dots, s_r\rangle$ and let $\varphi \colon\A\to\C$ be defined by
	$ \varphi (s_{i_1} \cdots s_{i_m}) = \# \ncpair^{[i]}(m)
	$
	as before. Then for all $m\geq 1$, $i_1,\dots,i_m \in[r]$ with
	$ i_1 \neq i_2, \;  i_2 \neq i_3, \; \dots, \; i_{m-1} \neq i_m
	$
	and all polynomials $p_1,\dots, p_m$ in one variable such that
	$ \varphi \left( p_k(s_{i_k}) \right) = 0$
	we have:
	\[ \varphi \left( p_1(s_{i_1}) p_2(s_{i_2})\cdots p_m(s_{i_m}) \right) = 0.
	\]
	In words: the alternating product of centered variables is centered.
\end{theorem}

We say that $s_1,\dots,s_r$ are \emph{free (or freely independent)}; in terms of the independent \GUE\ random matrices, we say that $A^{(N)}_1,\dots,A^{(N)}_r$ are \emph{asymtotially free}. Those notions and the results above are all due to Dan Voiculescu.

\begin{proof}
	It suffices to prove the statement for polynomials of the form
	\[ p_k(s_{i_k}) = s_{i_k}^{p_k} - \varphi \left( s_{i_k}^{p_k}   \right)
	\]
	for any power $p_k$, since general polynomials can be written as linear combinations of those. The general statement then follows by linearity.
	So we have to prove that
	\[ \varphi \left[  \left( s_{i_1}^{p_1} - \varphi \left( s_{i_1}^{p_1}  \right) \right)
	\cdots \left(  s_{i_m}^{p_m} - \varphi \left( s_{i_m}^{p_m} \right) \right) \right]
	= 0.
	\]
	We have
	\begin{align*}
	\varphi \left[  \left( s_{i_1}^{p_1} - \varphi \left( s_{i_1}^{p_1}  \right) \right)
	\cdots \left(  s_{i_m}^{p_m} - \varphi \left( s_{i_m}^{p_m} \right) \right) \right]
	&= \sum_{ M \subset [m] } (-1)^{\abs{M}} \prod_{j\in M}  \varphi \left( s_{i_j}^{p_j}  \right)
	 \varphi \left( \, \prod_{j \not\in M} s_{i_j}^{p_j}  \right)
	\end{align*}
	with
	\[ \varphi \left( s_{i_j}^{p_j}  \right)
	=  \varphi \left( s_{i_j} \cdots s_{i_j}  \right)
	= \# \ncpair(p_j)
	\]
	and
	\begin{align*}
	\varphi \left( \, \prod_{j \not\in M} s_{i_j}^{p_j}  \right)
	&= \# \ncpair^{[\text{respects indices}]} \left( \, \sum_{j\not\in M} p_j \right).
	\end{align*}
	Let us put 
	\begin{align*}
	I_1
	&= \{ 1, \dots, p_1 \} \\
	I_2
	&= \{ p_1+1, \dots, p_1+p_2 \} \\
	&\hspace{0.25cm}\vdots \\
	I_m
	&= \{ p_1+p_2 + \cdots + p_{m-1}+1, \dots, p_1+p_2 + \cdots + p_m \} 
	\end{align*}
	and $I = I_1 \cup I_2 \cup \cdots \cup I_m$. 
	Denote
	\[ [\dots] = [i_1,\dots, i_1, i_2, \dots, i_2, \dots, i_m, \dots, i_m].
	\]
	Then
	\begin{align*}
	\prod_{j\in M}  \varphi \left( s_{i_j}^{p_j}  \right)
	\varphi \left( \, \prod_{j \not\in M} s_{i_j}^{p_j}  \right)
	&= \# \{ \pi \in \ncpair^{[\dots]}(I) \mid \text{for all $j \in M$ all elements} \\
	&\qquad \qquad \text{in $I_j$ are only paired amongst each other} \}
	\end{align*}
	Let us denote
	\begin{align*}
	\ncpair^{[\dots]}(I:j)
	&:= \{ \pi \in \ncpair^{[\dots]}(I) \mid \text{elements in $I_j$ are only paired amongst each other} \}.
	\end{align*}
	Then, by the inclusion-exclusion formula, 
	\begin{align*}
	\varphi \left[  \left( s_{i_1}^{p_1} - \varphi \left( s_{i_1}^{p_1}  \right) \right)
	\cdots \left(  s_{i_m}^{p_m} - \varphi \left( s_{i_m}^{p_m} \right) \right) \right]
	&= \sum_{M\subset[m]} (-1)^{\abs{M}} \cdot \# \left(  \bigcap_{j \in M} \ncpair^{[\dots]}(I:j) \right) \\
	&= \# \left( \ncpair^{[\dots]}(I) \bs \bigcup_j \ncpair^{[\dots]}(I:j)  \right).
	\end{align*}
	These are $\pi \in \ncpair^{[\dots]}(I)$ such that at least one element of each interval $I_j$ is paired with an element from another interval $I_k$.
	Since
	$ i_1 \neq i_2, \;  i_2 \neq i_3, \; \dots, \; i_{m-1} \neq i_m
	$
	we cannot connect neighboring intervals and each interval must be connected to another interval in a non-crossing way. But there is no such $\pi$, hence
	\[\varphi \left[  \left( s_{i_1}^{p_1} - \varphi \left( s_{i_1}^{p_1}  \right) \right)
	\cdots \left(  s_{i_m}^{p_m} - \varphi \left( s_{i_m}^{p_m} \right) \right) \right]
	=\# \left( \ncpair^{[\cdots]}(I) \bs \bigcup_j \ncpair^{[\cdots]}(I:j)  \right) = 0,
	\]
	as claimed.
\end{proof}

\begin{remark}
	\begin{enumerate}
		\item Note that in Theorem \ref{thm:12.7} we have traded the explicit description of our moments for implicit relations between the moments.		
		\item For example, the simplest relations from Theorem \ref{thm:12.7} are
		\[ \varphi \left( 
		[  s_{i}^{p} - \varphi ( s_{i}^{p}  ) 1]
		[  s_{j}^{q} - \varphi ( s_{j}^{q}  ) 1] 
		\right) = 0,
		\]
		for $i \neq j$,	which can be reformulated to
		\begin{align*}
		\varphi (s_i^p s_j^q) - \varphi(s_i^p 1) \varphi(s_j^q) - \varphi(s_i^p) \varphi(s_j^q 1) 
		+ \varphi(s_i^p)\varphi(s_j^q) \varphi(1)
		&= 0,
		\end{align*}
		i.e.,
		\[ \varphi (s_i^p s_j^q) = \varphi(s_i^p)\varphi(s_j^q).
		\]
		Those relations are quickly getting more complicated. For example,
		\begin{align*}
		\varphi \left[
		(s_1^{p_1}- \varphi(s_1^{p_1})1)
		(s_2^{q_1}- \varphi(s_2^{q_1})1)
		(s_1^{p_2}- \varphi(s_1^{p_2})1)
		(s_2^{q_2}- \varphi(s_2^{q_2})1)
		\right]
		= 0
		\end{align*}
		leads to
		\begin{align*}
		\varphi \left( s_1^{p_1}  s_2^{q_1}  s_1^{p_2}  s_2^{q_2} \right)
		&= \varphi\left( s_1^{p_1+p_2} \right) \varphi\left( s_2^{q_1} \right) \varphi\left( s_2^{q_2} \right) \\
		&\quad+\varphi\left( s_1^{p_1} \right) \varphi\left( s_1^{p_2} \right) \varphi\left( s_2^{q_1+q_2} \right) \\
		&\quad
		-\varphi\left( s_1^{p_1} \right) \varphi\left( s_2^{q_1} \right)\varphi\left( s_1^{p_2} \right) \varphi\left( s_2^{q_2} \right).
		\end{align*}
These relations are to be considered as non-commutative versions for the factoriziation rules of expectations of independent random variables.
		\item One might ask: What is it good for to find those relations between the moments, if we know the moments in a more explicit form anyhow?
		
		Answer: Those relations occur in many more situations. For example, independent Wishart matrices satisfy the same relations, even though the explicit form of their mixed moments is quite different from the \GUE\ case.
		
Furthermore, we can control what happens with these relations much better than with the explicit moments if we deform our setting or construct new random matrices out of other ones.

		Not to mention that those relations also show up in very different corners of mathematics (like operator algebras).
		
		To make a long story short: Those relations from Theorem \ref{thm:12.7} are really worth being investigated further, not just in a random matrix context, but also for its own sake.
		This is the topic of a course on \textit{Free Probability Theory}, which can, for example, be found here:\newline
\url{rolandspeicher.files.wordpress.com/2019/08/free-probability.pdf}

	\end{enumerate}
\end{remark}

\chapter{Exercises}

\section{Assignment 1}

\begin{exercise}\label{exercise:1}
Make yourself familiar with MATLAB (or any other programming language which allows you to generate random matrices and calculate eigenvalues).
In particular, you should try to generate random matrices and calculate and plot their eigenvalues. 
\end{exercise}

\begin{exercise}\label{exercise:2}
In this exercise we want to derive the explicit formula for the Catalan numbers. We define numbers $c_k$ by the recursion
\begin{align}
c_k=\sum\limits_{l=0}^{k-1} c_lc_{k-l-1}\label{recursion}
\end{align}
for $k>0$, with the initial data $c_0=1$.
\begin{enumerate}
\item
Show that the numbers $c_k$ are uniquely defined by the recursion \eqref{recursion} and its initial data.
\item
Consider the (generating) function
\begin{align*}
f(z)=\sum\limits_{k=0}^\infty c_kz^k
\end{align*}
and show that the recursion \eqref{recursion} implies the relation
\begin{align*}
f(z)=1+zf(z)^2.
\end{align*}
\item
Show hat $f$ is a power series representation for
\begin{align*}
z\mapsto \frac{1-\sqrt{1-4z}}{2z}.
\end{align*}
\textit{Note: You may use the fact that the formal power series $f$, defined in (2), has a positive radius of convergence.}
\item
Conclude that
\begin{align*}
c_k=C_k=\frac{1}{k+1}\binom{2k}{k}.
\end{align*}
\end{enumerate}
\end{exercise}

\begin{exercise}\label{exercise:3}
Consider the semicircular distribution, given by the density function
\begin{align}
\frac{1}{2\pi}\sqrt{4-x^2}\mathbbm{1}_{[-2,2]},\label{density}
\end{align}
where $\mathbbm{1}_{[-2,2]}$ denotes the indicator function of the interval $[-2,2]$. Show that \eqref{density} indeed defines a probability measure, i.e.
\begin{align*}
\frac{1}{2\pi}\int\limits_{-2}^2\sqrt{4-x^2}\,\mathrm{d}x=1.
\end{align*}
Moreover show that the even moments of the measure are given by the Catalan numbers and the odd ones vanish, i.e.
\begin{align*}
\frac{1}{2\pi}\int\limits_{-2}^2 x^n \sqrt{4-x^2}\,\mathrm{d}x=\begin{cases}
0   & n\textrm{ is odd}\\
C_k & n=2k
\end{cases}.
\end{align*}
\end{exercise}

\section{Assignment 2}

\begin{exercise}\label{exercise:4}
Using your favorite programing language or computer algebra system, generate $N\times N$ random matrices for $N=3,9,100$. Produce a plot of the eigenvalue distribution for a single random matrix and as well as a plot for the average over a reasonable number of matrices of given size. The entries should be independent and identically distributed (i.i.d.) according to
\begin{enumerate}
\item
the Bernoulli distribution $\frac{1}{2}(\delta_{-1}+\delta_{1})$, where $\delta_x$ denotes the Dirac measure with atom $x$.
\item
the normal distribution.
\end{enumerate}
\end{exercise}

\begin{exercise}\label{exercise:5}
Prove Proposition \ref{prop:2.2}, i.e. compute the moments of a standart Gaussian random variable:
\begin{align*}
\frac{1}{\sqrt{2 \pi}}\int\limits_{-\infty}^\infty t^ne^{-\frac{t^2}{2}}\mathrm{d}t=\begin{cases} 0 & \textrm{$n$ odd},\\
(n-1)!! & \textrm{$n$ even}.\end{cases}
\end{align*}
\end{exercise}

\begin{exercise}\label{exercise:6}
Let $Z,Z_{1},Z_{2},\dots,Z_{n}$ be independent standard complex Gaussian random variables with mean $0$ and $\mathbb{E}[|Z_{i}|]=1$ for $i=1,\dots,n$.
\begin{enumerate}
\item
Show that
\begin{align*}
\mathbb{E}[Z_{i_1},\dots,Z_{i_{r}}\bar Z_{j_{1}}\dots\bar Z_{j_{r}}]=\#\{\sigma\in S_r\colon i_k=j_{\sigma(k)}\textrm{ for }k=1,\dots,r\}.
\end{align*}
\item
Show that
\begin{align*}
\mathbb{E}[Z^n\bar Z^m]=\begin{cases} 0 & m\neq n,\\
n! & m=n.
\end{cases}
\end{align*}
\end{enumerate}
\end{exercise}

\begin{exercise}\label{exercise:7}
Let $A=(a_{ij})_{i,j=1}^N$ be a Gaussian (\GUEN) random matrix with entries $a_{ii}=x_{ii}$ and $a_{ij}=x_{ij}+\sqrt{-1}y_{ij}$, i.e. the $x_{ij},y_{ij}$ are real i.i.d.~Gaussian random variables, normalized such that $\mathbb{E}[|a_{ij}^2|]={1}/{N}$. Consider the $N^2$ random vector
\begin{align*}
(x_{11},\dots,x_{NN},x_{12},\dots,x_{1N},\dots,x_{N-1N},y_{12},\dots,y_{1N},\dots,y_{N-1N})
\end{align*}
and show that it has the density
\begin{align*}
C\exp(-N\frac{\mathrm{Tr}(A^2)}{2})\textrm{d}A,
\end{align*}
where $C$ is a constant and
\begin{align*}
\textrm{d}A=\prod_{i=1}^{N}\textrm{d}x_{ii}\prod_{i<y}{\textrm{d}x_{ij}\textrm{d}y_{ij}}.
\end{align*}
\end{exercise}

\section{Assignment 3}

\begin{exercise}\label{exercise:8}
Produce histograms for various random matrix ensembles.
\begin{enumerate}
\item
Produce histograms for the averaged situation: average over 1000 realizations
for the eigenvalue distribution of a an $N\times N$ Gaussian random matrix
(or alternatively $\pm 1$ entries) and compare this with one random realization
for $N=5, 50,500,1000$.

\item 
Check via histograms that Wigner's semicircle law is insensitive to the common distribution of the entries as long as those are independent; compare
typical realisations for $N=100$ and $N=3000$ for different distributions of the entries: $\pm 1$, 
Gaussian, uniform distribution on the interval $[-1,+1]$.

\item
Check what happens when we give up the constraint that the the entries are
centered; take for example the uniform distribution on $[0,2]$.

\item
Check whether the semicircle law is sensitive to what happens on the diagonal of the matrix. Choose one distribution (e.g. Gaussian) for the off-diagonal elements and another distribution for the elements on the diagonal
(extreme case: put the diagonal equal to zero).

\item
Try to see what happens when we take a distribution for the entries which does not have finite second moment; for example, the Cauchy distribution.
\end{enumerate}
\end{exercise}

\begin{exercise}\label{exercise:9}
In the proof of Theorem \ref{thm:3.9} we have seen that the $m$-th moment of a Wigner matrix is asymptotically counted by the number of partitions $\sigma\in \PP(m)$, for which the corresponding graph $\G_\sigma$ is a tree; then the corresponding walk $i_1\to i_2\to\cdots\to i_m\to i_1$ (where $\text{ker } i=\sigma$) uses each edge exactly twice, in opposite directions. Assign to such a $\sigma$ a pairing by opening/closing a pair when an edge is used for the first/second time in the corresponding walk.
\begin{enumerate}
\item
 Show that this map gives a bijection between the $\sigma\in \PP(m)$ for which $\G_\sigma$ is a tree and non-crossing pairings $\pi\in NC_2(m)$.
\item
Is there a relation between $\sigma$ and $\gamma\pi$, under this bijection?

\end{enumerate}
\end{exercise}

\begin{exercise}\label{exercise:10}
For a probability measure $\mu$ on $\R$ we define its Stieltjes transform $S_\mu$ by
$$S_\mu(z):=\int_\R \frac 1{t-z} d\mu(t)$$
for all $z\in\C^+:=\{z\in\C\mid \Im(z)>0\}$. Show the following for a Stieltjes transform $S=S_\mu$.
\begin{enumerate}
\item
$S:\C^+\to C^+$.

\item
$S$ is analytic on $\C^+$.

\item
We have
$$\lim_{y\to\infty} iy S(iy)=-1\qquad\text{and}\qquad
\sup_{y>0,x\in\R} y\vert S(x+iy)\vert =1.$$

\end{enumerate}
\end{exercise}

\section{Assignment 4}

\begin{exercise}\phantomsection{\label{exercise:11}}
\begin{enumerate}
\item
Let $\nu$ be the Cauchy distribution, i.e., 
$$d\nu(t)=\frac 1\pi \frac 1{1+t^2}dt.$$
Show that the Stieltjes transform of $\nu$ is given by
$$S(z)=\frac 1{-i-z}\qquad \text{for $z\in \C^+$}.$$
(Note that this formula is not valid in $\C^-$.)\\
Recover from this the Cauchy distribution via the Stieltjes inversion formula.

\item
Let  $A$ be a selfadjoint matrix in $M_N(\C)$ and consider its spectral distribution $\mu_A = \frac{1}{N} \sum_{i=1}^N \delta_{\lambda_i}$, where $\lambda_1, \dots, \lambda_N$ are the eigenvalues (counted with multiplicity) of $A$.   Prove that for any $z\in \mathbb{C}^+$ the Stieltjes transform $S_{\mu_A}$ of $\mu_A$  is given  by
\[ S_{\mu_A}(z) =  \tr[(A-zI)^{-1}].
\]

\end{enumerate}
\end{exercise}

\begin{exercise}\label{exercise:12}
Let $(\mu_N)_{N\in\N}$ be a sequence of probability measures on $\mathbb{R}$ which converges vaguely to $\mu$. Assume that $\mu$ is also a probablity measure. Show the following.
\begin{enumerate}
\item
The sequence $(\mu_N)_{N\in\N}$ is tight, i.e., for each $\varepsilon >0$ there is a compact interval $I=[-R,R]$ such that $\mu_N(\R\backslash I)\leq \varepsilon$ for all $N\in\N$.
\item
$\mu_N$ converges to $\mu$ also weakly.
\end{enumerate}
\end{exercise}

\begin{exercise}\label{exercise:13}
The problems with being determined by moments and whether convergence in moments implies weak convergence are mainly coming from the behaviour of our probability measures around infinity. If we restrict everything to a compact interval, then the main statements follow quite easily by relying on the 
Weierstrass theorem for approximating continuous functions by polynomials.
In the following you should not use Theorem \ref{thm:4.12}.

In the following let $I=[-R,R]$ be a fixed compact interval in $\R$.
\begin{enumerate}

\item Assume that $\mu$ is a probability measure on $\R$ which has its support in $I$ (i.e., $\mu(I)=1$). Show that all moments of $\mu$ are finite and that $\mu$ is determined by its moments (among all probability measures on $\R$).

\item Consider in addition a sequence of probability measures $\mu_N$, such that 
$\mu_N(I)=1$ for all $N$. Show that the following are equivalent:
\begin{itemize}
\item
$\mu_N$ converges weakly to $\mu$;
\item
the moments of $\mu_N$ converge to the corresponding moments of $\mu$.
\end{itemize}

\end{enumerate}
\end{exercise}

\section{Assignment 5}

In this assignment we want to investigate the behaviour of the limiting eigenvalue distribution of matrices under certain perturbations. In order to do so, it is crucial to deal with different kinds of matrix norms. We recall the most important ones for the following exercises. Let $A\in M_N(\C)$, then we define the following norms.
\begin{itemize}
\item
The spectral norm (or operator norm): 
\begin{align*}
\|A\|=\max\{\sqrt{\lambda}\colon \textrm{$\lambda$ is an eigenvalue of $AA^\ast$}\}.
\end{align*}
Some of its important properties are:
\begin{enumerate}
\item[(i)]
 It is submultiplicative, i.e. for $A, B\in M_N(\C)$ one has
\begin{align*}
\|AB\|\leq \|A\|\cdot \|B\|.
\end{align*}
\item[(ii)]
It is also given as the operator norm
\begin{align*}
\|A\|=\sup\limits_{\substack{x\in\C^N \\ x\neq 0}}\frac{\|Ax\|_2}{\|x\|_2},
\end{align*}
where $\Vert x\Vert_2$ is here the Euclidean 2-norm of the vector $x\in\C^N$.
\end{enumerate}
\item
The Frobenius (or Hilbert-Schmidt or $L^2$) norm:
\begin{align*}
\|A\|_2=\bigl(\Tr(A^* A)\bigr)^{1/2}=\sqrt{\sum_{1\leq i,j\leq N}|a_{ij}|^2}
\end{align*}
\end{itemize}

\begin{exercise}\label{exercise:14}
In this exercise we will prove some useful facts about these norms, which you will  have to use in the next exercise when adressing the problem of perturbed random matrices.

Prove the following properties of the matrix norms.
\begin{enumerate}
\item
For $A,B\in M_N(\C)$ we have
$|\Tr(AB)|\leq\|A\|_2\cdot \|B\|_2$.
\item
Let $A\in M_N(\C)$ be positive and $B\in M_N(\C)$ arbitrary. Prove that 
\begin{align*}
|\Tr(AB)|\leq \|B\|\Tr(A).
\end{align*}
($A\in M_N(\C)$ is positive if there is a matrix $C\in M_N(\C)$ such that $A=C^\ast C$; this is equivalent to the fact that $A$ is selfadjoint and all the eigenvalues of $A$ are positive.) 
\item
Let $A\in M_N(\C)$ be normal, i.e. $AA^*=A^*A$, and $B\in M_N(\C)$ arbitrary. Prove that
\end{enumerate}
\begin{align*}
\max\{\|AB\|_2,\|BA\|_2\}\leq \|B\|_2\cdot \|A\|
\end{align*}
\textit{Hint: normal matrices are unitarily diagonalizable.}
\end{exercise}

\begin{exercise}\label{exercise:15}
In this main exercise we want to investigate the behaviour of the eigenvalue distribution of selfadjoint matrices with respect to certain types of perturbations. 
\begin{enumerate}
\item
Let $A\in M_N(\C)$ be selfadjoint, $z\in\C^+$ and $R_A(z)=(A-zI)^{-1}$. Prove that
\begin{align*}
 \|R_A(z)\|\leq\frac{1}{\textrm{Im}(z)}
\end{align*}
and that $R_A(z)$ is normal.
\item
First we study a general perturbation by a selfadjoint matrix.

Let, for any $N\in\N$, $X_N=(X_{ij})_{i,j=1}^N$ and  $Y_N=(Y_{ij})_{i,j=1}^N$  be selfadjoint matrices in $M_N(\C)$ and define $\tilde{X}_N=X_N+Y_N$. Show that
\begin{align*}
\bigg|\tr(R_{\frac{1}{\sqrt{N}}X_N}(z))-\tr(R_{\frac{1}{\sqrt{N}}\tilde{X}_N})(z)\bigg|\leq\frac{1}{(\textrm{Im}(z))^2}\sqrt{\frac{\tr(Y_N^2)}{N}}
\end{align*}
\item
In this part we want to show that the diagonal of a matrix does not contribute to the eigenvalue distribution in the large $N$ limit, if it is not too ill-behaved.

As before, consider a selfadjoint matrix $X_N=(X_{ij})_{i,j=1}^N\in M_N(\C)$; let $X_N^D=\textrm{diag}(X_{11},\dots,X_{NN})$ be the diagonal part of $X_N$ and
$X_N^{(0)}=X_N-X_N^D$ the part of $X_N$ with zero diagonal. Assume that $\|X_N^D\|_2\leq N$ for all $N\in\N$. Show that
\begin{align*}
\bigg|\tr(R_{\frac{1}{\sqrt{N}}X_N}(z))-\tr(R_{\frac{1}{\sqrt{N}}X_N^{(0)}})(z)\bigg|\to 0,\quad\textrm{as }N\to \infty.
\end{align*}

\end{enumerate}

\end{exercise}

\section{Assignment 6}

\begin{exercise}\label{exercise:16}
We will address here concentration estimates for the law of large numbers, and see that control of higher moments allows stronger estimates. 
Let $X_i$ be a sequence of independent and identically distributed random variables with common mean $\mu=E[X_i]$. We put
$$S_n:=\frac 1n \sum_{i=1}^n X_i.$$
\begin{enumerate}
\item
Assume that the variance $\var{X_i}$ is finite. Prove that we have then the weak law of large numbers, i.e., convergence in probability of $S_n$ to the mean: for any $\epsilon>0$
$$\P(\omega\mid \vert S_n(\omega)-\mu\vert \geq\epsilon)\to 0, \qquad \text{for $n\to\infty$}.$$
\item
Assume that the fourth moment of the $X_i$ is finite, $\ev{X_i^4}<\infty$. Show that we have then the strong law of large numbers, i.e., 
$$S_n\to\mu,\qquad\text{almost surely}.$$
(Recall that by Borel--Cantelli it suffices for almost sure convergence to show
that
$$\sum_{n=1}^\infty \P(\omega\mid \vert S_n(\omega)-\mu\vert \geq\epsilon)<\infty.)
$$
One should also note that our assumptions for the weak and strong law of large numbers are far from optimal. Even the existence of the variance is not needed for them, but for proofs of such general versions one needs other tools then our simple consequences of Cheyshev's inequality.
\end{enumerate}
\end{exercise}

\begin{exercise}\label{exercise:17}
Let $X_N=\frac 1{\sqrt N} (x_{ij})_{i,j=1}^N$, where the $x_{ij}$ are all (without symmetry condition) independent and identically distributed with standard complex Gaussian distribution. We denote the adjoint (i.e., congugate transpose) of $X_N$ by $X_N^*$. 
\begin{enumerate}
\item
By following the ideas from our proof of Wigner's semicircle law for the \GUE\ in Chapter \ref{ch:3} show the following: the averaged trace of any $*$-moment in $X_N$ and $X_N^*$, i.e., 
$$E[\tr(X_N^{p(1)}\cdots X_N^{p(m)})]\qquad
\text{where $p(1),\dots,p(m)\in \{1,*\}$}$$
is for $N\to\infty$ given by the number of non-crossing pairings $\pi$ in $NC_2(m)$ which satisfy the additional requirement that each block of $\pi$ connects an $X$ with an $X^*$.
\item
Use the result from part (1) to show that the asymptotic averaged eigenvalue distribution of $W_N:=X_NX_N^*$ is the same as the square of the semicircle distribution, i.e. the distribution of $Y^2$ if $Y$ has a semicircular distribution.
\item
Calculate the explicit form of the asymptotic averaged eigenvalue distribution of $W_N$.
\item
Again, the convergence is here also in probability and almost surely. Produce histograms of samples of the random matrix $W_N$ for large $N$ and compare it with the analytic result from (3).
\end{enumerate}
\end{exercise}

\begin{exercise}\label{exercise:18}

We consider now random matrices $W_N=X_NX_N^*$ as in Exercise \ref{exercise:17}, but now we allow the $X_N$ to be rectangular matrices, i.e., of the form
$$X_N=\frac 1{\sqrt p} (x_{ij})_{1\leq i\leq N\atop
1\leq j\leq p},$$
where again all $x_{ij}$ are independent and identically distributed. We allow now real or complex entries. (In case the entries are real, $X_N^*$ is of course just the transpose $X_N^T$.) Such matrices are called \emph{Wishart matrices}.
Note that we can now not multiply $X_N$ and $X_N^*$ in arbitrary order, but alternating products as in $W_N$ make sense.
\begin{enumerate}
\item
What is the general relation between the eigenvalues of $X_NX_N^*$ and the eigenvalues of $X_N^*X_N$. Note that the first is an $N\times N$ matrix, whereas the second is a $p\times p$ matrix.

\item
Produce histograms for the eigenvalues of $W_N:=X_NX_N^*$ for $N=50$, $p=100$ as well as for $N=500$, $p=1000$, for different distributions of the $x_{ij}$;
\begin{itemize}
\item
standard real Gaussian random variables
\item
standard complex Gaussian random variables
\item
Bernoulli random variables, i.e., $x_{ij}$ takes on values $+1$ and $-1$, each with probability $1/2$.
\end{itemize}

\item
Compare your histograms with the density, for $c=0.5=N/p$, of the \emph{Marchenko--Pastur} distribution which is given by
$$\frac {\sqrt{(\lambda^+-x)(x-\lambda^-)}}{2\pi c x} \mathbbm{1}_{[\lambda^-,\lambda^+]}(x),\qquad \text{where}\qquad \lambda^{\pm}:=\left( 1\pm\sqrt c\right)^2.$$

\end{enumerate}
\end{exercise}

\section{Assignment 7}

\begin{exercise}\label{exercise:19}
Prove -- by adapting the proof for the \GOE\ case and parametrizing a unitary matrix in the form $U=e^{-iH}$, where $H$ is a selfajoint matrix -- Theorem \ref{thm:7.6}: The joint eigenvalue distribution of the eigenvalues of a \GUEN\ is given by a density
$$\hat c_N e^{-\frac N2 (\lambda_1^2+\cdots+\lambda_N^2)}
\prod_{k<l} (\lambda_l-\lambda_k)^2,$$
restricted on $\lambda_1<\lambda_2<\dots <\lambda_N$,

\end{exercise}

\begin{exercise}\label{exercise:20}
In order to get a feeling for the repulsion of the eigenvalues of \GOE\ and \GUE\ compare histograms for the following situations:
\begin{itemize}
\item
the eigenvalues of a \GUEN\ matrix for one realization
\item
the eigenvalues of a \GOEN\ matrix for one realization
\item
$N$ independently chosen realizations of a random variable with semicircular distribution
\end{itemize}
for a few suitable values of $N$ (for example, take $N=50$ or $N=500$).
\end{exercise}

\begin{exercise}\label{exercise:21}
For small values of $N$ (like $N=2,3,4,5,10$) plot the histogram of averaged versions of \GUEN\ and of \GOEN\ and notice the fine structure in the \GUE\ case. In the next assignment we will compare this with the analytic expression for the \GUEN\ density from class.
\end{exercise}

\section{Assignment 8}

\begin{exercise}\label{exercise:22}
In this exercise we define the Hermite polynomials $H_n$ by
$$H_n(x)=(-1)^n e^{x^2/2} \frac{d^n}{dx^n} e^{-x^2/2}$$
and want to show that they are the same polynomials we defined in Definition \ref{def:7.10} and that
they satisfy the recursion relation. So, starting from the above definition show the following.
\begin{enumerate}
\item For any $n\geq 1$, 
$$xH_n(x)=H_{n+1}(x)+n H_{n-1}(x).$$

\item $H_n$ is a monic polynomial of degree $n$. Furthermore, it is an even function if $n$ is even and an odd function if $n$ is odd.

\item The $H_n$ are orthogonal with respect to the Gaussian measure
$$d\gamma(x)=(2\pi)^{-1/2} e^{-x^2/2}dx.$$
More precisely, show the following:
$$\int_\R H_n(x)H_m(x) d\gamma(x)=\delta_{nm} n!$$

\end{enumerate}
\end{exercise}

\begin{exercise}\label{exercise:23}
Produce histograms for the averaged eigenvalue distribution of a \GUEN\ and compare this with the exact analytic density from Theorem \ref{thm:7.21}.
\begin{enumerate}

\item Rewrite first the averaged eigenvalue density 
$$p_N(\mu)=\frac 1N K_N(\mu,\mu)=\frac 1{\sqrt{2\pi}} \frac 1N\sum_{k=0}^{N-1} \frac 1{k!} H_k(\mu)^2 e^{-\mu^2/2}$$
for the unnormalized \GUEN\ to the density $q_N(\lambda)$ for the normalized
\GUEN\ (with second moment normalized to 1).

\item Then average over sufficiently many normalized \GUEN, plot their histograms, and compare this to the analytic density $q_N(\lambda)$. Do this at least for $N=1,2,3,5,10,20,50$. 

\item Check also numerically that $q_N$ converges, for $N\to \infty$, to the semicircle.

\item For comparison, also average over \GOEN\ and over Wigner ensembles with non-Gaussian distribution for the entries, for some small $N$.
\end{enumerate}
\end{exercise}

\begin{exercise}\label{exercise:24}

In this exercise we will approximate the Dyson Brownian motions from Section \ref{section:8.3} by their discretized random walk versions and plot the corresponding walks of the eigenvalues.

\begin{enumerate}
\item
Approximate the Dyson Brownian motion by its discretized random walk version
$$A_N(k):=\sum_{i=1}^k \Delta\cdot A_N^{(i)},\qquad\text{for $1\leq k\leq K$}$$
where $A_N^{(1)},\dots,A_N^{(K)}$ are $K$ independent normalized \GUEN\ random matrices. $\Delta$ is a time increment. Generate a random realization of such a Dyson random walk $A_N(k)$ and plot the $N$ eigenvalues $\lambda_1(k),\dots,\lambda_N(k)$ of $A_N(k)$ versus $k$ in the same plot to see the time evolution of the $N$ eigenvalues. Produce at least plots for three different values of $N$.\\
\textit{Hint:} Start with $N=15$, $\Delta=0.01$, $K=1500$, but also play around with those parameters.
\item For the same parameters as in part (i) consider the situation where you replace \GUE\ by \GOE\ and produce corresponding plots. What is the effect of this on the behaviour of the eigenvalues?

\item For the three considered cases of $N$ in parts (1) and 2i), plot also $N$ independent random walks in one plot, i.e.,
$$\tilde\lambda_N(k):=\sum_{i=1}^k \Delta\cdot x^{(i)}, \qquad\text{for $1\leq k\leq K$}$$
where $x^{(1)},\dots,x^{(K)}$ are $K$ independent real standard Gaussian 
random variables.
\end{enumerate}
\end{exercise}
You should get some plots like in Section \ref{section:8.3}.

\section{Assignment 9}

\begin{exercise}\label{exercise:25}
Produce histograms for the Tracy--Widom distribution by plotting $(\lambda_{\max} -2) N^{2/3}$.
\begin{enumerate}
\item Produce histograms for the largest eigenvalue of \GUEN, for $N=50$, $N=100$, $N=200$, with at least 5000 trials in each case.

\item Produce histograms for the largest eigenvalue of \GOEN, for $N=50$, $N=100$, $N=200$, with at least 5000 trials in each case.

\item Consider also real and complex Wigner matrices with non-Gaussian distribution for the entries.

\item Check numerically whether putting the diagonal equal to zero (in \GUE\ or Wigner) has an effect on the statistics of the largest eigenvalue.

\item Bonus: Take a situation where we do not have convergence to semicircle, e.g., Wigner matrices with Cauchy distribution for the entries. Is there a reasonable guess for the asymptotics of the distribution of the largest eigenvalue?

\item Superbonus: Compare the situation of repelling eigenvalues with ``independent'' eigenvalues. Produce $N$ independent copies $x_1,\dots,x_N$ of variables distributed according to the semicircle distribution and take then
the maximal value $x_{\max}$ of these. Produce a histogram of the statistics of $x_{\max}$. Is there a limit of this for $N\to\infty$; how does one have to scale with $N$?
\end{enumerate}
\end{exercise}

\begin{exercise}\label{exercise:26}
Prove the estimate for the Catalan numbers
$$C_k\leq \frac {4^k}{k^{3/2} \sqrt\pi}\qquad \forall k\in \N.$$
Show that this gives the right asymptotics, i.e., prove that
$$\lim_{k\to\infty} \frac {4^k} { k^{3/2}C_k} =\sqrt\pi.$$
\end{exercise}

\begin{exercise}\label{exercise:27}
Let $H_n(x)$ be the Hermite polynomials. The Christoffel-Darboux identity says that
$$\sum_{k=0}^{n-1} \frac {H_k(x)H_k(y)}{k!}=\frac {H_n(x)H_{n-1}(y)-H_{n-1}(x)H_n(y)}{(x-y)\  (n-1)! }.$$
\begin{enumerate}
\item Check this identity for $n=1$ and $n=2$.

\item Prove the identity for general $n$.
\end{enumerate}
\end{exercise}

\section{Assignment 10}

\begin{exercise}\label{exercise:28}
Work out the details for the ``almost sure'' part of Corollary \ref{cor:9.6}, i.e., prove that
almost surely the largest eigenvalue of \GUEN\ converges, for $N\to\infty$, to 2.
\end{exercise}

\begin{exercise}\label{exercise:29}
Consider the rescaled Hermite functions
$$\tilde\Psi(x):=N^{1/{12}}\Psi_N(2\sqrt N + x N^{- 1/6}).$$
\begin{enumerate}
\item Check numerically that the rescaled Hermite functions have a limit for $N\to\infty$ by plotting them for different values of $N$.

\item Familarize yourself with the Airy function. Compare the above plots of $\tilde \Psi_N$ for large $N$ with
a plot of the Airy function.
\end{enumerate}
\textit{Hint:}  MATLAB has an implementation of the Airy function, see
\begin{center}
\url{https://de.mathworks.com/help/symbolic/airy.html}
\end{center}
\end{exercise}

\begin{exercise}\label{exercise:30}
Prove that the Hermite functions satisfy the following differential equations:
$$\Psi_n'(x)=-\frac x2 \Psi_n(x)+\sqrt n \Psi_{n-1}(x)$$
and
 $$\Psi_n''(x)+(n+\frac 12 -\frac {x^2}4)\Psi_n(x)=0.$$
\end{exercise}

\section{Assignment 11}

\begin{exercise}\label{exercise:31}
Read the notes ``Random Matrix Theory and its Innovative Applications'' by A. Edelman and Y. Wang,
\begin{center}
\small\url{http://math.mit.edu/~edelman/publications/random_matrix_theory_innovative.pdf}
\end{center}
and implement its ``Code 7'' for calculating the Tracy--Widom distribution (via solving the Painlev\'e II equation) and compare the output with the histogram for the rescaled largest eigenvalue for the \GUE\ from Exercise \ref{exercise:25}.
You should get a plot like after Theorem \ref{thm:9.10}.
\end{exercise}

\begin{exercise}\label{exercise:32}
For $N=100,1000,5000$ plot in the complex plane the eigenvalues of one $N\times N$ random matrix $\frac{1}{\sqrt{N}}A_N$, where all entries (without symmetry condition) are independent and identically distributed according to the
\begin{enumerate}
\item[(i)]
standard Gaussian distribution;
\item[(ii)]
symmetric Bernoulli distribution;
\item[(iii)]
Cauchy distribution.
\end{enumerate}
\end{exercise}

\chapter{Literature}
\vskip-.5cm
\subsection*{\qquad Books}
\begin{enumerate}
\item
    Gernot Akemann, Jinho Baik, Philippe Di Francesco: \textit{The Oxford Handbook of Random Matrix Theory}, Oxford Handbooks in Mathematics, 2011.
    
\item
    Greg Anderson, Alice Guionnet, Ofer Zeitouni: \textit{An Introduction to Random Matrices}, Cambridge University Press, 2010.
    
\item
    Zhidong Bai, Jack Silverstein: \textit{Spectral Analysis of Large Dimensional Random Matrices}, Springer-Verlag 2010.
    
\item
Patrick Billingsley: \textit{Probability and Measure},  John Wiley \&\ Sons, 3rd edition, 1995.

\item
St\'ephane Boucheron, G\'abor Lugosi, Pascal Massart:\textit{ Concentration inequalities: A nonasymptotic theory of independence},
Oxford University Press, Oxford, 2013. 

 \item 
    Alice Guionnet: \textit{Large Random Matrices: Lectures on Macroscopic Asymptotics}, Springer-Verlag 2009.
    
 \item 
    Madan Lal Mehta: \textit{Random Matices}, Elsevier Academic Press, 3rd edition, 2004.
    
 \item 
    James Mingo, Roland Speicher: \textit{Free Probability and Random Matrices}, Springer-Verlag, 2017.
    
 \item 
    Alexandru Nica, Roland Speicher: \textit{Lectures on the Combinatorics of Free Probability}, Cambridge University Press 2006.
    \vskip-1cm
\subsection*{Lecture Notes and Surveys}

\item
Nathana\"el Berstycki: \textit{Notes on Tracy--Widom Fluctuation Theory}, 2007.
\item
	Charles Bordenave, Djalil Chafa\"i: \textit{Around the circular law},
Probability Surveys 9 (2012) 1-89.
   \item  
    Alan Edelman, Raj Rao: \textit{Random matrix theory},
Acta Numer. 14 (2005), 233-297.
 \item 
    Alan Edelman, Raj Rao: \textit{The polynomial method for random matrices},
Found. Comput. Math. 8 (2008), 649-702.
  \item 
    Todd Kemp: \textit{MATH 247A: Introduction to Random Matrix Theory}, lecture notes, UC San Diego, fall 2013.

\end{enumerate}

\end{document}